\documentclass[10pt,a4paper]{amsart}
\usepackage[english]{babel}

\usepackage{amsmath}
\usepackage{scalerel}
\usepackage{amsfonts}
\usepackage{amssymb}
\usepackage{amsthm}
\usepackage{graphicx}

\usepackage[T1]{fontenc}
\usepackage{newpxtext,newpxmath}

\usepackage[margin=1.0in]{geometry} 
\usepackage{stmaryrd}
\usepackage{rsfso}  

\usepackage{subcaption}

\usepackage{mathtools,thmtools} %
\usepackage{empheq}
\usepackage{enumitem}
\usepackage{esint}
\usepackage{upgreek}
\usepackage[mathscr]{eucal}
\usepackage{dsfont}

\usepackage[colorlinks=true]{hyperref}
\hypersetup{citecolor=blue}

\usepackage{xparse}

\newcommand*{\cal}[1]{{\mathcal{#1}}}

\newcommand*{\Bop}{\cal{L}}
\newcommand*{\Lop}{\cal{L}}

\newcommand*{\sympf}{\sigma}

\newcommand*{\bigcdot}{\bullet}

\DeclareMathOperator{\ad}{\mathrm{ad}}
\DeclareMathOperator{\Ad}{\mathrm{Ad}}
\newcommand*{\lieder}{{\mathcal{L}}}

\DeclareRobustCommand{\longleftmapsto}

\newcommand*{\aff}{{\mathrm{Aff}}}

\newcommand*{\bdd}{b}

\newcommand*{\pushf}{\ast}
\newcommand*{\pullb}{\ast}

\DeclareMathOperator{\diag}{\mathrm{diag}}
\newcommand*{\udl}[1]{\underline{#1}}

\newcommand*{\gain}{h}

\newcommand*{\Conf}{\mathrm{Conf}}

\newcommand*{\hsclass}{\cal{L}^2}

\DeclareMathOperator*{\Moy}{\scalerel*{\#}{\sum}}

\newcommand*{\hor}{\textrm{h}}
\newcommand*{\ver}{\textrm{v}}

\newtheorem{innercustomthm}{Theorem}
\newenvironment{customthm}[1]
  {\renewcommand\theinnercustomthm{#1}\innercustomthm}
  {\endinnercustomthm}

\newtheorem{innercustomassum}{Assumption}
\newenvironment{customassum}[1]
  {\renewcommand\theinnercustomassum{#1}\innercustomassum}
  {\endinnercustomassum}

\numberwithin{equation}{section}

\newtheorem{lemma}{Lemma}[section]

\newtheorem{proposition}[lemma]{Proposition}
\newtheorem{corollary}[lemma]{Corollary}

\theoremstyle{definition}
\newtheorem{example}[lemma]{Example}
\newtheorem{definition}[lemma]{Definition}
\newtheorem{remark}[lemma]{Remark}

\newcommand*{\N}{\mathbf{N}}

\newcommand*{\R}{\mathbf{R}}
\newcommand*{\CC}{\mathbf{C}}

\renewcommand*{\cal}[1]{{\mathcal{#1}}}

\renewcommand*{\frak}[1]{{\mathfrak{#1}}}
\renewcommand*{\rm}[1]{{\mathrm{#1}}}
\renewcommand*{\bf}[1]{{\mathbf{#1}}}
\renewcommand*{\sf}[1]{{\mathsf{#1}}}
\newcommand*{\ovl}[1]{\overline{#1}}

\newcommand*{\one}{\mathbf{1}}

\newcommand*{\mfd}{M}
\newcommand*{\e}{e}	
\newcommand*{\ii}{i} 	
\newcommand*{\dd}{\mathop{}\mathopen{} d}
\newcommand*{\vol}{\mathrm{Vol}}

\newcommand*{\cont}{C}

\newcommand*{\sch}{{\mathcal S}}
\DeclareMathOperator{\ft}{{\mathcal F}}

\DeclarePairedDelimiterX{\brak}[2]{\langle}{\rangle}{#1, #2}
\DeclarePairedDelimiterX{\inp}[2]{(}{)}{#1, #2}
\DeclarePairedDelimiterX{\comm}[2]{[}{]}{#1, #2}
\DeclarePairedDelimiterX{\poiss}[2]{\{}{\}}{#1, #2}
\DeclarePairedDelimiterX{\liebrak}[2]{[}{]}{#1, #2}
\DeclarePairedDelimiter{\parens}{(}{)}
\DeclarePairedDelimiter\abs{\lvert}{\rvert}
\DeclarePairedDelimiter\norm{\lVert}{\rVert}
\DeclarePairedDelimiter\jap{\langle}{\rangle}

\NewDocumentCommand{\bnorm}{sO{}m}{%
  {\IfBooleanTF{#1}
    {\bnormaux{\left|}{\right|}{#3}}
    {\bnormaux{#2|}{#2|}{#3}}}
}
\makeatletter
\newcommand{\bnormaux}[3]{\mathpalette\bnormaux@i{{#1}{#2}{#3}}}
\newcommand{\bnormaux@i}[2]{\bnormaux@ii#1#2}
\newcommand{\bnormaux@ii}[4]{%
  \sbox\z@{$\m@th#1#2#4#3$}%
  \sbox\tw@{$\m@th\|$}%
  \mathopen{\hbox to\wd\tw@{\hss\vrule height \ht\z@ depth \dp\z@ width .3\wd\tw@\hss}}%
  #4
  \mathclose{\hbox to\wd\tw@{\hss\vrule height \ht\z@ depth \dp\z@ width .3\wd\tw@\hss}}%
}
\makeatother

\newcommand*{\comp}{c}

\DeclareMathOperator{\dom}{dom}
\DeclareMathOperator{\tr}{tr}

\DeclareMathOperator{\supp}{supp}
\DeclareMathOperator{\dist}{dist} 
\DeclareMathOperator{\dvg}{div}

\DeclareMathOperator{\ran}{Ran}

\DeclareMathOperator{\quantization}{Op}
\newcommand*{\quantizationw}[1][]{{\quantization}_{#1}^{\mathrm{\scriptscriptstyle W}}\hspace{-0.075em}}
\newcommand*{\Op}[2][]{{\quantization}_{#1}\hspace{-0.075em}\parens*{#2}}

\newcommand*{\Opw}[2][]{{\quantization}_{#1}^{\mathrm{\scriptscriptstyle W}}\hspace{-0.075em}\parens*{#2}}

\DeclareMathOperator{\id}{Id}

\newcommand*{\strongto}[2][]{\xrightarrow[#2]{#1}}
\newcommand*{\weakto}[2][]{\xrightharpoonup[#2]{#1}}

\newcommand*{\set}[2]{\left\{#1:#2\right\}}
\newcommand*{\moyal}{\mathbin{\#}}

\title{{E}gorov's theorem in the {W}eyl--{H}{\"o}rmander calculus}
\author{Antoine Prouff}
\address{\normalsize Department of Mathematics, Purdue University, West Lafayette, IN, USA}
\email{aprouff@purdue.edu}
\date{\today}

\begin{document}

\begin{abstract}
We prove a general version of Egorov's theorem for evolution propagators in the Euclidean space, in the Weyl--Hörmander framework of metrics on the phase space. Mild assumptions on the Hamiltonian allow for a wide range of applications that we describe in the paper, including Schrödinger, wave and transport evolutions. We also quantify an Ehrenfest time and describe the full symbol of the conjugated operator. Our main result is a consequence of a stronger theorem on the propagation of quantum partitions of unity.
\end{abstract}

\maketitle

{
  \hypersetup{linkcolor=black}
  \tableofcontents
}

\Large
\section{Introduction and results}
\normalsize

One of the simplest instances of the correspondence principle is the description of the evolution of coherent states of the quantum harmonic oscillator by Schrödinger in 1926~\cite{Schrodinger:coherentstates}. One can check by an explicit computation that in $\R^d$, the wave function
\begin{equation*}
\varphi(t, x)
	= \dfrac{\e^{- \frac{\ii}{2} t d}}{\pi^{d/4}} \exp\left( - \dfrac{\abs{x - x_t}^2}{2} \right) \e^{- \frac{\ii}{2} \xi_t \cdot x_t + \ii \xi_t \cdot x} ,
		\qquad (t, x) \in \R \times \R^d ,
\end{equation*}
is a solution to the Schrödinger equation
\begin{equation*}
D_t \psi + \tfrac{1}{2} (- \Delta + \abs{x}^2) \psi = 0 ,
	\qquad D_t = \tfrac{1}{\ii} \partial_t ,
\end{equation*}
if (and only if) $t \mapsto (x_t, \xi_t)$ satisfy Newton's second law of classical mechanics with a force field $\vec F(x) = - x = - \nabla (\frac{1}{2} \abs{x}^2)$, i.e.\
\begin{equation*}
\left\{
\begin{aligned}
\dot x_t &= \xi_t \\
\dot \xi_t &= - x_t
\end{aligned}
\right. .
\end{equation*}
We can formulate this observation as saying that the wave packet centered at $x_0$ with momentum $\xi_0$ propagates under the quantum evolution along the classical trajectory of a point mass initially at position and momentum $(x_0, \xi_0)$. This situation is quite exceptional since the correspondence between classical and quantum mechanics appears regardless of any asymptotic regime. Actually this is a very particular instance of \emph{exact} Egorov's theorem due to the fact that the Hamiltonian is quadratic.

The modern treatment of the quantum-classical (or wave-particle) correspondence principle is carried out in the framework of microlocal analysis---see the books~\cite{Taylor:91,Hoermander:V3,Lerner:10,MartinezBook,DS:book,Zworski:book} for a comprehensive account of this theory. A landmark paper is the work of Egorov~\cite{Egorov}, who shows that local canonical transformation $\chi$ (also called symplectomorphism) can be quantized, namely there exists a (microlocally unitary Fourier integral) operator $U_\chi$ such that
\begin{equation} \label{eq:EgorovbyEgorov}
U_\chi \Op{a} U_\chi^\ast
	\approx \Op{a \circ \chi} .
\end{equation}
In this expression, $\Op{a}$ is a quantization procedure, that associates to a symbol $a$ (a smooth function on phase space $\R^d \times \R^d$) an operator $\Op{a}$ acting on $L^2(\R^d)$ (see the next section for more details). This equality, which is usually valid only up to a remainder term (except in very particular situations), states that a transformation at the symbolic level, i.e.\ on the classical-mechanical side, has a counterpart on the quantum-mechanical side. In this paper, we are interested in going the other way around (a family of statements also called Egorov's theorem in the literature): starting from a self-adjoint operator $P = \Op{p}$ and the associated unitary propagator $\e^{- \ii t P}$, we prove a statement of the form
\begin{equation} \label{eq:egorovintro}
\e^{\ii t P} \Op{a} \e^{- \ii t P}
	\approx \Op{a \circ \phi^t} ,
\end{equation}
where $(\phi^t)_{t \in \R}$ is the so-called Hamiltonian flow associated with the symbol $p$, for a general class of Hamiltonians $p$ and symbols $a$. A statement like~\eqref{eq:egorovintro} allows to study the action of the propagator from the knowledge of the classical flow $(\phi^t)_{t \in \R}$ on phase space. Some motivations are listed below.

Since Egorov's paper~\cite{Egorov}, many results of the form~\eqref{eq:EgorovbyEgorov} and~\eqref{eq:egorovintro} have been proved in different contexts, all referred to as ``Egorov theorems". We refer to Taylor's book~\cite[Chapter 6]{Taylor:91} for a microlocal version of~\eqref{eq:egorovintro}. See Zworski's book~\cite[Chapter 11]{Zworski:book} for a comprehensive presentation in the semiclassical setting.

Although the present paper is mainly motivated by applications in control theory, Egorov-type theorems have been applied in a wide range of mathematical situations. Let us mention a few of them (the list is not exhaustive).
\begin{itemize}
\item They have multiple consequences in spectral geometry. They allow for instance to relate singularities of the trace of the half-wave propagator with the length spectrum of compact Riemannian manifolds; results in this direction go back to Duistermaat--Guillemin~\cite{DG:75} and Chazarain~\cite{Chazarain:74}. We refer to the recent paper by Canzani and Galkowski~\cite{CG:23} for an application to improved remainders in the Weyl law.
\item Quantum chaos is perhaps the field where Egorov's theorem has been mostly used. It has proved very powerful to describe the concentration or delocalization properties of Laplace eigenfunctions on compact Riemannian manifolds with a chaotic underlying classical dynamics, through semiclassical defect measures. See for example the paper of Colin de Verdière~\cite{CdV85} on quantum ergodicity. We also refer for instance to the work of Anantharaman~\cite{An:08}, Anantharaman--Nonnenmacher~\cite{AnNon}, Rivière~\cite{Riviere:10duke,Riviere:10nonpos,Riviere:14}, Anantharaman--Rivière~\cite{AR:12} and Dyatlov--Jin--Nonnenmacher~\cite{DJN}. See also the results of Dyatlov--Guillarmou~\cite{DG:14} in non-compact manifolds.
\item In control theory, the idea of using Egorov's theorem to prove observability inequalities was introduced for the wave equation by Laurent and Léautaud~\cite{LL:16}, based on earlier works by Dehman and Lebeau~\cite{DL:09}. We recently applied a related approach to study the observability of the Schrödinger equation in the Euclidean space~\cite{P:23}.
\item In the theory of scattering resonances, examples of application of Egorov's theorem can be found for instance in papers by Nonnenmacher and Zworski~\cite{NonZwo09,NonZwo15}, Dyatlov~\cite{Dyatlov:15} and Dyatlov--Galkowski~\cite{DyatlovGalkowski:17}. See also~\cite[Section 7.1]{DZ:19} for an interesting application to lower bound resolvent estimates in the context of geometric scattering.
\item In the wider context of resolvent estimates, Arnaiz and Rivière~\cite{AR:20} discuss Egorov's theorem in the Gevrey category. For non-self-adjoint versions of Egorov's theorem, we refer e.g.\ to the works of Royer~\cite{Royer:these,Royer:10a,Royer:10b}, Rivière~\cite{Riviere:12} and Léautaud~\cite{Leautaud:23}.
\item Ideas related to Egorov's theorem were also used recently to study fractal Weyl laws for Anosov flows by Faure and Tsujii~\cite{FT:23,FT:arxiv}. We shall mention this again in Section~\ref{subsubsec:X}.
\end{itemize}

The goal of the present paper is to prove a version of Egorov's theorem in the Weyl--Hörmander calculus, to provide a unifying viewpoint of the aforementioned approaches. In a forthcoming work, we plan to give applications to control theory (stabilization, observability), where the Weyl--Hörmander calculus arises as a natural tool to study various wave or Schrödinger-type equations in the Euclidean space. In the context of the Weyl--Hörmander calculus, a study of the quantization of canonical transformation in the continuation of Egorov's work~\cite{Egorov} (results of the form~\eqref{eq:EgorovbyEgorov}) was proposed by Beals~\cite[Section 5]{Beals:74II} and later by Hörmander~\cite[Section 9]{Hor:79}. We also discuss in more detail in Section~\ref{subsec:Bony} the contributions of Bony.

An important question, originating in physics~\cite{Chirikov,Zaslavsky}, it to figure out the time range on which the approximation~\eqref{eq:egorovintro} is true, the so-called Ehrenfest time. In a semiclassical setting, it was proved by Bambusi, Graffi and Paul~\cite{BambusiGraffiPaul}, and Bouzouina--Robert~\cite{BouzouinaRobert} that this time behaves logarithmically with respect to the semiclassical parameter. The inverse Lyapunov exponent of the classical dynamics appears as a constant factor in the Ehrenfest time, which means that the semiclassical approximation is valid all the longer as the underlying classical dynamics is stable~\cite[Section 11.4]{Zworski:book}. We shall recover this in Theorem~\ref{thm:main} below in a more general setting. See~\cite[Section 5.2]{AnNon},~\cite[Theorem 7.1]{Riviere:10duke},~\cite[Section 3.3 and Appendix C]{DG:14} and~\cite[Section 2.2.2 and Appendix A]{DJN} for various occurences of Ehrenfest times in the literature, as well as the recent work~\cite{GHZ:24} on Lindblad evolutions.

The rest of the introduction is structured as follows. In Section~\ref{subsec:classicalandquantum}, we set the geometric framework of this paper, we introduce the Weyl quantization and we define precisely the classical and quantum dynamics studied in the paper. We give some additional notation and conventions in Section~\ref{subsec:notationconventions}. Next we state our main result in Section~\ref{subsec:mainresult}, and comment it in the subsequent Section~\ref{subsec:comments}. Section~\ref{subsec:asadysonseries} is a discussion on the Egorov asymptotic expansion in terms of a Dyson series. The reader can find three concrete examples of applications of the main theorem of this paper (Theorem~\ref{thm:main}) in Section~\ref{subsec:examples}. Then we introduce the basics of the theory of metrics on the phase space in Section~\ref{subsec:metricsonphasespace}, which is not necessary to understand the statement of the main result but is needed as a tool box for the proofs. Then we introduce a result on the propagation of quantum partitions of unity adapted to a metric $g$ in Section~\ref{subsec:partition}. This is in fact the core result of this article (our main result is a consequence of the latter). We finish in Sections~\ref{subsec:Bony} and~\ref{subsec:strategy} with a review of related works by Jean-Michel Bony and describe the strategy of the proof.

\subsection{The classical and quantum dynamics} \label{subsec:classicalandquantum}

\subsubsection{Geometric framework}

We work in the manifold $\mfd = \R^d$, viewed as a finite-dimensional real affine space. The reason we stick to the notation $\mfd$ instead of $\R^d$ is that we will strive to avoid any use of the vector space or Euclidean structure of $\R^d$. Classical mechanics on the configuration space $\mfd$ takes place naturally on the cotangent bundle $T^\star \mfd$. Typical points of $T^\star \mfd$ will be denoted by $\rho = (x, \xi)$. The manifold $T^\star \mfd$ is called the \emph{phase space} of classical mechanics. It is equipped canonically with a symplectic form $\sympf$, defined intrinsically as the exterior derivative of the Liouville {$1$-form}~\cite[Chapter 1]{DS:book}. In coordinates, it reads $\sympf = \dd \xi \wedge \dd x$.

The Hilbert space that we consider is the usual space $L^2(\mfd)$ of square-integrable functions with respect to the Lebesgue measure on $\mfd$.\footnote{Rigorously, we would rather work with the intrinsic $L^2$ space, that is the completion of the space of half-densities with respect to the $L^2$ norm~\cite[Section 9.1]{Zworski:book}. In the present case, this space identifies with the space of square-integrable functions through the (non-canonical) choice of a normalization of the Lebesgue measure.}

\begin{remark}
The analysis of a quantum mechanical system on $\mfd$, from the Weyl--Hörmander framework perspective, essentially consists in considering a good Riemannian metric $g$ on $T^\star \mfd$ which is relevant to the typical scales of the problem. For this reason, it would be quite confusing to rely on a Euclidean structure on $\mfd$, in addition to a Weyl--Hörmander metric. This is the motivation for writing $\mfd$ instead of $\R^d$.
\end{remark}

\subsubsection{Quantization}

The affine structure of $\mfd$ is sufficient to make sense of the Weyl quantization which we use throughout the article. The affine structure is given by the transitive and free action by translation of a finite-dimensional real vector space $V \simeq \R^d$:
\begin{equation} \label{eq:affinestructureM}
\begin{split}
V \times \mfd &\longrightarrow M \\
(v, x) &\longmapsto x + v .
\end{split}
\end{equation}
This induces a natural action on the phase space by
\begin{equation} \label{eq:affinestructurephasespace}
\begin{split}
W \times T^\star \mfd &\longrightarrow T^\star M \\
(\zeta, \rho) &\longmapsto \rho + \zeta ,
\end{split}
\end{equation}
where $W = V \oplus V^\star$. Notice that all the tangent planes $T_\rho (T^\star \mfd)$ are naturally identified with $W$ through this action.

Given $a \in \sch(T^\star \mfd)$, the Weyl quantization of $a$ is the operator $\Opw{a}$ defined by
\begin{equation*}
\left[\Opw{a} u\right](x)
	= \int_{T^\star \mfd} \e^{\ii \xi.(x - y)} a\left(\dfrac{x + y}{2}, \xi\right) u(y) \dd y \dd \xi ,
		\qquad x \in \mfd ,
\end{equation*}
where the measure $\dd y \dd \xi$ on $T^\star \mfd \simeq \R^{2d}$ is a suitable normalization of the Lebesgue measure, or equivalently the symplectic volume,\footnote{Recall that the symplectic volume is the {$d$-fold} exterior product  $\sympf^d = \sympf \wedge \sympf \wedge \cdots \wedge \sympf$, where $d = \dim \mfd$.} so that $\Opw{1} = \id$. This quantization procedure can be extended to tempered distributions $a \in \sch'(T^\star \mfd)$. The operator $\Opw{a} : \sch(\mfd) \to \sch'(\mfd)$ is defined by
\begin{equation} \label{eq:defWeylquantization}
\brak*{\Opw{a} u}{v}_{\sch', \sch(\mfd)}
	:= \brak*{a}{u \ovee \bar v}_{\sch', \sch(T^\star \mfd)} ,
		\qquad \forall u, v \in \sch(\mfd) ,
\end{equation}
where $u_1 \ovee u_2$ is the Wigner transform of $u_1, u_2 \in \sch(\mfd)$:
\begin{equation} \label{eq:defWigner}
(u_1 \ovee u_2)(x, \xi)
	= \int_V \bar u_2\left(x + \dfrac{v}{2}\right) u_1\left(x - \dfrac{v}{2}\right) \e^{\ii \xi.v} \dd v ,
		\qquad \forall (x, \xi) \in T^\star \mfd .
\end{equation}
Here we follow the presentation of~\cite[Section 2.1]{Lerner:10}. Recall that the Weyl quantization provides an unequivocal notion of \emph{full symbol} of an operator (see Proposition~\ref{prop:SchwartzWeylkernel}). Classical properties of the Weyl quantization are recalled in Appendix~\ref{app:Moyal}. When $a_1$ and $a_2$ belong to appropriate classes of functions (for instance $\sch(\mfd)$ or symbol classes), the symbol of the composition of $\Opw{a_1}$ and $\Opw{a_2}$ is the Moyal product $a_1 \moyal a_2$, given by
\begin{equation} \label{eq:formulaMoyalproduct}
\left(a_1 \moyal a_2\right)(\rho)
	= \int_{W \times W} a_1\left(\rho + \dfrac{\zeta_1}{\sqrt{2}}\right) a_2\left(\rho + \dfrac{\zeta_2}{\sqrt{2}}\right) \e^{- \ii \sympf(\zeta_1, \zeta_2)} \dd \zeta_1 \dd \zeta_2
\end{equation}
(possibly to be understood as an oscillatory integral). It satisfies
\begin{equation} \label{eq:Moyalproddef}
\Opw{a_1 \moyal a_2}
	= \Opw{a_1} \Opw{a_2} .
\end{equation}
We discuss further composition of such operators in Section~\ref{subsec:pseudocalcWH}.

\subsubsection{The classical dynamics}

Consider a smooth real-valued function $p$ on $T^\star \mfd$ with temperate growth. The Hamiltonian vector field $H_p$ associated with $p$ is the unique vector field satisfying
\begin{equation} \label{eq:defHp}
\sympf\left(H_p, X\right)
	= - \dd p. X ,
\end{equation}
for all smooth vector field $X$ on $T^\star \mfd$.
The so-called Hamiltonian flow, denoted by $(\phi^t)_{t \in R}$, is generated by $H_p$ and satisfies the o.d.e.\
\begin{equation}
\dfrac{\dd}{\dd t} \phi^t
	= H_p \circ \phi^t ,
\end{equation}
which is Hamilton's equations of classical mechanics. Under mild assumptions on $p$, this flow is well-defined for all times. It is known that the Hamiltonian flow is a group of symplectomorphisms, that is to say $(\phi^t)^\pullb \sympf = \sympf$ for all times. As a consequence, it also preserves the symplectic volume (here the Lebesgue measure).

The Hamiltonian flow acts on Schwartz functions $a_0 \in \sch(T^\star \mfd)$ through $t \mapsto a_0 \circ \phi^t$, which solves the p.d.e.\
\begin{equation} \label{eq:pdeclassicaldyn}
\boxed{\partial_t a = H_p a}
\end{equation}
where $H_p$ is viewed as a differential operator of order~$1$. For any function $a$ on phase space, we define
\begin{equation} \label{eq:defcompositionHp}
\e^{t H_p} a
	:= a \circ \phi^t ,
		\qquad t \in \R .
\end{equation}
The notation is justified by~\eqref{eq:pdeclassicaldyn}. Actually, we can give a stronger meaning to $\e^{t H_p}$ for $L^2$~functions on phase space. Modulo some natural assumptions, we shall see $H_p$ as a skew-symmetric operator on $L^2(T^\star \mfd)$ with domain $\sch(T^\star \mfd)$ (here again, the space $L^2(T^\star \mfd)$ involves the Lebesgue measure, or symplectic volume, on $T^\star \mfd$). As a consequence of Stone's theorem~\cite[VIII.4]{RS:V1}, we show in Proposition~\ref{prop:isometrycalHp} that the family of operators $(\e^{t H_p})_{t \in \R}$, viewed as operators acting on $L^2(T^\star \mfd)$, is a strongly continuous one-parameter unitary group, generated by a skew-adjoint extension of $H_p$, that we call the \emph{classical dynamics}. We will use this unitary group to describe the quantum dynamics defined below.

\subsubsection{The quantum dynamics}

We consider a self-adjoint operator $P$ with domain $\dom P$, acting on $L^2(\mfd)$. We denote by $p$ the Weyl symbol of $P$, so that $P = \Opw{p}$, and assume it is a well-behaved symbol (smooth and with temperate growth). Stone's theorem states that the family of operators $(\e^{- \ii t P})_{t \in \R}$ constructed by functional calculus is a strongly-continuous one-parameter unitary group. For any $u \in L^2(\mfd)$, the map $t \mapsto \e^{- \ii t P} u$, belonging to $\cont^0(\R; L^2(\mfd))$, solves in a weak sense the initial value problem
\begin{equation} \label{eq:Scheq}
D_t \psi
	= - P \psi ,
		\qquad
\psi(0)
	= u ,
\end{equation}
with $D_t = \frac{1}{\ii} \partial_t$. For this reason, we will refer to $\e^{- \ii t P}$ as the Schrödinger propagator associated with $P$. The Schrödinger equation can be studied in the so-called Heisenberg picture of quantum mechanics: letting $A_0$ be a bounded operator on $L^2(\mfd)$ (a ``quantum observable"), we have
\begin{equation*}
\inp*{\e^{- \ii t P} u}{A_0 \e^{- \ii t P} u}_{L^2}
	= \inp*{u}{A(t) u}_{L^2} ,
		\qquad \forall t \in \R ,
\end{equation*}
where the map $t \mapsto A(t) := \e^{\ii t P} A_0 \e^{- \ii t P}$ solves the equation
\begin{equation} \label{eq:Heisenbergpicture}
D_t A
	= \comm*{P}{A} ,
		\qquad
A(0)
	= A_0 .
\end{equation}
Expressing~\eqref{eq:Heisenbergpicture} at the level of symbols, namely writing $A(t) = \Opw{a(t)}$ and using~\eqref{eq:Moyalproddef}, we obtain formally
\begin{equation*}
D_t \Opw{a}
	= \Opw{p \moyal a - a \moyal p} ,
\end{equation*}
where $P = \Opw{p}$. Therefore, $t \mapsto a(t)$ is a solution to the evolution equation
\begin{equation} \label{eq:pdequantumdyn}
\boxed{\partial_t a = \cal{H}_p a}
\end{equation}
where $\cal{H}_p$ is the operator
\begin{equation} \label{eq:defcalHp}
\cal{H}_p a
	:= \ii \left(p \moyal a - a \moyal p\right) , \qquad
		a \in \sch(T^\star \mfd) .
\end{equation}

We may now introduce the main object studied in this article.

\begin{definition} \label{def:defsymbconjugatedop}
Let $a \in \sch'(T^\star \mfd)$ and $t \in \R$, and suppose $\e^{\ii t P} \Opw{a} \e^{\ii t P}$ makes sense as a continuous operator $\sch(\mfd) \to \sch'(\mfd)$. Then we let $\e^{t \cal{H}_p} a$ be the Weyl symbol of the operator $\e^{\ii t P} \Opw{a} \e^{\ii t P}$, namely the unique tempered distribution on $T^\star \mfd$ such that
\begin{equation} \label{eq:implicitdef}
\Opw{\e^{t \cal{H}_p} a}
	= \e^{\ii t P} \Opw{a} \e^{- \ii t P} .
\end{equation}
\end{definition}
Observe that this definition applies in particular to any tempered distribution $a$ such that $\Opw{a}$ extends to a bounded operator on $L^2(\mfd)$.
Moreover this definition makes sense in virtue of Proposition~\ref{prop:SchwartzWeylkernel} (a consequence of the Schwartz kernel theorem), and the notation is justified by~\eqref{eq:pdequantumdyn}.

Actually, we may give a stronger meaning to $\e^{t \cal{H}_p}$ for $L^2$~symbols. The operator $\cal{H}_p$ introduced in~\eqref{eq:defcalHp} can be seen as a skew-symmetric operator on $L^2(T^\star \mfd)$ with domain $\sch(T^\star \mfd)$.
We show in Proposition~\ref{prop:isometrycalHp} that the family of operators $(\e^{t \cal{H}_p})_{t \in \R}$, viewed as operators acting on $L^2(T^\star \mfd)$, is in fact a strongly continuous one-parameter unitary group, generated by a skew-adjoint extension of $\cal{H}_p$, that we call the \emph{quantum dynamics}. Our goal in this paper is to study this unitary group.

\subsubsection{The quantum-classical (or wave-particle) correspondence principle}

The generator $H_p$ of the classical dynamics is a fairly simple object, namely a vector field. In contrast, the operator $\cal{H}_p$ looks much more difficult to describe.
The idea of a correspondence between the classical and the quantum dynamics follows from pseudo-differential calculus, as we explain now.
We define the operator $\cal{H}_p^{(3)}$ by
\begin{equation} \label{eq:calHpHpHp3}
\cal{H}_p^{(3)}
	:= \cal{H}_p - H_p .
\end{equation}
If $p$ has temperate growth (which is the case under our assumptions stated in Section~\ref{subsec:mainresult}), \eqref{eq:calHpHpHp3} defines a continuous operator $\sch(\mfd) \to \sch(\mfd)$ (a consequence of Proposition~\ref{prop:pseudocalcconf} below).
One can see~\eqref{eq:calHpHpHp3} as a pseudo-differential calculus identity at order~$3$ (see~\eqref{eq:pdidentity}):
\begin{equation} \label{eq:calcorder3}
\cal{H}_p a
	= \ii (p \moyal a - a \moyal p)
	= \underbrace{\ii (pa - ap)}_{= 0} + \poiss*{p}{a} + \cal{H}_p^{(3)} a ,
		\qquad a \in \sch(T^\star \mfd) ,
\end{equation}
where $\poiss*{p}{a} = H_p a$ is the Poisson bracket.\footnote{With the notation of~\eqref{eq:defhatPj}, we have
\begin{equation} \label{eq:defHp3}
\cal{H}_p^{(3)}
	= \widehat{\cal{P}}_3(p, \bigcdot) - \widehat{\cal{P}}_3(\bigcdot, p) .
\end{equation}
The notation $\cal{H}_p^{(3)}$ is justified by the fact that the {order-$2$} terms in the pseudo-differential calculus~\eqref{eq:calcorder3} cancel for symmetry reasons.}
The remainder operator $\cal{H}_p^{(3)}$ contains all the ``non-locality" of quantum mechanics, and quantifies the deviation of the latter from classical mechanics. It involves derivatives of order $\ge 3$ of $p$ (actually it can be thought as a non-local version of the differential operator $\partial^3p \cdot \partial^3$). Thus we shall make assumptions on $\nabla^3 p$ ensuring that $\cal{H}_p^{(3)} a$ in~\eqref{eq:calcorder3} can indeed be considered as a lower order term (in the usual semiclassical setting, it would be of order $\hslash^3$).

With these definitions as hand, a statement like~\eqref{eq:egorovintro}, which is the usual way Egorov's theorem is presented, may now be reformulated equivalently as
\begin{equation*}
\e^{t \cal{H}_p} a
	= \e^{t H_p} a + \textrm{remainder terms.}
\end{equation*}
This is exactly what we achieve in Theorem~\ref{thm:main}. In fact, we prove that the conjugated operator in the left-hand side of~\eqref{eq:egorovintro} is pseudo-differential, and quantify this property in terms of continuity estimates of the quantum dynamics $(\e^{t \cal{H}_p})_{t \in \R}$ on symbol classes. The difference between the quantum and the classical dynamics will be studied via the Dyson series expansion associated with the decomposition~\eqref{eq:calHpHpHp3} of $\cal{H}_p$. See Section~\ref{subsec:asadysonseries}.

\subsection{Notation and conventions} \label{subsec:notationconventions}

\subsubsection{Symplectic and Riemannian structures on phase space}

We follow the conventions of Lerner \cite[Section 4.4.1]{Lerner:10} here (see also Hörmander~\cite[Chapter XVIII]{Hoermander:V3}). Recall that the symplectic form $\sympf$ on $T^\star \mfd$ is a non-degenerate {$2$-form}. Therefore it induces a bundle isomorphism
\begin{equation} \label{eq:isomorphismsympf}
\sympf : T (T^\star \mfd) \longrightarrow T^\star (T^\star \mfd) ,
\end{equation}
still denoted by $\sympf$, through the formula
\begin{equation*}
\brak*{\sympf X}{Y}_{T^\star (T^\star \mfd), T (T^\star \mfd)}
	:= \sympf(X, Y) ,
\end{equation*}
for all pair of vector fields $X, Y$ on $T^\star \mfd$. The dual map
\begin{equation} \label{eq:sympfstar}
\sympf^\star : T (T^\star \mfd) \longrightarrow T^\star (T^\star \mfd)
\end{equation}
satisfies $\sympf^\star = - \sympf$ given that $\sympf$ is alternating.

Let $g$ be a Riemannian metric on $T^\star \mfd$, that is to say a smooth symmetric {$2$-tensor} field that is positive definite. We denote by $\vol_g$ the Riemannian volume associated with~$g$. We denote by $\abs{\zeta}_g(\rho) = \abs{\zeta}_{g_\rho} = \sqrt{g_\rho(\zeta, \zeta)}$ the norm induced by $g$ on the tangent space $T_\rho (T^\star \mfd) \simeq W \ni \zeta$. We will also write $g_\rho(\zeta) := g_\rho(\zeta, \zeta)$ as shorthand. Furthermore, we introduce the notation $\jap{\zeta}_g := \sqrt{1 + \abs{\zeta}_g^2}$.
Throughout this paper, given a quadratic form $q$ on $T^\star \mfd$, we denote by
\begin{equation*}
\dist_{q}(\rho_1, \rho_2)
	= \abs{\rho_2 - \rho_2}_{q} ,
		\qquad \rho_1, \rho_2 \in T^\star \mfd ,
\end{equation*}
the distance between $\rho_1$ and $\rho_2$ for the constant metric equal to $q$. This definition naturally extends to distances between points and sets or between two sets.

A metric $g$ can be regarded as a bilinear bundle map
\begin{equation*}
g_\rho : T_\rho (T^\star \mfd) \times T_\rho (T^\star \mfd)
	\longrightarrow \R ,
\end{equation*}
assigning to each phase space point $\rho \in T^\star \mfd$ an inner product on $T_\rho (T^\star \mfd)$; or we can see $g$ as a bundle isomorphism
\begin{equation*}
g : T (T^\star \mfd) \longrightarrow T^\star (T^\star \mfd) ,
\end{equation*}
still denoted by $g$, through the formula
\begin{equation*}
\brak*{g X}{Y}_{T^\star (T^\star \mfd), T (T^\star \mfd)}
	:= g(X, Y) ,
		\qquad \forall X, Y \in \Gamma(T^\star \mfd) ,
\end{equation*}
where $\Gamma(T^\star \mfd)$ refers to the space of smooth vector fields on $T^\star \mfd$. We will also use the inverse map
\begin{equation} \label{eq:defg-1}
g^{-1} : T^\star (T^\star \mfd) \longrightarrow T (T^\star \mfd) ,
\end{equation}
which can be seen as a smooth positive definite symmetric {$2$-tensor} field acting on {$1$-forms} instead of vector fields, and reads
\begin{equation*}
g^{-1}(\omega, \eta)
	= \brak*{\omega}{g^{-1} \eta}_{T^\star (T^\star \mfd), T (T^\star \mfd)}
	= g\left(g^{-1} \omega, g^{-1} \eta\right) ,
		\qquad \forall \omega, \eta \in \Lambda^1(T^\star \mfd) ,
\end{equation*}
where $\Lambda^1(T^\star \mfd)$ refers to the space of {$1$-forms} on $T^\star \mfd$. The maps $g$ and $g^{-1}$ are usually called the musical isomorphisms~\cite[Section 7.6]{Lee:mfddiffgeo} (respectively the flatting and sharping operators).

We now give an important definition in the Weyl--Hörmander theory.
\begin{definition}[{$\sympf$-dual} metric \--- {\cite[Definition 4.4.22]{Lerner:10}}] \label{def:sympfdual}
Let $g$ be a Riemannian metric on $T^\star \mfd$. We define the bundle map
\begin{equation} \label{eq:defsympfdual}
g^\sympf := \sympf^\star g^{-1} \sympf : T(T^\star \mfd) \longrightarrow T^\star (T^\star \mfd) .
\end{equation}
It induces a Riemannian metric\footnote{The verification that $g^\sympf$ gives a genuine Riemannian metric goes as follows: for all pair of vector fields $X, Y$ on $T^\star \mfd$, one has
\begin{equation*}
\brak*{g^\sympf X}{Y}_{T^\star (T^\star \mfd), T(T^\star \mfd)}
	= g^{-1}(\sympf X, \sympf Y)
	= g\left(g^{-1} \sympf X, g^{-1} \sympf Y\right) .
\end{equation*}} $g^\sympf$ on $T^\star \mfd$, called the {$\sympf$-dual} metric associated with $g$.
\end{definition}
This definition gives a representation of the metric $g^{-1}$ using the isomorphism~\eqref{eq:isomorphismsympf} given by the symplectic form. The {$\sympf$-duality} $g \mapsto g^\sympf$ is an involution. A metric satisfying $g = g^\sympf$ will be called a \emph{symplectic metric}. An alternative definition consists in introducing the (unique) bundle map $J_g : T (T^\star \mfd) \to T(T^\star \mfd)$ such that
\begin{equation} \label{eq:Jg}
g(X, Y)
	= \sympf(X, J_g Y) ,
\end{equation}
for all vector fields $X, Y$ on $T^\star M$. Then one can check that
\begin{equation} \label{eq:Jg-1}
g^\sympf(X, Y)
	= \sympf({J_g}^{-1} X, Y) .
\end{equation}
Just as the inverse $g \mapsto g^{-1}$, {$\sympf$-duality} reverses the order,\footnote{This is a consequence of $\displaystyle \abs*{\zeta}_{g^\sympf} = \abs*{\sympf \zeta}_{g^{-1}} = \sup_{\zeta' \in W \setminus \{0\}} \dfrac{\abs{\sympf(\zeta, \zeta')}}{\abs{\zeta'}_g}$.} in the sense that
\begin{equation} \label{eq:sympfduality}
g_1 \le g_2
	\qquad \Longleftrightarrow \qquad
g_1^\sympf \ge g_2^\sympf .
\end{equation}
We will often use the property $(cg)^\sympf = c^{-1} g^\sympf$ (where the conformal factor~$c$ is not necessarily constant).
Given a diffeomorphism $\varphi : T^\star \mfd \to T^\star \mfd$, we denote by $\varphi^\pullb g$ the pullback of $g$ by $\varphi$, that is
\begin{equation*}
\forall X, Y \in \Gamma(T^\star \mfd) , \qquad
	(\varphi^\pullb g)_\rho(X, Y)
		= g_{\varphi(\rho)}(\dd \varphi(\rho). X, \dd \varphi(\rho). Y) ,
\end{equation*}
which is still a Riemannian metric. In addition, given a Riemannian metric $g$ and a vector field $X$, we recall that the Lie derivative of $g$ with respect to $X$ is given by
\begin{equation} \label{eq:deflieder}
\lieder_X g
	= {\dfrac{\dd}{\dd t}}_{\vert t = 0} (\varphi_X^t)^\pullb g ,
\end{equation}
where $\varphi_X^t$ is the (locally well-defined) flow generated by the vector field $X$.

\subsubsection{Tensors and norms} \label{subsub:tensorsandnorms}

The affine structure on $M$ provides a connection $\nabla$ on $\mfd$ which is torsion free and has vanishing curvature, and extends naturally to $T^\star \mfd$. It allows to make sense of the {$k$-th} differential of a smooth function $a : T^\star \mfd \to \CC$, that we denote by $\nabla^k a$, regardless of a Euclidean structure on $T^\star \mfd$~\cite[Section 12.8]{Lee:mfddiffgeo}. The differentials $\nabla^k a$ are instances of covariant tensors, namely they are naturally paired with vector fields $X \in \Gamma(T^\star \mfd)$:
\begin{equation*}
\begin{aligned}
\nabla^k a : \Gamma(T^\star \mfd)^k
	&\longrightarrow \cont^\infty(T^\star \mfd) \\
	(X_1, X_2, \ldots, X_k)
		&\longmapsto \nabla^k a (X_1, X_2, \ldots, X_k) = \nabla_{X_1, X_2, \ldots, X_k} a ,
\end{aligned}
\end{equation*}
as opposed to vector fields (contravariant tensors) which are naturally paired with differential {$1$-forms} $\omega \in \Lambda^1(T^\star \mfd)$. Given a general {$k$-covariant}, {$l$-contravariant} tensor $B$, namely a {$\cont^\infty(T^\star \mfd)$-multlinear} map
\begin{equation} \label{eq:deftensor}
\begin{aligned}
T : \Gamma(T^\star \mfd)^k \times \Lambda^1(T^\star \mfd)^l
	&\longrightarrow \cont^\infty(T^\star \mfd) \\
	(X_1, X_2, \ldots, X_k, \omega_1, \omega_2, \ldots, \omega_l)
	&\longmapsto B(X_1, X_2, \ldots, X_k, \omega_1, \omega_2, \ldots, \omega_l) ,
\end{aligned}
\end{equation}
its norm with respect to a Riemannian metric $g$ on $T^\star \mfd$ is the map
\begin{equation*}
\abs{B}_g
	= \abs{B}_g(\rho)
	= \sup_{\substack{\abs{X_j}_g(\rho) \le 1 \\ \abs{\omega_n}_{g^{-1}}(\rho) \le 1}} \abs*{B(X_1, X_2, \ldots, X_k, \omega_1, \omega_2, \ldots, \omega_l)} ,
		\qquad \rho \in T^\star \mfd .
\end{equation*}
Analogously to tensors, if $\kappa : T^\star \mfd \to T^\star \mfd$ is a smooth map, we can also define for all $k \in \N^\ast$ and all $\rho \in T^\star \mfd$ a norm $\abs{\nabla^k \kappa}_g(\rho)$ (beware that $\nabla^k \kappa$ is not really a tensor since it maps a collection of tangent vectors at $\rho$ to a tangent vector at $\kappa(\rho)$). See Section~\ref{subsec:norms} for more details.

Throughout the paper, we shall write $\abs{B}_{g, \infty} = \sup_{T^\star \mfd} \abs{B}_g$ and $\abs{\nabla^k \kappa}_{g, \infty} = \sup_{T^\star \mfd} \abs{\nabla^k \kappa}_g \in [0, + \infty]$.

\begin{remark} \label{rmk:polarintro}
Since the connection $\nabla$ has vanishing torsion and curvature, derivatives of a symmetric tensor are still symmetric tensors. Moreover, for any symmetric {$\ell$-linear} form $B$, one can show that
\begin{equation*}
\sup_{\abs{X} = 1} \abs*{B(X^\ell)}
	= \sup_{\abs{X} = 1} \abs*{B(X, X, \ldots, X)}
	= \sup_{\abs{X_j} = 1} \abs*{B(X_1, X_2, \ldots, X_\ell)} ,
\end{equation*}
provided the norm $\abs{\bigcdot}$ comes from an inner product (see for instance \cite{Harris:96}).
Therefore, to estimate the norm of a tensor $T$ with respect to a metric $g$, it will be sufficient to give bounds for the diagonal elements of the tensor only. See Remark~\ref{rmk:polarization} for more details.
\end{remark}

\subsection{Metrics on the phase space}

We are now ready to introduce the Weyl--Hörmander framework of metrics on the phase space. Here again we follow~\cite{Lerner:10,Hoermander:V3}. Let us give a short account of the roots of this theory. General algebras of pseudo-differential operators where introduced by Kohn and Nirenberg~\cite{KN:65PDO} in the 1960's. Later, Hörmander gave a generalization of the Kohn--Nirenberg symbol classes motivated by the study of hypoelliptic operators~\cite{Hor:66}. In~\cite{BF:73Proc,BF:74,Beals:75}, Beals and Fefferman introduced even more general classes of symbols for the study of local solvability of linear partial differential equations. Finally, a unifying framework using Riemannian metrics on the phase space was designed by Hörmander~\cite{Hor:79} (see also~\cite{Dencker:86}).

Given a Riemannian metric $g$, a point $\rho_0 \in T^\star \mfd$ and $r \ge 0$, we introduce the open ball
\begin{equation} \label{eq:defball}
B_r^g(\rho_0)
	= \set{\rho \in T^\star \mfd}{\abs*{\rho - \rho_0}_{g_{\rho_0}} < r} .
\end{equation}

In our analysis, the most important quantity associated with a phase space metric $g$ is its \emph{gain function}. We recall that $W \simeq T_\rho (T^\star \mfd)$ is the real vector space given by the affine structure~\eqref{eq:affinestructurephasespace} on $T^\star \mfd$.
\begin{definition}[Gain function] \label{def:gaing}
Let $g$ be a Riemannian metric on $T^\star \mfd$. Then we define the gain function of $g$ as
\begin{equation*}
\gain_g(\rho)
	:= \sup_{\zeta \in W \setminus \{0\}} \dfrac{\abs{\zeta}_{g_\rho}}{\abs{\zeta}_{g_\rho^\sympf}} ,
		\qquad \rho \in T^\star \mfd ,
\end{equation*}
where $g^\sympf$ is the {$\sympf$-dual} metric defined in~\eqref{eq:defsympfdual}. We also define the maximal gain to be
\begin{equation*}
\udl{\gain}_g
	:= \sup_{T^\star \mfd} \gain_g .
\end{equation*}
\end{definition}

One way to compute this function is to recall the map $J_g$ introduced in~\eqref{eq:Jg}. One can check that
\begin{equation} \label{eq:hgJg}
\gain_g
	= \sup_{\zeta \in W \setminus \{0\}} \sqrt{\dfrac{g(J_g \zeta)}{g^\sympf(J_g \zeta)}}
	= \sup_{\zeta \in W \setminus \{0\}} \dfrac{\abs{J_g \zeta}_g}{\abs{\zeta}_g}
	= \abs*{J_g}_g .
\end{equation}

In order to study pseudo-differential operators, e.g. their $L^2$ boundedness, it is important that the metric $g$ satisfies several requirements.

\begin{definition}[Admissible metric] \label{def:admissiblemetric}
A Riemannian metric $g$ on $T^\star M$ is said to be admissible if it satisfies the following three properties:
\begin{enumerate}
\item (Slow variation:) there exist $C, r > 0$ such that for all $\rho_0, \rho \in T^\star M$,
\begin{equation*}
\rho \in \bar B_r^g(\rho_0)
	\qquad \Longrightarrow \qquad
C^{-2} g_{\rho_0} \le g_\rho \le C^2 g_{\rho_0} ;
\end{equation*}
\item (Temperance:) there exist $C, N > 0$ such that for all $\rho_0, \rho \in T^\star M$,
\begin{equation*}
g_\rho
	\le C^2 g_{\rho_0} \jap*{\rho - \rho_0}_{g_\rho^\sympf}^{2N} ,
\end{equation*}
\item (Uncertainty principle:) the gain function associated with $g$ introduced in Definition~\ref{def:gaing} satisfies $\gain_g(\rho) \le 1$ for any $\rho \in T^\star M$.
\end{enumerate}
The constants $r$, $C$, $N$ arising in this definition are called \emph{structure constants} of the metric $g$.
\end{definition}

\begin{remark}[Interpretation of the gain function in dimension $1$] \label{rmk:BealsFefferman}
Beals--Fefferman metrics are of the form
\begin{equation*}
g
	= \dfrac{\dd x^2}{\Phi(x, \xi)^2} + \dfrac{\dd \xi^2}{\Psi(x, \xi)^2} ,
\end{equation*}
where $\dd x^2 + \dd \xi^2$ is a fixed Euclidean metric on $T^\star \mfd$. One can check that
\begin{equation*}
g^\sympf
	= \Psi(x, \xi)^2 \dd x^2 + \Phi(x, \xi)^2 \dd \xi^2 ,
\end{equation*}
which leads to
\begin{equation*}
\gain_g
	= \dfrac{1}{\Phi \Psi} .
\end{equation*}
When $\dim \mfd = 1$ (i.e.\ $\dim T^\star \mfd = 2$), the gain $\gain_g(x, \xi)$ is proportional to the area of the unit box of~$g$ centered at $(x, \xi)$, measured with respect to the background metric $\dd x^2 + \dd \xi^2$.
\end{remark}

Now we turn to the definition of admissible weight functions on phase space.

\begin{definition}[$g$-admissible weight] \label{def:admissibleweight}
Let $g$ be a Riemannian metric on $T^\star M$. A smooth function $m : T^\star M \to \R_+^\ast$ is said to be a $g$-admissible weight if it satisfies the following two properties:
\begin{enumerate}
\item\label{it:slowvariationm} ($g$-slow variation) there exist $C, r > 0$ such that for all $\rho_0 \in T^\star M$,
\begin{equation*}
\forall \rho \in \bar B_r^g(\rho_0) , \qquad
	C^{-1} m(\rho_0) \le m(\rho) \le C m(\rho_0) \,;
\end{equation*}
\item ($g$-temperance) there exist $C, N > 0$ such that for all $\rho_0, \rho \in T^\star M$,
\begin{equation*}
m(\rho)
	\le C m(\rho_0) \jap*{\rho - \rho_0}_{g_\rho^\sympf}^N .
\end{equation*}
\end{enumerate}
The constants $r$, $C$, $N$ arising in this definition are called \emph{structure constants} of the weight $m$.
\end{definition}

\begin{remark} \label{eq:gaingadmissibleweight}
It is known that all real powers $\gain_g^s$ of the gain function of $g$ are {$g$-admissible} weights. See~\cite[Remark 2.2.17 and Lemma 2.2.22]{Lerner:10}.
\end{remark}

We are now able to introduce the Weyl--Hörmander symbol classes.

\begin{definition}[Symbol classes] \label{def:symbclasses}
Let $g$ be a Riemannian metric on $T^\star M$ and $m : T^\star \mfd \to \R_+^\ast$ be a positive function. Then the symbol class $S(m, g)$ is the vector space of $\cont^\infty$ functions $a : T^\star M \to \CC$ such that
\begin{equation*}
\forall j \in \N, \exists C_j > 0 : \qquad
	\abs*{\nabla^j a}_g \le C_j m
		\qquad \textrm{on $T^\star \mfd$.}
\end{equation*}
The space $S(m, g)$ becomes a Fréchet space once it is equipped with the countable family of seminorms given by
\begin{equation*}
\abs*{a}_{S(m, g)}^{(\ell)}
	:= \max_{0 \le j \le \ell} \sup_{\rho \in T^\star M} \dfrac{\abs*{\nabla^j a}_g(\rho)}{m(\rho)} ,
		\qquad \ell \in \N .
\end{equation*}
\end{definition}

From the definition of $S(m, g)$, given $m$ and $m'$ such that $m/m'$ is bounded on $T^\star \mfd$ and $g \le C g'$, we have
\begin{equation*}
S(m, g) \subset S(m', g') ,
\end{equation*}
where the embedding is continuous.
The definition of $S(m, g)$ makes sense for general positive weights $m$ and metrics $g$. If $g$ and $m$ are temperate in the sense of Definitions~\ref{def:admissiblemetric} and~\ref{def:admissibleweight}, then one can check that these symbol classes satisfy
\begin{equation*}
\sch(M) \subset S(m, g) \subset \sch'(M) ,
\end{equation*}
where the inclusions are continuous. (One may topologize $\sch(\mfd)$ with the usual Schwartz seminorms by choosing an arbitrary Euclidean structure on $\mfd$.)

\begin{remark}[Extension of symbol classes to tensors] \label{rmk:symbolclasstensors}
Given a tensor $T$ as in~\eqref{eq:deftensor}, we will say that $T \in S(m, g)$ if
\begin{equation*}
\forall j \in \N, \exists C_j > 0 : \qquad
	\abs*{\nabla^j T}_g \le C_j m
		\qquad \textrm{on $T^\star \mfd$.}
\end{equation*}
\end{remark}

We will also need to consider a quantity that is not standard in the classical theory of the Weyl--Hörmander calculus. It quantifies the temperance property of the metric (see Definition~\ref{def:admissiblemetric}).

\begin{definition}[Temperance weight] \label{def:thetag}
Let $g$ be an admissible metric in the sense of Definition~\ref{def:admissiblemetric}. Then we define the temperance weight of $g$ by
\begin{equation} \label{eq:deftheta}
\theta_g(\rho)
	:= \sup_{\zeta \in W \setminus \{0\}} \dfrac{\abs{\zeta}_{g_\rho^\sympf}}{\abs{\zeta}_{{\sf g}}}
	= \sup_{\zeta \in W \setminus \{0\}} \dfrac{\abs{\zeta}_{{\sf g}^\sympf}}{\abs{\zeta}_{g_\rho}} ,
\end{equation}
where ${\sf g}$ is a fixed Euclidean metric such that ${\sf g} = {\sf g}^\sympf$.
\end{definition}

\begin{remark}
General properties of the temperance weight are listed in Proposition~\ref{prop:temperanceweight}. Bear in mind that it satisfies $\theta_g \ge 1$ (Proposition~\ref{prop:thetag}). The justification for the equalities in~\eqref{eq:deftheta} is recalled in Lemma~\ref{lem:justificationequalities}.
\end{remark}

\subsection{Assumptions and main result} \label{subsec:mainresult}

The main statement of this article, Theorem~\ref{thm:main} below, establishes a correspondence principle between quantum and classical dynamics. 
Let us describe the assumptions on $p$ and $g$. The first set of assumptions is mandatory to make sense of the problem.

\begin{customassum}{{\textsf{A}}}[Classical and quantum well-posedness] \label{assum:mandatory}
The classical Hamiltonian $p \in \sch'(T^\star \mfd)$ is a real-valued function of class $\cont^\infty$ and the following holds.
\begin{itemize} [label=\textbullet]
\item (Completeness of the Hamiltonian flow:) The Hamiltonian flow $(\phi^t)_{t \in \R}$ is complete.
\item (Essential self-adjointness:) The operator $\Opw{p}$ with domain $\sch(\mfd)$ is essentially self-adjoint on $L^2(\mfd)$. We denote by $P$ its unique self-adjoint extension.
\end{itemize}
\end{customassum}

The second set of assumptions concerns Egorov's theorem specifically.

\begin{customassum}{{\textsf{B}}}[Compatibility of the Hamiltonian and the metric] \label{assum:p}
The Riemannian metric $g$ on $T^\star \mfd$ is admissible. Moreover, the metric $g$ and the Hamiltonian $p$ satisfy the following conditions.
\begin{enumerate}[label=(\roman*)]
\item\label{it:Lambda} (Flow expansion:) The following Lyapunov exponent is finite:
\begin{equation} \label{eq:defLambda}
\Lambda
	:= \dfrac{1}{2} \sup_{T^\star \mfd} \abs*{\lieder_{H_p} g}_g
		< \infty .
\end{equation}
\item\label{it:strongsubquad} (Strong sub-quadraticity:) There exists $\epsilon \in (0, 1/2]$ such that
\begin{equation} \label{eq:symbolclassesnablap3}
\nabla^3 p
	\in S\left((\theta_g \udl{\gain}_g^{1/2})^{-\epsilon} \left(\tfrac{\gain_g}{\udl{\gain}_g}\right)^{-3}, g\right) \cap S\left( \left(\tfrac{\gain_g}{\udl{\gain}_g}\right)^{-1}, g\right) .
\end{equation}
\item\label{it:metriccontrol} (Metric control along the flow:) There exist $\Upsilon \ge 0$ and $C_\Upsilon > 0$ such that
\begin{equation} \label{eq:constantsmetricalongtheflow}
\forall \rho \in T^\star \mfd, \forall t \in \R, \qquad
	g_{\phi^t(\rho)}
		\le C_\Upsilon^2 \e^{2 \Upsilon \abs{t}} g_\rho .
\end{equation}
\end{enumerate}
\end{customassum}
Here $\lieder_{H_p}$ is the Lie derivative with respect to the Hamiltonian vector field, whose definition is recalled in~\eqref{eq:deflieder}.

We introduce
\begin{equation} \label{eq:defg(t)}
g(t)
	:= \e^{2 (\Lambda + 2 \Upsilon) \abs{t}} g ,
		\qquad
g^\sympf(t)
	:= g(t)^\sympf
	= \e^{-2 (\Lambda + 2 \Upsilon) \abs{t}} g^\sympf
		\qquad t \in \R .
\end{equation}
Remark that the map $t \mapsto g(t)$ is increasing on $\R_+$. We call \emph{Ehrenfest time} the number:
\begin{equation} \label{eq:defTE}
T_E = T_E(p, g)
	:= \dfrac{1/2}{\Lambda + 2 \Upsilon} \log\left(\dfrac{1}{\udl{\gain}_g}\right) ,
		\qquad
\udl{\gain}_g
	= \sup_{T^\star \mfd} \gain_g .
\end{equation}
This time is set so that $g(t)$ complies with the uncertainty principle for all $t \in [-T_E, T_E]$, that is to say $\gain_{g(t)} \le 1$ (see Remark~\ref{rmk:uniformadmissibilityg(t)} for a proof of this statement).

With these definitions at hand, we now state the main result of the article.

\begin{customthm}{I} \label{thm:main}
Fix $T_0 > 0$. Let $g$ be an admissible metric, let $p$ satisfy Assumption~\ref{assum:mandatory}, and assume $(g, p)$ satisfies Assumption~\ref{assum:p}. Let $T \in [0, T_0 + \frac{1}{2} T_E]$. Let $m$ be a {$g$-admissible} weight and write $m(t) := \e^{t H_p} m$. Then for all $a \in S(m, g)$ and all $t \in [-T, T]$, the operator $\e^{\ii t P} \Opw{a} \e^{- \ii t P}$ makes sense as a continuous operator $\sch(\mfd) \to \sch(\mfd)$. Its Weyl symbol $\e^{t \cal{H}_p} a$ defined in Definition~\ref{def:defsymbconjugatedop} belongs to $S(m(t), g(t))$.

We have the following estimate: for all $\ell \in \N$, there exists $k \in \N$ and a constant $C_\ell > 0$ such that
\begin{equation} \label{eq:continuityquantumevolsymb}
\forall a \in S(m, g), \forall t \in [-T, T] , \qquad
	\abs*{\e^{t \cal{H}_p} a}_{S(m(t), g(t))}^{(\ell)}
		\le C_\ell \abs*{a}_{S(m, g)}^{(k)} .
\end{equation}
Moreover, we have an asymptotic expansion
\begin{equation} \label{eq:asymptoticexpansion}
\e^{t \cal{H}_p} a
	\sim \e^{t H_p} a + \sum_{j \ge 1} \cal{E}_j(t) a ,
\end{equation}
which holds in the following sense. For all $\ell \in \N$, there exist integers $(k_{j_0, \ell})_{j_0 \in \N}$ and constants $(C_{j_0, \ell})_{j_0 \in \N}$, $(C_{j_0, \ell}')_{j_0 \in \N}$ such that for all $a \in S(m, g)$ and all $t \in [-T, T]$, the following holds:
\begin{empheq}[left=\empheqlbrace]{align}
\abs*{\e^{t H_p} a}_{S(m(t), g(t))}^{(\ell)}
	&\le C_{0, \ell} \abs*{a}_{S(m, g)}^{(k_{0, \ell})} , \quad & &\label{eq:estimateethpclassical}\\
\abs*{\cal{E}_{j_0}(t) a}_{S(\gain_{g(t)}^{2 j_0} m(t), g(t))}^{(\ell)}
	&\le C_{j_0, \ell} \abs*{\nabla a}_{S(m, g)}^{(k_{j_0, \ell})} , \quad & j_0 &\ge 1 \label{eq:estimatecalEj0}\\
\abs*{\e^{t \cal{H}_p} a - \e^{t H_p} a - \sum_{j = 1}^{j_0} \cal{E}_j(t) a}_{S(\gain_{g(t)}^{2 (j_0+1)} m(t), g(t))}^{(\ell)}
	&\le C_{j_0+1, \ell}' \abs*{\nabla a}_{S(m, g)}^{(k_{j_0, \ell}')} , \quad & j_0 &\ge 0 . \label{eq:estimateremainder}
\end{empheq}
The operators $\cal{E}_j(t)$ are described in Section~\ref{subsec:asadysonseries}. In addition, all the seminorm indices and constants in the above seminorm estimates depend on $g$, $m$ and $p$ only through structure constants of $g$ and $m$, seminorms of $p$ in the symbol classes~\eqref{eq:symbolclassesnablap3}, the constant $C_\Upsilon$ in~\eqref{eq:constantsmetricalongtheflow} and on any constant $c$ such that $\Lambda \ge c \udl{\gain}_g$.\footnote{The dependence on $c$ degenerates as $c \to 0$. See Section~\ref{subsubsec:semiclassicalsetting} for more details.}
\end{customthm}

We shall explain in Section~\ref{subsubsec:semiclassicalsetting} how to apply this result in a semiclassical setting, which is an important particular case contained in the Weyl--Hörmander framework.

\subsection{Comments} \label{subsec:comments}

\subsubsection{Comments on Assumption~\ref{assum:mandatory}}

The smoothness assumption is very classical in microlocal analysis. Although it rules out some interesting cases where the Hamiltonian contains irregular terms (for instance a rough or a Coulomb potential), our assumptions allow for a fairly large class of Hamiltonians as illustrated in Section~\ref{subsec:examples}.

A simple sufficient condition on $p$ ensures that Assumption~\ref{assum:mandatory} is fulfilled.

\begin{proposition}[Quantum and classical well-posedness] \label{prop:classicalwellposedness}
Suppose that $p$ is sub-quadratic with respect to an admissible metric $g_0$, in the sense $\nabla^2 p \in S(1, g_0)$, and further assume that
\begin{equation} \label{eq:assumeessentialselfadjointness}
\exists E \in \R : \qquad
	\sup_{\{p = E\}} \abs{\nabla p}_{g_0} < \infty .
\end{equation}
Then Assumption~\ref{assum:mandatory} is satisfied.
\end{proposition}

Proposition~\ref{prop:classicalwellposedness} is proved in Sections~\ref{subsec:classicalwellposedness} and~\ref{subsec:quantumwellposedness}.

\begin{remark} \label{rmk:inpracticecheckassum}
In practice, if we are given a classical Hamiltonian~$p$ together with a metric~$g$, we can check the assumptions of the above proposition as follows. In general, we take $g_0 = g^\natural$ the symplectic intermediate metric associated with $g$ (see Lemma~\ref{lem:gnatural} or~\cite[Definition 2.2.19]{Lerner:10}). For any vector field $X$ on $T^\star \mfd$, we have
\begin{equation} \label{eq:reasoningforkge2}
\forall k \ge 2, \qquad
	\abs*{\nabla^k p. X^k}
		\le \gain_g^{-1} \abs*{X}_g^k \abs*{\nabla^2 p}_{S(\gain_g^{-1}, g)}^{(k-2)}
		\le \left(\gain_g^{-1/2} \abs*{X}_g\right)^k \abs*{\nabla^2 p}_{S(\gain_g^{-1}, g)}^{(k-2)}
		\le \abs*{X}_{g^\natural}^k \abs*{\nabla^2 p}_{S(\gain_g^{-1}, g)}^{(k-2)} ,
\end{equation}
because $\gain_g^{-1} g \le g^\natural$ by Lemma~\ref{lem:gnatural}. Therefore it suffices to have
\begin{equation*}
\nabla^2 p \in S\left(\gain_g^{-1}, g\right)
\end{equation*}
to conclude that $\nabla^2 p \in S(1, g_0)$.
The same reasoning as in~\eqref{eq:reasoningforkge2} for $k = 1$ shows that it suffices to check that
\begin{equation*}
\exists E \in \R : \qquad
	\sup_{\{p = E\}} \gain_g^{1/2} \abs{\nabla p}_g < \infty
\end{equation*}
in order to secure~\eqref{eq:assumeessentialselfadjointness}.

Incidentally, observe that~\eqref{eq:assumeessentialselfadjointness} is automatically fulfilled if $p$ is semi-bounded, by choosing $E$ sufficiently small or large so that $\{p = E\} = \varnothing$.
\end{remark}

\subsubsection{Comments on Assumption~\ref{assum:p}}

We start with a general observation on Assumption~\ref{assum:p}.

\begin{remark}[Homogeneity] \label{rmk:homogeneity}
Observe that if Assumption~\ref{assum:p} is verified for some metric $g$, then it is also satisfied with the metric $g(t)$ (and with the same Hamiltonian $p$) for any $t \in [0, T_E]$. Indeed, by definition, $\Lambda$ is {$0$-homogeneous} with respect to multiplication of the metric by a constant conformal factor. Likewise, the ratios $\gain_g/\udl{\gain}_g$ and $\theta_g \udl{\gain}_g^{1/2}$ are {$0$-homogeneous}. Finally, seminorms of~$p$ with respect to $g(t)$ decrease as $t$ grows. Item~\ref{it:metriccontrol} of Assumption~\ref{assum:p}, involving the parameter $\Upsilon$, is also invariant under scaling of the metric. As for the Ehrenfest time, one readily checks that it behaves as
\begin{equation*}
T_E\left(g(t)\right)
	= T_E - \abs*{t} .
\end{equation*}
\end{remark}

Next we observe that Assumption~\ref{assum:p} implies temperate growth of $p$, in the sense that
\begin{equation} \label{eq:temperategrowth}
\exists n \ge 0, \exists \rho_0 \in T^\star \mfd : \forall \ell \in \N, \exists C_\ell > 0 : \forall \rho \in T^\star \mfd \qquad
	\abs*{\nabla^\ell p}_{\sf g}(\rho)
		\le C_\ell \jap*{\rho - \rho_0}_{\sf g}^{n (1+\ell)} ,
\end{equation}
for some \emph{Euclidean} metric $\sf g$ on the phase space. This is the content of Lemma~\ref{lem:temperategrowth}.

\medskip
Let us explain each point of Assumption~\ref{assum:p} in more detail.

\begin{itemize}
\item \emph{Item~\ref{it:Lambda}.}
The condition $\Lambda < \infty$ on the Lyapunov exponent implies that
\begin{equation} \label{eq:Lyapunovcontrol}
(\phi^t)^\pullb g
	\le \e^{2 \Lambda \abs{t}} g ,
		\qquad \forall t \in \R ,
\end{equation}
by a Gr{\"o}nwall argument, or equivalently
\begin{equation} \label{eq:expflowexpansion}
\abs*{\dd \phi^t}_g
	\le \e^{\Lambda \abs{t}} ,
		\qquad \forall t \in \R .
\end{equation}
In other words, the exponential expansion rate of the Hamiltonian flow is controlled uniformly on the whole phase space (from the viewpoint of the metric~$g$). 

\begin{remark}[Sub-quadraticity] \label{rmk:subquad}
We shall see in Proposition~\ref{prop:aprioriexpgrowth}, assuming smoothness of $p$ only, that\footnote{Recall that $\lieder_{H_p}$ and $\nabla_{H_p}$ are not the same in general. The Lie derivative involves pullback by the flow $\phi^t$ (see~\eqref{eq:deflieder}), whereas the covariant derivative involves parallel transport according to the affine structure of~$T^\star \mfd$.}
\begin{equation*}
\Lambda
	\le \sup_{T^\star \mfd} \left(\gain_g \abs*{\nabla^2 p}_g + \dfrac{1}{2} \abs*{\nabla_{H_p} g}_g\right) .
\end{equation*}
Therefore in practice, to check that $\Lambda < \infty$, it will be sufficient to prove that
\begin{equation} \label{eq:assumsubquad}
\nabla^2 p
	\in S\left((\gain_g/\udl{\gain}_g)^{-1}, g\right)
		\qquad \rm{and} \qquad
\sup_{T^\star \mfd} \abs*{\nabla_{H_p} g}_g
	< \infty .
\end{equation}
The first condition on $\nabla^2 p$ can be regarded as a ``sub-quadraticity" property of the Hamiltonian with respect to the metric $g$. The sub-quadraticity assumption is quite classical in the context of Egorov's theorem (see \cite[Theorem (IV-10)]{Robert:book} for instance). Moreover, we make no ellipticity assumption (unlike~\cite[Chapter 11]{Zworski:book} for instance). 
\end{remark}

\item \emph{Item~\ref{it:strongsubquad}.}
The (relatively strong) version of sub-quadraticity stated in~\eqref{eq:symbolclassesnablap3}, which concerns all derivatives of $p$ of order larger than~$3$, is intended to control the time evolution both at the classical and quantum levels. It is useful in order to ensure that the non-local effects of quantum mechanics deviating from classical mechanics remain small. Notice that a condition on the third derivative of the form $\nabla^3 p \in S((\gain_g/\udl{\gain}_g)^{-1}, g)$ was already present in the work of Bony. See the section~$3$ of~\cite{Bony:evolfrancais,Bony:generalizedFIO} and the discussion in Section~\ref{subsec:Bony} below.

\item \emph{Item~\ref{it:metriccontrol}.}
The control of the metric along the Hamiltonian flow, involving the parameter $\Upsilon$, is convenient in our analysis, and it is clearly verified in most applications we have (see Section~\ref{subsec:examples}), for which $\nabla_{H_p} g = 0$ (i.e.\ $\Upsilon = 0$). However we will explain further in Remark~\ref{rmk:improvements} why this assumption is not so natural.
\end{itemize}

\subsubsection{Comments on Theorem~\ref{thm:main}} \label{subsec:commentsmainthm}

Before commenting on Theorem~\ref{thm:main}, let us stress the fact that it is a consequence of a more precise result on the propagation of partitions of unity adapted to the metric $g$. See Theorem~\ref{thm:partition} in Section~\ref{subsec:partition}. One of the main aspects of Theorem~\ref{thm:main} is that it provides an Egorov's theorem to any order (the higher order terms are described in Section~\ref{subsec:asadysonseries}), with a pseudo-differential remainder.

\begin{itemize} [label=\textbullet]
\item Notice that the asymptotic expansion~\eqref{eq:asymptoticexpansion} of the symbol consists in powers of $\gain_{g(t)}^2$ and not in powers of $\gain_{g(t)}$; see~\eqref{eq:estimatecalEj0}. Actually this is just a matter of how the terms contained in $\cal{E}_j(t) a$ are arranged. Using pseudo-differential calculus, one can see that $\cal{H}_p^{(3)}$ contains powers of order $\ge 3$ of $\gain_g$ actually (recall this is a ``non-local" operator). Thus, applying pseudo-differential calculus to decompose all the operators $\cal{H}_p^{(3)}$ appearing in~$\cal{E}_j(t)$ (see~\eqref{eq:defEj}) and grouping terms by powers of $\gain_{g(t)}$ would yield an asymptotic expansion in powers of $\gain_{g(t)}^2, \gain_{g(t)}^3, \gain_{g(t)}^4$, etc. The fact that there is no term of order $\gain_{g(t)}$ is due to the use of the Weyl quantization. The operators $\cal{E}_j(t)$ are described in more detail in Section~\ref{subsec:asadysonseries}.
\item For technical reasons we need to take into account both $\Upsilon$ and $\Lambda$ to measure the instability of the Hamiltonian flow. This is why $\Lambda + 2 \Upsilon$ appears in the definition~\eqref{eq:defg(t)} of $g(t)$ and of the Ehrenfest time $T_E$ in~\eqref{eq:defTE}. However, for a good and natural choice of $g$, we will have $\Upsilon = 0$ in many of cases (see the first two examples in Section~\ref{subsec:examples}).

It is not clear whether one can replace the definition of 
$\Lambda$ in~\eqref{eq:defLambda} with a weaker assumption on the maximal expansion rate:
\begin{equation*}
\Lambda_0
	= \limsup_{t \to \infty} \dfrac{1}{\abs{t}} \log \abs*{\dd \phi^t}_{g, \infty}
\end{equation*}
(which is used for instance in~\cite{AnNon,DG:14,DJN}).
The problem lies in the fact that we need a control of the form $\abs*{\dd \phi^t}_g \le C_\Lambda \e^{\Lambda \abs{t}}$ with a constant $C_\Lambda = 1$. It seems that having $C_\Lambda > 1$ poses a problem in Proposition~\ref{prop:continuityflowconf}---see Remark~\ref{rmk:CLambda1}.
\item Recall that the uncertainty principle reads $\udl{\gain}_g \le 1$ (Definition~\ref{def:admissiblemetric}). In the case where $\udl{\gain}_g = 1$, we end up with $T_E = 0$, and we obtain an Egorov theorem on fixed bounded intervals of time. 
\item Egorov's theorem is usually expected to work in time $T_E$ (which already contains a factor $1/2$---see~\eqref{eq:defTE}), whereas we have $\frac{1}{2} T_E$ in our statement. This $\frac{1}{2}$ factor seems to occur for technical reasons in our proof: it is convenient to have the stronger condition $\abs{t} \le \frac{1}{2} T_E$ instead of $\abs{t} \le T_E$ in several places in the paper (see for instance Proposition~\ref{prop:uniformm(t)} or Lemma~\ref{lem:multiaffconf}).
\end{itemize}

\subsubsection{The semiclassical setting} \label{subsubsec:semiclassicalsetting}

In semiclassical analysis~\cite{DS:book,Zworski:book}, one is concerned with operators $P = P_\hslash = \Opw{p_\hslash}$ depending on a small Planck parameter $\hslash \in (0, 1]$ (with suitable estimates). Such symbols enter into the framework of the Weyl--Hörmander calculus with appropriate families of metrics $(g_\hslash)_\hslash$ (for instance $g_\hslash = \dd x^2 + \hslash^2 \dd \xi^2$; see~\cite[p.\ 70 and Examples 2.2.10]{Lerner:10}), for which $\udl{\gain}_{g_\hslash} = \hslash$. Uniformity of structure constants with respect to $\hslash \in (0, 1]$ is very important to ensure that constants depend on $\hslash$ in a controlled way.

Beware that the semiclassical time scale usually considered in semiclassical analysis is different from the one in Theorem~\ref{thm:main}. Usually, one considers the propagator $\e^{- \ii \frac{\tau}{\hslash} P_\hslash}$ associated with $P_\hslash$ with a ``semiclassical time" $\tau = \hslash t$~\cite{DS:book,Zworski:book}.
Here we rather use the variable $t  = \hslash^{-1} \tau$. If the Lyapunov exponent of the semiclassical scaling is of order $\lambda \asymp 1$, then because of the semiclassical time rescaling, the Lyapunov exponent $\Lambda$ defined in~\eqref{eq:defLambda} is typically of order $\hslash \lambda \asymp \hslash$. Therefore the Ehrenfest time $T_E$ defined in~\eqref{eq:defTE} is of order $\hslash^{-1} \log\hslash^{-1}$. Thus, the bound $\abs{t} \lesssim \hslash^{-1} \log \hslash^{-1}$ is consistent with the more usual one $\abs{\tau} \lesssim \log \hslash^{-1}$.

We also mention at the end of Theorem~\ref{thm:main} that continuity estimates may depend on a constant~$c$ such that
\begin{equation} \label{eq:assumLambda}
\Lambda \ge c \udl{\gain}_g .
\end{equation}
In the semiclassical setting, this constant~$c$ should be independent of $\hslash$, in order to have uniform estimates (for instance in Proposition~\ref{prop:continuityEcal}). This is a natural assumption in view of the fact that the semiclassically rescaled Lyapunov exponent $\lambda = \Lambda/\hslash$ is generally independent of~$\hslash$.

To ensure a uniform control in the Egorov asymptotics when $t$ approaches the Ehrenfest time and $\hslash \to 0$, we need take care of the dependence of all the other quantities involved in Assumption~\ref{assum:p} on $\udl{\gain}_g$. Of course all the seminorms of $\nabla^3 p_\hslash$ measured with respect to $g_\hslash$ should be uniform in $\hslash$. We also require that $\Upsilon$ and $C_\Upsilon$ are bounded independently of $\hslash$.

\subsection{Egorov's theorem as a Dyson series} \label{subsec:asadysonseries}

In this section, we explain what are the operators $\cal{E}_j(t)$ appearing in the asymptotic expansion~\eqref{eq:asymptoticexpansion} of Theorem~\ref{thm:main}. The quantum and the classical dynamics, described by $\e^{t \cal{H}_p}$ and $\e^{t H_p}$ respectively and introduced in Section~\ref{subsec:classicalandquantum}, can be related through a so-called Dyson series expansion of $e^{t \cal{H}_p}$. The relevance of the Dyson series expansion comes from the fact that $\cal{H}_p^{(3)}$ introduced in~\eqref{eq:calHpHpHp3} allows to gain smallness in terms of powers of $\gain_g$ under Assumption~\ref{assum:p}. This formula arises naturally in the context of scattering theory or in the \emph{interaction picture} of quantum mechanics (also known as the ``Dirac picture"), where the generator of the dynamics under consideration (here $\cal{H}_p$) is a perturbation of the generator of the free dynamics (here $H_p$).

Throughout the article, we denote by $\Delta_k$ the {$k$-dimensional} simplex
\begin{equation} \label{eq:defsimplex}
\Delta_k
	= \set{(s_1, s_2, \ldots, s_k) \in \R^k}{0 \le s_1 \le s_2 \le \cdots \le s_k \le 1} ,
\end{equation}
and we set
\begin{equation*}
t \Delta_k
	= \set{t {\bf s} \in \R^k}{{\bf s} \in \Delta_k} ,
		\qquad t \in \R .
\end{equation*}
Typical points of the simplex $t \Delta_k$ will be denoted for simplicity by ${\bf s} = (s_1, s_2, \ldots, s_k)$, and the Lebesgue measure will be denoted by $\dd {\bf s} = \dd s_1 \dd s_2 \cdots \dd s_k$. Throughout this article, $\int_{t \Delta_k} f({\bf s}) \dd {\bf s}$ is understood as an ``oriented" integral; e.g.\ for $k = 1$, we have
\begin{equation*}
\int_{t \Delta_1} f(s) \dd s
	:= \int_0^t f(s) \dd s
	= - \int_t^0 f(s) \dd s ,
\end{equation*}
whatever the sign of~$t$ is.

\begin{proposition} \label{prop:Dyson}
Let $p$ and $g$ satisfy Assumptions~\ref{assum:mandatory} and~\ref{assum:p}. For any $j_0 \in \N$ and any time $t \in \R$, the following holds:
\begin{equation} \label{eq:Dyson}
\e^{t \cal{H}_p}
	= \sum_{j = 0}^{j_0} \cal{E}_j(t) + \widehat{\cal{E}}_{j_0+1}(t) ,
\end{equation}
where $\cal{E}_0(t) = \e^{t H_p}$ and for $j \ge 1$:
\begin{align}
\cal{E}_j(t)
	&= \int_{t \Delta_j} \e^{(t - s_j) H_p} \cal{H}_p^{(3)} \e^{(s_j - s_{j-1}) H_p} \cal{H}_p^{(3)} \cdots \cal{H}_p^{(3)} \e^{s_1 H_p} \dd {\bf s} , \label{eq:defEj}\\
\widehat{\cal{E}}_j(t)
	&= \int_{t \Delta_j} \e^{(t - s_j) \cal{H}_p} \cal{H}_p^{(3)} \e^{(s_j - s_{j-1}) H_p} \cal{H}_p^{(3)} \cdots \cal{H}_p^{(3)} \e^{s_1 H_p} \dd {\bf s} .\label{eq:defhatEj}
\end{align}
The equalities in~\eqref{eq:defEj} and~\eqref{eq:defhatEj} hold as operators $\sch(T^\star \mfd) \to L^2(T^\star \mfd)$, and the integrals converge absolutely with respect to the strong topology on $L^2(T^\star \mfd)$.
\end{proposition}

Proposition~\ref{prop:Dyson} is proved in Section~\ref{subsec:Egorovspecificclass}. It provides an explicit algebraic description of the higher order terms of the asymptotic expansion~\eqref{eq:asymptoticexpansion}. A key step to prove Theorem~\ref{thm:main} is to establish continuity estimates for the operators $\cal{E}_j(t)$ in symbol classes; see Proposition~\ref{prop:continuityEcal}.
Throughout the paper, we will use the notation
\begin{equation} \label{eq:notationlej}
\cal{E}_{\le j_0}(t)
	:= \sum_{j = 0}^{j_0} \cal{E}_j(t) .
\end{equation}

\begin{remark}[Recurrence relation]
It follows directly from the definition that for all $j \ge 0$:
\begin{equation} \label{eq:recurrencerelation}
\cal{E}_{j+1}(t)
	= \int_0^t \e^{(t - s) H_p} \cal{H}_p^{(3)} \cal{E}_{j}(s) \dd s
		\qquad \rm{and} \qquad
\widehat{\cal{E}}_{j+1}(t)
	= \int_0^t \e^{(t - s) \cal{H}_p} \cal{H}_p^{(3)} \cal{E}_{j}(s) \dd s .
\end{equation}
\end{remark}

\begin{remark}
To obtain estimates uniform in time on these symbols, it is important to keep in mind that the volume of the simplex is $\abs{t \Delta_j} = t^j/j!$.
\end{remark}

\subsection{Examples of application} \label{subsec:examples}

In this section, we present three different settings in which Theorem~\ref{thm:main} applies, namely a Schrödinger equation in the Euclidean space, a (half-)wave equation in a curved space and a transport equation. In each of these cases, we describe appropriate metrics $g$ on the phase space to which we can apply Theorem~\ref{thm:main}.

To ease the understanding of the computations presented in this section, let us discuss briefly the structure of the phase space $T^\star \mfd$. We have a natural splitting of the tangent space to $T^\star \mfd$:
\begin{equation} \label{eq:splitting}
T(T^\star \mfd)
	\simeq T \mfd \oplus T^\star \mfd .
\end{equation}
The first component (the horizontal direction) is parallel to $\mfd$, while the second one (the vertical direction) is perpendicular to $\mfd$ in the sense that it is tangent to the fibers of $T^\star \mfd$. See the beginning of Section~\ref{sec:proofsex} for a more details.

\subsubsection{Semiclassical Schrödinger operator in the (flat) Euclidean space} \label{subsubsec:Schhbar}

We equip $\mfd$ with a fixed Euclidean structure $\upgamma$. Consider a semiclassical Schrödinger operator with flat Laplacian, electric and magnetic potentials $V$ and $\beta$, i.e.\
\begin{equation} \label{eq:Schop}
P
	= \tfrac{1}{2} \abs*{\tfrac{\hslash}{\ii} \partial - \beta}_\upgamma^2 + V
	= - \tfrac{\hslash^2}{2} \Delta - \tfrac{\hslash}{\ii} \beta \cdot \partial - \tfrac{\hslash}{2 \ii} \dvg \beta + \tfrac{1}{2} \abs{\beta}_\upgamma^2 + V ,
\end{equation}
acting on Schwartz functions. The Planck parameter is such that $\hslash \in (0, 1]$. Here, $\partial$ refers to the gradient and $\cdot$ to the scalar product, both with respect to the Euclidean metric $\upgamma$, the divergence is its adjoint with respect to the $\upgamma$ inner product and $\Delta = \Delta_{\upgamma}$ is the flat Laplacian associated with $\upgamma$ on $\mfd$.

\begin{proposition} \label{prop:Schhbar}
Assume that $V$ and $\beta$ are $\cont^\infty$ and have temperate growth, we have
\begin{equation*}
P = \Opw{p}
	\qquad \rm{with} \qquad
p(x, \xi) = \tfrac{\hslash^2}{2} \abs*{\xi}_{\upgamma^{-1}}^2 - \hslash \xi. \beta + \tfrac{1}{2} \abs{\beta}_{\upgamma}^2 + V(x) .
\end{equation*}
Moreover, assume the following:
\begin{itemize} [label=\textbullet]
\item the vector potential $\beta : \mfd \to T \mfd$ is an affine vector field, namely $\nabla^2 \beta = 0$;
\item the electric potential $V : \mfd \to \R$ is sub-quadratic, namely
\begin{equation} \label{eq:assumV2}
\forall k \ge 0 , \qquad
	\sup_{T^\star \mfd} \abs*{\nabla^{2 + k} V}_\upgamma < \infty .
\end{equation}
and bounded from below.
\end{itemize}
Then Assumption~\ref{assum:mandatory} is satisfied, and Assumption~\ref{assum:p} is satisfied, uniformly with respect to $\hslash \in (0, 1]$, with the metric
\begin{equation} \label{eq:gSchhbar}
g = g_\hslash = {\upgamma} \oplus \hslash^2 {\upgamma}^{-1} .
\end{equation}
\end{proposition}

Proposition~\ref{prop:Schhbar} is proved in Section~\ref{subsec:proofsSchhbar}. Theorem~\ref{thm:main} thus applies in this setting and provides with a \emph{global} Egorov theorem, in the sense that it describes the quantum evolution of symbols defined globally on phase space, and not only in the vicinity of an energy shell.

\begin{remark}
The gain function associated with $g_\hslash$ is the semiclassical parameter $\udl{\gain}_{g_\hslash} = \gain_{g_\hslash} = \hslash$. The Lyapunov exponent defined in~\eqref{eq:defLambda} is then of order $\Lambda = \Lambda_\hslash \approx \hslash$ (while $\Upsilon = 0$), so that the Ehrenfest time~\eqref{eq:defTE} behaves like $\hslash^{-1} \log \hslash^{-1}$ (or more classically of order $\log \hslash^{-1}$ in the semiclassical time scale; see Section~\ref{subsubsec:semiclassicalsetting}).
\end{remark}

In view of this remark and Proposition~\ref{prop:Schhbar}, a consequence of Theorem~\ref{thm:main} formulates as follows. For all $a \in S(1, g_\hslash)$, with $g_\hslash$ in~\eqref{eq:gSchhbar} (see~\cite[(1.1.33)]{Lerner:10} or~\cite[(4.4.4)]{Zworski:book} with a different convention), we have
\begin{equation*}
\e^{\ii \frac{\tau}{\hslash} P} \Opw{a} \e^{- \ii \frac{\tau}{\hslash} P}
	= \Opw{a(\tau)}
\end{equation*}
where $a(\tau) \in S(1, g_\hslash(\tau))$ with seminorm estimates in this class uniform with respect to $(\tau, \hslash)$ subject to $\abs{\tau} \le \varepsilon \log \hslash^{-1}$, where $\varepsilon$ is a small enough constant (and where $g_\hslash(\tau)$ is defined in~\eqref{eq:defg(t)}). In particular, the family of symbols $a(\tau)$ belongs to a bounded subset of $S(1, \bar g_\hslash)$ where
\begin{equation*}
\bar g_\hslash
	= \hslash^{-2 \delta} \upgamma \oplus \hslash^{2(1 - \delta)} \upgamma^{-1} ,
\end{equation*}
for some parameter $\delta \in [0, 1/2]$ (which can be quantified in terms of $\varepsilon$ and of an upper bound of the semiclassically rescaled Lyapunov exponent $\Lambda_\hslash/\hslash$), uniformly in $\abs{\tau} \le \varepsilon \log \hslash^{-1}$ and $\hslash \in (0, 1]$.

Let us comment on the ``non-semiclassical" case ($\hslash$ fixed or $\hslash = 1$).

\begin{itemize}[label=\textbullet]
\item The metric $g$ in~\eqref{eq:gSchhbar} is roughly speaking the only relevant one to consider here.
Indeed, Item~\ref{it:Lambda} of Assumption~\ref{assum:p} is related to a sub-quadraticity property of the Hamiltonian in view of Remark~\ref{rmk:subquad}, so we may require~$g$ to be such that $\abs{\nabla^2 p}_g$ is bounded on~$T^\star \mfd$. The second derivative of $p$ with respect to $\xi$, namely $(\nabla^2 p)_{\xi \xi} = {\upgamma}^{-1}$, is constant on the whole phase space. Therefore, in order to ensure the boundedness of $\abs{\nabla^2 p}_g$, one is forced to consider a metric~$g$ whose unit boxes are bounded in the~$\xi$ direction (think of~$g$ as a Beals--Fefferman metric as in Remark~\ref{rmk:BealsFefferman}). As a consequence, to comply with the uncertainty principle ($\gain_g \le 1$), these boxes cannot be squeezed too much in the~$x$ direction. Thus, unit boxes of~$g$ should roughly look like squares. See Figure~\ref{fig:boxesSch} for an illustration.
\item It is very important here that the metric ${\upgamma}$ is perfectly flat. This was already evidenced in our earlier work~\cite{P:23} on the observability of the Schrödinger equation. The fact that $\nabla \upgamma = 0$ ensures that the position and momentum variables are somewhat ``separated". If it is not the case, the second derivative with respect to $x$, namely $(\nabla^2 p)_{x x}$, contains a term of the form $\nabla^2 {\upgamma}^{-1}(\xi, \xi)$, which may blow up like $\jap{\xi}_{{\upgamma}^{-1}}^2$ at fiber infinity. Having boundedness of the second derivative of $p$ with respect to $g$ would force us to chose $g$ of the form
\begin{equation*}
g = \jap*{\xi}_{{\upgamma}^{-1}}^2 {\upgamma} \oplus {\upgamma}^{-1} ,
\end{equation*}
which strongly violates the uncertainty principle since then $\gain_g \approx \jap{\xi}_{{\upgamma}^{-1}}$. This is related to the infinite speed of propagation of singularities for the Schrödinger equation: frequencies $\xi$ propagate at speed of the order of $\jap{\xi}_{{\upgamma}^{-1}}$, so that they see an effective Lyapunov exponent of order $\jap{\xi}_{{\upgamma}^{-1}} \Lambda$ instead of $\Lambda$ while propagating along bicharacteristics. This is a clear obstruction to having a \emph{global} Egorov's theorem. In Proposition~\ref{prop:Schhbar}, we also assume that the vector potential $\beta$ is affine for the same reason. This constraint was already remarked by Bouzouina and Robert~\cite[Remark 1.6]{BouzouinaRobert}. To relax these assumptions on $\upgamma$ and $\beta$, an alternative way to proceed is to truncate the Hamiltonian in the vicinity of a fixed energy shell and reduce our investigation to energy-localized symbols.
\end{itemize}

\begin{figure}
\centering
\includegraphics[scale=1.5]{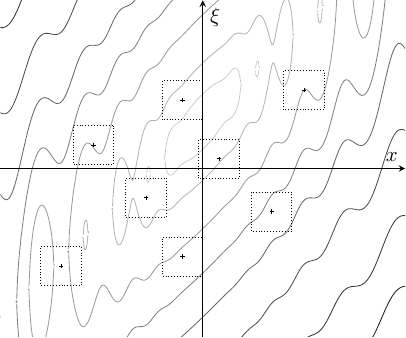}
\caption{Level sets of the classical Hamiltonian $p$ defined in dimension~$1$ by~\eqref{eq:Schop}, with a confining potential $V$ and $\nabla \beta$ constant. Dotted lines represent the unit boxes of the metric introduced in~\eqref{eq:gSchhbar} $g_1 = \upgamma \oplus \upgamma^{-1}$ (i.e.\ the product of unit balls of $\upgamma$ and $\upgamma^{-1}$). These boxes are squares with sidelength $\approx 1$ that saturate the uncertainty principle $\gain_{g_1} \le 1$.}
\label{fig:boxesSch}
\end{figure}

The Egorov theorem that we obtain in Theorem~\ref{thm:main} is consistent with the work of Bouzouina and Robert~\cite{BouzouinaRobert} (see also Robert~\cite[Theorem (IV-10)]{Robert:book}). Notice that the boundedness of derivatives of order larger that $2$~\eqref{eq:assumV2} was already required there; see~\cite[Theorem 1.2 (9)]{BouzouinaRobert} or~\cite[Theorem (IV-10) ii)]{Robert:book}. Theorem~\ref{thm:main} and Proposition~\ref{prop:Schhbar} generalize their work by considering symbol classes with general weight functions $m$, instead of symbols $a$ satisfying $\nabla a \in S(1, g_\hslash)$ as in~\cite[Theorem 1.2 (11)]{BouzouinaRobert}. In addition, we not only provide an asymptotic expansion but also prove that the full symbol of the conjugated operator belongs to the expected symbol class. Lastly, Theorem~\ref{thm:partition} on the quantum evolution of partitions of unity is new even in this context to our knowledge.

\subsubsection{Half-wave operator in a curved space} \label{subsubsec:halfwaves}

As explained in Section~\ref{subsubsec:Schhbar}, if one wants to consider a curved geometry on $\mfd$ instead of a Euclidean structure, it is hopeless to look for a global Egorov's theorem for the Laplace--Beltrami operator (unless we consider only symbols supported near a fixed energy shell). However, considering the wave operator instead of the Schrödinger operator, i.e.\ passing from $-\Delta$ (Schrödinger) to $\sqrt{-\Delta}$ (waves), brings us back in a setting where there exists admissible metrics fulfilling the assumptions of Theorem~\ref{thm:main} in the presence of curvature. The reason for this is that we go from infinite speed of propagation of energy for $\e^{\ii t \Delta}$ to finite speed of propagation for $\e^{- \ii t \sqrt{-\Delta}}$.

So here we let $\gamma$ be a smooth, non-necessarily flat, Riemannian metric on $\mfd = \R^d$. We assume that it satisfies
\begin{equation} \label{eq:assumgamma}
\exists C_\gamma > 0 : \forall x_1, x_2 \in \mfd , \quad
	\gamma_{x_1} \le C_\gamma^2 \gamma_{x_2}
		\qquad \rm{and} \qquad
\forall k \in \N , \quad \sup_\mfd \abs*{\nabla^k \gamma}_\gamma < \infty .
\end{equation}
This implies in particular a uniform control of the form $c^{-2} I \le \gamma \le c^2 I$.
This includes metrics with no specific asymptotic behavior at space infinity.

The Riemannian manifold $(M, \gamma)$ carries a natural volume form $\vol_\gamma$, and a Laplace--Beltrami operator $\Delta_\gamma$ acting on compactly supported smooth functions (or Schwartz functions). Fixing (global) Euclidean coordinates on $M$, it reads
\begin{equation*}
\Delta_\gamma
	= \abs*{\gamma}^{-1/2} \partial_j \gamma^{ij} \abs*{\gamma}^{1/2} \partial_i ,
\end{equation*}
where repeated indices are summed according to the Einstein convention. In this expression, $\abs{\gamma}$ refers to the determinant of $\gamma$ with respect to the chosen Euclidean coordinate system, so that $\dd \vol_\gamma = \abs{\gamma}^{1/2} \dd x$, and $\gamma^{ij}$ are the matrix components of $\gamma^{-1}$. The operator $\Delta_\gamma$ is seen naturally as an unbounded operator on the Hilbert space $L^2(M, \vol_\gamma)$. To fit this in the setting of this paper, we need to construct from $\Delta_\gamma$ an operator that acts on $L^2(\mfd)$. This is done by identifying elements $v \in L^2(\mfd, \vol_\gamma)$ with functions $u \in L^2(\mfd)$ through the correspondence
\begin{align*}
L^2(\mfd, \vol_\gamma) &\longrightarrow L^2(\mfd) \\
v &\longmapsto v \abs{\gamma}^{1/4} = u .
\end{align*}
One can check that this is an isometry (provided the chosen volume element on $L^2(\mfd)$ and the coordinate we consider agree). Through this identification, the Laplace--Beltrami operator is conjugated by $\abs{\gamma}^{1/4}$ and we obtain
\begin{equation} \label{eq:deftildeDeltagamma}
\widetilde{\Delta}_\gamma
	:= \abs*{\gamma}^{1/4} \Delta_\gamma \abs*{\gamma}^{-1/4}
	= \abs*{\gamma}^{-1/4} \partial_j \gamma^{ij} \abs*{\gamma}^{1/2} \partial_i \abs*{\gamma}^{-1/4} .
\end{equation}
The operator $- \widetilde{\Delta}_\gamma$ is indeed symmetric with respect to the scalar product of $L^2(\mfd)$, while $\Delta_\gamma$ is symmetric for the scalar product of $L^2(\mfd, \vol_\gamma)$.
Recalling the fact that
\begin{equation*}
\partial_t^2 - \Delta_\gamma
	= \left( \partial_t - \ii \sqrt{-\Delta_\gamma} \right) \left( \partial_t + \ii \sqrt{-\Delta_\gamma} \right) ,
\end{equation*}
the study of the wave equation in this context reduces\footnote{We refer to~\cite{DL:09,LL:16,Leautaud:23} for more details on this factorization.} to that of the evolution associated with the operator
\begin{equation*}
\abs*{\gamma}^{1/4} \sqrt{- \Delta_\gamma} \abs*{\gamma}^{-1/4}
	= \sqrt{- \widetilde{\Delta}_\gamma} .
\end{equation*}
In Section~\ref{subsec:proofshalfwaves}, we prove that $-\widetilde{\Delta}_\gamma = \Opw{p}$ with
\begin{equation*}
p(x, \xi)
	= \abs*{\xi}_{\gamma_x^{-1}}^2 + R_\gamma(x) ,
\end{equation*}
where $R_\gamma \in \cont_\bdd^\infty(\mfd)$.
See Lemma~\ref{lem:symbolDeltatilde} for the expression of the remainder $R_\gamma$.
We show in Lemma~\ref{lem:sqarerootclosetoP} (Section~\ref{subsec:proofshalfwaves}) that the operator $\Opw{\jap*{\xi}_{\gamma^{-1}}}$ is a good approximation of $\sqrt{- \widetilde{\Delta}_\gamma}$, and $\jap*{\xi}_{\gamma^{-1}}$ is smooth (whereas $\abs{\xi}_{\gamma^{-1}}$ has a singularity at $\xi = 0$). Thus in this section we consider the operator
\begin{equation} \label{eq:defPhalfwaves}
P
	:= \Opw{p}
		\qquad \rm{with} \qquad
			p(x, \xi) = \jap*{\xi}_{\gamma_x^{-1}} .
\end{equation}

The main result of this section is the following. Proofs can be found in Section~\ref{subsec:proofshalfwaves}.

\begin{proposition} \label{prop:halfwaves}
Under the assumptions~\eqref{eq:assumgamma}, the symbol $p$ defined in~\eqref{eq:defPhalfwaves} satisfies Assumptions~\ref{assum:mandatory} and~\ref{assum:p} with the family of metrics
\begin{equation} \label{eq:ghalfwave}
g_{(x, \xi)}
	= \jap*{\xi}_{\gamma_x^{-1}}^{1 - \alpha} \gamma_x \oplus \jap*{\xi}_{\gamma_x^{-1}}^{-(1+\alpha)} \gamma_x^{-1}
	= \dfrac{1}{p^\alpha} \left(p \gamma_x \oplus p^{-1} \gamma_x^{-1}\right) ,
		\qquad \alpha \in [0, 1] .
\end{equation}
\end{proposition}

Theorem~\ref{thm:main} in this context, and for a weight $m = \jap{\xi}^n$ ($n \in \R$), reads as follows: for any $a \in S(\jap{\xi}^n, g)$, the family $a(t)$ such that
\begin{equation*}
\e^{\ii t P} \Opw{a} \e^{- \ii t P}
	= \Opw{a(t)}
\end{equation*}
remains in a bounded subset of $S(\jap{\xi}^n, g)$ for all $t \in [-T_0, T_0]$, for any fixed $T_0 > 0$. We recover a result stated by Taylor~\cite[Section 6, Proposition 6.3.B]{Taylor:91} in the symbol classes
\begin{equation*}
S_{\rho, \delta}^n
	= \set{a \in \cont^\infty(\R^{2d})}{\forall \beta_1, \beta_2 \in \N^d, \exists C_{\beta_1, \beta_2} > 0 : \quad \abs*{\partial_x^{\beta_2} \partial_\xi^{\beta_1} a(x, \xi)} \le C_{\beta_1, \beta_2} \jap{\xi}^{n - \rho \abs{\beta_1} + \delta \abs{\beta_2}}} .
\end{equation*}
(See also~\cite[Chapter 7, Section 8]{Taylor:vol2} for a simpler version.) Our result clarifies the range of admissible indices $\rho = 1 - \delta > 1/2$ (here we have $\alpha = \rho - \delta$). In fact we can reach the boundary case $\rho = 1 - \delta = 1/2$ corresponding to $\alpha = 0$ in Proposition~\ref{prop:halfwaves}). Moreover, Theorem~\ref{thm:main} is valid for more general {$g$-admissible} weight functions and is not restricted to $m(x, \xi) = \jap{\xi}^n$.

Notice that the map $\alpha \mapsto g_\alpha$ is decreasing (the unit boxes are nested). The two boundary cases $\alpha = 1$ and $\alpha = 0$ are of particular interest.
\begin{itemize} [label=\textbullet]
\item The metric $g$ for $\alpha = 1$, namely
\begin{equation*}
g
	= \gamma \oplus \dfrac{1}{\jap{\xi}_{\gamma^{-1}}^2} \gamma^{-1}
\end{equation*}
is the most common one in microlocal analysis (it corresponds to $S_{\rho, \delta}$ with $\rho = 1$ and $\delta = 0$). This is the ``principal microlocal scale", namely the most natural way to put a metric on phase space while studying the wave equation. See Figure~\ref{subfig:alpha1} for an illustration.
\item The case $\alpha = 0$, i.e.\
\begin{equation*}
g
	= \jap{\xi}_{\gamma^{-1}} \gamma \oplus \dfrac{1}{\jap{\xi}_{\gamma^{-1}}} \gamma^{-1}
\end{equation*}
corresponds to a \emph{second-microlocal} scale of propagation of waves. It is related to that used in the articles~\cite{BZ:16,BG:20,Rouveyrol:24} on the (damped) wave equation. From these articles, in a semiclassical setting, one can establish that a {$o(h)$-quasimode} of the Laplacian of typical frequency $\jap{\xi}_{\gamma^{-1}} \asymp h^{-1}$ cannot concentrate (in the space variable $x$) at scales smaller than $\sqrt{h}$. See Figure~\ref{subfig:alpha0} for an illustration. This metric is also relevant for the problem studied by the author in~\cite{P:24}, concerning the uniform stability of the damped wave equation in the Euclidean space. Theorem~\ref{thm:main} was mainly motivated by this paper and we plan to apply it (or rather Theorem~\ref{thm:partition}) to tackle~\cite[Conjecture 1.11]{P:24}. Finally, we refer to~\cite{BonyLerner} for a comprehensive approach to second microlocalization in the framework of the Weyl--Hörmander calculus.
\end{itemize}

\begin{figure}
		\centering
		\begin{subfigure}[t]{0.45\textwidth}
		\includegraphics[scale=1.0]{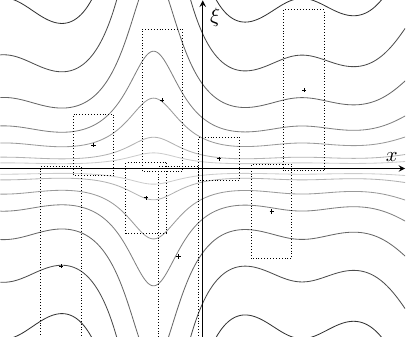}
		\caption{Case $\alpha = 1$. The unit boxes have sidelength $\approx 1$ in the space variable and $\approx \jap{\xi}_{\gamma^{-1}}$ in the momentum variable. Their area increases as their center goes away from the null section $\{\xi = 0\}$. This is consistent with the expression of the gain function $\gain_g(x, \xi) \approx \jap{\xi}_{\gamma^{-1}}^{-1}$ associated with this metric.}
		\label{subfig:alpha1}
		\end{subfigure}
		\hfill
		\begin{subfigure}[t]{0.45\textwidth}
		\includegraphics[scale=1.0]{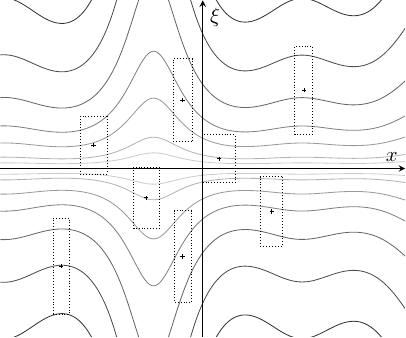}
        \caption{Case $\alpha = 0$. The unit boxes have sidelength $\approx \jap{\xi}_{\gamma^{-1}}^{-1/2}$ in the space variable and $\approx \jap{\xi}_{\gamma^{-1}}^{1/2}$ in the momentum variable. Their area is of order~$1$, which illustrates the fact that the metric saturates the uncertainty principle $\gain_g \le 1$.}
        \label{subfig:alpha0}
        \end{subfigure}
        \caption{Level sets of the classical Hamiltonian $p(x, \xi) = \jap{\xi}_{\gamma^{-1}}$. Dotted lines represent unit boxes of the metric $g$ defined in~\eqref{eq:ghalfwave}. They correspond to products of unit balls of $\gamma_1$ and $\gamma_2$ associated with the decomposition $g = \gamma_1 \oplus \gamma_2$ in~\eqref{eq:ghalfwave}.}
        \label{fig:halfwave}
\end{figure}

\subsubsection{Vector fields} \label{subsubsec:X}

Our last application concerns differential operators of order $1$, namely vector fields. These are quite different from the previous cases for several reasons. The first difference with the examples discussed in Sections~\ref{subsubsec:Schhbar} and~\ref{subsubsec:halfwaves} is that vector fields are neither elliptic nor semibounded. Another interesting feature is that the relevant metric on phase space that we introduce below is not invariant by the Hamiltonian flow, contrary to the first two examples discussed above.

Although the analysis of such operators seems to fall under classical mechanics at first glance, microlocal techniques have proved very powerful and natural in this context, as evidenced e.g.\ by the works of Faure--Sjöstrand~\cite{FJ:11}, Dyatlov--Zworski~\cite{DZ:16} and Faure--Tsujii~\cite{FT:23,FT:arxiv}.

In what follows, we place ourselves in the following setting: a vector field on (the non-compact manifold) $\mfd$, bounded with respect to a fixed Euclidean metric. We first consider the case of a vector field $X$ on $\mfd$ that preserves a smooth density $\mu = \abs{\mu} \dd x$, in such a way that the derivation operator $\frac{1}{\ii} X$ acting on compactly supported functions is symmetric as an operator on $L^2(\mfd, \mu)$. To fit in the framework of this paper, we identify $L^2(\mfd, \mu)$ and $L^2(\mfd)$ via
\begin{align*}
T_\mu : L^2(\mfd, \mu) &\longrightarrow L^2(\mfd) \\
u &\longmapsto u \abs*{\mu}^{1/2} .
\end{align*}
This is an isometric isomorphism provided $\abs*{\mu}$ does not vanish. The operator $\frac{1}{\ii} X$ is then conjugated by $\abs*{\mu}^{1/2}$, so that it acts on $L^2(\mfd)$ as
\begin{equation} \label{eq:conjX}
\tfrac{1}{\ii} T_\mu X T_\mu^{-1}
	= \tfrac{1}{\ii} \abs*{\mu}^{1/2} X \abs*{\mu}^{-1/2}
	= \tfrac{1}{\ii} X - \dfrac{1}{2 \ii} \left(X \log \abs*{\mu}\right) .
\end{equation}
In Lemma~\ref{lem:logderivativemu}, we compute the logarithmic derivative of $\abs{\mu}$, which reads $X \log \abs{\mu} = - \dvg X$ (here $\dvg X = \tr \nabla X$ where $\nabla X(x)$ is viewed as an endomorphism of $T_x \mfd$).
Therefore the operator of interest here is
\begin{equation} \label{eq:defPX}
P
	:= \tfrac{1}{\ii} X + \tfrac{1}{2 \ii} \dvg X ,
\end{equation}
which is symmetric as an operator on $L^2(\mfd)$. In the sequel, we assume that the vector field satisfies
\begin{equation} \label{eq:assumX}
\forall k \in \N , \qquad
	\sup_\mfd \abs*{\nabla^k X}_I < \infty ,
\end{equation}
where $I$ refers to the standard Euclidean metric on $\mfd$. The induced metric on $T^\star \mfd$ is denoted by $\dd x^2 + \dd \xi^2$. The main result of this section is the following.

\begin{proposition} \label{prop:exX}
Assume $X$ satisfies~\eqref{eq:assumX}. With $P$ in~\eqref{eq:defPX}, we have
\begin{equation} \label{eq:HamiltonianPX}
P = \Opw{p}
	\qquad \rm{with} \qquad
p(x, \xi) = \xi. X(x) .
\end{equation}
Then Assumptions~\ref{assum:mandatory} and~\ref{assum:p} are satisfied with the metric
\begin{equation} \label{eq:gvectorfield}
g_{(x, \xi)}
	= \jap*{\xi}_{I}^{2 \alpha_1} \dd x^2 + \jap*{\xi}_{I}^{- 2 \alpha_2} \dd \xi^2 ,
\end{equation}
where the parameters $\alpha_1$ and $\alpha_2$ are subject to
\begin{equation} \label{eq:conditionalphanu}
0 \le \alpha_1 \le \alpha_2 \le 1 ,
	\qquad
\alpha_1 + \alpha_2 \ge 1
	\quad \rm{and} \quad
\alpha_1 < 1 .
\end{equation}
\end{proposition}

The proof of this result is presented in Section~\ref{subsec:proofsX}. Notice that the conjugated operator $\e^{\ii t P} \Opw{a} \e^{- \ii t P}$ can be described explicitly in the case where $a = a(x)$ depends only on the position variable. Indeed, we have
\begin{equation*}
\e^{- \ii t P} u
	= \exp\left(- \dfrac{1}{2} \int_0^t \e^{-s X} (\dvg X) \dd s \right) \e^{- t X} u ,
		\qquad \forall u \in \sch(\mfd) ,
\end{equation*}
and one can then check that
\begin{equation*}
\e^{\ii t P} \Opw{a(x)} \e^{- \ii t P}
	= \Opw{\e^{t X} a} .
\end{equation*}
This is an instance of exact Egorov's theorem related to the fact that the Hamiltonian~\eqref{eq:HamiltonianPX} is linear in the momentum variable (derivatives of order larger than~$2$ with respect to $\xi$ vanish). Therefore, the main interest of our result concerns the evolution of observables depending on both variables $x$ and $\xi$. In this context, Theorem~\ref{thm:main} gives the following, say in the symbol class $S(1, g)$: for any $a \in S(1, g)$, we have
\begin{equation*}
\e^{\ii t P} \Opw{a} \e^{- \ii t P}
	= \Opw{a(t)} ,
\end{equation*}
where the family of symbols $a(t)$ remains in a bounded subset of $S(1, g)$ for all $t \in [-T_0, T_0]$ with fixed $T_0 > 0$.

In~\cite{FT:23,FT:arxiv}, Faure and Tsujii introduce a specific family of metrics on the phase space to study flows of Anosov vector fields on compact manifolds. The principal symbol of the Hamiltonian that the authors consider is the same as ours, namely $\xi. X(x)$. To define the relevant metrics adapted to the dynamics, they introduce flow box coordinates: position and momentum variables split into perpendicular and parallel components
\begin{equation*}
x
	= (x_\perp, x_\parallel) \in \R^{d-1} \times \R ,
		\qquad
\xi
	= (\xi_\perp, \xi_\parallel) \in \R^{d - 1} \times \R ,
\end{equation*}
in such a way that the vector field $X$ generating the flow corresponds to $- X = \partial_{x_\parallel}$. Then the family of metrics that they consider is defined as
\begin{equation} \label{eq:gFT}
g_{(x, \xi)}
	= \jap*{\xi}^{2 \alpha^\perp} \dd x_\perp^2 + \dfrac{1}{\jap*{\xi}^{2 \alpha^\perp}} \dd \xi_\perp^2 + \jap*{\xi}^{2 \alpha^\parallel} \dd x_\parallel^2 + \dfrac{1}{\jap*{\xi}^{2 \alpha^\parallel}} \dd \xi_\parallel^2 ,
\end{equation}
in those coordinates. The parameters $\alpha^\perp, \alpha^\parallel$ satisfy
\begin{equation*}
0 \le \alpha^\parallel \le \alpha^\perp < 1
	\qquad \rm{and} \qquad
\dfrac{1}{2} \le \alpha^\perp < 1 .
\end{equation*}
These conditions ensure that changing charts does not affect the metric, up to a global conformal factor (actually they define an equivalence class of metrics invariant under flow box coordinate changes). A relevant choice of these parameters then allows Faure and Tsujii to describe the Ruelle spectrum of the flow and analyze the corresponding resonant states.

Conditions~\eqref{eq:conditionalphanu} imply in particular that $\alpha_2 \ge 1/2$. The metric of Faure and Tsujii~\eqref{eq:gFT} is symplectic, which means that $g = g^\sympf$. In our case, the metric $g$ defined in~\eqref{eq:gvectorfield} is symplectic when $\alpha_1 = \alpha_2$, and then it coincides with Faure and Tsujii's one if we take $\alpha^\parallel = \alpha^\perp$ there. An interesting feature of Faure and Tsujii's metric~\eqref{eq:gFT} is the fact that it is adapted to $X$ in the sense that it allows measurements at different scales in the direction of the flow and in the transverse direction. It is not clear how to reproduce such an anisotropy in our framework. This is certainly due to the subtle construction of Faure and Tsujii using flow box coordinates, while we work in a global Euclidean coordinate system.

\subsection{More on admissibility of phase space metrics} \label{subsec:metricsonphasespace}

In this article, we work with a metric $g$ on the phase space, but we will see that the metric $g(t)$ introduced in~\eqref{eq:defg(t)} will arise naturally while considering the action of the Hamiltonian flow $\e^{t H_p}$ on symbol classes. Indeed, $g(t)$ is a metric conformal to $g$, for which~\eqref{eq:Lyapunovcontrol} holds.
In the Weyl--Hörmander framework, the Ehrenfest time~\eqref{eq:defTE} arises as the time from which admissibility of $g(t)$ breaks down, due to the failure of the uncertainty principle $\gain_{g(t)} \le 1$.
The purpose of the two propositions below is to check that the family of metrics $g(t)$ is uniformly admissible, namely its structure constants are uniform with respect to $t \in [-T_E, T_E]$, and that $m(t) := \e^{t H_p} m$ is a {$g(t)$-admissible} weight, uniformly in $\abs{t} \le \frac{1}{2} T_E$. One of the tools involved in the proofs is the so-called \emph{symplectic intermediate metric} $g^\natural$, which is defined as the geometric mean of $g$ and $g^\sympf$. More precisely, $g^\natural$ is the largest non-negative symmetric map $q : W \to W^\star$ such that the symmetric map
\begin{equation} \label{eq:defgnatural}
\begin{pmatrix}
g & q \\ q & g^\sympf
\end{pmatrix}
	: W \oplus W \longrightarrow W^\star \oplus W^\star
\end{equation}
is non-negative---see~\cite{PW:75,Ando:79} and~\cite[Definition 2.2.19]{Lerner:10}. The main property of $g^\natural$ is that it is symplectic, namely $g^\natural = (g^\natural)^\sympf$, and it satisfies
\begin{equation} \label{eq:chainineq}
g
	\le \gain_g g^\natural
	\le g^\natural
	= (g^\natural)^\sympf
	\le \dfrac{1}{\gain_g} g^\natural
	\le g^\sympf
\end{equation}
(see Lemma~\ref{lem:gnatural}).

\begin{proposition}[Improved admissibility] \label{prop:improvedadmissibility}
Let $g$ be an admissible metric. Then there exist $r_g \in (0, 1], C_g > 0, N_g \ge 0$, depending only on structure constants of~$g$, such that
\begin{equation} \label{eq:unifadmnatural}
\forall t \in \R , \forall \rho_0, \rho \in T^\star \mfd, \qquad
	g_\rho(t)
		\le C_g^2 \jap*{\dist_{(g_{\rho_0}^\natural + g_\rho^\natural)^\sympf}\left( B_{r_g}^g(\rho_0), B_{r_g}^g(\rho) \right)}^{2 N_g} g_{\rho_0}(t) .
\end{equation}
Similarly, for any {$g$-admissible} weight $m$ for which $r_g$ is a slow variation radius, there exist $C > 0, N \ge 0$ such that
\begin{equation} \label{eq:unifadmweight}
\forall \rho_0, \rho \in T^\star \mfd, \qquad
	m(\rho)
		\le C \jap*{\dist_{(g_{\rho_0}^\natural + g_\rho^\natural)^\sympf}\left( B_{r_g}^g(\rho_0), B_{r_g}^g(\rho) \right)}^{N} m(\rho_0) ,
\end{equation}
with constants $C, N$ depending only on structure constants of $g$ and $m$.
\end{proposition}

This improved admissibility property is a consequence of a stronger form of temperance involving the metric $g^\natural$ (see \cite[Proposition 2.2.20]{Lerner:10}). We refer to Appendix~\ref{app:improvedadmissibility} for a proof.

\begin{remark}[Uniform admissibility of the metrics $g(t)$] \label{rmk:uniformadmissibilityg(t)}
By {$\sympf$-duality}~\eqref{eq:sympfduality}, the estimate~\eqref{eq:unifadmnatural} holds replacing $g_\rho(t)$ and $g_{\rho_0}(t)$ with $g_\rho^\sympf(t)$ and $g_{\rho_0}^\sympf(t)$ respectively. We also provide less precise upper bounds for the right-hand side of~\eqref{eq:unifadmnatural} and~\eqref{eq:unifadmweight} that will be useful throughout the paper. From~\eqref{eq:chainineq}, we have $g^\natural \ge \frac{\gain_g}{\udl{\gain}_g} g^\natural \ge \udl{\gain}_g^{-1} g \ge g(t) \ge g$ for all $\abs{t} \le T_E$, so that
\begin{align*}
\dist_{(g_{\rho_1}^\natural + g_{\rho_2}^\natural)^\sympf}\left( B_{r}^g(\rho_1), B_{r}^g(\rho_2) \right)
	&\le \dist_{(g_{\rho_1}(t) + g_{\rho_2}(t))^\sympf}\left( B_{r}^g(\rho_1), B_{r}^g(\rho_2) \right)
	\le \min_{j \in \{1, 2\}} \dist_{g_{\rho_j}^\sympf(t)}\left( B_{r}^g(\rho_1), B_{r}^g(\rho_2) \right) \\
	&\le \min_{j \in \{1, 2\}} \abs{\rho_2 - \rho_1}_{g_{\rho_j}^\sympf(t)} .
\end{align*}
In particular,~\eqref{eq:unifadmnatural} and~\eqref{eq:unifadmweight} can be seen as a compact way of writing both slow variation and temperance properties (Definitions~\ref{def:admissiblemetric} and~\ref{def:admissibleweight}) with a single inequality. Lastly, $g(t)$ satisfies the uncertainty principle for all $\abs{t} \le T_E$, since by definition of the gain function (Definition~\ref{def:gaing}) and of the Ehrenfest time~\eqref{eq:defTE}:
\begin{equation*}
\gain_{g(t)}
	= \sup_{\zeta \in W \setminus \{0\}} \dfrac{\e^{(\Lambda + 2 \Upsilon) \abs{t}} \abs{\zeta}_{g_\rho}}{\e^{-(\Lambda + 2 \Upsilon) \abs{t}} \abs{\zeta}_{g_\rho^\sympf}}
	= \e^{2 (\Lambda + 2 \Upsilon) \abs{t}} \gain_g
	\le \e^{2 (\Lambda + 2 \Upsilon) T_E} \udl{\gain}_g
	= 1 .
\end{equation*}
\end{remark}

From now on, we fix common structure constants of the family of metrics $g(t)$, that is to say constants $r_g \in (0, 1]$, $C_g > 0$ and $N_g \ge 0$ such that Definition~\ref{def:admissiblemetric} is satisfied by $g(t)$ for all $t \in [-T_E, T_E]$. We will often call $r_g$ a slow variation radius and $C_g$ a slow variation constant of $g$. See Appendix~\ref{app:slowvariationradius} for more information on slow variation radii.

Uniform {$g(t)$-admissibility} of $m(t)$ is provided in the proposition below. The proof is given at the end of Section~\ref{sec:classical}.

\begin{proposition} \label{prop:uniformm(t)}
Suppose $p$ and $g$ satisfy Assumptions~\ref{assum:mandatory} and~\ref{assum:p}, and let $m$ be a {$g$-admissible} weight. Then $m(t) := \e^{t H_p} m$ is a {$g(t)$-admissible} weight for all $\abs{t} \le \frac{1}{2} T_E$ with uniform structure constants.
\end{proposition}

We proceed with a sanity check on temperance weights.

\begin{proposition}  \label{prop:temperanceweight}
The temperance weight $\theta_g$ defined in Definition~\ref{def:thetag} is a {$g$-admissible} weight. Moreover $r_g$ (introduced above) is a slow variation radius of $\theta_g$. If $\theta_g'$ is the temperance weight defined with a flat metric $\sf{g}'$ instead of $\sf{g}$, one has
\begin{equation} \label{eq:temperanceweightscomparable}
\exists C = C({\sf g}, {\sf g}') > 0 : \qquad
	C^{-1} \theta_g \le \theta_g' \le C \theta_g
		\quad \rm{on} \; T^\star \mfd .
\end{equation}
\end{proposition}

Proposition~\ref{prop:temperanceweight} says that the temperance weight $\theta_g$ is essentially independent of the background Euclidean metric ${\sf g}$ chosen in its definition. The temperance weight and the gain function are somewhat related through the following observation.

\begin{proposition} \label{prop:thetag}
The following holds:
\begin{equation*}
\forall \rho_0, \rho \in T^\star \mfd , \qquad
	g_\rho^\sympf
		\le \theta_g(\rho_0) \theta_g(\rho) g_{\rho_0} .
\end{equation*}
Moreover, one has
\begin{equation*}
\gain_g \times \theta_g^2 \ge 1 .
\end{equation*}
In particular, $\theta_g \ge 1$ provided $g$ satisfies the uncertainty principle $\gain_g \le 1$.
\end{proposition}

Proposition~\ref{prop:temperanceweight} and~\ref{prop:thetag} are proved in Appendix~\ref{app:metrics}.
We end this section with a discussion on possible improvements regarding Assumption~\ref{assum:p}, in connection with the use of the temperance weight.

\begin{remark} \label{rmk:improvements}
Let us comment on two assumptions that could perhaps be relaxed.
\begin{itemize} [label=\textbullet]
\item First, the introduction of the temperance weight $\theta_g$ relies on the choice of a background Euclidean metric ${\sf g}$. Hence it is not an intrinsic feature of the metric $g$. However it is crucial in our argument in order to apply Beals' theorem. Roughly speaking, it allows to control derivatives of the propagator $\ad_F \e^{- \ii t P}$ with $F = \Opw{f}$ being the quantization of an affine function $f$ (see the symbols of the form $c_\pi$ in Corollary~\ref{cor:iteratedcommutatorsaff}). It ensures that the long-range effects of the possible blow up of the metric at phase space infinity (i.e.\ $\theta_g$ unbounded) can be balanced by the gain $\theta_g^{-\epsilon}$ appearing at each step of the Egorov expansion~\eqref{eq:asymptoticexpansion}.
\item Second, the use of $\theta_g$ seems to require an estimate of the growth of the metric along the flow, hence the importance of the parameter $\Upsilon$ introduced in Assumption~\ref{assum:p}~\ref{it:metriccontrol}. Lemma~\ref{lem:gain(t)} is quite illuminating in this respect: the gain function $\gain_g$ is defined intrinsically, and understanding $\e^{t H_p} \gain_g$ only requires the very natural control~\eqref{eq:Lyapunovcontrol}, whereas $\theta_g$ requires the choice of a background Euclidean metric, and understanding $\e^{t H_p} \theta_g$ involves the extra Item~\ref{it:metriccontrol} of Assumption~\ref{assum:p} on the growth of $g_{\phi^t}$.
\end{itemize}
\end{remark}

\subsection{Propagation of quantum partitions of unity} \label{subsec:partition}

Theorem~\ref{thm:main} is in fact a consequence of the more general Theorem~\ref{thm:partition} below, namely an Egorov theorem for confined family of symbols. We introduce first partitions of unity adapted to an admissible metric $g$.

\begin{proposition}[Existence of partitions of unity \--- {\cite[Theorem 2.2.7]{Lerner:10}}] \label{prop:existencepartitionofunity}
Let $g$ be an admissible metric. Let $r_g > 0$ be a slow variation radius given in Proposition~\ref{prop:improvedadmissibility}. For any $r \in (0, r_g]$, there exists a family of functions $(\varphi_{\rho_0})_{\rho_0 \in T^\star \mfd}$, bounded in $S(1, g)$, namely
\begin{equation} \label{eq:S(1,g)partitionbounded}
\sup_{\rho_0 \in T^\star \mfd} \abs*{\varphi_{\rho_0}}_{S(1, g)}^{(\ell)} < \infty , \qquad
 \forall \ell \in \N ,
\end{equation}
such that for all $\rho_0 \in T^\star \mfd$, $\supp \varphi_{\rho_0} \subset B_r^g(\rho_0)$ and
\begin{equation} \label{eq:integralequal1}
\forall \rho \in T^\star \mfd , \qquad
	\int_{T^\star \mfd} \varphi_{\rho_0}(\rho) \dd \vol_g(\rho_0)
		= 1 .
\end{equation}
More precisely, there exists a constant $c$ depending only on structure constants of $g$, but not on $r$, such that
\begin{equation} \label{eq:r-dependentestimate}
\forall \ell \in \N, \forall \rho_0 \in T^\star \mfd, \qquad
	\abs*{\nabla^\ell \varphi_{\rho_0}}_{g, \infty}
		\le c r^{- \ell - 2 \dim \mfd} .
\end{equation}
\end{proposition}

\begin{remark}
The {$r$-dependent} estimate~\eqref{eq:r-dependentestimate} is not stated in~\cite[Theorem 2.2.7]{Lerner:10}, but it follows from the proof. To check that this is the good scaling with respect to $r$, one can argue as follows: if we take a smooth function $\chi$ supported in $B_{r_g}^g(\rho_0)$, then the function $\chi_r = \chi(\rho_0 + \tfrac{r_g}{r}(\bigcdot - \rho_0))$ is compactly supported in $B_r^g(\rho_0)$ and its {order-$\ell$} derivatives behave indeed as $r^{-\ell}$. The extra factor $r^{-2 \dim \mfd}$ is due to the fact that we want~\eqref{eq:integralequal1} to be true. It is needed in order to compensate for the fact that the integral of $\chi_r$ is of the same order as the {$g$-volume} of $B_r^g(\rho_0)$, namely $r^{2 \dim \mfd}$.
\end{remark}

We introduce spaces of \emph{confined symbols}.

\begin{definition}[Spaces of confined symbols] \label{def:Conf}
Let $g$ be an admissible metric on $T^\star \mfd$, let $r > 0$ and $\rho_0 \in T^\star \mfd$.
We say a smooth function $\psi$ belongs to the class $\Conf_r^g(\rho_0)$ if it satisfies
\begin{equation*}
\forall \ell \in \N, \forall k \in \N, \exists C_{\ell, k} > 0 : \forall \rho \in T^\star \mfd , \qquad
	\abs*{\nabla^k \psi(\rho)}_g
		\le \dfrac{C_{\ell,k}}{\jap{\dist_{g_{\rho_0}^\sympf}\left(\rho, B_r^{g_{\rho_0}}(\rho_0)\right)}^\ell} .
\end{equation*}
For $\ell \in \N$ fixed, the largest of the optimal constants $C_{\ell, k}$, with $k$ ranging in $\{0, 1, \ldots, \ell\}$, is written $\abs{\psi}_{\Conf_r^g(\rho_0)}^{(\ell)}$. These are seminorms which endow the space $\Conf_r^g(\rho_0)$ with a structure of a Fréchet space.
\end{definition}

\begin{remark} \label{rmk:equivalenceSchConf}
It turns out that the spaces $\Conf_r^g(\rho_0)$ coincide with the Schwartz class $\sch(T^\star \mfd)$ as Fréchet spaces. However, the seminorms are designed in a way that quantifies the confinement of symbols around the ball $B_r^g(\rho_0)$ introduced in~\eqref{eq:defball} with respect to the metric $g$.
\end{remark}

\begin{definition}[Uniformly confined family of symbols \--- {\cite[Definition 2.3.14]{Lerner:10}}] \label{def:unifconfinedfamily}
Let $g$ be an admissible metric on $T^\star \mfd$, and let $r_g > 0$ be a slow variation radius of $g$, introduced in Proposition~\ref{prop:improvedadmissibility}. We say a family of functions $(\psi_{\rho_0})_{\rho_0 \in T^\star \mfd}$ on $T^\star \mfd$ is a {$g$-uniformly} confined family of symbols if there exists $r \in (0, r_g]$ such that
\begin{equation*}
\forall \ell \in \N, \qquad
	\sup_{\rho_0 \in T^\star \mfd} \abs{\psi_{\rho_0}}_{\Conf_r^g(\rho_0)}^{(\ell)} < \infty .
\end{equation*}
The parameter $r$ is called a confinement radius of the family of symbols $(\psi_{\rho_0})_{\rho_0 \in T^\star \mfd}$.
\end{definition}

\begin{example}
Partitions of unity given by Proposition~\ref{prop:existencepartitionofunity} are instances of {$g$-uniformly} confined families of symbols (confinement is clear since each function has compact support in $B_r^g(\rho_0)$, and the uniform seminorm estimates come from~\eqref{eq:S(1,g)partitionbounded}). 
\end{example}

At the level of the classical dynamics, the confinement radius of a confined symbol is expected to grow exponentially in time under the action of the Hamiltonian flow. This is the reason why, given a radius $r_0 > 0$, we introduce
\begin{equation} \label{eq:defr(t)}
r(t)
	:= r_0 \e^{(2 (\Lambda + \Upsilon) + C_g^3 C_p \udl{\gain}_g) \abs{t}} ,
		\quad t \in \R , \qquad
C_p := \abs{\nabla^3 p}_{S((\gain_g/\udl{\gain}_g)^{-1}, g)}^{(0)} ,
\end{equation}
where $C_g$ is a slow variation constant from Proposition~\ref{prop:improvedadmissibility}. This particular definition is motivated by Proposition~\ref{prop:Lipschitzproperty}.
To make the analysis work, we need to ensure that the condition
\begin{equation} \label{eq:conditionr(t)}
r(t)
	\le r_g
		\qquad \Longleftrightarrow \qquad
\abs{t}
	\le \dfrac{\log (r_g/r_0)}{2 (\Lambda + \Upsilon) + C_g^3 C_p \udl{\gain}_g}
\end{equation}
is satisfied.
We shall always assume that $r_0$ is chosen in such a way that~\eqref{eq:conditionr(t)} is fulfilled on the time range under consideration (typically $\abs{t} \le T_E$). With this definition, the inclusion
\begin{equation} \label{eq:monotonicityConf}
\Conf_{r_0}^g(\rho_0)
	\xlongrightarrow{} \Conf_{r(t)}^{g(t)}(\rho_0)
\end{equation}
is ``$1$-Lipschitz" for any $t \in \R$ and $\rho_0 \in T^\star \mfd$, in the sense that
\begin{equation*}
\forall \ell \in \N , \forall \psi \in \Conf_{r_0}^g(\rho_0), \qquad
	\abs*{\psi}_{\Conf_{r(t)}^{g(t)}(\rho_0)}^{(\ell)}
		\le \abs*{\psi}_{\Conf_{r_0}^g(\rho_0)}^{(\ell)} .
\end{equation*}
This follows from Definition~\ref{def:Conf} and~\eqref{eq:conditionr(t)} (use the fact that $g^\sympf \ge g(t)^\sympf$ together with $B_{r_0}^g(\rho_0) \subset B_{r(t)}^{g(t)}(\rho_0)$).

The following result says that the quantum evolution of a {$g$-uniformly} confined family of symbols remains {$g$-uniformly confined} for times not exceeding a fraction of the Ehrenfest time~\eqref{eq:defTE}.

\begin{customthm}{II}[Quantum evolution of {$g$-uniformly} confined families of symbols] \label{thm:partition}
Let $g$ be an admissible metric, let $p$ satisfy Assumption~\ref{assum:mandatory}, and assume $(g, p)$ satisfies Assumption~\ref{assum:p}. Let $T \in [0, \frac{1}{2} T_E]$ and $r_0 > 0$ such that~$r(t)$ defined in~\eqref{eq:defr(t)} satisfies $r(T) \le r_g$.
Then for any $\ell \in \N$, there exist $k \in \N$ and a constant $C_\ell > 0$ such that
\begin{equation} \label{eq:seminormestimateegorovconfined}
\forall t \in [-T, T], \forall \rho_0 \in T^\star \mfd, \forall \psi \in \Conf_{r_0}^g(\rho_0) , \qquad
	\abs*{\e^{t \cal{H}_p} \psi}_{\Conf_{r(t)}^{g(t)}(\phi^{-t}(\rho_0))}^{(\ell)}
		\le C_\ell \abs*{\psi}_{\Conf_{r_0}^g(\rho_0)}^{(k)} .
\end{equation}
Let $(\psi_{\rho_0})_{\rho_0 \in T^\star \mfd}$ be a {$g$-uniformly} confined family of symbols with radius $r_0$ and define
\begin{equation*} \label{eq:quantumevolutionucfs}
\psi_{\rho_0}^t
	:= \e^{t \cal{H}_p} \psi_{\phi^t(\rho_0)} ,
\end{equation*}
in order that
\begin{equation*}
\e^{\ii t P} \Opw{\psi_{\rho_0}} \e^{- \ii t P}
	= \Opw{\e^{t \cal{H}_p} \psi_{\rho_0}}
	= \Opw{\psi_{\phi^{-t}(\rho_0)}^t} ,
		\qquad \forall \rho_0 \in T^\star \mfd, \forall t \in \R .
\end{equation*}
Then for any $t \in [-T, T]$, the family of symbols $(\psi_{\rho_0}^t)_{\rho_0}$ is {$g(t)$-uniformly} confined with radius $r(t)$, and we have
\begin{equation*}
\forall \ell \in \N, \exists k \in \N, \exists C_\ell > 0 : \forall t \in [-T, T] , \qquad
	\sup_{\rho_0 \in T^\star \mfd} \abs*{\psi_{\rho_0}^t}_{\Conf_{r(t)}^{g(t)}(\rho_0)}^{(\ell)}
		\le C_\ell \sup_{\rho_0 \in T^\star \mfd} \abs*{\psi_{\rho_0}}_{\Conf_{r_0}^g(\rho_0)}^{(k)} .
\end{equation*}
All the seminorm indices and constants in the seminorm estimates~\eqref{eq:seminormestimateegorovconfined} depend on $g$, $m$ and $p$ only through structure constants of $g$ and $m$, seminorms of $p$ in the symbol classes~\eqref{eq:symbolclassesnablap3}, the constant $C_\Upsilon$ in~\eqref{eq:constantsmetricalongtheflow} and on any constant $c$ such that $\Lambda \ge c \udl{\gain}_g$.\footnote{The dependence on $c$ degenerates as $c \to 0$. See Section~\ref{subsubsec:semiclassicalsetting} for more details.}
\end{customthm}

\begin{remark}
The asymptotic expansion~\eqref{eq:asymptoticexpansion} is also valid in spaces of confined symbols. In particular, we have
\begin{equation*}
\e^{t \cal{H}_p} \psi_{\rho_0}
	= \e^{t H_p} \psi_{\rho_0} + \tilde \psi_{\rho_0, t}
\end{equation*}
where $\tilde \psi_{\rho_0, t}$ belongs to $\gain_{g(t)}^2 \Conf_{r(t)}^{g(t)}(\phi^{-t}(\rho_0))$. In other words, $\psi_{\phi^{-t}(\rho_0)}^t$ is approximated at leading order by $\psi_{\rho_0} \circ \phi^t$.
\end{remark}

\begin{remark} \label{rmk:positiveTE}
In the case where $T_E(g) = 0$, one could rather consider the metric $\tilde g := \e^{-2T} g$ instead of $g$, for some fixed $T > 0$, in order to have a positive Ehrenfest time. Indeed, one can check that the metric~$\tilde g$ is admissible with the same structure constants as $g$ (except the slow variation which reads $r_{\tilde g} = \e^{-T} r_g$). In addition, Assumption~\ref{assum:p} is satisfied, with the same values of $\Lambda$ and $\Upsilon$ as for~$g$, so that
\begin{equation*}
T_E(\tilde g)
	= T_E(g) + \dfrac{T}{\Lambda + 2 \Upsilon} .
\end{equation*}
All the seminorms built with the metric $\tilde g$ are then equivalent to those defined with~$g$. One could have stated Theorem~\ref{thm:partition} on a time interval $[0, T_0 + \frac{1}{2} T_E]$ instead $[0, \frac{1}{2} T_E]$, as we do for Theorem~\ref{thm:main}, but we chose not to do so to simplify the statement, specifically concerning the growth of the confinement radius~$r(t)$, which may depend on the scaling factor~$\e^{-2T}$. This is not a problem provided we do not seek to let~$T$ go to infinity.
\end{remark}

Several consequences can be drawn from Theorem~\ref{thm:partition}. One is the fact that the quantum (and also the classical) dynamics act continuously on $\sch(T^\star \mfd)$, so that it extends to $\sch'(T^\star \mfd)$.

\begin{corollary} \label{cor:contonshwartz}
Let $p$ satisfy Assumption~\ref{assum:mandatory} and suppose there exists a metric $g$ such that $(p, g)$ satisfies Assumption~\ref{assum:p}. Then the unitary groups $(\e^{t \cal{H}_p})_{t \in \R}$ and $(\e^{t H_p})_{t \in \R}$ act continuously on $\sch(T^\star \mfd)$ and extend continuously to $\sch'(T^\star \mfd)$.
\end{corollary}

Corollary~\ref{cor:contonshwartz} follows directly from Theorem~\ref{thm:partition} together with Remark~\ref{rmk:equivalenceSchConf} for $(\e^{\cal{H}_p t})_{t \in \R}$, and from Proposition~\ref{prop:continuityflowconf} together with Remark~\ref{rmk:equivalenceSchConf} for $(\e^{H_p t})_{t \in \R}$.
A less evident consequence of Theorem~\ref{thm:partition} is the fact that the Schrödinger propagator itself acts continuously on the Schwartz class, and can thus be extended to a continuous operator on $\sch'(\mfd)$.

\begin{corollary} \label{cor:continuitypropagatorSchwartz}
Let $p$ satisfy Assumption~\ref{assum:mandatory} and suppose there exists a metric $g$ such that $(p, g)$ satisfies Assumption~\ref{assum:p}. Then for any $t \in \R$, the propagator $\e^{- \ii t P}$ maps $\sch(\mfd)$ to itself continuously.
\end{corollary}

The latter corollary is proved in Section~\ref{subsec:continuityschclass}. Notice that the conclusion relies only on the existence of an admissible metric $g$ which is compatible with $p$ in the sense of Assumption~\ref{assum:p}, whatever the specific properties of this metric are.

Another consequence is the Egorov theorem in general symbols classes $S(m, g)$ stated in Theorem~\ref{thm:main}.
The proof of Theorem~\ref{thm:main} from Theorem~\ref{thm:partition} relies on the fact that any symbol in $S(m, g)$ can be decomposed as a superposition of {$g$-confined} symbols thanks to a partition of unity adapted to $g$. Conversely, a superposition of a {$g$-uniformly} confined family of symbols, weighted according to some admissible weight $m$, gives a symbol in $S(m, g)$. See Proposition~\ref{prop:reconstructsymbol}.
Proofs can be found in Section~\ref{sec:consequences}.

We finish with an important remark concerning the semiclassical regime.

\begin{remark}[Semiclassical regime $\udl{\gain}_g \to 0$]
We draw the reader's attention to the fact that Theorem~\ref{thm:partition} is mostly relevant for fixed bounded times $T > 0$, independent of $\udl{\gain}_g$, while considering a family of metrics $g = (g_\hslash)_{\hslash \in (0, 1]}$ with $\udl{\gain}_{g_\hslash} = \hslash \to 0$. Indeed, pushing $T = T_\hslash$ up to a fraction of $\log \hslash^{-1}$ would force us to consider a very small confinement radius $r_0 = r_0(\hslash)$, of order $\hslash^\beta$ for some $\beta > 0$, so that the requirement $r(T) \le r_g$ is fulfilled.

A {$g_\hslash$-partition} of unity $(\psi_{\rho_0, \hslash})_{\rho_0 \in T^\star \mfd}$ with such a confinement radius would then consist of confined symbols with roughly $\abs{\nabla^\ell \psi_\hslash}_g = O(\hslash^{- \beta \ell -\beta \dim T^\star \mfd})$ (the factor involving the dimension of $T^\star \mfd$ ensures that the partition integrates to~$1$ as in~\eqref{eq:integralequal1}). That means that seminorms in spaces of confined symbols would blow up as negative powers of $\hslash$.

Although this could seem to be troublesome, we will be able to deduce Theorem~\ref{thm:main} from the above Theorem~\ref{thm:partition} up to a fraction of the Ehrenfest time as $\hslash$ goes to zero by considering the expansion~\eqref{eq:asymptoticexpansion} at a sufficiently high order to cancel the negative powers of $\udl{\gain}_g$ that arise from~\eqref{eq:r-dependentestimate}. See Step~$2$ of the proof of Theorem~\ref{thm:main} in Section~\ref{sec:consequences}.
\end{remark}

\subsection{Idea of proof of Theorem~\ref{thm:main}}

Studying the quantum dynamics $(\e^{t \cal{H}_p})_{t \in \R}$ amounts to solving the equation~\eqref{eq:pdequantumdyn}. The latter can be solved in $L^2(T^\star \mfd)$ by classical semi-group arguments as we shall see in Section~\ref{subsec:quantumwellposedness}. However Theorems~\ref{thm:main} and~\ref{thm:partition} boil down to solve~\eqref{eq:pdequantumdyn} in symbol classes and spaces of confined symbols, which turns out to be much more difficult. Indeed, the only a priori information that we have on the quantum dynamics is that it is a unitary group on $L^2(T^\star \mfd)$ (or equivalently $\e^{-\ii t P}$ is a unitary group acting on $L^2(\mfd)$).

The strategy of our proof consists in showing that $\e^{\ii t P} \Opw{a} \e^{-\ii t P} = \Opw{\e^{t \cal{H}_p} a}$ is a pseudo-differential operator through the characterization known as Beals' theorem (see Proposition~\ref{prop:Beals}). Applying this criterion requires quite intricate computations involving iterated commutators with operators of the form $\Opw{f}$ where $f \in \aff(T^\star \mfd)$ is an affine function.

A very delicate point, especially in the proof of Theorem~\ref{thm:partition}, is to care about the dependence on $\rho_0$ in all the computations. The dependence on $t$ is also to be taken into account carefully. In particular, when applying the pseudo-differential calculus, we will always make sure that the metrics $g(t)$ and weight $m(t)$ involved have structure constants bounded independently of $t$. Fortunately, the Weyl--Hörmander theory is well-suited to handle this uniformity question.

A key step of our approach consists in describing the action of $\e^{t H_p}$ and $\cal{H}_p^{(3)}$, arising in the definition of $\cal{E}_j(t)$ and $\widehat{\cal{E}}_j(t)$ (defined in~\eqref{eq:defEj} and~\eqref{eq:defhatEj}), on symbol classes $S(m, g)$. Let us sketch here some important arguments that will be used throughout the proofs. On the one hand, we show that the Hamiltonian flow $\e^{t H_p}$ acts continuously on symbol classes as follows:
\begin{equation} \label{eq:continuityHamflow}
\e^{t H_p} : S(m, g)
	\longrightarrow S\left(m(t), g(t)\right) ,
\end{equation}
where $m(t) = \e^{t H_p} m$ and $g(t)$ is defined in~\eqref{eq:defg(t)}. The corresponding seminorm estimates are uniform in $t \in \R$:
\begin{multline*}
\forall \ell \in \N, \exists C_\ell > 0 : \forall a \in S(m, g), \forall t \in \R, \forall \rho \in T^\star \mfd , \qquad \\
	\abs*{\nabla^\ell \e^{t H_p} a}_{g(t)}(\rho)
		= \abs*{\nabla^\ell \e^{t H_p} a}_g(\rho) \e^{- \ell (\Lambda + 2 \Upsilon) \abs{t}}
		\le C_\ell m(t; \rho) \max_{0 \le j \le \ell} \sup_{T^\star \mfd} \dfrac{\abs{\nabla^j a}_g}{m} .
\end{multline*}
The continuity of the operator $\e^{t H_p}$ in~\eqref{eq:continuityHamflow} is a concise way to describe the exponential growth of derivatives of the flow:
\begin{equation*}
\forall k \in \N^\ast , \exists C_k > 0 : \qquad
	\abs*{\nabla^k \phi^t}_g
		\le C_k \e^{k \Lambda \abs{t}} ,
\end{equation*}
where $\Lambda$ is the Lyapunov exponent defined in Item~\ref{it:Lambda} of Assumption~\ref{assum:p}. On the other hand, we prove that the operator $\cal{H}_p^{(3)}$ acts continuously on symbol classes as follows:
\begin{equation} \label{eq:continuitycalHp3inintrosketch}
\cal{H}_p^{(3)} : S\left(m, g(s)\right) \longrightarrow S\left( \gain_{g(s)}^2 \udl{\gain}_g \e^{- (\Lambda + 2 \Upsilon) \abs{s}} m, g(s) \right) ,
\end{equation}
uniformly in $\abs{s} \le T_E$. The limitation on $s$ is due to the use of pseudo-differential calculus to establish seminorm estimates corresponding to~\eqref{eq:continuitycalHp3inintrosketch}. Let us explain why such a mapping property~\eqref{eq:continuitycalHp3inintrosketch} holds. The operator $\cal{H}_p^{(3)}$ is related to the remainder of order $3$ in the pseudo-differential calculus with the symbol $p$ (we recall that the definition of $\cal{H}_p^{(3)}$ in~\eqref{eq:defHp3} involves $\cal{H}_p$ defined in~\eqref{eq:defcalHp}, hence Moyal products with $p$). If we consider a symbol $a \in S(m, g(s))$, computing $\cal{H}_p a$ amounts to applying pseudo-differential calculus with $a \in S(m, g(s))$ and $\nabla^3 p \in S((\gain_g/\udl{\gain}_g)^{-1}, g)$ (Item~\ref{it:strongsubquad} of Assumption~\ref{assum:p}). The gain of the associated pseudo-differential calculus corresponds to
\begin{equation*}
\sqrt{\gain_g \gain_{g(s)}}
	= \sqrt{\gain_g^2 \e^{2 (\Lambda + 2 \Upsilon) \abs{s}}}
	= \gain_g \e^{(\Lambda + 2 \Upsilon) \abs{s}}
\end{equation*}
(we have used the fact that $\gain_{g(s)} = \e^{2 (\Lambda + 2 \Upsilon) \abs{s}} \gain_g$; see Remark~\ref{rmk:uniformadmissibilityg(t)}). Therefore at order $3$, we ``gain"
\begin{equation*}
\gain_g^3 \e^{3 (\Lambda + 2 \Upsilon) \abs{s}}
	= \gain_{g(s)}^2 \gain_g \e^{- (\Lambda + 2 \Upsilon) \abs{s}} .
\end{equation*}
Mutliplying by the weights $(\gain_g/\udl{\gain}_g)^{-1}$ and $m$ associated with $\nabla^3 p$ and $a$ respectively, pseudo-differential calculus implies that $\cal{H}_p^{(3)} a$ belongs to the class $S(m', g(s))$, where
\begin{equation*}
m'
	= \gain_{g(s)}^2 \gain_g \e^{- (\Lambda + 2 \Upsilon) \abs{s}} \left(\dfrac{\gain_g}{\udl{\gain}_g}\right)^{-1} m
	= \gain_{g(s)}^2 \e^{- (\Lambda + 2 \Upsilon) \abs{s}} \udl{\gain}_g m .
\end{equation*}
Then we prove the estimate~\eqref{eq:estimatecalEj0} on the operator $\cal{E}_j(t)$ defined in~\eqref{eq:defEj} by induction on $j$, taking advantage of the recurrence relation~\eqref{eq:recurrencerelation}. If we assume that~\eqref{eq:estimatecalEj0} is true for $\cal{E}_j(t)$, then composing this estimate with those obtained on $\e^{t H_p}$ in~\eqref{eq:continuityHamflow} and on $\cal{H}_p^{(3)}$ in~\eqref{eq:continuitycalHp3inintrosketch}, we obtain
\begin{equation*}
\e^{(t - s) H_p} \cal{H}_p^{(3)} \cal{E}_{j}(s) : S(m, g) \longrightarrow S\left( m(t) \gain_{g(t)}^{2(j+1)} \udl{\gain}_g \e^{-(\Lambda + 2 \Upsilon) \abs{s}}, g(t) \right) .
\end{equation*}
The factor $\udl{\gain}_g \e^{-(\Lambda + 2 \Upsilon) \abs{s}}$ disappears while integrating over $t \in [0, t]$ (we use $\Lambda \ge c \udl{\gain}_g$ here), and we deduce that the estimate~\eqref{eq:estimatecalEj0} holds for $\cal{E}_{j+1}(t)$.

The most delicate part of the proof of Theorem~\ref{thm:main} consists in showing the estimate~\eqref{eq:estimateremainder} on the remainder of the Dyson expansion, i.e.\ the operator $\widehat{\cal{E}}_j(t)$ defined in~\eqref{eq:defhatEj}. An additional difficulty is due to the fact that the operator $\e^{t \cal{H}_p}$, which appears in the definition of $\widehat{\cal{E}}_j(t)$, is defined only implicitly by~\eqref{eq:implicitdef} (whereas the operator $\e^{t H_p}$, involved in the definition of $\cal{E}_j(t)$, is more tractable since it corresponds to an explicit operation on symbols, namely the composition by the Hamiltonian flow). This step of the proof relies on Beals' characterization of pseudo-differential operators~\cite{Beals:77}. This is where the assumption on $\nabla^3 p$ involving the temperance weight (Item~\ref{it:strongsubquad} of Assumption~\ref{assum:p}), as well as Item~\ref{it:metriccontrol} of Assumption~\ref{assum:p}, come into play.

\subsection{Related works: the contributions of J.-M.\ Bony} \label{subsec:Bony}

Our results are in line with previous investigations conducted by Bony in the late 1990s and the 2000s. His contributions essentially consist in generalizing the Fourier integral operator theory to the framework of the Weyl--Hörmander calculus. In a series of works that we review below, he introduces an efficient algebraic approach to Fourier integral operators.

Let $\kappa$ be a canonical transformation and $g_0, g_1$ be two Weyl--Hörmander symplectic metrics subject to $g_0 = \kappa^\pullb g_1$. In the seminar notes~\cite{Bony:FIOX,Bony:FIOandWeyl}, Bony introduces classes of Fourier integral operators $\rm{FIO}(\kappa, g_0, g_1)$ as superpositions of metaplectic operators (quantizations of the tangent map of $\kappa$ at each phase space point), weighted by {$g_0$-confined} symbols. Later, he gives in~\cite{Bony:FIOandWeyl2} an alternative definition of the classes $\rm{FIO}(\kappa, g_0, g_1)$ for more general metrics. This definition involves ``twisted commutators" and relies on delicate characterizations of pseudo-differential operators in Weyl--Hörmander classes~\cite{Bony:characterization}, that generalize Beals' criterion~\cite{Beals:77,Beals:77correction} to non-Euclidean metrics. With this definition, the fact that Fourier integral operators conjugate operators with symbol in $S(1, g_0)$ to operators with symbol in $S(1, g_1)$ becomes practically tautological. This axiomatic approach allows him to check that these classes of Fourier integral operators obey a natural calculus. However, these characterizations of pseudo-differential operators do not go along with precise estimates and require some additional assumptions on the metrics under consideration (geodesic temperance and above all absence of symplectic eccentricity). In later conference proceedings~\cite{Bony:evol,Bony:evolfrancais,Bony:generalizedFIO}, Bony shows the equivalence between the previous definitions of the class $\rm{FIO}(\kappa, g_0, g_1)$, introduces an abstract notion of principal symbol and discusses boundedness properties of these operators in Sobolev spaces attached to a Weyl--Hörmander metric. He also points out that the case of propagators of the form $\e^{-\ii t P}$ should fit in this framework, taking $\kappa = \kappa_t$ the Hamiltonian flow associated with the generator $P$ and $g_t = (\phi^t)_\pullb g_0$. Nevertheless, checking that a concrete unitary group $\e^{- \ii t P}$, with simple assumptions on $P$, actually belongs to a class $\rm{FIO}(\phi^t, g_0, g_t)$ is a highly non-trivial task. In the present paper, we propose instead in Theorem~\ref{thm:main} a statement with explicit assumptions on the generator $P$ and on the metric $g$, and provide with precise continuity estimates and detailed proofs. We also relax some of the technical assumptions of~\cite{Bony:evol,Bony:evolfrancais,Bony:generalizedFIO}.

\subsection{Plan of the article} \label{subsec:strategy}

The article is organized as follows.
\begin{itemize}
\item Section~\ref{sec:WHcalculus} introduces basic notation and the so-called pseudo-differential Weyl--Hörmander calculus. Our main reference for this is the treatise of Lerner~\cite[Chapter 2]{Lerner:10}, but our presentation is also inspired from Hörmander~\cite[Chapter XVIII]{Hoermander:V3}. Proofs of this section are collected in Appendix~\ref{app:Moyal}. We chose to redo some of the proofs to obtain more precise seminorm estimates needed in view of~Theorem~\ref{thm:main}.
\item Then in Section~\ref{sec:wellposedness} we discuss the well-posedness of the classical and quantum dynamics on $L^2(T^\star \mfd)$. We show that Assumption~\ref{assum:p} is sufficient to make sense of the classical and quantum dynamics globally in time. This part is not related to microlocal analysis but follows from classical evolution equations and spectral theory arguments.
\item Next we study the classical dynamics in Section~\ref{sec:classical}. We essentially prove global estimates on the Hamiltonian flow and its derivatives on the whole phase space. The content of this section is quite classical but we chose to include the proofs with our notation to make the paper self-contained.
\item In Section~\ref{sec:mapping}, we discuss the mapping properties in symbol classes and spaces of confined symbols of the operators $\cal{E}_j(t)$ appearing in the Dyson expansion~\eqref{eq:Dyson}. More precisely, we prove continuity estimates for these operators, taking care of the dependence on time and on the parameter $\rho_0 \in T^\star \mfd$ for confined families of symbols.
\item In Section~\ref{sec:quantum} we deal with the delicate computations involving commutators, in preparation for the application of Beals' theorem. We discuss the boundedness of operators given by iterated commutators with affine symbols. This section contains the key arguments to prove smoothness and decay of the conjugated operator $\e^{\ii t P} \Opw{\psi_{\rho_0}} \e^{- \ii t P}$ in Theorem~\ref{thm:partition}.
\item The proof of Theorem~\ref{thm:partition} on the propagation of partitions of unity is presented in Section~\ref{sec:proofpartition}. Roughly speaking, it consists in putting together the estimates given in Section~\ref{sec:quantum} related to iterated commutators, taking care of the dependence on the parameters $t$ and $\rho_0$. Then the sought seminorm estimates~\eqref{eq:seminormestimateegorovconfined} follow from Beals' theorem (Proposition~\ref{prop:Beals}).
\item Then Section~\ref{sec:consequences} is devoted to the proof of Theorem~\ref{thm:main} and several other consequences of Theorem~\ref{thm:partition}.
\item Proofs of the results of Section~\ref{subsec:examples} concerning applications to Schrödinger, wave and transport equations are collected in the final Section~\ref{sec:proofsex}.
\item The paper ends with four appendices. Appendix~\ref{app:Moyal} is devoted to the proofs of the precise estimates for the pseudo-differential calculus in the Weyl--Hörmander framework given in Section~\ref{sec:WHcalculus}. Then basic results on pseudo-differential operators are recalled in Appendix~\ref{app:operators}, and some technical lemmata on phase space metrics are gathered in Appendix~\ref{app:metrics}. We finally recall the Faà di Bruno formula for ``vector-valued" functions in Appendix~\ref{app:faadibruno}.
\end{itemize}

\subsection*{Acknowledgments} I am grateful to Matthieu L\'{e}autaud for carefully reading an early version of this article and suggesting countless improvements. His advice and encouragement were extremely helpful.
Most of this project has been completed while affiliated with Laboratoire de Math\'{e}matiques d'Orsay, Universit\'{e} Paris-Saclay, France. I thank this institution for the outstanding mathematical environment and working conditions I enjoyed there.

\Large
\section{The Weyl--Hörmander calculus} \label{sec:WHcalculus}
\normalsize

In this section, we introduce the Weyl--Hörmander calculus of pseudo-differential operators, that is to say we describe how the composition of these operators works in the Weyl--Hörmander symbol classes and spaces of confined symbols.

\subsection{Functional framework} \label{subsec:functionalframework}

Recall that, for admissible~$g$ and~$m$, the symbol classes $S(m, g)$ introduced in Definition~\ref{def:symbclasses} are Fréchet spaces contained in $\sch'(T^\star \mfd)$. Also recall that we extended the definition of $S(m, g)$ to tensors in Remark~\ref{rmk:symbolclasstensors}. For any $k \in \N$, we define
\begin{equation*}
\nabla^{-k} S(m, g)
	= \set{a \in \sch'(T^\star \mfd)}{\nabla^k a \in S(m, g)} .
\end{equation*}
Notice that similar spaces were already considered by other authors, such as Bony~\cite[Definition 2.5]{Bony:characterization}.
This is a linear subspace of $\sch'(T^\star \mfd)$. We endow it with the following family of seminorms:
\begin{equation*}
a \longmapsto \abs*{\nabla^k a}_{S(m, g)}^{(\ell)} ,
		\qquad \ell \in \N, a \in \nabla^{-k} S(m, g) .
\end{equation*}
The topology induced by these seminorms coincides with the topology induced by $S(m, g)$ through the linear map
\begin{equation} \label{eq:mapnablak}
\nabla^k : \nabla^{-k} S(m, g) \longrightarrow S(m, g) ,
\end{equation}
which is then continuous. Notice that $\nabla^{-k} S(m, g)$ is not a topological vector space (singletons are not closed), so in particular it is not a Fréchet space. Open sets of $\nabla^{-k} S(m, g)$ are of the form $(\nabla^k)^{-1} U$ where $U$ is an open set of $S(m, g)$. Also notice that the ``fiber" of the map $\nabla^k$ in~\eqref{eq:mapnablak} is given by $\ker \nabla^k$, that is the set of degree $k-1$ polynomial functions, in the sense that if $\nabla^k a_1 = \nabla^k a_2$, then~$a_1$ and~$a_2$ differ by a degree $k - 1$ polynomial function.
Given a linear map $L : \nabla^{-k_1} S(m_1, g_1) \to \nabla^{-k_2} S(m_2, g_2)$,
if we have an estimate of the form
\begin{equation*}
\forall \ell \in \N, \exists k \in \N, \exists C_\ell > 0 : \forall a \in S(m_1, g) , \qquad
	\abs*{\nabla^{k_2} L a}_{S(m_2, g_2)}^{(\ell)}
		\le C_\ell \abs*{\nabla^{k_1} a}_{S(m_1, g_1)}^{(k)} ,
\end{equation*}
then we shall write
\begin{equation*}
\nabla^{-k_1} S(m_1, g_1)
	\xrightarrow{L} \nabla^{-k_2} S(m_2, g_2) .
\end{equation*}
One readily checks that
\begin{equation*}
\forall k_0 \in \{0, 1, 2, \ldots, k\} , \qquad
	\nabla^{-k} S(m, g)
		\xrightarrow{\nabla^{k_0}} \nabla^{-(k - k_0)} S(m, g) .
\end{equation*}
When considering a family of operators $L(t)$ depending on some parameter $t \in \R$ (usually time), it is interesting to keep track of the dependence of the continuity constants $C_\ell$ on $t$. Given an interval $I \subset \R$ and $C \in \cont^0(I; \R_+)$, we will write
\begin{equation*}
\nabla^{-k_1} S(m_1, g_1)
	\xrightarrow{L(t)} C(t) \nabla^{-k_2} S(m_2, g_2) ,
		\qquad t \in I ,
\end{equation*}
to mean
\begin{equation*}
\forall \ell \in \N, \exists k \in \N, \exists C_\ell > 0 : \forall t \in \R, \forall a \in S(m_1, g_1) , \qquad
	\abs*{\nabla^{k_2} L a}_{S(m_2, g_2)}^{(\ell)}
		\le C_\ell C(t) \abs*{\nabla^{k_1} a}_{S(m_1, g_1)} .
\end{equation*}
In particular, we shall write
\begin{equation*}
\nabla^{-k_1} S(m_1, g_1)
	\xrightarrow{L(t)} \nabla^{-k_2} S(m_2, g_2)
\end{equation*}
when seminorm estimates are uniform with respect to $t$. We also extend this notation accordingly to the case where weights $m_1, m_2$ and metrics $g_1, g_2$ depend on the parameter~$t$.

\subsection{Pseudo-differential calculus} \label{subsec:pseudocalcWH}

Quantization of quantum observables follows well-known rules. Let $a_1$ and $a_2$ be tempered distributions. The composition $\Opw{a_1} \Opw{a_2}$ makes sense as a continuous operator $\sch(\R^d) \to \sch'(\R^d)$ as soon as one of the two operators maps $\sch(\R^d)$ to itself continuously. In such a case, the Weyl symbol $b$ of the corresponding operator $\Opw{b} = \Opw{a_1} \Opw{a_2}$, given by the Schwartz kernel theorem (Proposition~\ref{prop:SchwartzWeylkernel}), is called the Moyal product of $a_1$ and $a_2$ and is denoted by $a_1 \moyal a_2$, so that
\begin{equation} \label{eq:pseudodefmoyal}
\Opw{a_1} \Opw{a_2} 
	= \Opw{a_1 \moyal a_2} .
\end{equation}
The Moyal product is a bilinear map. Assume now that $a_1$ and $a_2$ are Schwartz functions. Then one has $a_1 \moyal a_2 \in \sch(T^\star \mfd)$~\cite[Theorem 4.11]{Zworski:book}, and this symbol can be computed according to the formula~\eqref{eq:formulaMoyalproduct}.

The proposition below ensures that the composition of $\Opw{a_1}$ and $\Opw{a_2}$ makes sense as a continuous map $\sch(\mfd) \to \sch'(\mfd)$ as soon as $a_1$ or $a_2$ belongs to some symbol class $S(m, g)$.

\begin{proposition}[{\cite[Theorem 18.6.2]{Hoermander:V3}}] \label{prop:continuityinschpseudodiff}
Let $g$ be an admissible metric and $m$ be a $g$-admissible weight. Then for any $a \in S(m, g)$, the operator $\Opw{a}$ maps $\sch(M) \to \sch(M)$ and $\sch'(M) \to \sch'(M)$ continuously.
\end{proposition}

Notice that if $a_1 \in S(m_1, g_1)$ and $a_2 \in S(m_2, g_2)$, with admissible metrics and weights, then~\eqref{eq:pseudodefmoyal} and~\eqref{eq:formulaMoyalproduct} hold, and~\eqref{eq:formulaMoyalproduct} can be understood as an oscillatory integral.

The formula~\eqref{eq:formulaMoyalproduct} for the Moyal product can be rewritten
\begin{equation} \label{eq:exprMoyal}
a_1 \moyal a_2
	= \e^{- \frac{\ii}{2} \frak{P}} (a_1 \otimes a_2)_{\vert \diag} .
\end{equation}
In this expression, $a_1 \otimes a_2$ is the map $(\rho_1, \rho_2) \mapsto a_1(\rho_1) a_2(\rho_2)$ defined on $T^\star \mfd \oplus T^\star \mfd$, $\diag$ refers to the restriction to the diagonal $\{\rho_1 = \rho_2\}$ and the operator $\e^{- \frac{\ii}{2} \frak{P}}$ is a Fourier multiplier that one can see as the exponential of $\frak{P} = \poiss*{\nabla^{\rho_1}}{\nabla^{\rho_2}}$, which is a symbolic notation for
\begin{equation*}
\frak{P} (a_1 \otimes a_2)(\rho_1, \rho_2)
	= \sympf\left(H_{a_1}(\rho_1), H_{a_2}(\rho_2)\right) ,
		\qquad \forall \rho_1, \rho_2 \in T^\star \mfd .
\end{equation*}
(See Sections~\ref{subsec:Poissonop} and~\ref{subsec:Moyalop} of Appendix~\ref{app:Moyal} for further details on these definition.) A formal Taylor expansion of the exponential in~\eqref{eq:exprMoyal}, which can be made rigorous, gives
\begin{equation} \label{eq:pdidentity}
a_1 \moyal a_2
	= \sum_{j = 0}^{j_0} \cal{P}_j(a_1, a_2) + \widehat{\cal{P}}_{j_0+1}(a_1, a_2) ,
\end{equation}
where for all $j \ge 0$,
\begin{empheq}[left={\empheqlbrace}]{alignat=2}
\cal{P}_j(a_1, a_2)
	&= \dfrac{(2 \ii)^{-j}}{j!} \frak{P}^j (a_1 \otimes a_2) _{\vert \diag} , \label{eq:defPj}\\
\widehat{\cal{P}}_{j+1}(a_1, a_2)
	&= \int_0^1 (1 - s)^j \dfrac{(2 \ii)^{-(j+1)}}{j!} \e^{- \ii \frac{s}{2} \frak{P}} \frak{P}^{j+1} (a_1 \otimes a_2) _{\vert \diag} \dd s , \label{eq:defhatPj}
\end{empheq}
and $\widehat{\cal{P}}_0(a_1, a_2) = a_1 \moyal a_2$ by convention. Throughout the paper, $\cal{P}_j(a_1, a_2)$ is called the ``$j$-th order term" of the expansion.

The operators $(\cal{P}_j)_{j \in \N}$ and $(\widehat{\cal{P}}_j)_{j \in \N}$ are bilinear maps acting on the space of Schwartz symbols. The most important ones are the {$0$th} order and the {$1$st} order terms in the expansion, namely
\begin{equation*}
\cal{P}_0(a_1, a_2)
	= a_1 a_2 ,
		\qquad
\cal{P}_1(a_1, a_2)
	= \dfrac{1}{2 \ii} \poiss*{a_1}{a_2} .
\end{equation*}
The term $\cal{P}_j$ depends on derivatives of order $j$ of $a_1$ and $a_2$, while $\widehat{\cal{P}}_j$ depends on derivatives of order larger than $j$. When $a_1$ and $a_2$ belong to symbol classes $S(m, g)$, we can make sense of the asymptotic expansion
\begin{equation} \label{eq:expansionpseudocalc}
a_1 \moyal a_2
	\sim \sum_{j \ge 0} \cal{P}_j(a_1, a_2) ,
		\qquad a_1, a_2 \in \sch(T^\star \mfd) ,
\end{equation}
by describing how the bilinear operators $\cal{P}_j$ and $\widehat{\cal{P}}_j$ extend to symbol classes or spaces of confined symbols.
In Proposition~\ref{prop:pseudocalcsymb} below, we describe the mapping properties of these bilinear operators on symbol classes $S(m_1, g_1)$ and $S(m_2, g_2)$, with precise seminorm estimates. Derivatives of $a_1$ and $a_2$ are measured with respect to two possibly different metrics $g_1$ and $g_2$. Beyond admissibility, some compatibility condition is required on these metrics. We introduce this notion of compatibility in Definition~\ref{def:compatiblemetrics}, based on~\cite[Proposition 18.5.3]{Hoermander:V3}.

An important object appearing in the statements below is the \emph{joint gain function} of $g_1$ and $g_2$, defined by\footnote{The justification for the equality in~\eqref{eq:defjointgain} can be found in Lemma~\ref{lem:justificationequalities}.}
\begin{equation} \label{eq:defjointgain}
\forall \rho \in T^\star \mfd , \qquad
	\gain_{g_1, g_2}(\rho)
		:= \sup_{\zeta \in W \setminus \{0\}} \dfrac{\abs{\zeta}_{g_{1, \rho}}}{\abs{\zeta}_{g_{2, \rho}^\sympf}}
		= \sup_{\zeta \in W \setminus \{0\}} \dfrac{\abs{\zeta}_{g_{2, \rho}}}{\abs{\zeta}_{g_{1, \rho}^\sympf}} .
\end{equation}
Notice that when $g_1 = g_2 = g$, one recovers the usual gain function of $g$ of Definition~\ref{def:gaing}.\footnote{Observe that the definition of the temperance weight $\theta_g$ in~\eqref{eq:deftheta} goes somewhat in the other way around compared with the joint gain function $\gain_{g, {\sf g}}$ of $g$ and ${\sf g}$.}

\begin{proposition}[Pseudo-differential calculus in symbol classes] \label{prop:pseudocalcsymb}
Let $g_1$ and $g_2$ be admissible compatible metrics in the sense of Definitions~\ref{def:admissiblemetric} and~\ref{def:compatiblemetrics}. Let $m_1$ and $m_2$ be $(g_1 + g_2)$-admissible weights in the sense of Definition~\ref{def:admissibleweight}. Then we have for all $j \in \N$:
\begin{align}
\nabla^{-j} S\left(m_1, g_1\right) \times \nabla^{-j} S\left(m_2, g_2\right)
	&\xrightarrow{\cal{P}_j} S\left(\gain_{g_1, g_2}^j m_1 m_2, g_1 + g_2\right) , \label{eq:continuitycalPjsymb}\\
\nabla^{-j} S\left(m_1, g_1\right) \times \nabla^{-j} S\left(m_2, g_2\right)
	&\xrightarrow{\widehat{\cal{P}}_j} S\left(\gain_{g_1, g_2}^j m_1 m_2, g_1 + g_2\right) .\label{eq:continuityhatcalPjsymb}
\end{align}
More precisely, for all $j \in \N$ and all $\ell \in \N$, there exist $k = k_\ell \in \N$ and positive constants $(C_\ell, C_\ell')$ such that for all $a_1 \in \nabla^{-j} S(m_1, g_1)$ and $a_2 \in \nabla^{-j} S(m_2, g_2)$:
\begin{align*}
\abs*{\cal{P}_j(a_1, a_2)}_{S(\gain_{g_1, g_2}^j m_1 m_2, g_1 + g_2)}^{(\ell)}
	&\le C_\ell \abs*{\nabla^j a_1}_{S(m_1, g_1)}^{(\ell)} \abs*{\nabla^j a_2}_{S(m_2, g_2)}^{(\ell)} , \\
\abs*{\widehat{\cal{P}}_j(a_1, a_2)}_{S(\gain_{g_1, g_2}^j m_1 m_2, g_1 + g_2)}^{(\ell)}
	&\le C_\ell' \abs*{\nabla^j a_1}_{S(m_1, g_1)}^{(k)} \abs*{\nabla^j a_2}_{S(m_2, g_2)}^{(k)} .
\end{align*}
The estimates depend on $g_1, g_2$ and $m_1, m_2$ only through structure constants.
\end{proposition}

\begin{remark}[Pseudo-differential calculus with polynomial symbols] \label{rmk:polynomialpseudocalc}
Proposition~\ref{prop:pseudocalcsymb} has the following consequence: if $a_1$ or $a_2$ is polynomial of degree~$n$, that is to say $\nabla^{n+1} a_1 = 0$ or $\nabla^{n+1} a_2 = 0$, then pseudo-differential calculus is ``exact at order $n$", namely
\begin{equation*}
\cal{P}_j(a_1, a_2) = 0
	\quad \rm{and} \quad
\widehat{\cal{P}}_j(a_1, a_2) = 0 ,
\qquad \forall j > n .
\end{equation*}
\end{remark}

Another version of pseudo-differential calculus consists in pairing a symbol in a Weyl--Hörmander symbol class with a confined symbol.

\begin{proposition}[Pseudo-differential calculus with confined symbols] \label{prop:pseudocalcconf}
Let $g_0$ and $g$ be admissible compatible metrics in the sense of Definitions~\ref{def:admissiblemetric} and~\ref{def:compatiblemetrics}. Suppose in addition that $g \le g_0$ and that~$r_0 \in (0, 1]$ is a common slow variation radius of both metrics. Let $m$ be a {$g_0$-admissible} weight such that $r_0$ is a slow variation radius\footnote{Namely Item~\eqref{it:slowvariationm} of Definition~\ref{def:admissibleweight} holds with this radius.} for $m$. Let $r \in (0, r_0]$ and $\rho_0 \in T^\star \mfd$. Then for all $j \in \N$, the following maps are continuous:
\begin{align}
\nabla^{-{j}} S(m, g) \times \Conf_r^{g_0}(\rho_0)
	&\xlongrightarrow{\cal{P}_{j}} \gain_{g, g_0}^{j}(\rho_0) m(\rho_0) \Conf_r^{g_0}(\rho_0) , \label{eq:continuitycalPj}\\
\nabla^{-{j}} S(m, g) \times \Conf_r^{g_0}(\rho_0)
	&\xlongrightarrow{\widehat{\cal{P}}_{j}} \gain_{g, g_0}^{j}(\rho_0) m(\rho_0) \Conf_r^{g_0}(\rho_0) . \label{eq:continuityhatcalPj}
\end{align}
More precisely, for all $j \in \N$ and all $\ell \in \N$, there exist $k = k_\ell \in \N$ and positive constants $(C_\ell, C_\ell')$ such that for all $a \in \nabla^{-j} S(m, g)$ and $\psi \in \Conf_r^{g_0}(\rho_0)$:
\begin{align*}
\abs*{\cal{P}_j(a, \psi)}_{\Conf_r^{g_0}(\rho_0)}^{(\ell)}
	&\le C_\ell \gain_{g, g_0}^{j}(\rho_0) m(\rho_0) \abs*{\nabla^j a}_{S(m, g)}^{(\ell)} \abs*{\nabla^j \psi}_{\Conf_r^{g_0}(\rho_0)}^{(\ell)} , \\
\abs*{\widehat{\cal{P}}_j(a, \psi)}_{\Conf_r^{g_0}(\rho_0)}^{(\ell)}
	&\le C_\ell' \gain_{g, g_0}^{j}(\rho_0) m(\rho_0) \abs*{\nabla^j a}_{S(m, g)}^{(k)} \abs*{\nabla^j \psi}_{\Conf_ r^{g_0}(\rho_0)}^{(k)} .
\end{align*}
The estimates are uniform with respect to $\rho_0 \in T^\star \mfd$ and $r$, and depend on $g_0, g$ and $m$ only through structure constants.
\end{proposition}

Several subtleties in these statements are not covered by~\cite{Hor:79,Hoermander:V3,Lerner:10}. For this reason we prove Propositions~\ref{prop:pseudocalcsymb} and~\ref{prop:pseudocalcconf} in Appendix~\ref{app:Moyal}, following the strategy of~\cite[Section 2.3]{Lerner:10}.

In practice, $g_1$ and $g_2$ (or $g$ and $g_0$) will always be conformal, so that in particular $g_1^\natural = g_2^\natural$. Compatibility in the sense of Definition~\ref{def:compatiblemetrics} is always true for such metrics as stated in Proposition~\ref{prop:sufficientconditioncompatibility} and Remark~\ref{rmk:sufficientconditioncompatibility}. So we will just ignore this assumption in the sequel. Similarly for weight functions, in the case of conformal metrics, it will be sufficient to check that $m$ is {$g_1$-admissible} or {$g_2$-admissible} in order to ensure that it is {$(g_1 + g_2)$-admissible} (see Proposition~\ref{prop:improvedadmissibility}).

\begin{remark}
We will often apply the results of this section with $(g_1, g_2)$ or $(g, g_0)$ equal to $(g(s), g(t))$ for different times $t$ and $s$ smaller than the Ehrenfest time. Since the structure constants of these metrics are uniform with respect to $t$ (Proposition~\ref{prop:improvedadmissibility} and the subsequent Remark~\ref{rmk:uniformadmissibilityg(t)}), the implicit constants appearing in the continuity estimates will not depend on time.
\end{remark}

\Large
\section{Classical and quantum well-posedness} \label{sec:wellposedness}
\normalsize

In this section, we prove Proposition~\ref{prop:classicalwellposedness} which provides with a sufficient condition for the global existence of the Hamiltonian flow and the essential self-adjointness of $P$.
A first step is the following lemma.

\begin{lemma}[Estimate on the first derivative of $p$] \label{lem:estimatefirstderivativep}
Under the assumptions of Proposition~\ref{prop:classicalwellposedness}, the following holds:
\begin{equation} \label{eq:gradient...}
\exists c > 0 : \forall \rho \in T^\star \mfd , \qquad
	\abs*{\nabla p}_{g_0}(\rho)
		\le c \jap*{p(\rho)} .
\end{equation}
\end{lemma}

\begin{proof}
Recall $r_{g_0}$ and $C_{g_0}$ the slow variation radius and slow variation constant of $g_0$ from Proposition~\ref{prop:improvedadmissibility}. Also recall the real number $E$ from~\eqref{eq:assumeessentialselfadjointness}. For any $\tilde \rho, \tilde \rho' \in T^\star \mfd$, we have by Taylor's theorem:
\begin{equation} \label{eq:Taylorp2}
p(\tilde \rho')
	= p(\tilde \rho) + \nabla p(\tilde \rho) . (\tilde \rho' - \tilde \rho) + \int_0^1 (1 - t) \nabla^2 p\left((1 - t) \tilde \rho + t \tilde \rho'\right). (\tilde \rho' - \tilde \rho)^2 \dd t .
\end{equation}
In the particular case where $\tilde \rho$ is such that $p(\tilde \rho) = E$ and $\abs{\tilde \rho' - \tilde \rho}_{g_{0, \tilde \rho}} \le r_{g_0}$, we have in addition
\begin{align} \label{eq:controlnearnodalset}
\abs*{p(\tilde \rho') - E}
	&\le \abs*{\nabla p}_{g_0}(\tilde \rho) \abs*{\tilde \rho' - \tilde \rho}_{g_{0, \tilde \rho}} + \tfrac{1}{2} \abs*{\nabla^2 p}_{g_0, \infty} \sup_{t \in [0, 1]} \abs*{\tilde \rho' - \tilde \rho}_{g_{0, (1 - t) \tilde \rho + t \tilde \rho'}}^2 \nonumber\\
	&\le r_{g_0} \sup_{\{p = E\}} \abs*{\nabla p}_{g_0} + \tfrac{1}{2} C_{g_0}^2 r_{g_0}^2 \abs*{\nabla^2 p}_{g_0, \infty}
	=: c_0 .
\end{align}
We used slow variation of $g_0$ in the second inequality. Now we let $\rho \in T^\star \mfd$ and we show~\eqref{eq:gradient...}. We distinguish two cases.

\medskip
\emph{First case: $p(\rho) \ge E$.} For any $|s| \le r_{g_0} C_{g_0}^{-1} =: r_{g_0}' (\le r_{g_0})$, we introduce $\rho_s' := \rho + s \zeta$, where $\zeta \in W$ is chosen in such a way that
\begin{equation*}
\abs*{\zeta}_{g_{0, \rho}} = 1
	\qquad \rm{and} \qquad
\nabla p(\rho). \zeta
	= \abs*{\nabla p}_{g_0}(\rho) .
\end{equation*}
If for some $|s| \le r_{g_0}'$ we have $p(\rho_s') \le E$, then the intermediate value theorem ensures that there exists $s_0$ between~$0$ and~$s$ such that $p(\rho_{s_0}') = E$. Then we can apply~\eqref{eq:controlnearnodalset} with $(\rho_{s_0}', \rho_s')$, satisfying $\abs{\rho_s' - \rho_{s_0}'}_{g_{0, \rho_{s_0}'}} \le C_{g_0} \abs{\zeta}_{g_{0, \rho}} \abs{s_0 - s} \le r_{g_0}$ by slow variation of~$g_0$, to deduce that
\begin{equation*}
E - c_0
	\le p(\rho_s') + \abs*{E - p(\rho_s')} - c_0
	\le p(\rho_s') .
\end{equation*}
This inequality is obviously also true if $p(\rho_s') \ge E$ for all~$s$.

With the estimate $p(\rho_s') \ge E - c_0$ at hand, we now take $\tilde \rho = \rho$ and $\tilde \rho' = \rho_s'$ in~\eqref{eq:Taylorp2} to obtain
\begin{equation*}
\forall |s| \le r_{g_0}' , \qquad
	E - c_0
		\le p(\rho_s')
		\le p(\rho) + s \abs*{\nabla p}_{g_0}(\rho) + \dfrac{s^2}{2} \abs*{\nabla^2 p}_{g_0, \infty} C_{g_0}^2 ,
\end{equation*}
where we used slow variation of $g_0$ again, hence
\begin{equation} \label{eq:polnonneg}
\forall |s| \le r_{g_0}' , \qquad
	0
		\le p(\rho) - (E - c_0) + s \abs*{\nabla p}_{g_0}(\rho) + \dfrac{s^2}{2} \jap*{\abs*{\nabla^2 p}_{g_0, \infty}} C_{g_0}^2 .
\end{equation}
Seeing the above function of $s$ as a degree-two polynomial, either the discriminant $\Delta$ is non-positive, namely
\begin{equation*}
\abs*{\nabla p}_{g_0}^2(\rho)
	\le 2 \jap*{\abs*{\nabla^2 p}_{g_0, \infty}} C_{g_0}^2 \left(p(\rho) - E + c_0\right) ,
\end{equation*}
which implies in particular the desired result $\abs{\nabla p}_{g_0}(\rho) \le c \jap{p(\rho)}^{1/2} \le c \jap{p(\rho)}$, or it is positive and we have two real distinct roots $s_- < s_+$. Combining the fact that $p(\rho) \ge E \ge E - c_0$ with~\eqref{eq:polnonneg}, we deduce that the polynomial is non-negative on $[- r_{g_0}', + \infty)$. Therefore $s_+ \le - r_{g_0}'$, namely
\begin{align*}
\dfrac{- \abs*{\nabla p}_{g_0}(\rho) + \sqrt{\Delta}}{\jap*{\abs*{\nabla^2 p}_{g_0, \infty}} C_{g_0}^2}
	&\le - r_{g_0}' \\
&\Downarrow \\
\Delta
	&\le \left( \abs*{\nabla p}_{g_0}(\rho) - r_{g_0}' \jap*{\abs*{\nabla^2 p}_{g_0, \infty}} C_{g_0}^2 \right)^2 \\
&\Downarrow \\
- 2 \jap*{\abs*{\nabla^2 p}_{g_0, \infty}} C_{g_0}^2 \left(p(\rho) - E + c_0\right)
	&\le - 2 \left(r_{g_0}' \jap*{\abs*{\nabla^2 p}_{g_0, \infty}} C_{g_0}^2\right) \abs*{\nabla p}_{g_0}(\rho) + \left(r_{g_0}' \jap*{\abs*{\nabla^2 p}_{g_0, \infty}} C_{g_0}^2\right)^2 \\
&\Downarrow \\
\abs*{\nabla p}_{g_0}(\rho)
	&\le \dfrac{p(\rho) - E + c_0}{r_{g_0}'} + \dfrac{r_{g_0}'}{2} \jap*{\abs*{\nabla^2 p}_{g_0, \infty}} C_{g_0}^2 ,
\end{align*}
which yields the desired result $\abs*{\nabla p}_{g_0}(\rho) \le c \jap{p(\rho)}$.

\medskip
\emph{Second case: $p(\rho) \le E$.} If we take $\rho$ such that $p(\rho) \le E$, we can reduce to the previous case by considering $-p$ and $-E$ instead of $p$ and $E$. This concludes the proof.
\end{proof}

\subsection{Proof of Proposition~\ref{prop:classicalwellposedness}: global-in-time existence of the Hamiltonian flow} \label{subsec:classicalwellposedness}

Here we prove the ``classical part" of Proposition~\ref{prop:classicalwellposedness}, namely that the Hamiltonian flow is well-defined globally in time.
Let $\rho_0 \in T^\star \mfd$ and consider $t \mapsto \rho(t)$ solution to the Cauchy problem
\begin{equation*}
\dfrac{\dd}{\dd t} \rho(t)
	= H_p\left(\rho(t)\right) , \qquad
\rho(0)
	= \rho_0 .
\end{equation*}
Let $I \subset \R$ be a relatively compact interval on which the solution exists. We show that the range of $I \ni \tau \to \rho(\tau)$ is relatively compact.
From Lemma~\ref{lem:estimatefirstderivativep} (and Lemma~\ref{lem:seminormsHp}), we know that $\abs{H_p}_{g_0} \le \abs{\nabla p}_{g_0} \le c \jap{p(\rho)} = c \jap{p(\rho_0)}$ on the entire energy shell $\{p = p(\rho_0)\}$, which contains $\tau \mapsto \rho(\tau)$. Fix $t \in I$ and let
\begin{equation*}
s = s(t)
	:= \sup \set{s' \ge 0}{\forall \tau \in (t-s', t+s') \cap I, \;\, \rho(\tau) \in \bar B_{r_{g_0}}^{g_0}(\rho(t))} .
\end{equation*}
Then by the mean value inequality, we have for all $\tau \in (t - s, t + s) \cap I$:
\begin{align*}
\abs*{\rho(\tau) - \rho(t)}_{g_{\rho(t)}}
	&\le \abs{s} \sup_{\tau \in (t - s, t + s) \cap I} \abs*{H_p\left(\rho(\tau)\right)}_{g_{0, \rho(t)}}
	\le C_{g_0} \abs{s} \sup_{\tau \in (t - s, t + s)  \cap I} \abs*{H_p\left(\rho(\tau)\right)}_{g_{0, \rho(\tau)}} \\
	&\le c C_{g_0} \abs{s} \jap{p(\rho_0)} .
\end{align*}
Since $\tau \mapsto \rho(\tau)$ is continuous, we deduce that $s(t) \ge r_{g_0}/(c C_{g_0} \jap{p(\rho_0)})$. Therefore, we can cover $I$ with a finite number of compact intervals of size $2 r_{g_0}/(c C_{g_0} \jap{p(\rho_0)})$, on which $\tau \mapsto \rho(\tau)$ takes values in a {$g_0$-ball} of radius $r_{g_0}$. We deduce that the range of $I \ni \tau \mapsto \rho(\tau)$ is relatively compact. Therefore the solution can be extended to the whole real line.

\subsection{Proof of Proposition~\ref{prop:classicalwellposedness}: essential self-adjointness} \label{subsec:quantumwellposedness}

We finish the proof of Proposition~\ref{prop:classicalwellposedness}, showing essential self-adjointness of $P$.
In this proof $P = \Opw{p}$ acts on $\sch(\mfd)$. Since $P - E$ is essentially self-adjoint if and only if $P$ is essentially self-adjoint, we can assume without loss of generality that $E = 0$. The result will follow from the classical characterization of essential self-adjointness~\cite[Theorem X.2]{RS:V2}:
\begin{equation} \label{eq:criterionessentialselfadjointness}
\exists \lambda > 0 : \qquad
	\ker\left(P^\ast + \ii \lambda \right) = \{0\}
		\quad \rm{and} \quad
	\ker\left(P^\ast - \ii \lambda \right) = \{0\} .
\end{equation}
Let $\lambda \ge 1$ to be chosen properly later. We first give an estimate on derivatives of $q = q_\lambda = (p \pm \lambda \ii)^{-1}$ using the Faà di Bruno formula (see Appendix~\ref{app:faadibruno}): for any $\ell \in \N$,
\begin{equation*}
\abs*{\frac{1}{\ell!} \nabla^\ell q}_{g_0}
	\le \sum_{j = 1}^\ell \abs*{q}^{j+1} \sum_{\substack{\bf{n} \in (\N^\ast)^j \\ \abs{\bf{n}} = \ell}} \abs*{\dfrac{1}{\bf{n}!} \nabla^{\bf{n}} p}_{g_0} .
\end{equation*}
We can bound the right-hand side from above by using Lemma~\ref{lem:estimatefirstderivativep} for the indices $\bf{n}_l = 1$ and using the sub-quadraticity assumption $\nabla^2 p \in S(1, g_0)$ for the indices $\bf{n}_l \ge 2$:
\begin{equation*}
\abs*{\nabla^\ell q}_{g_0}
	\le C_\ell(p) \max_{1 \le j \le \ell} \abs{q}^{j+1} \jap*{p}^j ,
\end{equation*}
where the constant $C_\ell(p)$ depends on $\ell$, on seminorms of $\nabla^2 p$ in $S(1, g_0)$ and on the constant~$c$ from Lemma~\ref{lem:estimatefirstderivativep} (which depends also on $p$).
In addition we have $\abs{q} \le \min\{\frac{1}{\jap{p}}, \frac{1}{\lambda}\}$, so that
\begin{equation*}
\abs*{\nabla^\ell q}_{g_0}
	\le \dfrac{C_\ell(p)}{\lambda}
		\qquad \Longleftrightarrow \qquad
q \in S(\lambda^{-1}, g_0) ,
\end{equation*}
uniformly with respect to $\lambda \ge 1$.
Now from pseudo-differential calculus (Proposition~\ref{prop:pseudocalcsymb}), we obtain
\begin{equation*}
\Opw{p \pm \lambda \ii} \Opw{\dfrac{1}{p \pm \lambda \ii}}
	= \id + \Opw{\widehat{\cal{P}}_2(p \pm \lambda \ii, q)}
\end{equation*}
with 
\begin{equation*}
\nabla^{-2} S(1, g_0) \times S(\lambda^{-1}, g_0)
	\xlongrightarrow{\widehat{\cal{P}}_2} S(\lambda^{-1}, g_0) .
\end{equation*}
The Calder\'{o}n--Vaillancourt theorem (Proposition~\ref{prop:CV}) applies in $S(1, g_0)$, so we arrive at
\begin{equation*}
\norm*{\Opw{\widehat{\cal{P}}_2(p \pm \lambda \ii, q)}}_{\Bop(L^2)}
	\le C_0 \abs*{\widehat{\cal{P}}_2(p \pm \lambda\ii, q)}_{S(1, g_0)}^{(k_0)}
	= O\left(\dfrac{1}{\lambda}\right) .
\end{equation*}
Taking $\lambda$ large enough ensures that $R := \Opw{\widehat{\cal{P}}_2(p \pm \lambda \ii, q)}$ satisfies $\norm{R}_{\Bop(L^2)} \le 1/2$. Now writing $A_n = \id - R + R^2 - \cdots + (-1)^n R^n$ for any $n \in \N$, we have
\begin{equation*}
\Opw{p \pm \lambda \ii} \Opw{\dfrac{1}{p \pm \lambda \ii}} A_n
	= \id + (-1)^{n} R^{n+1} ,
\end{equation*}
where $\Opw{q} A_n$ and $R^{n+1}$ map $\sch(\mfd)$ to itself continuously (Proposition~\ref{prop:continuityonSchwartz}) and $\norm*{R^{n+1}} \le 2^{-(n+1)}$. We deduce that for any $u_0 \in \dom P^\ast$ and $u \in \sch(\mfd)$, we have
\begin{align*}
\inp*{(P \pm \lambda \ii)^\ast u_0}{\Opw{q} A_n u}_{L^2}
	&= \inp*{u_0}{(P \pm \lambda \ii) \Opw{q} A_n u}_{L^2} \\
	&= \inp*{u_0}{u}_{L^2} + (-1)^{n} \inp*{u_0}{R^{n+1} u}_{L^2}
	= \inp*{u_0}{u}_{L^2} + O\left(2^{-(n+1)}\right) .
\end{align*}
The first equality is justified by the fact that $\Opw{q} A_n u \in \sch(\mfd)$ (this is the reason why we did not take $n = \infty$ directly). We let $n \to \infty$, observing that $A_n \to (\id + R)^{-1}$ in operator norm by Neumann series, to obtain
\begin{equation*}
\inp*{(P \pm \lambda \ii)^\ast u_0}{\Opw{q} (\id + R)^{-1} u}_{L^2}
	= \inp*{u_0}{u}_{L^2} .
\end{equation*}
Therefore, $u_0 \in \ker(P^\ast \mp \lambda \ii)$ implies that $u_0 \perp \sch(\mfd)$ in $L^2(\mfd)$, that is to say $u_0 = 0$. This concludes the proof of~\eqref{eq:criterionessentialselfadjointness} and thus of Proposition~\ref{prop:classicalwellposedness}.
\qed

\subsection{The quantum and classical dynamics as evolution groups on $L^2$ symbols} \label{subsec:quantumdynamicsL2}

In the proof of the proposition below, we use an important aspect of the Weyl quantization: it is an isometry between $L^2(T^\star \mfd)$ and the space of Hilbert--Schmidt operators (Proposition~\ref{prop:quantizationisometry}):
\begin{equation*}
\quantizationw : L^2(T^\star \mfd) \longrightarrow \cal{L}^2\left(L^2(\mfd)\right) .
\end{equation*}
That is to say for all $a_1, a_2 \in L^2(T^\star \mfd)$, the operators $\Opw{a_1}$ and $\Opw{a_2}$ are Hilbert--Schmidt and
\begin{equation} \label{eq:isomHS}
\inp*{a_2}{a_1}_{L^2(T^\star \mfd)}
	= \tr\left(\Opw{a_2}^\ast \Opw{a_1}\right) .
\end{equation}

\begin{proposition} \label{prop:isometrycalHp}
For all $t \in \R$ let $\cal{U}(t) := \e^{t \cal{H}_p}$ and $U(t) := \e^{t H_p}$ defined in Definition~\ref{def:defsymbconjugatedop} and~\eqref{eq:defcompositionHp} respectively, viewed as operators acting on $L^2(T^\star \mfd)$, in such a way that
\begin{equation*}
\Opw{\cal{U}(t) a}
	= \e^{\ii t P} \Opw{a} \e^{- \ii t P}
		\quad \rm{and} \quad
U(t) a
	= a \circ \phi^t ,
		\qquad \forall a \in L^2(T^\star \mfd) .
\end{equation*}
Then under Assumption~\ref{assum:mandatory}, and if additionally $p$ has temperate growth in the sense of~\eqref{eq:temperategrowth}, then $(\cal{U}(t))_{t \in \R}$ and $(U(t))_{t \in \R}$ are strongly continuous unitary groups on $L^2(T^\star \mfd)$. Their respective generators are self-adjoint extensions of $\ii \cal{H}_p$ and $\ii H_p$ defined in~\eqref{eq:defcalHp} and~\eqref{eq:defHp} respectively.
\end{proposition}

\begin{proof}
We first mention that $\cal{U}(t)$ clearly satisfies the group property $\cal{U}(t) \cal{U}(s) = \cal{U}(t + s)$, for all $t, s \in \R$, since $(\e^{- \ii t P})_{t \in \R}$ itself is a group of isometries. That $\cal{U}(t)$ is an isometry follows from the cyclicity of the trace and~\eqref{eq:isomHS}: for all $a \in L^2(T^\star \mfd)$, we have
\begin{equation*}
\norm*{\cal{U}(t) a}_{L^2(T^\star \mfd)}^2
	= \tr\left(\e^{\ii t P} \Opw{a}^\ast \Opw{a} \e^{- \ii t P}\right)
	= \tr\left( \Opw{a}^\ast \Opw{a} \right)
	= \norm*{a}_{L^2(T^\star \mfd)}^2 .
\end{equation*}
The group property of $U(t)$ and the unitarity of $U(t)$ follow from the fact that $(\phi^t)_{t \in \R}$ is a group of symplectomorphisms (Liouville's theorem tells us that it is measure preserving~\cite[Section 16]{Arnold:book}).

It remains to show that the maps $t \mapsto \cal{U}(t)$ and $t \mapsto U(t)$ are weakly continuous. In view of the density of $\sch(T^\star \mfd)$ in $L^2(T^\star \mfd)$, it is sufficient to verify that the maps $t \mapsto \cal{U}(t) w$ and $t \mapsto U(t) w$ are strongly continuous for $w \in \sch(T^\star \mfd)$.
For any $u \in \sch(\mfd) \subset \dom P$ we have
\begin{equation} \label{eq:exprtomakesense}
\Opw{\cal{U}(t) w} u - \Opw{w} u
	= \ii \int_0^t \e^{\ii s P} \comm*{P}{\Opw{w}} \e^{- \ii s P} u \dd s .
\end{equation}
The operator $\comm*{P}{\Opw{w}}$ can be extended to a bounded operator on $L^2(\mfd)$ since $w \in \sch(T^\star \mfd)$ and $p$ has temperate growth by assumption (use the pseudo-differential calculus in Proposition~\ref{prop:pseudocalcsymb} and the Calder\'{o}n--Vailancourt theorem in Proposition~\ref{prop:CV}). Its symbol is even in the Schwartz class, so that $\comm*{P}{\Opw{w}}$ is in fact a Hilbert--Schmidt operator. Hence the expression~\eqref{eq:exprtomakesense} still makes sense for any $u \in L^2(\mfd)$.

We fix $(e_j)_j$ an orthonormal basis of $L^2(\mfd)$. By definition of the Hilbert--Schmidt norm and of the trace, we have
\begin{align*}
\norm*{\Opw{\cal{U}(t) w} u - \Opw{w}}_{\hsclass(L^2(\mfd))}^2
	&= \tr\abs*{\int_0^t \e^{\ii s P} \comm*{P}{\Opw{w}} \e^{- \ii s P} \dd s}^2 \\
	&= \sum_j \norm*{\int_0^t \e^{\ii s P} \comm*{P}{\Opw{w}} \e^{- \ii s P} e_j \dd s}_{L^2(\mfd)}^2 .
\end{align*}
Then by Jensen's inequality, and again the definition of the trace, we obtain
\begin{align*}
\norm*{\Opw{\cal{U}(t) w} u - \Opw{w}}_{\hsclass(L^2(\mfd))}^2
	&\le \sum_j \abs*{t} \int_{[0, t]} \norm*{\comm*{P}{\Opw{w}} \e^{\ii s P} e_j}_{L^2(\mfd)}^2 \dd s \\
	&= \abs*{t} \int_{[0, t]} \tr\left(\abs*{\comm*{P}{\Opw{w}}}^2\right) \dd s .
\end{align*}
This works because $(\e^{\ii s P} e_j)_j$ is an orthonormal basis for all $s \in \R$ if and only if $(e_j)_j$ is a Hilbert basis of $L^2(\mfd)$. Finally, we change the order of the sum and the integral since the integrand is non-negative. We finish with the definition of the Hilbert--Schmidt norm in the right-hand side again, and we use~\eqref{eq:isomHS} in the left-hand side to deduce that
\begin{equation*}
\norm*{\cal{U}(t) w - w}_{L^2(T^\star \mfd)}^2
	\le \abs*{t}^2 \norm*{\comm*{P}{\Opw{w}}}_{\hsclass(L^2(\mfd))}^2 .
\end{equation*}
This means that strong continuity holds at $t = 0$, and the group property of $\cal{U}(t)$ allows to extend continuity to the whole real line. Thus we can apply Stone's theorem~\cite[VIII.4]{RS:V1} to prove that $\cal{U}(t)$ is indeed a unitary group. To see that the generator of $\cal{U}(t)$ is a self-adjoint extension of $\ii \cal{H}_p$, we simply remark that the generator and $\ii \cal{H}_p$ both agree on $\sch(T^\star \mfd)$ since we have
\begin{equation*}
\forall w \in \sch(T^\star \mfd) , \qquad
	\dfrac{\dd}{\dd t} \Opw{\cal{U}(t) w}_{\vert t = 0}
		= \ii \comm*{P}{\Opw{w}}
		= \Opw{\ii \cal{H}_p w} ,
\end{equation*}
where the last equality follows directly from the definition of $\cal{H}_p$ in~\eqref{eq:defcalHp}.

The situation is simpler for $U(t)$. We prove that for any $w \in \sch(T^\star \mfd)$, the map $t \mapsto U(t) w$ is strongly continuous by observing that
\begin{equation*}
\norm*{U(t) w - w}_{L^2}
	\le \int_{[0, t]} \norm*{U(s) H_p w}_{L^2} \dd s
	= \abs*{t} \norm*{H_p w}_{L^2} .
\end{equation*}
The first inequality comes from the mean-value inequality an the equality is due to the fact that $U(s)$ is an isometry. The norm in the right-hand side is finite since $p$ has temperate growth~\eqref{eq:temperategrowth} by assumption.
The conclusion follows from Stone's theorem once again.
\end{proof}

\Large
\section{The classical dynamics} \label{sec:classical}
\normalsize

To goal of this section is to provide general estimates on derivatives of the Hamiltonian flow $(\phi^t)_{t \in \R}$ associated with a Hamiltonian $p$. The derivatives we consider are in the sense of the covariant derivative, using the natural connection $\nabla$ on $T^\star M$, $M = \R^d$. We will set by convention $\nabla^0 \phi^t = \phi^t$ and $\nabla \phi^t = \dd \phi^t$. Similarly to the definition of covariant derivatives of tensors, higher-order derivatives are defined by induction so that the Leibniz formula holds:
\begin{equation*}
\nabla_{X_1, X_2, \ldots, X_{k+1}}^{k+1} \phi^t
	= \nabla_{X_{k+1}} \left(\nabla_{X_1, X_2, \ldots, X_k}^k \phi^t\right) - \sum_{j = 1}^k \nabla_{X_1, X_2, \ldots, \nabla_{X_{k+1}} X_j, \ldots, X_k}^k \phi^t ,
\end{equation*}
for any family of vector fields $X_1, X_2, \ldots, X_{k+1}$ on $T^\star M$. For all $k \ge 1$, the input of $\nabla^k \phi^t$ is a $k$-tuple of vector fields. The result is a vector field on $T^\star M$. Beware that $(\nabla_{X_1, X_2, \ldots, X_k}^k \phi^t)(\rho)$ is a tangent vector at the point $\phi^t(\rho)$, and not at $\rho$. In other words, $\nabla^k \phi^t(\rho)$ sends~$k$ tangent vectors at $\rho$ to a tangent vector at $\phi^t(\rho)$. Thus we can think of these covariant derivatives as tensors of type $(1, k)$. One can check that $\nabla^k \phi^t$ is symmetric with respect to the input vector fields, for any $k \ge 0$, owing to the fact that the connection is torsion free and has vanishing curvature~\cite[Section 12.8]{Lee:mfddiffgeo}.

\subsection{Norms with respect to a metric} \label{subsec:norms}

The size of these derivatives will be measured using a Riemannian metric $g$ on $T^\star M$. We denote by $\abs{\nabla^k \phi^t}_g$ the function on $T^\star M$ that corresponds to the smallest constant $C = C(\rho)$ such that
\begin{equation*}
\abs*{\nabla_{X_1, X_2, \ldots, X_k}^k \phi^t}_{g_{\phi^t(\rho)}}
	\le C \prod_{j = 0}^k \abs*{X_j}_{g_\rho} ,
\end{equation*}
for all vector fields $X_1, X_2, \ldots, X_k$. Throughout Section~\ref{sec:classical}, we shall use the notation
\begin{equation} \label{eq:notationgE}
\abs*{\bullet}_{g, E}
	:= \sup_{\{p = E\}} \abs*{\bullet}_g ,
		\qquad E \in \R .
\end{equation}

\begin{remark}[Norm of a volume-preserving diffeomorphism] \label{rmk:normsymplecto}
For any Riemannian metric $g$, we always have $\abs{\dd \phi^t}_g \ge 1$. This is a consequence of the fact that $\phi^t$ is a symplectomorphism. Indeed, given a real finite-dimensional normed vector space $(F, \norm*{\bigcdot})$, denoting by $\mu$ any multiple of the Lebesgue measure, then for any isomorphism $A$, we have
\begin{equation*}
\abs*{\det A} \mu\left(B_1(0)\right)
	= \mu\left(A B_1(0)\right)
	\le \mu\left(B_{\norm{A}}(0)\right)
	= \norm*{A}^{\dim F} \mu\left(B_1(0)\right) ,
\end{equation*}
where $B_r(0)$ is the ball of radius $r$ centered at the origin. Therefore
\begin{equation*}
\abs*{\det A} \le \norm*{A}^{\dim F} .
\end{equation*}
Applying this with $(F, \norm*{\bigcdot}) = (T_\rho (T^\star M), \abs{\bigcdot}_g)$ and $A = \dd \phi^t(\rho)$, and recalling that $\abs{\det \dd \phi^t(\rho)} = 1$, we obtain $\abs{\dd \phi^t(\rho)}_g \ge 1$.
\end{remark}

We start with a lemma that relates derivatives of $H_a$, defined as in~\eqref{eq:defHp}, to seminorms of $a$.

\begin{lemma}[Seminorms of the Hamiltonian vector field] \label{lem:seminormsHp}
Let $a \in \cont^\infty(T^\star \mfd)$. One has for any $k \in \N$:
\begin{equation*}
\abs*{\nabla^k H_a}_g
	\le \gain_g \abs*{\nabla^{k+1} a}_g
	\quad \rm{and} \quad
\abs*{\nabla^k H_a}_{g^\sympf}
	\le \gain_g^k \abs*{\nabla^{k+1} a}_g ,
		\qquad \forall k \ge 0 .
\end{equation*}
One also has
\begin{equation*}
\abs*{H_a}_g
	= \abs*{\nabla a}_{g^\sympf}
		\quad \rm{and} \quad
\abs*{H_a}_{g^\sympf}
	= \abs*{\nabla a}_g .
\end{equation*}
\end{lemma}

\begin{proof}
Let us check by induction that for any $k \ge 0$, the following is true: for all vector fields $X_0, X_1, \ldots, X_k$,
\begin{equation} \label{eq:nablakHp}
- \nabla^{k + 1} a (X_0, X_1, X_2, \ldots, X_k)
	= \sympf\left(\nabla^k H_a \left(X_1, X_2, \ldots, X_k\right), X_0\right) .
\end{equation}
The case $k = 0$ is merely the definition of the Hamiltonian vector field. Now assume~\eqref{eq:nablakHp} holds for some $k \ge 0$. We pick $k + 2$ vector fields $X_0, X_1, \ldots, X_{k + 1}$ and we differentiate~\eqref{eq:nablakHp} with respect to $X_{k+1}$:
\begin{equation*}
- \dd \left(\nabla^{k + 1} a \left(X_0, X_1, \ldots, X_k\right)\right) X_{k+1}
	= \dd \left(\sympf(\nabla^k H_a  (X_1, X_2, \ldots, X_k), X_0)\right) X_{k+1} .
\end{equation*}
Using the Leibniz rule on both sides and using that $\nabla \sympf = 0$, we obtain
\begin{multline*}
- \nabla^{k+2} a \left(X_0, X_1, \ldots, X_{k+1}\right) - \sum_{j = 0}^k \nabla^{k+1} a \left(X_0, X_1, \ldots, \nabla_{X_{k+1}} X_j, \ldots, X_k\right) \\
	= \sympf\left(\nabla^k H_a\left(X_1, X_2, \ldots, X_k\right), \nabla_{X_{k+1}} X_0) + \sympf(\nabla^{k+1} H_a\left(X_1, X_2, \ldots, X_{k+1}\right), X_0\right) \\
		+ \sum_{j = 1}^k \sympf\left(\nabla^k H_a\left(X_1, X_2, \ldots, \nabla_{X_{k+1}} X_j, \ldots, X_k\right), X_0\right) .
\end{multline*}
The induction hypothesis~\eqref{eq:nablakHp} implies that the $j$-th term in the left-hand side sum, for $j \ge 1$, coincides with the {$j$-th} in the right-hand side sum, while the term corresponding to $j = 0$ in the left-hand side is equal to the first term in the right-hand side. We obtain~\eqref{eq:nablakHp} at step $k+1$, hence the induction is complete.

Now recall the bundle map $J_g$ such that
\begin{equation*}
g(X, Y)
	= \sympf(X, J_g Y) ,
\end{equation*}
for all vector fields $X, Y$ on $T^\star M$. Also recall that $g = {J_g}^\pullb g^\sympf$ (see~\eqref{eq:Jg} and~\eqref{eq:Jg-1}), so that in particular $\abs*{\bigcdot}_g = \abs*{J_g \bigcdot}_{g^\sympf}$, or equivalently $\abs*{J_g^{-1} \bigcdot}_g = \abs*{\bigcdot}_{g^\sympf}$.

We use~\eqref{eq:nablakHp} with $X_0 = J_g X$ for some vector field $X$ to obtain
\begin{equation*}
g\left(\nabla^k H_a \left(X_1, X_2, \ldots, X_k\right), X\right)
	= \sympf\left(\nabla^k H_a \left(X_1, X_2, \ldots, X_k\right), X_0\right)
	= - \nabla^{k + 1} a (X_0, X_1, X_2, \ldots, X_k)
\end{equation*}
for any $k \ge 0$, and it follows that
\begin{equation*}
\abs*{g\left(\nabla^k H_a \left(X_1, X_2, \ldots, X_k\right), X\right)}
	\le \abs*{\nabla^{k+1} a}_g \abs*{J_g}_g \abs*{X}_g \times \prod_{j = 1}^k \abs*{X_j}_g .
\end{equation*}
Using~\eqref{eq:hgJg}, we deduce that
\begin{equation*}
\abs*{\nabla^k H_a \left(X_1, X_2, \ldots, X_k\right)}_g
	\le \gain_g \abs*{\nabla^{k+1} a}_g \prod_{j = 1}^k \abs*{X_j}_g ,
\end{equation*}
hence the first claim.

We proceed similarly to obtain the second estimate: we use~\eqref{eq:nablakHp} again to obtain
\begin{align*}
g^\sympf\left(X, \nabla^k H_a(X_1, X_2, \ldots, X_k)\right)
	&= \sympf\left(J_g^{-1} X, \nabla^k H_a(X_1, X_2, \ldots, X_k)\right) \\
	&= \nabla^{k+1} a\left(J_g^{-1} X, X_1, X_2, \ldots, X_k\right) .
\end{align*}
This implies that
\begin{equation*}
\abs*{g^\sympf\left(X, \nabla^k H_a(X_1, X_2, \ldots, X_k)\right)}
	\le \abs*{\nabla^{k+1} a\left(\bigcdot, X_1, X_2, \ldots, X_k\right)}_g \abs*{X}_{g^\sympf} ,
\end{equation*}
hence
\begin{align*}
\abs*{\nabla^k H_a(X_1, X_2, \ldots, X_k)}_{g^\sympf}
	&\le \abs*{\nabla^{k+1} a\left(\bigcdot, X_1, X_2, \ldots, X_k\right)}_g
	\le \abs*{\nabla^{k+1} a}_g \prod_{j = 1}^k \abs*{X_j}_g \\
	&\le \abs*{\nabla^{k+1} a}_g \gain_g^k \prod_{j = 1}^k \abs*{X_j}_{g^\sympf} ,
\end{align*}
which is the desired estimate.

The third claim can be handled similarly: writing $X_0 = J_g X$ again and assuming that this vector field is non-vanishing, we have
\begin{equation*}
\dfrac{\abs{g\left(H_a, X\right)}}{\abs{X}_g}
	= \dfrac{\abs{\sympf\left(H_a, J_g X\right)}}{\abs{X}_g}
	= \dfrac{\abs{\nabla a. X_0}}{\abs{X_0}_{g^\sympf}}
		\qquad \rm{and} \qquad
\dfrac{\abs{g^\sympf\left(H_a, X_0\right)}}{\abs{X_0}_{g^\sympf}}
	= \dfrac{\abs{\sympf\left(H_a, J_g^{-1} X_0\right)}}{\abs{X_0}_{g^\sympf}}
	= \dfrac{\abs{\nabla a. X}}{\abs{X}_g} ,
\end{equation*}
and $J_g$ is an isomorphism, hence the sought equalities.
\end{proof}

The following proposition relates norms with respect to~$g$ to norms with respect to the {$\sympf$-dual} metric~$g^\sympf$ (Definition~\ref{def:sympfdual}) and to the symplectic intermediate metric~$g^\natural$ (introduced in~\eqref{eq:defgnatural}). Applying~\eqref{eq:4norms} below to $\phi = \phi^t$, we deduce that the expansion of the flow in the future with respect to $g^\sympf$ is the same as the expansion of the flow in the past with respect to $g$.

\begin{proposition} \label{prop:normssympfnatural}
Let $g$ be a Riemannian metric on $T^\star \mfd$ and $\phi : T^\star \mfd \to T^\star \mfd$ be a symplectomorphism. Then we have
\begin{equation} \label{eq:2pullbacks}
\left(\phi^\pullb g\right)^\sympf
	= \phi^\pullb (g^\sympf)
		\qquad \rm{and} \qquad
\left(\phi^\pullb g\right)^\natural
	= \phi^\pullb (g^\natural) .
\end{equation}
Moreover, we have:
\begin{equation} \label{eq:4norms}
\abs*{\dd \phi}_{g_\rho}
	= \abs*{(\dd \phi)^{-1}}_{g_{\phi(\rho)}^\sympf}
		\qquad \rm{and} \qquad
\abs*{\dd \phi}_{g_\rho^\natural}
	= \abs*{(\dd \phi)^{-1}}_{g_{\phi(\rho)}^\natural}
	\le \abs*{\dd \phi}_{g_\rho}^{1/2} \abs*{\dd \phi}_{g_\rho^\sympf}^{1/2} , 
		\qquad \forall \rho \in T^\star \mfd .
\end{equation}
\end{proposition}

\begin{proof}
Using the fact that $\dd \phi$ is symplectic, namely
\begin{equation*}
\phi^\pullb \sympf_\rho
	= \left(\dd_\rho \phi\right)^\star \sympf_{\phi(\rho)} \dd_\rho \phi
	= \sympf_\rho ,
		\qquad \forall \rho \in T^\star \mfd ,
\end{equation*}
together with the definition of $g^\sympf$ (Definition~\ref{def:sympfdual}), we have
\begin{align*}
\left(\phi^\pullb g\right)_\rho^\sympf
	&= \sympf_\rho^\star \left( \left(\dd_\rho \phi\right)^\star g_{\phi(\rho)} \dd_\rho \phi \right)^{-1} \sympf_\rho
	= \sympf_\rho^\star \left( \dd_\rho \phi \right)^{-1} g_{\phi(\rho)}^{-1} \left(\left(\dd_\rho \phi \right)^\star\right)^{-1} \sympf_\rho \\
	&= \left(\dd_\rho \phi\right)^\star \sympf_{\phi(\rho)}^\star g_{\phi(\rho)}^{-1} \sympf_{\phi(\rho)} \dd_\rho \phi
	= \left(\phi^\pullb (g^\sympf)\right)_\rho .
\end{align*}
This shows that $\left(\phi^\pullb g\right)^\sympf = \phi^\pullb (g^\sympf)$. Now we check that $\left(\phi^\pullb g\right)^\natural = \phi^\pullb (g^\natural)$. Introducing the map $\phi \oplus \phi$ acting naturally on $T^\star \mfd \oplus T^\star \mfd$, we have for any symmetric map $q : W \to W^\star$:
\begin{equation*}
(\phi \oplus \phi)^\pullb
\begin{pmatrix}
g & q \\ q & g^\sympf
\end{pmatrix}
	=
\begin{pmatrix}
\phi^\pullb g & \phi^\pullb q \\ \phi^\pullb q & \phi^\pullb (g^\sympf)
\end{pmatrix}
	=
\begin{pmatrix}
\phi^\pullb g & \phi^\pullb q \\ \phi^\pullb q & (\phi^\pullb g)^\sympf
\end{pmatrix} ,
\end{equation*}
so that
\begin{equation*}
\begin{pmatrix}
g & q \\ q & g^\sympf
\end{pmatrix}
	\ge 0
		\qquad \Longleftrightarrow \qquad
\begin{pmatrix}
\phi^\pullb g & \phi^\pullb q \\ \phi^\pullb q & (\phi^\pullb g)^\sympf
\end{pmatrix}
	\ge 0 .
\end{equation*}
Taking $q = g^\natural$, so that the left-hand side is true, we infer from the right-hand side that $\phi^\pullb (g^\natural) 
\le (\phi^\pullb g)^\natural$, by definition of the symplectic intermediate metric~\eqref{eq:defgnatural}. Applying this inequality with $\phi^\pullb g$ in place of~$g$ and $\phi^{-1}$ in place of~$\phi$, we deduce that
\begin{equation*}
(\phi^{-1})^\pullb \left((\phi^\pullb g)^\natural\right)
	\le \left((\phi^{-1})^\pullb (\phi^\pullb g)\right)^\natural
	= g^\natural ,
\end{equation*}
hence the converse inequality $(\phi^\pullb g)^\natural \le \phi^\pullb (g^\natural)$. This finishes the proof of~\eqref{eq:2pullbacks}.

It remains to prove~\eqref{eq:4norms}. On the one hand, from~\eqref{eq:2pullbacks} together with {$\sympf$-duality}, we deduce that for any function $C = C_\rho : T^\star \mfd \to \R_+^\ast$:
\begin{equation*}
\phi^\pullb g
	\le C_\rho^2 g
		\qquad \Longleftrightarrow \qquad
g^\sympf
	\le C_\rho^2 \phi^\pullb g^\sympf
		\qquad \Longleftrightarrow \qquad
(\phi^{-1})^\pullb g^\sympf
	\le C_{\phi^{-1}(\rho)}^2 g^\sympf .
\end{equation*}
Taking $C_\rho = \abs{\dd \phi}_{g_\rho}$ (for which the left-hand side inequality is true) and $C_\rho = \abs{(\dd \phi)^{-1}}_{g_{\phi(\rho)}^\sympf}$ (for which the right-hand side inequality is true) yields
\begin{equation*}
\abs*{\dd \phi}_{g_\rho}
	\le \abs*{(\dd \phi)^{-1}}_{g_{\phi(\rho)}^\sympf}
		\qquad \rm{and} \qquad
\abs*{(\dd \phi)^{-1}}_{g_{\phi(\rho)}^\sympf}
	\le \abs*{\dd \phi}_{g_\rho}
\end{equation*}
respectively, hence the equality of norms.
This is valid in particular with $g^\natural = (g^\natural)^\sympf$ instead of $g$. Lastly, if $\phi^\pullb g \le C_1^2 g$ and $\phi^\pullb g^\sympf \le C_2^2 g^\sympf$, then by definition of $g^\natural$:
\begin{equation*}
0
	\le (\phi \oplus \phi)^\pullb \begin{pmatrix}
g & g^\natural \\ g^\natural & g^\sympf
\end{pmatrix}
	= \begin{pmatrix}
\phi^\pullb g & \phi^\pullb g^\natural \\ \phi^\pullb g^\natural & \phi^\pullb g^\sympf
\end{pmatrix}
	\le \begin{pmatrix}
C_1^2 g & \phi^\pullb g^\natural \\ \phi^\pullb g^\natural & C_2^2 g^\sympf
\end{pmatrix} ,
\end{equation*}
which yields
\begin{equation} \label{eq:reasoningsimilarlater}
0
	\le \dfrac{1}{C_1 C_2} \Psi^\pullb \begin{pmatrix}
C_1^2 g & \phi^\pullb g^\natural \\ \phi^\pullb g^\natural & C_2^2 g^\sympf
\end{pmatrix}
	= \Psi^\pullb \begin{pmatrix}
\frac{C_1}{C_2} g & \frac{1}{C_1 C_2} \phi^\pullb g^\natural \\ \frac{1}{C_1 C_2} \phi^\pullb g^\natural & \frac{C_2}{C_1} g^\sympf
\end{pmatrix}
	= \begin{pmatrix}
g & \frac{1}{C_1 C_2} \phi^\pullb g^\natural \\ \frac{1}{C_1 C_2} \phi^\pullb g^\natural & g^\sympf
\end{pmatrix} ,
\end{equation}
where $\Psi(\zeta_1, \zeta_2) = (\sqrt{\frac{C_2}{C_1}} \zeta_1, \sqrt{\frac{C_1}{C_2}} \zeta_2)$. We deduce from the definition of~$g^\natural$ in~\eqref{eq:defgnatural} that
\begin{equation*}
\frac{1}{C_1 C_2} \phi^\pullb g^\natural
	\le g^\natural .
\end{equation*}
Putting $C_1 = \abs{\dd \phi}_g$ and $C_2 = \abs{\dd \phi}_{g^\sympf}$ yields the sought inequality in~\eqref{eq:4norms}.
\end{proof}

\subsection{A general estimate} \label{subsec:generalestimate}

In this section, we provide a general estimate on the derivatives of the Hamiltonian flow. Here $g$ is an arbitrary Riemannian metric on $T^\star \mfd$. Our presentation is inspired by~\cite[Section 2]{BouzouinaRobert} and~\cite[Appendix C.1]{DG:14}.
The rough idea is to differentiate successively the equation on the differential:
\begin{equation*}
\dfrac{\dd}{\dd t} \dd \phi^t(\rho)
	= \nabla H_p\left(\phi^t(\rho)\right). \dd \phi^t(\rho) ,
\end{equation*}
and use Duhamel's formula to express $\nabla^{k+1} \phi^t$ as a function of lower order derivatives of $\phi^t$.

We define several functions of time quantifying the growth of the differential of the Hamiltonian flow: on the one hand, we consider the successive primitives of $\abs{\dd \phi^t(\rho)}_g$ with respect to time:
\begin{equation*}
\forall k \in \N^\ast, \quad \cal{I}_{k+1}(t, \rho)
	:= \int_{[0, t]} \cal{I}_k(s, \rho) \dd s , \qquad \rm{with} \;\; \cal{I}_0(t, \rho) := \abs*{\dd \phi^t(\rho)}_g .
\end{equation*}
On the other hand, we handle products arising from iterating Duhamel's formula by introducing for any $t \in \R$ and $\rho \in T^\star \mfd$:
\begin{equation*}
\cal{N}_0(t, \rho)
	:= \sup_{s \in [0, t]} \abs*{\dd \phi^s(\rho)}_g
		\qquad \rm{and} \qquad
\forall k \in \N^\ast , \quad
	\cal{N}_k(t, \rho)
		:= \sup_{s \in [0, t]} \sup_{\bf{s} \in s \Delta_k} \prod_{j = 0}^k \abs*{\dd \phi^{s_{j+1} - s_j}\left(\phi^{s_j}(\rho)\right)}_g .
\end{equation*}
In the latter definition, $\bf{s} = (s_1, s_2, \ldots, s_k) \in s \Delta_k$ refers to a point in the $k$-dimensional simplex defined in~\eqref{eq:defsimplex}, with the boundary conventions $s_0 = 0$ and $s_{k+1} = s$.
We also set for any energy level $E \in \R$:
\begin{equation*}
\cal{I}_k(t, E)
	:= \sup_{\rho \in \{p = E\}} \cal{I}_k(t, \rho)
		\qquad \rm{and} \qquad
\cal{N}_k(t, E)
	:= \sup_{\rho \in \{p = E\}} \cal{N}_k(t, \rho) .
\end{equation*}
Note that $\cal{I}_k$ and $\cal{N}_k$ are non-decreasing as functions of $t$ by definition.

We set for any smooth function $f$ on $T^\star M$ (or more generally for any tensor over $T^\star M$):
\begin{equation} \label{eq:defbnorm}
\bnorm{f}_n^k
	= \bnorm{f}_n^k(\rho)
	:= \max_{\substack{\bf{n} \in \N^n \\ \abs{\bf{n}} = k}} \prod_{j = 1}^n \dfrac{\abs{\nabla^{n_j} f(\rho)}_g}{n_j!} , \qquad
		\rho \in T^\star M, \; k \in \N, n \in \N^\ast .
\end{equation}
We also set $\bnorm{f}_{n, E}^k := \sup_{\{p = E\}} \bnorm{f}_n^k$. Note that in case $f$ is an analytic function, the quantity $\bnorm{f}_n^k$ behaves roughly as $r^{-k}$, where $r$ is the radius of analyticity of $f$ around $\rho$. Such quantities will appear naturally while applying the Faà di Bruno formula (Appendix~\ref{app:faadibruno}). This interpretation of this quantity justifies why we chose to emphasize the dependence on $k$ by putting it as an exponent.

To simplify notation, the dependence of the quantities $\cal{I}_k, \cal{N}_k$ and $\bnorm{\bullet}_n^k$ on the metric $g$ will not be made explicit.

\begin{proposition}[Derivative estimates of the flow] \label{prop:estimatesflow}
Let $p \in \cont^\infty(T^\star \mfd)$ and assume the associated Hamiltonian flow $(\phi^t)_{t \in \R}$ is complete. Let $g$ be a Riemannian metric on $T^\star \mfd$.
Then for all $k \ge 2$, we have for all $t \in \R$ and $\rho \in T^\star \mfd$:
\begin{equation} \label{eq:estimateflow}
\dfrac{1}{k!} \abs{\nabla^k \phi^t}_g(\rho) 
	\le k^{k-1} \max_{1 \le n \le k - 1} \left( \dfrac{\bnorm*{\nabla^2 H_p}_{n, E}^{k-1 - n}}{2^n} \max_{1 \le j \le k/2} \max_{\substack{\bf{n} \in \N^j \\ \abs{\bf{n}} = n}} \cal{N}_0^{k - 2 j}(t, \rho) \prod_{\ell = 1}^j \cal{N}_{n_\ell}(t, \rho) \cal{I}_{n_\ell}(t, \rho) \right) .
\end{equation}
\end{proposition}

The proof of Proposition~\ref{prop:estimatesflow} consists in computing the successive derivatives of $\phi^t$ by iterating Duhamel's formula. The index~$n$ in~\eqref{eq:estimateflow} corresponds to a term, in the resulting expression of $\nabla^k \phi^t$, containing~$n$ nested Duhamel integrals.

\begin{proof}
We proceed by induction. The starting point is the equation defining the Hamiltonian flow:
\begin{equation*}
\dfrac{\dd}{\dd t} \phi^t(\rho)
	= H_p\left(\phi^t(\rho)\right) .
\end{equation*}
We differentiate it $k$ times with respect to the variable $\rho$ using the Faà di Bruno formula (Appendix~\ref{app:faadibruno}): isolating the term of order $j = 1$, one has
\begin{equation} \label{eq:nablakphit}
\dfrac{\dd}{\dd t} \left( \dfrac{1}{k!} \nabla^k \phi^t \right)_\rho
	= (\nabla H_p)_{\phi^t(\rho)}. \left( \dfrac{1}{k!} \nabla^k \phi^t \right)_\rho + \sum_{j = 2}^k \; \sum_{\substack{\bf{n} \in (\N^\ast)^j \\ \abs{\bf{n}} = k}} \left(\dfrac{1}{j!} \nabla^j H_p\right)_{\phi^t(\rho)}. \left( \dfrac{1}{\bf{n}!} \nabla^{\bf{n}} \phi^t \right)_\rho .
\end{equation}
The notation $\frac{1}{\bf{n}!} \nabla^{\bf{n}}$ is defined in~\eqref{eq:compactnotation}.

We observe from~\eqref{eq:nablakphit} that the equation satisfied by $A(t) := \frac{1}{k!} \nabla^k \phi^t$ for $k \ge 2$ is of the form
\begin{equation} \label{eq:equationA(t)}
\dfrac{\dd}{\dd t} A_\rho(t)
	= (\nabla H_p)_{\phi^t(\rho)} . A_\rho(t) + B_\rho(t) ,
\end{equation}
where the source term $B(t)$ depends on derivatives of order less than $k - 1$ of $\phi^t$. We observe that the solution to the homogeneous equation
\begin{equation*}
\dfrac{\dd}{\dd t} a_\rho(t)
	= (\nabla H_p)_{\phi^t(\rho)} . a_\rho(t) ,
\end{equation*}
is given by the formula
\begin{equation*}
a_\rho(t)
	= \dd \phi^t(\rho) a_\rho(0) ,
\end{equation*}
so that Duhamel's formula in~\eqref{eq:equationA(t)} (with initial datum $\nabla^k \phi^0 = \nabla^k \id = 0$ since $k \ge 2$) yields
\begin{equation} \label{eq:solA(t)}
A(t)
	= \int_0^t \dd \phi^{t - s}\left(\phi^s(\rho)\right) B_\rho(s) \dd s .
\end{equation}

With this at hand, we can start the induction. We check the basis step $k = 2$ of~\eqref{eq:estimateflow}: we have
\begin{equation*}
\dfrac{1}{2!} \nabla^2 \phi^t = \int_0^t \dd \phi^{t - s}\left(\phi^s(\rho)\right) \dfrac{1}{2!} \nabla^2 H_p\left(\phi^s(\rho)\right).\left(\nabla \phi^s(\rho)\right)^2 \dd s ,
\end{equation*}
so that
\begin{equation*}
\dfrac{1}{2!} \abs*{\nabla^2 \phi^t(\rho)}_g
	\le \dfrac{1}{2!} \abs*{\nabla^2 H_p}_{g, E} \int_{[0, t]} \abs*{\nabla \phi^{t - s}\left(\phi^s(\rho)\right)}_g \abs*{\nabla \phi^s(\rho)}_g^2 \dd s
	\le \dfrac{1}{2} \bnorm*{\nabla^2 H_p}_{1, E}^0 \cal{N}_1(t, \rho) \cal{I}_1(t, \rho) ,
\end{equation*}
which is the expected bound. Now assume~\eqref{eq:estimateflow} is true for some $k \ge 2$. From the Faà di Bruno formula written in~\eqref{eq:nablakphit}, we know that the source term in~\eqref{eq:equationA(t)} with $A(t) = \frac{1}{(k+1)!} \nabla^{k+1} \phi^t$ is bounded by derivatives of order less than $k$ of~$\phi^t$,
\begin{equation} \label{eq:B(t)FdB}
\abs*{B_\rho(s)}_g
 	\le \sum_{j = 2}^{k+1} \; \sum_{\substack{n_1 + n_2 + \cdots + n_j = k+1 \\ n_\ell \ge 1}}  \dfrac{1}{j!} \abs*{\nabla^j H_p}_{g, E} \; \prod_{\ell = 1}^j \dfrac{1}{n_\ell!} \abs*{\nabla^{n_\ell} \phi^s(\rho)}_g .
\end{equation}
We plug this into Duhamel's formula~\eqref{eq:solA(t)} to compute $A(t)$. We are going to bound each term of the double sum separately, and sum the estimates ultimately. We deal with the term $j = k+1$ in the above sum separately. We bound all but one factors $\abs{\nabla \phi^s}_g$ by the sup norm in time, keeping only one of them in the integral. It yields
\begin{equation*}
\int_{[0, t]} \abs*{\dd \phi^{t - s}\left(\phi^s(\rho)\right)}_g \dfrac{\abs*{\nabla^{k+1} H_p}_{g, E}}{(k+1)!} \abs*{\nabla \phi^s(\rho)}_g^{k+1} \dd s
	\le \dfrac{\bnorm*{\nabla^2 H_p}_{1, E}^{k-1}}{2} \cal{N}_0^{k-2}(t, \rho) \cal{N}_1(t, \rho) \cal{I}_1(t, \rho) ,
\end{equation*}
which matches~\eqref{eq:estimateflow} with $n = 1$ and $j = 1$. Now we use the induction hypothesis to handle the terms of order $j \in \{2, 3, \ldots, k\}$. Fix such a $j$ and a multi-index $\bf{n} = (n_1, n_2, \ldots, n_j)$, with $n_1 + n_2 + \cdots + n_j = k+1$. Denote by $j_0$ the number of indices $\ell$ such that $n_\ell = 1$. Up to relabeling, we assume that those indices are $n_1, n_2, \ldots, n_{j_0}$. Note that we know that $j_0 < j$ since the case $j = k+1$ is excluded. The product of $\frac{1}{n_\ell!} \abs{\nabla^{n_\ell} \phi^s}_g$ for $\ell > j_0$ is bounded from above by the multiplication of the bounds~\eqref{eq:estimateflow}, with $n_\ell$ in place of $k$. For any $\ell > j_0$, we fix $n_\ell^\star$ between $1$ and $n_\ell - 1$ for which the first maximum in~\eqref{eq:estimateflow} (over $1 \le n \le k - 1$) is attained in~\eqref{eq:estimateflow}. Here, $n_\ell^\star$ and $n_\ell$ play the role of $n$ and $k$ in~\eqref{eq:estimateflow} respectively, for the {$\ell$-th} term. The combinatorial factors (namely the factor $k^{k-1}$ in~\eqref{eq:estimateflow}) for each $\ell$ are multiplied, and the product is bounded from above by:
\begin{equation} \label{eq:combifac}
\prod_{\ell = j_0 + 1}^j n_\ell^{n_\ell - 1}
	\le \prod_{\ell = 1}^j k^{n_\ell - 1}
	= k^{k+1 - j} .
\end{equation}
We used $n_\ell \le k$, as a consequence of the fact that $j \ge 2$ (the $n_\ell$'s form a family of at least two positive integers whose sum is equal to $k+1$).

Next we look at the multiplication of factors involving $\nabla H_p$ in~\eqref{eq:estimateflow}: for any $\ell > j_0$, we have
\begin{align*}
\dfrac{1}{j!} \abs*{\nabla^j H_p}_{g, E} \prod_{\ell = j_0 + 1}^j \dfrac{1}{2^{n_\ell^\star}} \bnorm*{\nabla^2 H_p}_{n_\ell^\star, E}^{n_\ell - 1 - n_\ell^\star}
	&\le \dfrac{1}{2} \bnorm*{\nabla^2 H_p}_{1, E}^{j-2} \times \dfrac{1}{2^{n^\star}} \bnorm*{\nabla^2 H_p}_{n^\star, E}^{k+1 - j - n^\star} \\
	&= \dfrac{1}{2^{n^\star+1}} \bnorm*{\nabla^2 H_p}_{n^\star+1, E}^{k - (n^\star+1)} ,
\end{align*}
where $1 \le n^\star := n_{j_0 + 1}^\star + \cdots + n_j^\star \le n_1 + n_2 + \cdots + n_j - j = k+1 - j \le k - 1$. This yields the expected factor in~\eqref{eq:estimateflow}, with $n^\star+1$ between $2$ and $k$ playing the role of $n$.

Lastly we deal with the factors involving $\cal{N}(t, \rho)$ and $\cal{I}(t, \rho)$. For any $\ell > j_0$, we pick $\bf{m}_\ell = (m_{\ell, 1}, m_{\ell, 2}, \ldots, m_{\ell, q_\ell})$ in $\N^{q_\ell}$ with $\abs*{\bf{m}_\ell} \le n_\ell - 1$. We choose them such that the ``second" and ``third" maxima (over~$j$ and~$\bf{n}$) in~\eqref{eq:estimateflow} are achieved. Thus $\bf{m}_\ell$ and $q_\ell$ play the role of~$\bf{n}$ and~$j$ in~\eqref{eq:estimateflow} respectively, for the {$\ell$-th} term. In particular, we have $1 \le q_\ell \le n_\ell/2$. Going back to Duhamel's formula, the multiplication of the bounds of the form~\eqref{eq:estimateflow} gives terms of the form
\begin{equation} \label{eq:termduhameljlargerthan2}
\int_{[0, t]} \abs*{\nabla \phi^{t - s}\left(\phi^s(\rho)\right)}_g \abs*{\nabla \phi^s(\rho)}_g^{j_0} \prod_{\ell = j_0 + 1}^j \cal{N}_0^{n_\ell - 2 q_\ell}(s, \rho) \prod_{i = 1}^{q_\ell} \cal{N}_{m_{\ell, i}}(s, \rho) \cal{I}_{m_{\ell, i}}(s, \rho) \dd s .
\end{equation}
We concatenate the {$\bf{m}_\ell$'s} to get a multi-index $\bf{m} = (\bf{m}_{j_0 + 1}, \ldots, \bf{m}_j)$ such that
\begin{equation*}
\abs*{\bf{m}}
	= \sum_{\ell = j_0 + 1}^j \abs*{\bf{m}_\ell}
	\le \sum_{\ell = 1}^j \left(n_\ell - 1\right)
	= k+1 - j .
\end{equation*}
This multi-index belongs to $(\N^\ast)^m$ with $1 \le m := q_{j_0 + 1} + \cdots + q_j \le \frac{1}{2}(n_1 + n_2 +\cdots + n_j) = \frac{1}{2} (k+1)$. The product of $\cal{N}_0^{n_\ell - 2 q_\ell}(s, \rho)$, for $\ell = j_0 + 1, \ldots, j$, can be bounded by the same product at time $t$, which yields $\cal{N}_0^{k+1 - j_0 - 2 m}(t, \rho)$. As for $\abs{\nabla \phi^s(\rho)}_g^{j_0}$, we can bound it by $\cal{N}_0^{j_0}(t, \rho)$. Now in~\eqref{eq:termduhameljlargerthan2}, we gather $\abs{\nabla \phi^{t - s}(\phi^s(\rho))}_g$ with $\cal{N}_{m_{j, q_j}}(s, \rho)$, and we use that
\begin{equation*}
\abs{\nabla \phi^{t - s}(\phi^s(\rho))} \cal{N}_{m_{j, q_j}}(s, \rho)
	\le \cal{N}_{m_{j, q_j} + 1}(t, \rho) ,
\end{equation*}
for any $t$ and $\rho$. Finally, using monotonicity in time, we bound from above all the factors $\cal{N}_{m_{\ell, i}}(s, \rho)$ and $\cal{I}_{m_{\ell, i}}(s, \rho)$, for $j_0 < \ell < j$ and $1 \le i \le q_\ell$, or $\ell = j$ and $1 \le i < q_j$, by the same factor at time $t$. As for the remaining factor $\cal{I}_{m_{j, q_j}}$, we keep it in the integral over~$s$, and use that by definition:
\begin{equation*}
\int_{[0, t]} \cal{I}_{m_{j, q_j}}(s, \rho) \dd s = \cal{I}_{m_{j, q_j} + 1}(t, \rho) .
\end{equation*}
In the end, the term~\eqref{eq:termduhameljlargerthan2} admits a bound of the form
\begin{equation*}
\cal{N}_0^{j_0}(t, \rho) \cal{N}_0^{k+1 - j_0 - 2m}(t, \rho) \prod_i \cal{N}_{\tilde m_i}(t, \rho) \cal{I}_{\tilde m_i}(t, \rho) ,
\end{equation*}
where the multi-index $\tilde{\bf{m}} = (\tilde m_i)_i$ is built from $\bf{m}$ by modifying solely its last component, changing $m_{j, q_j}$ into $m_{j, q_j} + 1$. We obtain a bound of the form of the right-hand side of~\eqref{eq:estimateflow}, since $m \le (k+1)/2$ and $\abs{\tilde{\bf{m}}} = \abs{\bf{m}} + 1 = m + 1 \le k+1 - j + 1 \le k$, in view of $j \ge 2$.

It remains to handle the sums over $j$ and~$\bf{n}$ in~\eqref{eq:B(t)FdB}. We observe that only the combinatorial factors obtained in~\eqref{eq:combifac} depend on $j$ (the seminorms of $p$ and the quantities $\cal{I}_k, \cal{N}_k$ can be factored outside the sums). We use Remark~\ref{rk:numberoftermsfaadibruno} to arrive at
\begin{equation*}
\sum_{j = 2}^{k+1} \sum_{\substack{\bf{n} \in (\N^\ast)^j \\ \abs{\bf{n}} = k+1}} k^{k+1 - j}
	= \sum_{j = 2}^{k+1} \binom{k}{j-1} k^{k+1 - j}
	\le (k+1)^k .
\end{equation*}
This completes the induction.
\end{proof}

\subsection{Exponentially expanding flows}

Proposition~\ref{prop:estimatesflow} in Section~\ref{subsec:generalestimate} gives a general bound, which can be used to study the growth of derivatives of polynomially expanding flows. This case will not be discussed further in this work since it corresponds to very particular situation (see for instance~\cite{BouzouinaRobert}). From now on, we focus on the generic situation in which the differential of the flow grows like an exponential, which occurs under Assumptions~\ref{assum:mandatory} and~\ref{assum:p} in particular (see~\eqref{eq:Lyapunovcontrol}).

\begin{proposition}[Exponential growth of flow derivatives] \label{prop:estflowexpgrowth}
Let $p \in \cont^\infty(T^\star \mfd)$ and assume the associated Hamiltonian flow $(\phi^t)_{t \in \R}$ is complete. Let $g$ be a Riemannian metric on $T^\star \mfd$.
Let $E \in \R$ and assume that the flow has an exponential growth on the energy layer $\{p = E\}$, i.e.
\begin{equation*}
\exists C = C(E) \ge 1, \exists \Lambda = \Lambda(E) > 0 : \qquad
	\abs*{\dd \phi^t}_{g, E} \le C \e^{\Lambda \abs{t}} , \quad \forall t \in \R .
\end{equation*}
(Recall the notation $\abs{\bullet}_{g, E}$ in~\eqref{eq:notationgE}.)
Then one has
\begin{equation*}
\dfrac{1}{k!} \abs*{\nabla^k \phi^t}_{g, E}
	\le k^{k-1} \left(C \e^{\Lambda \abs{t}}\right)^k \max_{1 \le n \le k - 1} \left( \left(\dfrac{C}{2 \Lambda}\right)^n \bnorm*{\nabla^2 H_p}_{n, E}^{k-1 - n} \right) , \qquad
		\forall t \in \R, \forall k \ge 2 .
\end{equation*}
\end{proposition}

\begin{proof}
This is a consequence of Proposition~\ref{prop:estimatesflow}. We estimate each term in the right-hand side of~\eqref{eq:estimateflow}.
The product in the definition of $\cal{N}_k(t, \rho)$ is telescopic:
\begin{equation*}
\cal{N}_k(t, E)
	\le \sup_{s \in [0, t]} \sup_{\bf{s} \in s \Delta_k} \prod_{j = 0}^k C \e^{\Lambda \abs{s_{j+1} - s_j}}
	= C^{k+1} \sup_{s \in [0, t]} \e^{\Lambda \abs{s}}
	= C^{k+1} \e^{\Lambda \abs{t}} .
\end{equation*}
As for $\cal{I}_k(t, E)$, we can proceed by induction and check that each integration yields a factor $\Lambda^{-1}$, so that:
\begin{equation*}
\cal{I}_k(t, E)
	\le \dfrac{C}{\Lambda^k} \e^{\Lambda \abs{t}} .
\end{equation*}
For any $n \in \{1, 2, \ldots, k - 1\}$, it leads to
\begin{multline*}
\max_{1 \le j \le k/2} \max_{\substack{\bf{n} \in \N^j \\ \abs{\bf{n}} = n}} \cal{N}_0^{k - 2 j}(t, E) \prod_{\ell = 1}^j \cal{N}_{n_\ell}(t, E) \cal{I}_{n_\ell}(t, E) \\
	\le \max_{1 \le j \le k/2} \max_{\substack{\bf{n} \in \N^j \\ \abs{\bf{n}} = n}} C^{k - 2j} \e^{(k - 2j) \Lambda \abs{t}} \prod_{\ell = 1}^j C^{n_\ell + 1} \e^{\Lambda \abs{t}} \dfrac{C}{\Lambda^{n_\ell}} \e^{\Lambda \abs{t}}
	= \dfrac{C^{k + n}}{\Lambda^n} \e^{k \Lambda \abs{t}} .
\end{multline*}
Proposition~\ref{prop:estimatesflow} give the sought result.
\end{proof}

Under Assumption~\ref{assum:p}, it is sufficient to pick $\Lambda$ defined in~\eqref{eq:defLambda} and $C = 1$ according to the following a priori estimate.

\begin{proposition}[A priori exponential growth estimate] \label{prop:aprioriexpgrowth}
Suppose $p$ and $g$ satisfy Assumptions~\ref{assum:mandatory} and~\ref{assum:p}. Then the following holds:
\begin{equation} \label{eq:diffflow}
\abs*{\dd \phi^t(\rho)}_g
	\le \exp\left( \dfrac{1}{2} \int_{[0, t]} \abs*{\lieder_{H_p} g}_g\left(\phi^s(\rho)\right) \dd s \right) , \qquad
		\forall \rho \in T^\star \mfd, \forall t \in \R ,
\end{equation}
so in particular
\begin{equation} \label{eq:diffflowexp}
\abs*{\dd \phi^t}_{g, \infty}
	\le \e^{\Lambda \abs{t}} ,
		\qquad \forall t \in \R ,
\end{equation}
where $\Lambda$ is the Lyapunov exponent introduced in~\eqref{eq:defLambda}.
Moreover we have
\begin{equation} \label{eq:upperboundLambda}
\Lambda
	\le \sup_{T^\star \mfd} \left( \gain_g \abs*{\nabla^2 p}_g + \dfrac{1}{2} \abs*{\nabla_{H_p} g}_g \right) .
\end{equation}
\end{proposition}

\begin{proof}
We have for any $\rho \in T^\star \mfd$ and any $t \in \R$:
\begin{equation} \label{eq:derivativepullbmetric}
\dfrac{\dd}{\dd t} \left(\phi^t\right)^\pullb g
	= \left(\phi^t\right)^\pullb \lieder_{H_p} g
	\le \abs*{\lieder_{H_p} g}_g\left(\phi^t(\rho)\right) \left(\phi^t\right)^\pullb g ,
\end{equation}
so Grönwall's inequality yields~\eqref{eq:diffflow} and~\eqref{eq:diffflowexp} follows from the definition of $\Lambda$ in~\eqref{eq:defLambda}.
It remains to prove~\eqref{eq:upperboundLambda}. For any couple of vector fields $X, Y$ on $T^\star \mfd$, we have
\begin{align*}
(\lieder_{H_p} g)_\rho(X, Y)
	&= \dfrac{\dd}{\dd t}_{\vert t =0} g_{\phi^t(\rho)}\left(\dd \phi^t(\rho). X, \dd \phi^t(\rho). Y\right) \\
	&= g_\rho\left(\nabla H_p. X, Y\right) + g_\rho\left(X, \nabla H_p. Y\right) + (\nabla_{H_p} g)_\rho(X, Y) ,
\end{align*}
and combining this with $\abs{\nabla H_p}_g \le \gain_g \abs{\nabla^2 p}_g$ (see Lemma~\ref{lem:seminormsHp}), we obtain inequality~\eqref{eq:upperboundLambda}.
\end{proof}

Now we provide estimates for the composition of a symbol with the flow. We shall study this composition as an operator, acting on various spaces. We recall the notation
\begin{equation*}
\e^{t H_p} a = a \circ \phi^t , \qquad \forall t \in \R, \, \forall a \in \cont^\infty(T^\star M) .
\end{equation*}

\begin{lemma}[Composition by the flow] \label{lem:compositionbytheflow}
Let $p \in \cont^\infty(T^\star \mfd)$ and assume the associated Hamiltonian flow $(\phi^t)_{t \in \R}$ is complete. Let $g$ be a Riemannian metric on $T^\star \mfd$.
Let $E \in \R$ and assume that the flow has an exponential growth on the energy layer $\{p = E\}$, i.e.
\begin{equation*}
\exists C = C(E) \ge 1, \exists \Lambda = \Lambda(E) > 0 : \qquad
	\abs*{\dd \phi^t}_{g, E} \le C \e^{\Lambda \abs{t}} , \quad \forall t \in \R .
\end{equation*}
Then for all $a \in \cont^\infty(T^\star M)$ and all $k \in \N^\ast$, we have for any $\rho \in \{p = E\}$:
\begin{equation} \label{eq:estcomposflow}
\dfrac{1}{k!} \abs*{\nabla^k \e^{t H_p} a}_g(\rho)
	\le \left(C k \e^{\Lambda \abs{t}}\right)^k \max_{1 \le j \le k} \left( \dfrac{1}{j!} \abs*{(\nabla^j a)_{\phi^t(\rho)}}_g \max_{1 \le n \le k - j} \left(\left(\dfrac{C}{2 \Lambda}\right)^n \bnorm*{\nabla^2 H_p}_{n, E}^{k - j - n}\right) \right)  .
\end{equation}
When $j = k$, the maximum over $n \in \{1, 2, \ldots, k-j\}$ is set equal to $1$ by convention.
\end{lemma}

\begin{remark} \label{rem:compositionbytheflowsimpler}
This estimate implies the simpler estimate
\begin{equation*}
\forall k \ge 1 , \qquad
	\dfrac{1}{k!} \abs*{\nabla^k \e^{t H_p} a}_g(\rho)
		\le \left(C k \e^{\Lambda \abs{t}}\right)^k \abs*{\nabla a}^{(k-1)}_g\left(\phi^t(\rho)\right) \jap*{\dfrac{C}{2 \Lambda} \abs*{\nabla^2 H_p}_{E}^{(k-2)}}^{k-1} ,
\end{equation*}
where $\abs*{\nabla a}_g^{(k-1)} = \max_{0 \le j \le k-1} \abs{\nabla^{1 + j} a}_g$ and $\abs{\nabla^2 H_p}_E^{(k-2)} = \max_{0 \le j \le k-2} \abs{\nabla^{2+j} H_p}_{g, E}$ (recall the notation $\abs{\bullet}_{g, E}$ in~\eqref{eq:notationgE}). This follows directly from the definition of $\bnorm{\nabla^2 H_p}_n^k$ in~\eqref{eq:defbnorm}.
\end{remark}

\begin{proof}
The statement for $k = 1$ follows from the chain rule. Thus we consider the case $k \ge 2$ in the sequel. We use the Faà di Bruno formula (Lemma~\ref{lem:FaadiBruno}) and Proposition~\ref{prop:estflowexpgrowth}. Lemma~\ref{lem:FaadiBruno} gives
\begin{equation*}
\dfrac{1}{k!} \abs*{\nabla^k (a \circ \phi^t)}_g(\rho)
	\le \sum_{j = 1}^k \dfrac{1}{j!} \abs*{(\nabla^j a)_{\phi^t(\rho)}}_g \sum_{\substack{n_1 + n_2 + \cdots + n_j = k \\ n_\ell \ge 1}} \prod_{\ell = 1}^j \dfrac{1}{n_\ell!} \abs*{\nabla^{n_\ell} \phi^t}_g .
\end{equation*}
We observe that there is a single term corresponding to $j = k$, given by the multi-index $(n_1, n_2, \ldots, n_j) = (1, 1, \ldots, 1)$. It is bounded by $\frac{1}{k!} \abs*{(\nabla^k a)_{\phi^t(\rho)}}_g (C \e^{\Lambda t})^k$. Now we focus on terms corresponding to $j < k$. We fix such a~$j$ and a multi-index $\bf{n} = (n_1, n_2, \ldots, n_j)$ such that $|\bf{n}| = k$. Up to relabeling, we can assume that $n_1 = n_2 = \cdots = n_{j_0} = 1$ and $n_\ell \ge 2$ for all $\ell > j_0$ ($j_0 < j$ due to the fact that $\bf{n} \in (\N^\ast )^j$ with $j < |\bf{n}| = k$).) We apply Proposition~\ref{prop:estflowexpgrowth} to the factors $\dfrac{1}{n_\ell!} \abs*{\nabla^{n_\ell} \phi^t}_g$ for $\ell > j_0$, and the other factors corresponding to $\ell \le j_0$ are bounded from above by $C e^{\Lambda \abs{t}}$. That yields
\begin{equation*}
\prod_{\ell = 1}^j \dfrac{1}{n_\ell!} \abs*{\nabla^{n_\ell} \phi^t}_g
	\le \left(\prod_{\ell = j_0 + 1}^j n_\ell^{n_\ell - 1}\right) \left(C \e^{\Lambda \abs{t}}\right)^k \max_{1 \le n \le k - j} \left(\left(\dfrac{C}{2 \Lambda}\right)^n \bnorm*{\nabla^2 H_p}_{n, E}^{k - j - n}\right) .
\end{equation*}
The maximum over~$j$ above comes from multiplying the maxima over $\{1, 2, \ldots, n_\ell - 1\}$ from Proposition~\ref{prop:estflowexpgrowth}.
The combinatorial factor can be estimated as follows:
\begin{equation*}
\prod_{\ell = j_0 + 1}^j n_\ell^{n_\ell - 1}
	\le \prod_{\ell = 1}^j (k-1)^{n_\ell - 1}
	= (k - 1)^{k - j} .
\end{equation*}
We used that $n_\ell \le k - 1$, owing to the fact that $j < k$, as well as $\abs{\bf{n}} = k$ and $n_\ell = 1$ for all $\ell \le j_0$.
Taking into account the factor depending on the $j$-th derivative of $a$, we obtain that each term in the sum over~$j$ is bounded by
\begin{equation*}
(k - 1)^{k - j} \left(C \e^{\Lambda t}\right)^k \max_{1 \le j \le k} \left( \dfrac{1}{j!} \abs*{(\nabla^j a)_{\phi^t(\rho)}}_g \max_{1 \le n \le k - j} \left(\left(\dfrac{C}{2 \Lambda}\right)^n \bnorm*{\nabla^2 H_p}_{n, E}^{k - j - n}\right) \right) ,
\end{equation*}
with the convention mentioned in the statement in case $j = k$. It remains to sum over $j$ the combinatorial factors $\sum_{j=1}^k (k - 1)^{k - j} \le k^k$. This yields the desired estimate~\eqref{eq:estcomposflow}.
\end{proof}

\subsection{Lipschitz-type properties of the flow}

We wish to prove a Lipschitz-type property of the flow. Here the slow variation of the metric $g$ matters. Recall the slow variation radius $r_g$ and the slow variation constant $C_g$ from Proposition~\ref{prop:improvedadmissibility}, as well as the Lyapunov exponent $\Lambda$ and the number $\Upsilon$ introduced in Assumption~\ref{assum:p}.

\begin{proposition} \label{prop:Lipschitzproperty}
Suppose $p$ and $g$ satisfy Assumptions~\ref{assum:mandatory} and~\ref{assum:p}, and define $C_p := \abs*{\nabla^3 p}_{S((\gain_g/\udl{\gain}_g)^{-1}, g)}^{(0)}$. Then for all $r \in (0, r_g]$ and all $t \in \R$ such that
\begin{equation*}
r
	 \le r_g \e^{- (\Lambda + C_g^3 C_p \udl{\gain}_g) \abs{t}} ,
\end{equation*}
we have
\begin{equation*}
\forall \rho_0 \in T^\star \mfd, \forall \rho \in \bar B_r^g(\rho_0), \qquad
	\abs*{\phi^t(\rho) - \phi^t(\rho_0)}_{g_{\phi^t(\rho_0)}}
		\le \abs*{\rho - \rho_0}_{g_{\rho_0}} \e^{(\Lambda + C_g^3 C_p \udl{\gain}_g) \abs{t}} .
\end{equation*}
\end{proposition}

\begin{proof}
Let $r \in (0, r_g]$. Given $\rho_0, \rho \in T^\star \mfd$ with $\abs{\rho - \rho_0}_{g_{\rho_0}} \le r$, we differentiate with respect to $t$ the quantity
\begin{equation*}
\cal{N}(t)
	:= \abs*{\phi^t(\rho) - \phi^t(\rho_0)}_{g_{\phi^t(\rho_0)}}^2 .
\end{equation*}
Writing for short $\rho_t = \phi^t(\rho_0)$, $\rho_{t, s} = (1 - s) \phi^t(\rho_0) + s \phi^t(\rho)$ and $\zeta_t = \phi^t(\rho) - \phi^t(\rho_0)$, we obtain
\begin{align} \label{eq:computationgronwall}
\dfrac{\dd}{\dd t} \cal{N}(t)
	&= 2 g_{\rho_t}\left( H_p\left(\phi^t(\rho)\right) - H_p\left(\phi^t(\rho_0)\right), \zeta_t \right) + \nabla_{H_p} g_{\rho_t} \left( \zeta_t, \zeta_t\right) \nonumber\\
	&= \int_0^1 2 g_{\rho_t}\left( \nabla_{\zeta_t} H_p(\rho_{t, s}), \zeta_t \right) \dd s + \nabla_{H_p} g_{\rho_t} \left( \zeta_t, \zeta_t \right) \nonumber\\
	&= (\lieder_{H_p} g)_{\rho_t}\left( \zeta_t, \zeta_t \right) + \int_0^1 g_{\rho_t}\left(\nabla_{\zeta_t} H_p(\rho_{t, s}) - \nabla_{\zeta_t} H_p(\rho_t), \zeta_t \right) \dd s \nonumber\\
	&= (\lieder_{H_p} g)_{\rho_t}\left( \zeta_t, \zeta_t \right) + \int_0^1 \int_0^1 g_{\rho_t}\left((\nabla_{\zeta_t, s \zeta_t}^2 H_p\left((1 - \tau) \rho_t + \tau \rho_{t,s}\right), \zeta_t \right) \dd \tau \dd s \nonumber\\
	&\le 2 \Lambda \cal{N}(t) + \sup_{s, \tau \in [0, 1]} \abs*{\nabla^2 H_p\left( (1 - \tau) \rho_t + \tau \rho_{t,s} \right)}_{g_{\rho_t}} \cal{N}(t)^3 .
\end{align}
The second equality comes from Taylor's theorem, then we make the Lie derivative appear by adding and subtracting $\nabla H_p(\rho_t)$, we use Taylor's theorem again in the fourth equality, and finally we use the definition of $\Lambda$ in Assumption~\ref{assum:p}~\ref{it:Lambda}.
We now set
\begin{equation*}
t_\star = \sup A
	\qquad \rm{where} \qquad
A = \set{T \ge 0}{\forall t \in [-T, T], \;\, \abs*{\phi^t(\rho) - \phi^t(\rho_0)}_{g_{\rho_t}} \le r_g} .
\end{equation*}
Notice that $0 \in A$, and $t_\star > 0$ by a continuity argument. If $t \in (-t_\star, t_\star)$, we have $(1 - \tau) \rho_t + \tau \rho_{t,s} = \rho_t + \tau s \zeta_t \in \bar B_{r_g}^g(\rho_t)$ for all $s, \tau \in [0, 1]$, so that slow variation of the metric $g$ combined with Lemma~\ref{lem:seminormsHp} gives
\begin{equation*}
\abs*{\nabla^2 H_p\left( (1 - \tau) \rho_t + \tau \rho_{t,s} \right)}_{g_{\rho_t}}
	\le C_g^3 \abs*{\nabla^2 H_p\left( \rho_t + \tau s \zeta_t \right)}_{g_{\rho_t + \tau s \zeta_t}}
	\le C_g^3 \sup_{T^\star \mfd} \gain_g \abs*{\nabla^3 p}_g
	\le C_g^3 C_p \udl{\gain}_g .
\end{equation*}
In addition, we have $\cal{N}(t) \le r_g^2 \le 1$, hence from~\eqref{eq:computationgronwall}:
\begin{align*}
\dfrac{\dd}{\dd t} \cal{N}(t)
	&\le \left( 2 \Lambda + C_g^3 C_p \udl{\gain}_g \right) \cal{N}(t)  .
\end{align*}
Therefore Grönwall's inequality yields
\begin{equation*}
\cal{N}(t)
	\le \cal{N}(0) \e^{(2 \Lambda + C_g^3 C_p \udl{\gain}_g) \abs{t}} ,
\end{equation*}
that is to say
\begin{equation*}
\abs*{\phi^t(\rho) - \phi^t(\rho_0)}_{g_{\phi^t(\rho_0)}}
	\le \abs*{\rho - \rho_0}_{g_{\rho_0}} \e^{(\Lambda + C_g^3 C_p \udl{\gain}_g) \abs{t}}
	\le r \e^{(\Lambda + C_g^3 C_p \udl{\gain}_g) t_\star} .
\end{equation*}
By a continuity argument, we conclude that $t_\star$ is such that $r \e^{(\Lambda + C_g^3 C_p \udl{\gain}_g) t_\star} \ge r_g$, which is a condition independent of $\rho_0, \rho$. Given a time $t \in \R$, in order to ensure that $t_\star \ge \abs{t}$, it suffices to require that
\begin{equation*}
\dfrac{r_g}{r}
	\ge \e^{(\Lambda + C_g^3 C_p \udl{\gain}_g) \abs{t}} ,
\end{equation*}
hence the result.
\end{proof}

\begin{proposition} \label{prop:temperancepropertyflow}
Suppose $p$ and $g$ satisfy Assumptions~\ref{assum:mandatory} and~\ref{assum:p}. Then there exist constants $C > 0$ and $N \ge 0$, depending only on structure constants of $g$, such that the following holds: for all $\rho_0 \in T^\star \mfd$ and all $\tilde \rho \in \bar B_{r_g}^g(\rho_0)$, we have
\begin{empheq}[left={\empheqlbrace},right={\qquad \forall \rho \in T^\star \mfd, \forall t \in [- T_E, T_E] ,}]{alignat=2}
    \abs*{\phi^t(\rho) - \phi^t(\tilde \rho)}_{g_{\phi^t(\rho_0)}^\sympf(t)}
		&\le C \abs*{\rho - \tilde \rho}_{g_{\rho_0}^\sympf} \jap*{\rho - \tilde \rho}_{g_{\rho_0}^\natural}^N , \label{eq:Lipgsympf}\\
    \abs*{\phi^t(\rho) - \phi^t(\tilde \rho)}_{g_{\phi^t(\rho_0)}^\natural}
		&\le C \e^{(\Lambda + 2 \Upsilon) \abs{t}} \abs*{\rho - \tilde \rho}_{g_{\rho_0}^\natural} \jap*{\rho - \tilde \rho}_{g_{\rho_0}^\natural}^N , \label{eq:Lipgnatural}
\end{empheq}
where $T_E$ is the Ehrenfest time introduced in~\eqref{eq:defTE}, and $g^\natural$ is the symplectic intermediate metric defined as the geometric mean of $g$ and $g^\sympf$ (see~\eqref{eq:defgnatural}).
\end{proposition}

\begin{proof}
We prove first the estimate~\eqref{eq:Lipgsympf} for the {$g^\sympf(t)$-norm}. Let $\tilde \rho, \rho$ be as in the statement. Introduce for any $s \in [0, 1]$ the point $\rho_s = (1 - s) \tilde \rho + s \rho$. Then we have
\begin{equation*}
\abs*{\phi^t(\rho) - \phi^t(\tilde \rho)}_{g_{\phi^t(\rho_0)}^\sympf}
	\le C_\Upsilon \e^{\Upsilon \abs{t}} \int_0^1 \abs*{\dd \phi^t(\rho_s).(\rho - \tilde \rho)}_{g_{\rho_0}^\sympf} \dd s
	\le C_g C_\Upsilon \e^{\Upsilon \abs{t}} \int_0^1 \abs*{\dd \phi^t(\rho_s).(\rho - \tilde \rho)}_{g_{\tilde \rho}^\sympf} \dd s .
\end{equation*}
This first inequality consists in applying the fundamental theorem of analysis, the triangle inequality and the fact that $g_{\rho_0} \le C_\Upsilon^2 \e^{2 \Upsilon \abs{t}} g_{\phi^t(\rho_0)}$ (Assumption~\ref{assum:p}~\ref{it:metriccontrol}) combined with {$\sympf$-duality}~\eqref{eq:sympfduality}. The second inequality comes from slow variation of $g$, which applies since $\tilde \rho \in \bar B_{r_g}^g(\rho_0)$, with {$\sympf$-duality} again. Then we use the improved admissibility of $g$ (Proposition~\ref{prop:improvedadmissibility}) with the points $\tilde \rho$ and $\rho_s$ to obtain
\begin{equation*}
\abs*{\phi^t(\rho) - \phi^t(\tilde \rho)}_{g_{\phi^t(\rho_0)}^\sympf}
	\le C_g^2 C_\Upsilon \e^{\Upsilon \abs{t}} \int_0^1 \abs*{\dd \phi^t(\rho_s).(\rho - \tilde \rho)}_{g_{\rho_s}^\sympf} \jap*{\rho_s - \tilde \rho}_{g_{\tilde \rho}^\natural}^{N_g} \dd s .
\end{equation*}
Next we use the fact that $g_{\rho_s} \le C_\Upsilon^2 \e^{2 \Upsilon \abs{t}} g_{\phi^t(\rho_s)}$ combined with {$\sympf$-duality}, together with the fact\footnote{This ``slow variation" property of $g^\natural$ on $\bar B_{r_g}^g(\rho_0)$ is a consequence of that of $g$ and $g^\sympf$, and the definition of $g^\natural$ as the geometric mean of $g$ and $g^\sympf$~\eqref{eq:defgnatural}. See~\eqref{eq:slowvargnatural} in Appendix~\ref{app:improvedadmissibility} for more details.} that $g_{\tilde \rho}^\natural \le C_g^2 g_{\rho_0}^\natural$:
\begin{equation*}
\abs*{\phi^t(\rho) - \phi^t(\tilde \rho)}_{g_{\phi^t(\rho_0)}^\sympf}
	\le C_g^{2 + N_g} C_\Upsilon^2 \e^{2 \Upsilon \abs{t}} \int_0^1 \abs*{\dd \phi^t(\rho_s).(\rho - \tilde \rho)}_{g_{\phi^t(\rho_s)}^\sympf} \jap*{s (\rho - \tilde \rho)}_{g_{\rho_0}^\natural}^{N_g} \dd s .
\end{equation*}
Then we apply the improved admissibility property again (Proposition~\ref{prop:improvedadmissibility}) for $g^\sympf$ with the points $\rho_s$ and $\tilde \rho$:
\begin{align*}
\abs*{\phi^t(\rho) - \phi^t(\tilde \rho)}_{g_{\phi^t(\rho_0)}^\sympf}
	&\le C_g^{3 + N_g} C_\Upsilon^2 \e^{2 \Upsilon \abs{t}} \int_0^1 \abs*{\dd \phi^t(\rho_s)}_{g_{\rho_s}^\sympf} \abs*{\rho - \tilde \rho}_{g_{\tilde \rho}^\sympf} \jap*{\rho_s - \tilde \rho}_{g_{\tilde \rho}^\natural}^{N_g} \jap*{\rho - \tilde \rho}_{g_{\rho_0}^\natural}^{N_g} \dd s \\
	&\le C_g^{4 + 2 N_g} C_\Upsilon^2 \e^{(2 \Upsilon + \Lambda) \abs{t}} \abs*{\rho - \tilde \rho}_{g_{\rho_0}^\sympf} \jap*{\rho - \tilde \rho}_{g_{\rho_0}^\natural}^{2 N_g} .
\end{align*}
For the last inequality, we used the estimate $\abs{\dd \phi^t}_{g^\sympf} \le e^{\Lambda \abs{t}}$ as a consequence of Propositions~\ref{prop:aprioriexpgrowth} and~\ref{prop:normssympfnatural}. We also used the fact that $g_{\tilde \rho}^\sympf \le C_g^2 g_{\rho_0}^\sympf$, thanks to slow variation of~$g$ and {$\sympf$-duality}. Finally, we obtain~\eqref{eq:Lipgsympf} recalling that $g^\sympf(t) = \e^{- 2 (2 \Upsilon + \Lambda) \abs{t}} g^\sympf$ by definition~\eqref{eq:defg(t)}.

We now prove the estimate~\eqref{eq:Lipgnatural} in the {$g^\natural$-norm}. The previous estimate applies with $g^\natural$ instead of $g$, with the following adaptations. The metric $g^\natural$ is admissible by~\cite[Proposition 2.2.20]{Lerner:10} and varies slowly on {$g$-balls} (namely $g_{\tilde \rho}^\natural \le C_g^2 g_{\rho_0}^\natural$ for $\abs{\tilde \rho - \rho_0}_{g_{\rho_0}} \le r_g$). In addition, Assumption~\ref{assum:p}~\ref{it:metriccontrol} is true for $g^\sympf$ with the same constants $\Upsilon, C_\Upsilon$ than $g$, as a consequence of {$\sympf$-duality}~\eqref{eq:sympfduality}. We deduce that the same is true for $g^\natural$, defined in~\eqref{eq:defgnatural} as the geometric mean of $g$ and $g^\sympf$ (apply a reasoning similar to~\eqref{eq:reasoningsimilarlater} in the proof of Proposition~\ref{prop:normssympfnatural}), namely $g_{\phi^t(\rho)}^\natural \le C_\Upsilon^2 \e^{2 \Upsilon \abs{t}} g_\rho^\natural$. We also have the estimate $\abs{\dd \phi^t}_{g^\natural} \le \e^{\Lambda \abs{t}}$ by Proposition~\ref{prop:normssympfnatural}. This concludes the proof of the proposition.
\end{proof}

\begin{remark}
In the proof of Proposition~\ref{prop:temperancepropertyflow}, we used the assumption that $g_{\phi^t(\rho)} \le C_\Upsilon^2 \e^{2 \Upsilon \abs{t}} g_\rho$ (Assumption~\ref{assum:p}~\ref{it:metriccontrol}). In fact we could bypass the use of this assumption here by slightly strengthening the admissibility assumption on $g$. Instead of assuming that $g$ is temperate, we could assume $g$ to be \emph{geodesically} temperate, namely to verify
\begin{equation*}
\exists C>0, N \ge 0 : \forall \rho_0, \rho \in T^\star \mfd, \qquad
	g_\rho
		\le C^2 \jap*{d_g(\rho_0, \rho)}^{2N} g_{\rho_0} ,
\end{equation*}
where $d_g$ is the geodesic distance. This implies temperance in the sense of Definition~\ref{def:admissiblemetric}, although we are not aware of examples of temperate metrics that are not geodesically temperate. This assumption is much more difficult to check in practice, although it is fulfilled for all the examples discussed in Section~\ref{subsec:examples} (see~\cite[Section 2.6.5]{Lerner:10}).
\end{remark}

An important corollary of Proposition~\ref{prop:temperancepropertyflow} is Proposition~\ref{prop:uniformm(t)} stated in Section~\ref{subsec:metricsonphasespace}, namely the uniform {$g(t)$-admissibility} of $\e^{t H_p} m$ for {$g$-admissible} weights~$m$.

\begin{proof}[Proof of Proposition~\ref{prop:uniformm(t)}]
\emph{Slow variation.}
Let $r_0 > 0$ be a common slow variation radius of $g$ and $m$ (given by Proposition~\ref{prop:improvedadmissibility}), and set $\tilde r_0 := r_0/(1+C_g^2)$. Let $\rho_0 \in T^\star \mfd$ and introduce
\begin{equation*}
t_\star
	= \sup \set{\tau \ge 0}{\forall \abs{t} \le \tau, \; \phi^t\left(\bar B_{\tilde r_0}^{g(t)}(\rho_0)\right) \subset \bar B_{r_0}^g\left(\phi^t(\rho_0)\right)} .
\end{equation*}
By a continuity argument, we have $t_\star > 0$, since $\tilde r_0 < r_0$. So let $t \in (-t_\star, t_\star)$. Using the fundamental theorem of calculus, we have for all $\rho \in \bar B_{\tilde r_0}^{g(t)}(\rho_0)$:
\begin{align*}
\abs*{\phi^t(\rho) - \phi^t(\rho_0)}_{g_{\phi^t(\rho_0)}}
	&\le \int_0^1 \abs*{\dd \phi^t\left((1-s) \rho_0 + s \rho\right). (\rho - \rho_0)}_{g_{\phi^t(\rho_0)}} \dd s \\
	&\le C_g \int_0^1 \abs*{\dd \phi^t\left((1-s) \rho_0 + s \rho\right). (\rho - \rho_0)}_{g_{\phi^t((1-s) \rho_0 + s \rho)}} \dd s \\
	&\le C_g \e^{\Lambda \abs{t}} \int_0^1 \abs*{\rho - \rho_0}_{g_{(1-s) \rho_0 + s \rho}} \dd s \\
	&\le C_g^2 \e^{\Lambda \abs{t}} \abs*{\rho - \rho_0}_{g_{\rho_0}}
	\le C_g^2 \abs*{\rho - \rho_0}_{g_{\rho_0}(t)}
	\le r_0 \dfrac{C_g^2}{1 + C_g^2} .
\end{align*}
We used slow variation of~$g$ in~$\bar B_{r_0}^g(\phi^t(\rho_0))$ in the second inequality, Proposition~\ref{prop:aprioriexpgrowth}~\eqref{eq:diffflowexp} in the third one, and slow variation of~$g$ in~$\bar B_{r_0}^g(\rho_0)$ in the fourth one.
We deduce that
\begin{equation*}
\forall t \in (-t_\star , t_\star), \qquad
	\phi^t\left(\bar B_{\tilde r_0}^{g(t)}(\rho_0)\right) \subset \bar B_{\frac{r_0 C_g^2}{1 + C_g^2}}^g\left(\phi^t(\rho_0)\right) ,
\end{equation*}
therefore by a continuity argument, one concludes that $t_\star = + \infty$. By slow variation of $m$, we deduce that
\begin{equation*}
\forall \rho_0 \in T^\star \mfd, \forall \rho \in \bar B_{\tilde r_0}^{g(t)}(\rho_0) , \qquad
	m\left(\phi^t(\rho)\right)
		\le C m\left(\phi^t(\rho_0)\right) ,
			\qquad \forall t \in \R ,
\end{equation*}
hence uniform {$g(t)$-slow} variation of $\e^{t H_p} m$.

\medskip
\emph{Temperance.} Applying Proposition~\ref{prop:improvedadmissibility} (improved admissibility), we have for any $\rho_0, \rho \in T^\star \mfd$:
\begin{equation*}
m\left(\phi^t(\rho)\right)
	\le C m\left(\phi^t(\rho_0)\right) \jap*{\phi^t(\rho) - \phi^t(\rho_0)}_{g_{\phi^t(\rho_0)}^\natural}^N .
\end{equation*}
The {$g^\natural$-norm} estimate from Proposition~\ref{prop:temperancepropertyflow}~\eqref{eq:Lipgnatural} yields
\begin{equation*}
\jap*{\phi^t(\rho) - \phi^t(\rho_0)}_{g_{\phi^t(\rho_0)}^\natural}
		\le C \jap*{\rho - \rho_0}_{\e^{2 (\Lambda + 2 \Upsilon) \abs{t}} g_{\rho_0}^\natural}^{N'+1} .
\end{equation*}
Now from~\eqref{eq:chainineq}, we have
\begin{equation}
\e^{2 (\Lambda + 2 \Upsilon) \abs{t}} g^\natural
	\le \e^{2 (\Lambda + 2 \Upsilon) \abs{t}} \gain_g g^\sympf
	\le \e^{4 (\Lambda + 2 \Upsilon) \abs{t}} \udl{\gain}_g g^\sympf(t) ,
\end{equation}
which yields $\e^{2 (\Lambda + 2 \Upsilon) \abs{t}} g^\natural \le g^\sympf(t)$ for $\abs{t} \le \frac{1}{2} T_E$. Therefore
\begin{equation*}
m\left(\phi^t(\rho)\right)
	\le C m\left(\phi^t(\rho_0)\right) \jap*{\rho - \rho_0}_{g_{\rho_0}^\sympf(t)}^{N (N'+1)} ,
\end{equation*}
hence uniform {$g(t)$-temperance} of $\e^{t H_p} m$ in the time range $\abs{t} \le \frac{1}{2} T_E$.
\end{proof}

\Large
\section{Mapping properties of the Egorov expansion operators} \label{sec:mapping}
\normalsize

We provide estimates on the operators $\cal{E}_j(t)$ that appear in the Dyson series~\eqref{eq:Dyson}. Throughout, we use the notation of Section~\ref{subsec:functionalframework}. In all this section, we work with a classical Hamiltonian $p$ and a Riemannian metric $g$ on $T^\star \mfd$ that satisfy Assumptions~\ref{assum:mandatory} and~\ref{assum:p}.

Let us first state a technical result explaining how the gain function and the temperance weight (defined in Definitions~\ref{def:gaing} and~\ref{def:thetag}) behave when composed by the Hamiltonian flow.

\begin{lemma} \label{lem:gain(t)}
The following holds:
\begin{equation*}
\e^{t H_p} \gain_g
	\le \gain_{g(t)}
		\quad {\rm and} \quad
\e^{t H_p} \theta_g
	\ge C_\Upsilon^{-1} \theta_{g(t)} ,
		\qquad \forall t \in \R ,
\end{equation*}
where $C_\Upsilon$ is the constant from Assumption~\ref{assum:p}~\ref{it:metriccontrol}.
\end{lemma}

\begin{proof}
Given $\rho \in T^\star \mfd$, since $\dd \phi^t$ is an isomorphism:
\begin{equation*}
\e^{t H_p} \gain_g(\rho)
	= \sup_{\zeta \in W \setminus \{0\}} \dfrac{\abs{\zeta}_{g_{\phi^t(\rho)}}}{\abs{\zeta}_{g_{\phi^t(\rho)}^\sympf}}
	= \sup_{\zeta \in W \setminus \{0\}} \dfrac{\abs{\dd \phi^t(\rho). \zeta}_{g_{\phi^t(\rho)}}}{\abs{\dd \phi^t(\rho). \zeta}_{g_{\phi^t(\rho)}^\sympf}}
	= \gain_{(\phi^t)^\pullb g}(\rho) .
\end{equation*}
Using Propositions~\ref{prop:aprioriexpgrowth} and~\ref{prop:normssympfnatural}, we have
\begin{equation*}
\left\{
\begin{aligned}
\abs*{\dd \phi^t(\rho). \zeta}_{g_{\phi^t(\rho)}}
	&\le \e^{\Lambda \abs{t}} \abs*{\zeta}_{g_\rho}
	\le \abs*{\zeta}_{g_\rho(t)} \\
\abs*{\zeta}_{g_\rho^\sympf(t)}
	&= 	\e^{- (\Lambda + 2 \Upsilon) \abs{t}} \abs*{\zeta}_{g_\rho^\sympf}
	\le \e^{- (\Lambda + 2 \Upsilon) \abs{t}} \abs*{\dd \phi^{-t}(\phi^t(\rho))}_{g^\sympf} \abs*{\dd \phi^t(\rho). \zeta}_{g_{\phi^t(\rho)}^\sympf}
	\le \abs*{\dd \phi^t(\rho). \zeta}_{g_{\phi^t(\rho)}^\sympf}
\end{aligned}
\right. .
\end{equation*}
Therefore, we deduce that
\begin{equation*}
\e^{t H_p} \gain_g(\rho)
	\le \sup_{\zeta \in W \setminus \{0\}} \dfrac{\abs{\zeta}_{g_\rho(t)}}{\abs{\zeta}_{g_\rho^\sympf(t)}}
	\le \gain_{g(t)} .
\end{equation*}
As for the temperance weight, we use Assumption~\ref{assum:p}~\ref{it:metriccontrol} together with {$\sympf$-duality} to obtain
\begin{equation*}
\forall \zeta \in W \setminus \{0\} , \qquad
	 \dfrac{\abs{\zeta}_{g_{\phi^t(\rho)}^\sympf}}{\abs{\zeta}_{{\sf g}}}
		\ge \dfrac{C_\Upsilon^{-1} \e^{-\Upsilon \abs{t}} \abs{\zeta}_{g_\rho^\sympf}}{\abs{\zeta}_{{\sf g}}}
		\ge C_\Upsilon^{-1} \dfrac{\abs{\zeta}_{g_\rho(t)^\sympf}}{\abs{\zeta}_{{\sf g}}} ,
\end{equation*}
so that taking the supremum over $\zeta$ yields the sought result.
\end{proof}

\subsection{Mapping properties in symbol classes} \label{subsec:contEcalsymb}

We investigate the mapping properties of $\e^{t H_p}$ on symbol classes. We use the notation introduced in Section~\ref{subsec:functionalframework}.

\begin{corollary}[Continuity of the Hamiltonian flow on symbol classes] \label{cor:continuityofflowinsymbolclasses}
Let $p$ and $g$ satisfy Assumptions~\ref{assum:mandatory} and~\ref{assum:p}. Let $m : T^\star \mfd \to \R_+^\ast$. Setting $g(t) = \e^{2 (\Lambda + 2 \Upsilon) \abs{t}} g$ and $m(t) = \e^{t H_p} m$, one has
\begin{equation*}
S(m, g) \xrightarrow{\e^{t H_p}} S\left(m(t), g(t)\right)
	\qquad \rm{and} \qquad
\nabla^{-1} S(m, g) \xrightarrow{\e^{t H_p}} \nabla^{-1} S\left(m(t), g(t)\right)
\end{equation*}
uniformly with respect to $t \in \R$.
\end{corollary}

\begin{remark}
We do not discuss admissibility of $g(t)$ or $m(t)$ in Corollary~\ref{cor:continuityofflowinsymbolclasses}, since symbol classes make sense for metrics and weights that are not necessarily admissible. However, we shall take care of this issue while using pseudo-differential calculus.
\end{remark}

\begin{proof}
This is a direct consequence of Proposition~\ref{prop:aprioriexpgrowth} and Lemma~\ref{lem:compositionbytheflow} (or rather the simplified estimate of Remark~\ref{rem:compositionbytheflowsimpler}). Indeed, if $a \in S(m, g)$, we have for all $k \ge 1$ and all $\rho \in \{p = E\}$:
\begin{align} \label{eq:intermediateineqflow}
\abs*{\nabla^k \e^{t H_p} a}_g(\rho)
		&\le C_{k} \e^{k \Lambda \abs{t}} \abs*{\nabla a}^{(k-1)}_g\left(\phi^t(\rho)\right) \jap*{\dfrac{1}{2 \Lambda} \abs*{\nabla^2 H_p}_{E}^{(k-2)}}^{k-1} \nonumber\\
		&\le C_{k} \e^{k (\Lambda + 2 \Upsilon) \abs{t}} \abs*{\nabla a}_{S(m, g)}^{(k-1)} m\left(\phi^t(\rho)\right) \jap*{\dfrac{1}{2 \Lambda} \abs*{\nabla^2 H_p}_{E}^{(k-2)}}^{k-1} .
\end{align}
Using Lemma~\ref{lem:seminormsHp} and the fact that $\nabla^3 p \in S((\gain_g/\udl{\gain}_g)^{-1}, g)$ (Assumption~\ref{assum:p}~\ref{it:strongsubquad}), we have
\begin{equation*}
\forall j \in \N , \qquad
	\abs*{\nabla^{2+j} H_p}_{g, E}
		\le \sup_{\{p = E\}} \gain_g \abs*{\nabla^{3+j} p}_g
		\le \udl{\gain}_g \abs*{\nabla^3 p}_{S((\gain_g/\udl{\gain}_g)^{-1}, g)}^{(j)} .
\end{equation*}
Plugging this into~\eqref{eq:intermediateineqflow}, and using~\eqref{eq:assumLambda}, we deduce that
\begin{equation*}
\abs*{\nabla^k \e^{t H_p} a}_{g(t)}(\rho)
		\le C_{k} \abs*{\nabla a}_{S(m, g)}^{(k-1)} m\left(\phi^t(\rho)\right) \jap*{\dfrac{1}{2 c} \abs*{\nabla^3 p}_{S((\gain_g/\udl{\gain}_g)^{-1}, g)}^{(k-2)}}^{k-1} .
\end{equation*}
This proves that $\e^{t H_p} : \nabla^{-1} S(m, g) \to \nabla^{-1} S(m(t), g(t))$, with continuity constants depending only on seminorms of $p$ (and not on $t$).
For $k = 0$, we also have
\begin{equation*}
\abs{a \circ \phi^t}(\rho)
	\le \abs*{a}_{S(m, g)}^{(0)} m(\phi^t(\rho)) ,
\end{equation*}
so that $\e^{t H_p} : S(m, g) \to S(m(t), g(t))$ too. This finishes the proof of the proposition.
\end{proof}

Recall that from Assumption~\ref{assum:p}~\ref{it:strongsubquad}, we have
\begin{equation*}
\nabla^3 p
	\in S\left( \left(\tfrac{\gain_g}{\udl{\gain}_g}\right)^{-1}, g\right) \cap S\left((\theta_g \udl{\gain}_g^{1/2})^{-\epsilon} \left(\tfrac{\gain_g}{\udl{\gain}_g}\right)^{-3}, g\right) ,
\end{equation*}
with $\epsilon \le 1/2$.
The key ingredient in the proof of the result below is pseudo-differential calculus.

\begin{lemma}[Continuity of $\cal{H}_p^{(3)}$ on symbol classes] \label{lem:continuityHp3}
Let $p$ and $g$ satisfy Assumptions~\ref{assum:mandatory} and~\ref{assum:p}. For any {$g$-admissible} weight $m$, the following holds:
\begin{equation*}
\nabla^{-3} S\left(m, g(t)\right)
	\xrightarrow{\cal{H}_p^{(3)}} S\left(\theta_{g(t)}^{-\epsilon} \udl{\gain}_g^{3/2} m, g(t)\right) \cap S\left(\gain_{g(t)}^2 \udl{\gain}_g \e^{-(\Lambda + 2 \Upsilon) \abs{t}} m, g(t)\right) ,
		\qquad \forall \abs{t} \le T_E .
\end{equation*}
\end{lemma}

\begin{proof}
In view of the definition of the operator $\cal{H}_p^{(3)}$ in~\eqref{eq:defHp3}, this is a direct application of pseudo-differential calculus (Proposition~\ref{prop:pseudocalcsymb}), namely on the one hand:
\begin{equation*}
\nabla^{-3} S\left((\theta_g \udl{\gain}_g^{1/2})^{-\epsilon} \left(\tfrac{\gain_g}{\udl{\gain}_g}\right)^{-3}, g\right) \times \nabla^{-3} S\left(m, g(t)\right)
	\xlongrightarrow{\widehat{\cal{P}}_3} S\left(\gain_{g(t), g}^3 m (\theta_g \udl{\gain}_g^{1/2})^{-\epsilon} \left(\tfrac{\gain_g}{\udl{\gain}_g}\right)^{-3}, g(t)\right) .
\end{equation*}
Recalling that we have
\begin{equation*}
\gain_{g(t), g}
	= \e^{(\Lambda + 2 \Upsilon) \abs{t}} \gain_g
\end{equation*}
(see the definition of the joint gain function~\eqref{eq:defjointgain}), we obtain
\begin{equation*}
\gain_{g(t), g}^3 \theta_g^{-\epsilon} \udl{\gain}_g^{-\epsilon/2} \left(\tfrac{\gain_g}{\udl{\gain}_g}\right)^{-3}
	= \theta_{g(t)}^{-\epsilon} \e^{(3-\epsilon) (\Lambda + 2 \Upsilon) \abs{t}} \udl{\gain}_g^{3 - \epsilon/2}
	\le \theta_{g(t)}^{-\epsilon} \udl{\gain}_g^{3/2} ,
		\qquad \forall \abs{t} \le T_E .
\end{equation*}
We used $\abs{t} \le T_E$ in the last inequality. Therefore, we deduce that
\begin{equation*}
\nabla^{-3} S\left(m, g(t)\right)
	\xrightarrow{\cal{H}_p^{(3)}}
		S\left(\theta_{g(t)}^{-\epsilon} \udl{\gain}_g^{3/2} m, g(t)\right) .
\end{equation*}
On the other hand, we have similarly
\begin{equation*}
\nabla^{-3} S\left(\left(\tfrac{\gain_g}{\udl{\gain}_g}\right)^{-1}, g\right) \times \nabla^{-3} S\left(m, g(t)\right)
	\xlongrightarrow{\widehat{\cal{P}}_3} S\left(\gain_{g(t), g}^3 m \left(\tfrac{\gain_g}{\udl{\gain}_g}\right)^{-1}, g(t)\right) ,
\end{equation*}
with
\begin{equation*}
\gain_{g(t), g}^3 \left(\tfrac{\gain_g}{\udl{\gain}_g}\right)^{-1}
	= \gain_{g(t)}^2 \e^{-(\Lambda + 2 \Upsilon) \abs{t}} \udl{\gain}_g .
\end{equation*}
Therefore, we deduce that
\begin{equation*}
\nabla^{-3} S\left(m, g(t)\right)
	\xrightarrow{\cal{H}_p^{(3)}}
		S\left(\gain_{g(t)}^2 \udl{\gain}_g \e^{-(\Lambda + 2 \Upsilon) \abs{t}} m, g(t)\right) ,
\end{equation*}
hence the result.
\end{proof}

We immediately deduce the following result on the mapping properties of the operators~$\cal{E}_j(t)$ introduced in Proposition~\ref{prop:Dyson}. These operators are defined as multiple compositions of~$\e^{t H_p}$ and of~$\cal{H}_p^{(3)}$, which we studied in Corollary~\ref{cor:continuityofflowinsymbolclasses} and Lemma~\ref{lem:continuityHp3} above. Let us highlight the fact that the following holds for $j \ge 1$ only. The case $j=0$ is covered by Corollary~\ref{cor:continuityofflowinsymbolclasses}, since $\cal{E}_0(t) = \e^{t H_p}$.

\begin{proposition}[Continuity estimates for $\cal{E}_j(t)$ on symbol classes, $j \ge 1$] \label{prop:continuityEcal}
Let $p$ and $g$ satisfy Assumptions~\ref{assum:mandatory} and~\ref{assum:p}. Let $m$ be a {$g$-admissible} weight and write $m(t) := \e^{t H_p} m$.
Then for any $j \in \N^\ast$, we have
\begin{equation} \label{eq:continuityEcalsymbolsjneq0}
\nabla^{-1} S(m, g)
	\xrightarrow{\cal{E}_j(t)} S\left(m(t) \theta_{g(t)}^{-j \epsilon}, g(t)\right) \cap S\left( m(t) \gain_{g(t)}^{2j}, g(t) \right) ,
		\qquad \forall \abs{t} \le \tfrac{1}{2} T_E .
\end{equation}
If in addition $m$ satisfies
\begin{equation} \label{eq:growthmalongtheflow}
\e^{t H_p} m
	\le C_\kappa \e^{\kappa \abs{t}} m ,
		\qquad \forall t \in \R ,
\end{equation}
for some $\kappa \ge 0$ and $C_\kappa > 0$, then we have
\begin{equation} \label{eq:continuityEcalsymbolsjneq0bis}
\nabla^{-1} S(m, g)
	\xrightarrow{\cal{E}_j(t)} S\left(\e^{\kappa \abs{t}} m \theta_{g(t)}^{-j \epsilon}, g(t)\right) \cap S\left( \e^{\kappa \abs{t}} m \gain_{g(t)}^{2j}, g(t) \right) ,
		\qquad \forall \abs{t} \le \tfrac{1}{2} T_E ,
\end{equation}
with implicit constants depending only on structure constants of~$m$ and~$g$, the constant $C_\kappa$, as well as seminorms of~$p$ and the constant~$C_\Upsilon$ from Assumption~\ref{assum:p}.
\end{proposition}

\begin{proof}
We prove~\eqref{eq:continuityEcalsymbolsjneq0} by induction on $j \in \N^\ast$.
The basis step $j = 1$ is as follows: by Corollary~\ref{cor:continuityofflowinsymbolclasses}, we have
\begin{equation} \label{eq:step11}
\nabla^{-1} S(m, g)
	\xlongrightarrow{\e^{s_1 H_p}}
		\nabla^{-1} S\left(m(s_1), g(s_1)\right) .
\end{equation}
For any $\abs{s_1} \le T_E$, we apply Lemma~\ref{lem:continuityHp3} (dependence on structure constants of $m(t)$ does not degenerate since those constants are uniform in $\abs{t} \le \frac{1}{2} T_E$ by Proposition~\ref{prop:uniformm(t)}):
\begin{equation}  \label{eq:step12}
\nabla^{-1} S\left(m(s_1), g(s_1)\right)
	\xlongrightarrow{\cal{H}_p^{(3)}} S\left(m(s_1) \theta_{g(s_1)}^{-\epsilon} \udl{\gain}_g^{3/2}, g(s_1)\right) \cap S\left(\gain_{g(s_1)}^2 \udl{\gain}_g \e^{-(\Lambda + 2 \Upsilon) \abs{s_1}} m(s_1), g(s_1)\right)  .
\end{equation}
Next we apply Corollary~\ref{cor:continuityofflowinsymbolclasses} again to obtain
\begin{equation} \label{eq:step13}
S\left(m(s_1) \theta_{g(s_1)}^{-\epsilon} \udl{\gain}_g^{3/2}, g(s_1)\right)
	\xlongrightarrow{\e^{(t - s_1) H_p}} S\left(\e^{(t - s_1) H_p} (m(s_1) \theta_{g(s_1)}^{-\epsilon}) \udl{\gain}_g^{3/2}, g(t)\right)  .
\end{equation}
From Lemma~\ref{lem:gain(t)}, we also have
\begin{equation} \label{eq:step14}
\e^{(t - s_1) H_p} (m(s_1) \theta_{g(s_1)}^{-\epsilon})
	\le m(t) \theta_{g(t)}^{-\epsilon} C_\Upsilon^\epsilon .
\end{equation}
Similarly, we deduce from Corollary~\ref{cor:continuityofflowinsymbolclasses} and Lemma~\ref{lem:gain(t)} that
\begin{equation} \label{eq:step15}
S\left(\gain_{g(s_1)}^2 \udl{\gain}_g \e^{-(\Lambda + 2 \Upsilon) \abs{s_1}} m(s_1), g(s_1)\right)
	\xlongrightarrow{\e^{(t - s_1) H_p}} S\left(\gain_{g(t)}^2 \udl{\gain}_g \e^{-(\Lambda + 2 \Upsilon) \abs{s_1}} m(t), g(t)\right) .
\end{equation}
Combining the estimates~\eqref{eq:step11}, \eqref{eq:step12}, \eqref{eq:step13}, \eqref{eq:step14} and~\eqref{eq:step15}, we obtain
\begin{equation} \label{eq:firstestcontin}
\nabla^{-1} S(m, g)
	\xlongrightarrow{\e^{(t - s_1) H_p} \cal{H}_p^{(3)} \e^{s_1 H_p}} S\left(\theta_{g(t)}^{-\epsilon} \udl{\gain}_g^{3/2} m(t), g(t)\right) \cap S\left(\gain_{g(t)}^2 \udl{\gain}_g \e^{-(\Lambda + 2 \Upsilon) \abs{s_1}} m(t), g(t)\right) .
\end{equation}
Then it remains to integrate over $s_1$. For all $\abs{t} \le T_E$, we have in view of the definition~\eqref{eq:defTE} of $T_E$ and the assumption~\eqref{eq:assumLambda}:
\begin{equation} \label{eq:getridoffactort}
\int_{[0, t]} \udl{\gain}_g^{3/2} \dd t
	\le T_E \udl{\gain}_g^{3/2}
	= \dfrac{\udl{\gain}_g^{3/2}}{2 (\Lambda + 2 \Upsilon)} \log\left(\dfrac{1}{\udl{\gain}_g}\right)
	\le \dfrac{C_{1/2}}{2 c} ,
\end{equation}
where $C_\delta = \sup_{\tau \in [0, 1]} \tau^\delta \log\abs{\tau}^{-1}$. Likewise, we find
\begin{equation} \label{eq:getrid2}
\int_{t \Delta_1} \udl{\gain}_g \e^{- (\Lambda + 2 \Upsilon) \abs{s_1}} \dd s_1
	\le \dfrac{\udl{\gain}_g}{\Lambda + 2 \Upsilon}
	\le \dfrac{1}{c} .
\end{equation}
Combining~\eqref{eq:firstestcontin}, \eqref{eq:getridoffactort} and~\eqref{eq:getrid2}, we obtain~\eqref{eq:continuityEcalsymbolsjneq0} for $j = 1$.

Now we assume that the assertion~\eqref{eq:continuityEcalsymbolsjneq0} is true at step $j$, and prove it at step $j+1$. Given that we have from~\eqref{eq:recurrencerelation}:
\begin{equation} \label{eq:inductiononDysonterms}
\cal{E}_{j+1}(t)
	= \int_0^t \e^{(t - s_{j+1}) H_p} \cal{H}_p^{(3)} \cal{E}_j(s_{j+1}) \dd s_{j+1} ,
\end{equation}
it suffices to check that
\begin{align} \label{eq:inclsuffices2}
S\left(m(s_{j+1}) \theta_{g(s_{j+1})}^{-j \epsilon}, g(s_{j+1})\right)
	&\xrightarrow{\e^{(t - s_{j+1}) H_p} \cal{H}_p^{(3)}}
		S\left(m(t) \theta_{g(t)}^{-(j+1) \epsilon} \udl{\gain}_g^{3/2}, g(t)\right) \nonumber\\
S\left(m(s_{j+1}) \gain_{g(s_{j+1})}^{2j}, g(s_{j+1})\right)
	&\xrightarrow{\e^{(t - s_{j+1}) H_p} \cal{H}_p^{(3)}}
		S\left(m(t) \gain_{g(t)}^{2(j+1)} \udl{\gain}_g \e^{-(\Lambda + 2 \Upsilon) \abs{s_{j+1}}}, g(t)\right)
\end{align}
($s_{j+1}$ now plays the role of $t$ in the {$j$-th} step).
This follows first from an application of Lemma~\ref{lem:continuityHp3} for $\cal{H}_p^{(3)}$:
\begin{align*}
S\left(m(s_{j+1}) \theta_{g(s_{j+1})}^{-j \epsilon}, g(s_{j+1})\right)
	&\xrightarrow{\cal{H}_p^{(3)}}
		S\left(m(s_{j+1}) \theta_{g(s_{j+1})}^{-(j+1) \epsilon} \udl{\gain}_g^{3/2}, g(s_{j+1})\right) \nonumber\\
S\left(m(s_{j+1}) \gain_{g(s_{j+1})}^{2j}, g(s_{j+1})\right)
	&\xrightarrow{\cal{H}_p^{(3)}}
		S\left(m(s_{j+1}) \gain_{g(s_{j+1})}^{2(j+1)} \udl{\gain}_g \e^{- (\Lambda + 2 \Upsilon) \abs{s_{j+1}}}, g(s_{j+1})\right)
\end{align*}
together with Corollary~\ref{cor:continuityofflowinsymbolclasses} to handle the composition with $\e^{(t - s_{j+1}) H_p}$, and the fact that
\begin{equation} \label{eq:ineqff}
\e^{(t - s_{j+1}) H_p} \left(m(s_{j+1}) \theta_{g(s_{j+1})}^{-(j+1) \epsilon}\right) \udl{\gain}_g^{3/2}
	\le m(t) \theta_{g(t)}^{-(j+1) \epsilon} \udl{\gain}_g^{3/2} C_\Upsilon^\epsilon ,
\end{equation}
as well as
\begin{equation} \label{eq:ineqfff}
\e^{(t - s_{j+1}) H_p} \left(m(s_{j+1}) \gain_{g(s_{j+1})}^{2(j+1)}\right)
	\le m(t) \gain_{g(t)}^{2(j+1)}
\end{equation}
from Lemma~\ref{lem:gain(t)}.
The factor $\udl{\gain}_g^{3/2}$ in~\eqref{eq:ineqff} allows to absorb the factor $t$ that appears while considering the integral over $s_{j+1}$ in~\eqref{eq:inductiononDysonterms}, exactly as in~\eqref{eq:getridoffactort}. Similarly, the factor $\udl{\gain}_g \e^{- (\Lambda + 2 \Upsilon) \abs{s_{j+1}}}$ in the right-hand side of~\eqref{eq:inclsuffices2} allows to get rid of the extra integral over $s_{j+1}$ as in~\eqref{eq:getrid2}.
This finishes the induction.

The case where $m$ satisfies in addition~\eqref{eq:growthmalongtheflow} follows by plugging $C_\kappa \e^{\kappa \abs{t}} m$ in place of $m(t)$.
\end{proof}

\subsection{Mapping properties in spaces of confined symbols}

We conduct exactly the same investigation as in Section~\ref{subsec:contEcalsymb} but on spaces of confined symbols instead of symbol classes.

\begin{proposition}[Continuity of the flow] \label{prop:continuityflowconf}
Let $p$ and $g$ satisfy Assumptions~\ref{assum:mandatory} and~\ref{assum:p}. Let $r_0 \in (0, r_g]$ and $T \in [0, T_E]$, and suppose they satisfy
\begin{equation} \label{eq:conditionr}
r(T)
	\le r_g
		\qquad \rm{with } \;\, r(\tau) := r_0 \e^{(2 (\Lambda + \Upsilon) + C_g^3 C_p \udl{\gain}_g) \abs{\tau}} .
\end{equation}
Then for any $t_1, t_2$ such that $t := \abs{t_1} + \abs{t_2} \le T$ we have
\begin{equation*}
\Conf_{r(t_1)}^{g(t_1)}(\rho_0)
	\xrightarrow{\e^{t_2 H_p}} \Conf_{r(t)}^{g(t)}\left(\phi^{-t_2}(\rho_0)\right) ,
\end{equation*}
uniformly in $\rho_0 \in T^\star \mfd$.
\end{proposition}

\begin{remark} \label{rmk:CLambda1}
Here it seems that it is important to have a control of the flow of the form $\abs{\dd \phi^t}_g \le C_\Lambda \e^{\Lambda \abs{t}}$ with $C_\Lambda = 1$ (this is the case under Assumption~\ref{assum:p}; see Proposition~\ref{prop:aprioriexpgrowth}). Indeed, if $C_\Lambda > 1$, the confinement radius $r$ could be instantaneously increased by a factor $C_\Lambda$ under the action of $\e^{t H_p}$, $t > 0$. This could prevent us from keeping $r(t)$ under control after several applications of the Hamiltonian flow in the higher order terms of the Dyson expansion~\eqref{eq:Dyson}.
\end{remark}

\begin{proof}
In the proof, we write $\tilde r(\bullet) := r_0 \e^{(\Lambda + C_g^3 C_p \udl{\gain}_g) \abs{\bullet}}$.

Firstly, by definition of~$g(t)$ (see~\eqref{eq:defg(t)}) and of~$\tilde r$ above, we have
\begin{equation} \label{eq:tilderhodef}
B_{r(t_1)}^{g(t_1)}(\rho_0)
	 = B_{\tilde r(t_1)}^g(\rho_0) ,
	 	\qquad \forall \rho_0 \in T^\star \mfd .
\end{equation}
Therefore Proposition~\ref{prop:Lipschitzproperty} (with $\tilde r(t_1)$ in place of~$r$) gives
\begin{equation} \label{eq:insideball}
\phi^{-t_2}\left(\bar B_{r(t_1)}^{g(t_1)}(\rho_0)\right)
	= \phi^{-t_2}\left(\bar B_{\tilde r(t_1)}^g(\rho_0)\right)
	\subset \bar B_{\tilde r(t)}^g \left( \phi^{-t_2}(\rho_0) \right)
	= \bar B_{r(t)}^{g(t)} \left( \phi^{-t_2}(\rho_0) \right) .
\end{equation}
Here, we can apply Proposition~\ref{prop:Lipschitzproperty} since $\tilde r(t_1) \le \tilde r(T) e^{- (\Lambda + C_g^3 C_p \udl{\gain}_g) \abs{t_2}} \le r_g e^{- (\Lambda + C_g^3 C_p \udl{\gain}_g) \abs{t_2}}$ in view of~\eqref{eq:conditionr}.

Secondly, given $\rho$ outside $B_1 := B_{r(t_1)}^{g(t_1)}(\rho_0) = B_{\tilde r(t_1)}^g(\rho_0)$, we estimate the {$g_{\rho_0}^\sympf(t_1)$-distance} from $\phi^{-t_2}(\rho)$ to $\phi^{-t_2}(B_1)$. We pick $\tilde \rho \in \bar B_1$ such that
\begin{equation} \label{eq:choiceoftilderho}
\abs*{\rho - \tilde \rho}_{g_{\rho_0}^\sympf(t_1)}
	= \dist_{g_{\rho_0}^\sympf(t_1)}\left(\rho, B_1\right) .
\end{equation}
Notice that by definition of $\tilde \rho$, we have $\abs{\tilde \rho - \rho_0}_{g_{\rho_0}(t_1)} \le r(t_1)$, so
\begin{equation} \label{eq:slowvariationapplies}
\abs*{\tilde \rho - \rho_0}_{g_{\rho_0}}
	\le \abs*{\tilde \rho - \rho_0}_{g_{\rho_0}(t_1)}
	\le r_g ,
\end{equation}
due to the assumption~\eqref{eq:conditionr}. Applying Proposition~\ref{prop:temperancepropertyflow}, we have
\begin{equation*}
\abs*{\phi^{-t_2}(\rho) - \phi^{-t_2}(\tilde \rho)}_{g_{\phi^{-t_2}(\rho_0)}^\sympf(t_2)}
	\le C \abs*{\rho - \tilde \rho}_{g_{\rho_0}^\sympf} \jap*{\rho - \tilde \rho}_{g_{\rho_0}^\natural}^{N} .
\end{equation*}
Now recall that by~\eqref{eq:chainineq}, we have
\begin{equation*}
g^\natural
	\le \gain_g g^\sympf
	\le \udl{\gain}_g g^\sympf
	= \left(\udl{\gain}_g^{-1} g\right)^\sympf
	= g^\sympf(T_E)
	\le g^\sympf(t) ,
		\qquad \forall t \in [-T_E, T_E]
\end{equation*}
(we used~\eqref{eq:sympfduality} in the last inequality). Combining this with the definition of $\tilde \rho$ in~\eqref{eq:choiceoftilderho}, we obtain
\begin{equation*}
\abs*{\rho - \tilde \rho}_{g_{\rho_0}^\natural}
	\le \abs*{\rho - \tilde \rho}_{g_{\rho_0}^\sympf(t_1)}
	= \dist_{g_{\rho_0}^\sympf(t_1)}\left(\rho, B_1\right) .
\end{equation*}
We deduce that
\begin{equation*}
\dist_{g_{\phi^{-t_2}(\rho_0)}^\sympf(t_2)}\left(\phi^{-t_2}(\rho), \phi^{-t_2}\left(B_1\right)\right)
	\le C(g) \dist_{g_{\rho_0}^\sympf}\left(\rho, B_1\right) \jap*{\dist_{g_{\rho_0}^\sympf(t_1)}\left(\rho, B_1\right)}^{N(g)} ,
\end{equation*}
so that multiplying by $\e^{- (2 \Upsilon + \Lambda) \abs{t_1}}$ we arrive at
\begin{equation} \label{eq:outsideball}
\dist_{g_{\phi^{-t_2}(\rho_0)}^\sympf(t)}\left(\phi^{-t_2}(\rho), \phi^{-t_2}\left(B_1\right)\right)
	\le C(g) \dist_{g_{\rho_0}^\sympf(t_1)}\left(\rho, B_1\right) \jap*{\dist_{g_{\rho_0}^\sympf(t_1)}\left(\rho, B_1\right)}^{N(g)} .
\end{equation}
Now that we have~\eqref{eq:insideball} and~\eqref{eq:outsideball}, we can write for any $\rho \in T^\star \mfd$:
\begin{align} \label{eq:flowweight}
\dist_{g_{\phi^{-t_2}(\rho_0)}^\sympf(t)}\left(\phi^{-t_2}(\rho), B_{r(t)}^{g(t)}(\phi^{-t_2}(\rho_0))\right)
	&\le \dist_{g_{\phi^{-t_2}(\rho_0)}^\sympf(t)}\left(\phi^{-t_2}(\rho), \phi^{-t_2}\left(B_1\right)\right) \nonumber\\
	&\le C(g) \jap*{\dist_{g_{\rho_0}^\sympf(t_1)}\left(\rho, B_{r(t_1)}^{g(t_1)}(\rho_0)\right)}^{N(g)+1} .
\end{align}

To conclude the proof of the proposition, it suffices to have bounds on the {$g(t)$-norm} of $\nabla^\ell \e^{t_2 H_p} \psi_{\rho_0}$ with $\psi_{\rho_0} \in \Conf_{r(t_1)}^{g(t_1)}(\rho_0)$, for any $\ell \in \N$. We have from Lemma~\ref{lem:compositionbytheflow}:
\begin{align*}
\abs*{\nabla^\ell \e^{t_2 H_p} \psi_{\rho_0}(\rho)}_{g(t)}
	&\le C_\ell(p) \max_{0 \le j \le \ell} \abs*{\nabla^j \psi_{\rho_0}}_g\left(\phi^{t_2}(\rho)\right) \\
	&\le C_\ell(p) \abs*{\psi_{\rho_0}}_{\Conf_{r(t_1)}^{g(t_1)}(\rho_0)}^{(\ell k)} \jap*{\dist_{g_{\rho_0}^\sympf(t_1)}\left(\phi^{t_2}(\rho), B_{r(t_1)}^{g(t_1)}(\rho_0)\right)}^{-\ell k} \\
	&\le C_\ell(p) C'(g) \abs*{\psi_{\rho_0}}_{\Conf_{r(t_1)}^{g(t_1)}(\rho_0)}^{(\ell k)} \jap*{\dist_{g_{\rho_0}^\sympf(t)}\left(\rho, B_{r(t)}^{g(t)}(\phi^{-t_2}(\rho_0))\right)}^{-\frac{\ell k}{N(g)+1}} ,
\end{align*}
using~\eqref{eq:flowweight} in the last inequality.
Notice that it was important to consider the {$g(t)$-norm} of $\nabla^\ell \e^{t_2 H_p} \psi_{\rho_0}$, instead of the norm with respect to the constant metric $g_{\rho_0}(t)$, in order to apply Lemma~\ref{lem:compositionbytheflow} (see the discussion in Appendix~\ref{app:Confseminorms} on equivalent seminorms on the spaces $\Conf_r^g(\rho_0)$). Choosing $k = \lceil N(g)+1 \rceil$, we conclude that
\begin{equation*}
\abs*{\e^{t_2 H_p} \psi_{\rho_0}}_{\Conf_{r(t)}^{g(t)}(\phi^{-t_2}(\rho_0))}^{(\ell)}
	\le C \abs*{\psi_{\rho_0}}_{\Conf_{r(t_1)}^{g(t_1)}(\rho_0)}^{(\ell k)} .
\end{equation*}
The constant $C$ is uniform in $\abs{t_1} + \abs{t_2} \le T$ and $\rho_0 \in T^\star \mfd$.
\end{proof}

We now turn to mapping properties of $\cal{H}_p^{(3)}$. Proposition~\ref{prop:mappingremainderConf} below is an instance of a strong version of pseudo-differential calculus in spaces of confined symbols, and follows from Proposition~\ref{prop:pseudocalcconf} (proved in Appendix~\ref{subsec:proofspdeudocalc}).

\begin{proposition} \label{prop:mappingremainderConf}
Let $p$ and $g$ satisfy Assumptions~\ref{assum:mandatory} and~\ref{assum:p}. Let $r_0 \in (0, r_g]$ and $T \in [0, T_E]$, and suppose they satisfy $r(T) \le r_g$ where $r(\bigcdot)$ is defined in~\eqref{eq:defr(t)}. Then the following holds:
\begin{equation*}
\Conf_{r(t)}^{g(t)}(\rho_0)
	\xlongrightarrow{\cal{H}_p^{(3)}} \theta_{g(t)}^{-\epsilon}(\rho_0) \udl{\gain}_g^{3/2} \Conf_{r(t)}^{g(t)}(\rho_0)
\end{equation*}
uniformly with respect to $\rho_0 \in T^\star \mfd$ and $\abs{t} \le T_E$.
\end{proposition}

\begin{proof}[Proof of Proposition~\ref{prop:mappingremainderConf} from Proposition~\ref{prop:pseudocalcconf}]
Given that $\cal{H}_p^{(3)} = \widehat{\cal{P}}_3(p, \bigcdot) - \widehat{\cal{P}}_3(\bigcdot, p)$ (recall~\eqref{eq:defHp3}) with $p \in \nabla^{-3} S((\theta_g \udl{\gain}_g^{1/2})^{- \epsilon} (\gain_g/ \udl{\gain}_g)^{-3}, g)$, we apply Proposition~\ref{prop:pseudocalcconf} with $j = 3$ and $(g, g(t))$ in place of $(g, g_0)$. Noticing that~$g$ and~$g(t)$ have common slow variation radius~$r_g$ (see Remark~\ref{rmk:uniformadmissibilityg(t)}) and the weight $\theta_g$ is {$g(t)$-admissible} with slow variation radius $r_g$ (Propositions~\ref{prop:temperanceweight} and~\ref{prop:improvedadmissibility}). From~\eqref{eq:continuityhatcalPj},
\begin{equation*}
\Conf_{r(t)}^{g(t)}(\rho_0) \times \nabla^{-3} S\left((\theta_g \udl{\gain}_g^{1/2})^{-\epsilon} (\gain_g/\udl{\gain}_g)^{-3}, g\right) \\
	\xlongrightarrow{\widehat{\cal{P}}_3} \left(\dfrac{\gain_{g(t), g}(\rho_0)}{\gain_g(\rho_0)} \udl{\gain}_g\right)^3 \theta_g^{-\epsilon}(\rho_0) \udl{\gain}_g^{-\epsilon/2} \Conf_{r(t)}^{g(t)}(\rho_0) .
\end{equation*}
(Since $\widehat{\cal{P}}_3$ is skew-symmetric, we can also invert the order of the two factors in the left-hand side.)
In addition, for any $\abs{t} \le T_E$, the definitions of $\gain_{g(t), g}$ in~\eqref{eq:defjointgain} and of $\theta_g$ (Definition~\ref{def:thetag}) yield
\begin{equation*}
\left(\dfrac{\gain_{g(t), g}(\rho_0)}{\gain_g(\rho_0)} \udl{\gain}_g\right)^3 \theta_g^{-\epsilon}(\rho_0) \udl{\gain}_g^{-\epsilon/2}
	= \e^{(3 - \epsilon) (\Lambda + 2 \Upsilon) \abs{t}} \udl{\gain}_g^{3 - \epsilon/2} \theta_{g(t)}^{-\epsilon}(\rho_0)
	\le \theta_{g(t)}^{-\epsilon}(\rho_0) \udl{\gain}_g^{3/2} ,
\end{equation*}
and the result follows.
\end{proof}

Putting together Proposition~\ref{prop:continuityflowconf} and Proposition~\ref{prop:mappingremainderConf}, we obtain the following.

\begin{proposition}[Continuity estimates for $\cal{E}_j(t)$ on spaces of confined symbols] \label{prop:continuityEcalconf}
Let $r_0 \in (0, r_g]$ and $T \in [0, T_E]$, and suppose they satisfy $r(T) \le r_g$ where $r(\bigcdot)$ is defined in~\eqref{eq:defr(t)}.
Then for any $j \in \N$, we have
\begin{equation*}
\Conf_{r_0}^g(\rho_0)
	\xrightarrow{\cal{E}_j(t)} \theta_{g(t)}^{-j \epsilon}\left(\phi^{-t}(\rho_0)\right) \Conf_{r(t)}^{g(t)}\left(\phi^{-t}(\rho_0)\right) ,
\end{equation*}
uniformly with respect to $\rho_0 \in T^\star \mfd$ and $\abs{t} \le T$.
\end{proposition}

\begin{proof}
We prove the statement by induction. For $j = 0$, we have $\cal{E}_0(t) = \e^{t H_p}$, so the result follows from Proposition~\ref{prop:continuityflowconf}. Assume that the statement is true at step $j \ge 0$. We prove that for all $t \in [-T, T]$ and all $s \in t \Delta_1$, the following holds:
\begin{equation*}
\Conf_{r(s)}^{g(s)}\left(\phi^{-s}(\rho_0)\right)
	\xrightarrow{\e^{(t - s) H_p} \cal{H}_p^{(3)}} \theta_{g(t)}^{-\epsilon}\left(\phi^{-t}(\rho_0)\right) \udl{\gain}_g^{3/2} \Conf_{r(t)}^{g(t)}\left(\phi^{-t}(\rho_0)\right) .
\end{equation*}
Applying Proposition~\ref{prop:mappingremainderConf} and Proposition~\ref{prop:continuityflowconf}, it follows that
\begin{multline*}
\Conf_{r(s)}^{g(s)}\left(\phi^{-s}(\rho_0)\right)
	\xrightarrow{\cal{H}_p^{(3)}} \theta_{g(s)}^{-\epsilon}\left(\phi^{-s}(\rho_0)\right) \udl{\gain}_g^{3/2} \Conf_{r(s)}^{g(s)}\left(\phi^{-s}(\rho_0)\right) \\
	\xrightarrow{\e^{(t - s) H_p}} \theta_{g(s)}^{-\epsilon}\left(\phi^{-s}(\rho_0)\right) \udl{\gain}_g^{3/2} \Conf_{r(t)}^{g(t)}\left(\phi^{-t}(\rho_0)\right) .
\end{multline*}
By Lemma~\ref{lem:gain(t)} we have
\begin{equation*}
\theta_{g(s)}^{-\epsilon}\left(\phi^{-s}(\rho_0)\right)
	= \theta_{g(s)}^{-\epsilon}\left(\phi^{t-s}(\phi^{-t}(\rho_0))\right)
	\le \theta_{g(t)}^{-\epsilon}\left(\phi^{-t}(\rho_0)\right) C_\Upsilon^\epsilon .
\end{equation*}
Combining this with the induction hypothesis, in view of~\eqref{eq:recurrencerelation}, we deduce that
\begin{equation*}
\Conf_r^g(\rho_0)
	\xrightarrow{\cal{E}_{j+1}(t)} I(t) \Conf_{r(t)}^{g(t)}\left(\phi^{-t}(\rho_0)\right) ,
\end{equation*}
with
\begin{align*}
I(t)
	= \abs{t} \theta_{g(t)}^{-(j+1) \epsilon}\left(\phi^{-t}(\rho_0)\right) \udl{\gain}_g^{3/2}
	\le \dfrac{C_{1/2}}{2 c} \theta_{g(t)}^{-(j+1) \epsilon}\left(\phi^{-t}(\rho_0)\right)
\end{align*}
as we did in~\eqref{eq:getridoffactort}, where the constants $C_{1/2}, c$ are independent of $\udl{\gain}_g$. We obtain the sought statement, and the induction is finished.
\end{proof}

\Large
\section{The quantum dynamics} \label{sec:quantum}
\normalsize

In this section, we establish Egorov's theorem for a specific class of symbols, whose gradient is controlled by (a power of) the temperance weight $\theta_g$ (Corollary~\ref{cor:Egorovwithnabla-1symbols}). This covers in particular affine symbols. Another important result of this section is Lemma~\ref{lem:iteratedcommutators} in which we investigate quantitatively the action of commutators of symbols of the form $\e^{t \cal{H}_p} a$ with affine symbols, in preparation for the application of Beals' theorem (Proposition~\ref{prop:Beals}). In the same vein, Lemma~\ref{lem:combinatorialmultiplication} deals with one-sided compositions with affine symbols, so as to investigate the decay of $\e^{t \cal{H}_p} a$ when $a$ is a confined symbol.

\subsection{Egorov expansion for a specific class of symbols} \label{subsec:Egorovspecificclass}

We start with the proof of the Dyson series expansion, valid on the Schwartz class.

\begin{proof}[Proof of Proposition~\ref{prop:Dyson}]
First recall that $\sch(T^\star \mfd) \subset \dom \cal{H}_p$ and that $\cal{H}_p$ preserves the Schwartz class (a consequence of Proposition~\ref{prop:pseudocalcconf}), so that $\sch(T^\star \mfd) \subset \dom \cal{H}_p^k$ for all $k \in \N$. The same holds for $H_p$, namely $\sch(T^\star \mfd) \subset \dom H_p^k$ for all $k \in \N$, since $p$ has temperate growth under Assumption~\ref{assum:p} (see Lemma~\ref{lem:temperategrowth}). In addition, $\e^{t H_p}$ preserves the Schwartz class for all times (by Proposition~\ref{prop:continuityflowconf} for instance). We deduce by the chain rule that for any $w \in \sch(T^\star \mfd)$ and $t \in \R$, the map
\begin{equation*}
s \longmapsto \e^{(t - s) \cal{H}_p} \e^{s H_p} w
\end{equation*}
is of class $C^1(\R; L^2(T^\star \mfd))$. We have
\begin{equation*}
\dfrac{\dd}{\dd s} \e^{(t - s) \cal{H}_p} \e^{s H_p} w
	= - \e^{(t - s) \cal{H}_p} \cal{H}_p \e^{s H_p} w + \e^{(t - s) \cal{H}_p} H_p \e^{s H_p} w
	= - \e^{(t - s) \cal{H}_p} \cal{H}_p^{(3)} \e^{s H_p} w .
\end{equation*} 
Integrating over $s \in [0, t]$ yields
\begin{equation*}
\int_0^t \e^{(t - s) \cal{H}_p} \cal{H}_p^{(3)} \e^{s H_p} w \dd s
	= \e^{t \cal{H}_p} w - \e^{t H_p} w .
\end{equation*}
This is exactly the Dyson expansion~\eqref{eq:Dyson} at order $j_0 = 0$. We proceed by induction to prove the higher order Dyson expansion: if~\eqref{eq:Dyson} is valid for $j_0 \ge 0$, then we observe that, given $w \in \sch(T^\star \mfd)$, we have
\begin{equation*}
\tilde w
	:= \cal{H}_p^{(3)} \e^{(s_{j_0+1} - s_{j_0}) H_p} \cal{H}_p^{(3)}  \e^{(s_{j_0} - s_{j_0-1}) H_p}  \cdots \e^{(s_2 - s_1) H_p} \cal{H}_p^{(3)} \e^{s_1 H_p} w
	\in \sch(T^\star \mfd)
\end{equation*}
for any $(s_1, s_2, \ldots, s_{j_0+1}) \in t \Delta_{j_0+1}$. Therefore the Dyson expansion at order $0$ allows to write
\begin{equation} \label{eq:inductionDyson}
\e^{(t - s_{j_0+1}) \cal{H}_p} \tilde w
	= \e^{(t - s_{j_0+1}) H_p} \tilde w
		+ \int_0^{t - s_{j_0+1}} \e^{(t - s_{j_0+1} - s) \cal{H}_p} \tilde w \dd s
\end{equation}
We change variables in the integral by setting $s_{j_0+2} = s + s_{j_0+1}$, which ranges between $s_{j_0+1}$ and $t$. This term becomes
\begin{equation*}
\int_{s_{j_0+1}}^t \e^{(t - s_{j_0+2}) \cal{H}_p} \cal{H}_p^{(3)} \e^{(s_{j_0+1} - s_{j_0}) H_p} \cal{H}_p^{(3)}  \e^{(s_{j_0} - s_{j_0-1}) H_p}  \cdots \e^{(s_2 - s_1) H_p} \cal{H}_p^{(3)} \e^{s_1 H_p} w \dd s_{j_0+2} .
\end{equation*}
Plugging this into~\eqref{eq:inductionDyson} and integrating over $(s_1, s_2, \ldots, s_{j_0+1}) \in t \Delta_{j_0+1}$, we obtain
\begin{equation*}
\widehat{\cal{E}}_{j_0+1}(t)
	= \cal{E}_{j_0+1}(t) + \widehat{\cal{E}}_{j_0+2}(t) ,
\end{equation*}
which concludes the proof. The operators involved map $\sch(T^\star \mfd)$ to $L^2(T^\star \mfd)$ in view of Proposition~\ref{prop:isometrycalHp}.
\end{proof}

By a density argument, one can extend the validity of the Dyson expansion to a wider class of symbols on phase space. The difficulty is that we do not know yet that $\e^{t \cal{H}_p}$ preserves the Schwartz class (it only preserves $\dom \cal{H}_p$ a priori), although its generator $\cal{H}_p$ preserves $\sch(T^\star \mfd)$.

However, we already know that the operator $\e^{t \cal{H}_p}$ is well defined as a map
\begin{equation*}
\e^{t \cal{H}_p} :
	S(1, g) \longrightarrow \sch'(T^\star \mfd) .
\end{equation*}
This follows from the unitarity of the propagator $\e^{- \ii t P}$, the Calder{\'o}n--Vaillancourt theorem (Proposition~\ref{prop:CV}) and Proposition~\ref{prop:SchwartzWeylkernel}.
In the following corollary, we show that it can be extended to classes of the form $S(\theta_g^N, g)$ for any $N \in \R$. This will be used in fact for $N = 1$ to justify Egorov's theorem for affine symbols, in the prospect of applying Beals' theorem. The motivation for considering $S(\theta_g, g)$ comes from the fact that affine symbols, namely those symbols $f$ such that $\nabla^2 f = 0$, belong to $\nabla^{-1}S(\theta_g, g)$ (see Lemma~\ref{lem:seminormaff}).

\begin{corollary} \label{cor:Egorovwithnabla-1symbols}
Assume $p$ and $g$ satisfy Assumptions~\ref{assum:mandatory} and~\ref{assum:p}. Let $f \in \cont^\infty(T^\star \mfd)$ be such that
\begin{equation} \label{eq:assumptionaffine}
\exists N \in \R : \qquad
	f \in \nabla^{-1} S(\theta_g^N, g) .
\end{equation}
Then the operators
\begin{gather*}
\Opw{f} \e^{- \ii t P} :
	\sch(\mfd) \longrightarrow L^2(\mfd) , \\
\Opw{\cal{E}_j(t) f} : \sch(\mfd) \longrightarrow L^2(\mfd) ,
	\quad
\Opw{\widehat{\cal{E}}_j(t) f} : \sch(\mfd) \longrightarrow L^2(\mfd) ,
	\qquad j \in \N
\end{gather*}
are continuous, and we have
\begin{equation*}
\Opw{f} \e^{- \ii t P}
	= \sum_{j = 0}^{j_0} \e^{- \ii t P} \Opw{\cal{E}_j(t) f} + \e^{- \ii t P} \Opw{\widehat{\cal{E}}_{j_0+1}(t) f} ,
		\qquad \forall \abs{t} \le \tfrac{1}{2} T_E , \forall j_0 \in \N ,
\end{equation*}
or equivalently, $\e^{t \cal{H}_p} f$ is well-defined and for any $j_0 \in \N$ we have
\begin{equation} \label{eq:egorovontermediatefaff}
\e^{t \cal{H}_p} f
	= \sum_{j = 0}^{j_0} \cal{E}_j(t) f + \widehat{\cal{E}}_{j_0+1}(t) f ,
		\qquad \forall \abs{t} \le \tfrac{1}{2} T_E .
\end{equation}
\end{corollary}

The proof uses the following technical lemma.

\begin{lemma} \label{lem:continuityquantumvf}
Let $(a_n)_{n \in \N}$ be a sequence in $\sch'(T^\star \mfd)$ converging to some $a \in \sch'(T^\star \mfd)$ as $n \to \infty$. Assume that the operators $\Opw{a_n}$ extend to bounded operators on $L^2(\mfd)$, and that
\begin{equation} \label{eq:assumeuniformboundedness}
\sup_{n \in \N} \; \norm*{\Opw{a_n}}_{\Bop(L^2(\mfd))}
	< \infty .
\end{equation}
Then we have
\begin{equation*}
\Opw{\e^{t \cal{H}_p} a_n}
	\strongto{n \to \infty} \Opw{\e^{t \cal{H}_p} a} ,
		\qquad \forall t \in \R ,
\end{equation*}
for the weak operator topology on $L^2(\mfd)$. In particular, convergence holds as continuous operators $\sch(\mfd) \to \sch'(\mfd)$.
\end{lemma}

\begin{proof}[Proof of Lemma~\ref{lem:continuityquantumvf}]
Since $a_n \weakto{} a$ in $\sch'(T^\star \mfd)$, we know by definition of the Weyl quantization that
\begin{equation*}
\Opw{a_n} u
	\weakto[\sch'(\mfd)]{n \to \infty} \Opw{a} u ,
		\qquad \forall u \in \sch(\mfd) .
\end{equation*}
Using the uniform boundedness assumption~\eqref{eq:assumeuniformboundedness}, weak compactness in $L^2(\mfd)$ results in
\begin{equation} \label{eq:convergenceforsch}
\Opw{a_n} u
	\weakto[L^2(\mfd)]{n \to \infty} \Opw{a} u ,
		\qquad \forall u \in \sch(\mfd) .
\end{equation}
We want to upgrade this to functions $u \in L^2(\mfd)$. Fix $u \in L^2(\mfd)$ and let $u' \in \sch(\mfd)$. We write
\begin{equation*}
\Opw{a_n - a} u
	= \Opw{a_n - a} (u - u') + \Opw{a_n - a} u' ,
\end{equation*}
so that for all $v \in L^2(\mfd)$, we have
\begin{equation*}
\abs*{\inp*{\Opw{a_n - a} u}{v}_{L^2}}
	\le \norm*{\Opw{a_n - a}}_{\Bop(L^2(\mfd))} \norm*{u - u'}_{L^2} \norm*{v}_{L^2} + \abs*{\inp*{\Opw{a_n - a} u'}{v}_{L^2}} .
\end{equation*}
By~\eqref{eq:convergenceforsch} and~\eqref{eq:assumeuniformboundedness}, we obtain
\begin{equation*}
\limsup_{n \to \infty} \abs*{\inp*{\Opw{a_n - a} u}{v}_{L^2}}
	\le \left( \sup_{n \in \N} \norm*{\Opw{a_n}}_{\Bop(L^2(\mfd))} + \norm*{\Opw{a}}_{\Bop(L^2(\mfd))} \right) \norm*{u - u'}_{L^2} \norm*{v}_{L^2} .
\end{equation*}
This is valid for all $v \in L^2(\mfd)$ and all $u' \in \sch(\mfd)$, so by density of $\sch(\mfd)$ in $L^2(\mfd)$, we infer that
\begin{equation*}
\Opw{a_n} u
	\weakto[L^2(\mfd)]{n \to \infty} \Opw{a} u ,
		\qquad \forall u \in L^2(\mfd) ,
\end{equation*}
or in other words, $\Opw{a_n} \to \Opw{a}$ weakly as operators on $L^2(\mfd)$. This property is preserved by conjugation with the propagator $\e^{- \ii t P}$, hence the result. Testing against Schwartz functions instead of $L^2$ functions, we obtain the convergence in the sense of continuous operators $\sch(\mfd) \to \sch'(\mfd)$.
\end{proof}

We now go back to the proof of the corollary.

\begin{proof}[Proof of Corollary~\ref{cor:Egorovwithnabla-1symbols}]
Observe first that throughout the proof, we do not care about the dependence of seminorms with respect to time. Only mapping properties for fixed times $t$ matter.

We fix a {$g$-partition} of unity $(\varphi_{\rho_0})_{\rho_0 \in T^\star \mfd}$ (Proposition~\ref{prop:existencepartitionofunity}). The {$\varphi_{\rho_0}$'s} are compactly supported. Let $(B_n)_{n \in \N}$ be an exhaustive sequence of compact sets that cover the phase space: $\bigcup_n B_n = T ^\star \mfd$. We define
\begin{equation*}
\chi_n(\rho)
	:= \int_{B_n} \varphi_{\rho_0}(\rho) \dd \vol_g(\rho_0) ,
		\qquad n \in \N .
\end{equation*}
This symbol is compactly supported. Applying Proposition~\ref{prop:reconstructsymbol} to the uniformly confined family of symbols $(\one_{B_n}(\rho_0) \varphi_{\rho_0})_{\rho_0 \in T^\star \mfd}$, we know that $\chi_n \in S(1, g)$. In addition, the sequence $(\chi_n)_{n \in \N}$ is bounded in $S(1, g)$ since the seminorms of $\one_{B_n}(\rho_0) \varphi_{\rho_0}$ are bounded uniformly in $n$.

Now by dominated convergence, one checks that $\chi_n \weakto{} 1$ in $\sch'(T^\star \mfd)$ as $n \to \infty$, so that $f_n := \chi_n f \weakto{} f$ in $\sch'(T^\star \mfd)$ as well.

Then Proposition~\ref{prop:Dyson} applies to all functions $f_n$ and yields for any $j_0 \in \N$:
\begin{equation} \label{eq:Dysoncomposedleft}
\Opw{f_n} \e^{- \ii t P}
	= \sum_{j = 0}^{j_0} \e^{- \ii t P} \Opw{\cal{E}_j(t) f_n} + \e^{- \ii t P} \Opw{\widehat{\cal{E}}_{j_0+1}(t) f_n} ,
		\qquad \forall t \in \R .
\end{equation}
We wish to pass to the limit as $n \to \infty$ in both sides of~\eqref{eq:Dysoncomposedleft}.
\begin{itemize}[label=\textbullet]
\item To pass to the limit in the terms $\cal{E}_j(t)$, we simply argue that the operators $\cal{H}_p^{(3)}$ and $\e^{s H_p}$ are continuous on $\sch(T^\star \mfd)$, so that they extend to continuous operators on $\sch'(T^\star \mfd)$. This implies that
\begin{equation*}
\cal{E}_j(t) f_n
	\weakto[\sch'(T^\star \mfd)]{n \to \infty} \cal{E}_j(t) f .
\end{equation*}
We deduce that
\begin{equation*}
\Opw{\cal{E}_j(t) f_n}
	\weakto[\Lop(\sch, \sch'(\mfd))]{n \to \infty} \Opw{\cal{E}_j(t) f} .
\end{equation*}
Let us check that the operator $\Opw{\cal{E}_j(t) f}$ is actually mapping $\sch(\mfd)$ to $L^2(\mfd)$ continuously.
For $j=0$, in view of the assumption~\eqref{eq:assumptionaffine} on $\nabla f$, Corollary~\ref{cor:continuityofflowinsymbolclasses} implies that $\nabla \cal{E}_0(t) f \in S(\e^{t H_p} \theta_g, g(t)) \subset S(C_\Upsilon \e^{\Upsilon \abs{t}} \theta_g, g(t))$ (recall that $\e^{t H_p} \theta_g \le C_\Upsilon \e^{\Upsilon \abs{t}} \theta_g$ from Assumption~\ref{assum:p}~\ref{it:metriccontrol}). Using the admissibility of $g(t)$ (Remark~\ref{rmk:uniformadmissibilityg(t)}) and $\theta_g$ (Proposition~\ref{prop:temperanceweight}), we deduce from the improved temperance property of Proposition~\ref{prop:improvedadmissibility} that $\cal{E}_0(t) f$ belongs for instance to $S(f(\rho_0) + \theta_g(\rho_0) \jap{\rho - \rho_0}_{g_{\rho_0}^\natural}^k, g_{\rho_0})$ for some fixed $\rho_0 \in T^\star \mfd$ and some large $k \ge 0$. (Dependence of seminorms on $t$ does not matter here.) In other words, the symbol $\cal{E}_0(t) f$ has temperate growth with respect to the Euclidean metric~$g_{\rho_0}$. Therefore we have that
\begin{equation*}
\Opw{\cal{E}_0(t) f}
	: \sch(\mfd) \xlongrightarrow{} \sch(\mfd)
\end{equation*}
continuously by Proposition~\ref{prop:continuityonSchwartz}. Similarly for $j \in \N^\ast$, Proposition~\ref{prop:continuityEcal} implies that $\cal{E}_j(t) f$ has temperate growth so that Proposition~\ref{prop:continuityonSchwartz} again yields
\begin{equation*}
\Opw{\cal{E}_j(t) f}
	: \sch(\mfd) \xlongrightarrow{} \sch(\mfd) .
\end{equation*}
In particular, the operators $\Opw{\cal{E}_j(t) f}$ for all $j \in \N$ map $\sch(\mfd)$ to $L^2(\mfd)$ continuously and can be composed on the left by $\e^{- \ii t P}$.
\item To pass to the limit in $\widehat{\cal{E}}_{j_0+1}(t) f_n$, we use the fact that this term involves only derivatives of $f_n$ and not $f_n$ itself. We assume first that $j_0$ is large enough so that $\epsilon j_0 \ge N$, but this is not restrictive for our purpose.
In view of the recurrence relation~\eqref{eq:recurrencerelation}, we first study the action of $\cal{E}_{j_0}(t)$ and $\cal{H}_p^{(3)}$.
From Proposition~\ref{prop:continuityEcal} (applied with $\kappa = \Upsilon$), combined with Assumption~\ref{assum:p}~\ref{it:metriccontrol}, and Proposition~\ref{lem:continuityHp3}, we have
\begin{multline*}
\nabla^{-1} S(\theta_g^N, g)
	\xlongrightarrow{\cal{E}_{j_0}(t-s)} S\left(\e^{N \Upsilon \abs{t-s}} \theta_g^N \theta_{g(t-s)}^{- j_0 \epsilon}, g(t-s)\right) \\
	\xlongrightarrow{\cal{H}_p^{(3)}} S\left(\e^{N \Upsilon \abs{t-s}} \theta_g^N \theta_{g(t-s)}^{- j_0 \epsilon} \gain_{g(t)}^2 \udl{\gain}_g \e^{-(\Lambda + 2 \Upsilon) \abs{t - s}}, g(t-s)\right) .
\end{multline*}
Using that $\theta_{g(t)} = \theta_g \e^{-(\Lambda + 2 \Upsilon) \abs{t}}$, we can crudely bound the weight in the right-hand side symbol class by
\begin{equation} \label{eq:dontcareexpfactor}
\e^{N \Upsilon \abs{t-s}} \theta_g^N \theta_{g(t-s)}^{- j_0 \epsilon} \gain_{g(t)}^2 \udl{\gain}_g \e^{-(\Lambda + 2 \Upsilon) \abs{t - s}}
	\le \e^{N (\Lambda + 3 \Upsilon) \abs{t}} \theta_{g(t)}^{N - j_0 \epsilon}
	= C(t) \theta_{g(t)}^{N - j_0 \epsilon} .
\end{equation}
Recalling that $\epsilon j_0 \ge N$, we deduce that the symbol $\cal{H}_p^{(3)} \cal{E}_{j_0}(t-s) f_n$ belongs to $S(1, g(t))$ with uniform seminorms with respect to $n$ and $s, t$ in a bounded interval, since $f_n = \chi_n f$ belongs to a bounded subset of $\nabla^{-1} S(\theta_g^N, g)$.

Therefore, the Calder\'{o}n--Vaillancourt theorem (Proposition~\ref{prop:CV}) ensures that the operator $\Opw{\cal{H}_p^{(3)} \cal{E}_{j_0}(t-s) f_n}$ is bounded on $L^2(\mfd)$ uniformly with respect to $n \in \N$ and $t, s$ in a compact set. In addition, since the operators $\cal{H}_p^{(3)}$ and $\cal{E}_{j_0}(t-s)$ are continuous on $\sch'(T^\star \mfd)$, we also have
\begin{equation*}
\cal{H}_p^{(3)} \cal{E}_{j_0}(t-s) f_n
	\weakto[\sch'(T^\star \mfd)]{n \to \infty} \cal{H}_p^{(3)} \cal{E}_{j_0}(t-s) f .
\end{equation*}
Therefore Lemma~\ref{lem:continuityquantumvf} below applies and we obtain
\begin{equation} \label{eq:pyramidetointegrate}
\Opw{\e^{(t - s_{j_0+1}) \cal{H}_p} \cal{H}_p^{(3)} \cal{E}_{j_0}(t-s) f_n}
	\strongto{n \to \infty} \Opw{\e^{(t - s_{j_0+1}) \cal{H}_p} \cal{H}_p^{(3)} \cal{E}_{j_0}(t-s) f}
\end{equation}
for the weak operator topology on $L^2(\mfd)$, which finally results in
\begin{equation*}
\Opw{\widehat{\cal{E}}_{j_0+1}(t) f_n}
	\strongto{n \to \infty} \Opw{\widehat{\cal{E}}_{j_0+1}(t) f} ,
\end{equation*}
after integrating~\eqref{eq:pyramidetointegrate} over $s \in t \Delta_1$, again for the weak operator topology on $L^2(\mfd)$. Notice that in particular the right-hand side is a bounded operator.
\item Lastly, we pass to the limit in the left-hand side of~\eqref{eq:Dysoncomposedleft}, namely
\begin{equation} \label{eq:toprovelefthandsidee}
\forall u \in L^2(\mfd), \qquad
	\Opw{f_n} \e^{- \ii t P} u
		\weakto[\sch'(\mfd)]{n \to \infty} \Opw{f} \e^{- \ii t P} u .
\end{equation}
(Notice that contrary to above, Lemma~\ref{lem:continuityquantumvf} does not apply since the operators $\Opw{f_n}$ are not necessarily bounded.) We claim that it suffices to prove that
\begin{equation} \label{eq:toproveuniformbddness}
\forall v \in \sch(\mfd) , \qquad
	\sup_{n \in \N} \norm*{\Opw{\bar f_n} v}_{L^2} < \infty .
\end{equation}
Indeed, if this is true, since $f_n \weakto{} f$ in $\sch'(T^\star \mfd)$, we deduce that $\Opw{\bar f_n} v \weakto{} \Opw{\bar f} v$ weakly in $L^2(\mfd)$,  so that
\begin{equation*}
\inp*{\e^{- \ii t P} u}{\Opw{\bar f_n} v}_{L^2}
	 \strongto{n \to \infty} \inp*{\e^{- \ii t P} u}{\Opw{\bar f} v}_{L^2} ,
	 	\qquad \forall u \in L^2(\mfd) ,
\end{equation*}
which implies~\eqref{eq:toprovelefthandsidee}. So let us prove~\eqref{eq:toproveuniformbddness}. In view of the definition of $f_n = \chi_n f$, we use the pseudo-differential calculus (Proposition~\ref{prop:pseudocalcsymb}) with $\chi_n \in S(1, g)$ and $f \in \nabla^{-1} S(\theta_g^N, g)$ to have
\begin{equation*}
f_n
	= \chi_n f
	= \chi_n \moyal f - \widehat{\cal{P}}_1(\chi_n, f) ,
\end{equation*}
where $\widehat{\cal{P}}_1(\chi_n, f)$ lies in a bounded subset of $S(\theta_g^N, g)$, given that $(\chi_n)_{n \in \N}$ is a bounded sequence in $S(1, g)$. Now by~\cite[Corollary 2.6.16]{Lerner:10}, there exist $a \in S(\theta_g^{-N}, g)$ and $b \in S(\theta_g^N, g)$ such that $a \moyal b = 1$. Thus we can write
\begin{equation} \label{eq:toquant}
f_n
	= \chi_n \moyal f - \widehat{\cal{P}}_1(\chi_n, f)  \moyal a   \moyal b ,
\end{equation}
where $\widehat{\cal{P}}_1(\chi_n, f) \moyal a$ lies in a bounded subset of $S(1, g)$ by pseudo-differential calculus. In particular, its quantization gives rise to a bounded operator on $L^2(\mfd)$, with norm bounded independently of $n$. Quantizing~\eqref{eq:toquant} and applying the resulting operator to a function $v \in \sch(\mfd)$, we obtain
\begin{equation*}
\Opw{f_n} v
	= \Opw{\chi_n} \Opw{f} v - \Opw{\widehat{\cal{P}}_1(\chi_n, f)  \moyal a} \Opw{b} v .
\end{equation*}
As we saw, $\Opw{\chi_n}$ and $\Opw{\widehat{\cal{P}}_1(\chi_n, f)  \moyal a}$ are bounded uniformly with respect to $n$, so that
\begin{equation*}
\norm*{\Opw{f_n} v}_{L^2}
	\le c \left( \norm*{\Opw{f} v}_{L^2} + \norm*{\Opw{b} v}_{L^2} \right) .
\end{equation*}
for some constant $c > 0$ independent of $n$. Notice that $\Opw{f} v$ and $\Opw{b} v$ lie in $\sch(\mfd)$, so in particular in $L^2(\mfd)$, by Proposition~\ref{prop:continuityonSchwartz}.
\end{itemize}
Thus one can pass to the limit $n \to \infty$ in~\eqref{eq:Dysoncomposedleft} to obtain the same equality with $f$ instead of $f_n$. The right-hand side of~\eqref{eq:Dyson} makes sense as an operator $\sch(\mfd) \to L^2(\mfd)$, so that the operator $\Opw{f} \e^{- \ii t P}$ maps $\sch(\mfd) \to L^2(\mfd)$ continuously. Therefore it can be composed on the left by the Schrödinger propagator, so that the distribution $\e^{t \cal{H}_p} f$ is well-defined (recall Definition~\ref{def:defsymbconjugatedop}), and we obtain~\eqref{eq:egorovontermediatefaff}.
\end{proof}

\subsection{Technical lemmata for affine symbols}

In preparation for Section~\ref{subsec:twocommutatorlemmata}, we collect some results concerning affine symbols (whose study is motivated by Beals' theorem) and symbol classes of the form $S(\theta_g^N, g)$.

Given a smooth function $f$, we recall that $H_f$ is the Hamiltonian vector field associated with $f$, defined in~\eqref{eq:defHp}, and that its ``quantum counterpart" $\cal{H}_f$ is defined in~\eqref{eq:defcalHp}. Notice that $H_f$ satisfies the product rule (or ``Leibniz formula") with respect to the pointwise product, while $\cal{H}_f$ satisfies the Leibniz formula with respect to the Moyal product, that is to say:
\begin{equation} \label{eq:productrules}
H_f (a_1 a_2)
	= (H_f a_1) a_2 + a_1 (H_f a_2)
		\qquad \rm{and} \qquad
\cal{H}_f (a_1 \moyal a_2)
	= (\cal{H}_f a_1) \moyal a_2 + a_1 \moyal (\cal{H}_f a_2) .
\end{equation}
In the sequel, we denote by $\aff(T^\star \mfd; \R)$ the set of real-valued affine functions, namely functions $f$ such that $\nabla^2 f = 0$. In the first lemma below, we relate affine symbols to the class $\nabla^{-1} S(\theta_g, g)$, for any admissible metric $g$ on phase space. This justifies the introduction of the temperance weight in Definition~\ref{def:thetag}.

\begin{lemma} \label{lem:seminormaff}
Let $g$ be an admissible Riemannian metric on $T^\star \mfd$. Let $f \in \aff(T^\star \mfd; \R)$. Then we have $\nabla f \in S(\theta_g, g)$ and
\begin{equation*}
\forall \ell \in \N, \forall \rho_0 \in T^\star \mfd, \forall t \in \R , \qquad
	\abs*{\nabla f}_{S(\theta_g, g)}^{(\ell)}
		= \abs*{\nabla f}_{S(\theta_g, g)}^{(0)}
		\le \abs*{H_f}_{g_{\rho_0}(t)} \theta_{g(t)}(\rho_0) .
\end{equation*}
\end{lemma}

\begin{proof}
This is a direct consequence of Lemma~\ref{lem:seminormsHp} and of the definition of the temperance weight in Definition~\ref{def:thetag}:
\begin{equation*}
\forall \rho \in T^\star \mfd , \qquad
	\abs*{\nabla f}_{g_\rho}
		= \abs*{H_f}_{g_\rho^\sympf}
		\le \abs*{H_f}_{{\sf g}} \theta_g(\rho)
		= \abs*{H_f}_{{\sf g}^\sympf} \theta_g(\rho)
		\le \abs*{H_f}_{g_{\rho_0}(t)} \theta_g(\rho) \theta_{g(t)}(\rho_0) .
\end{equation*}
Then we recall that $\nabla^2 f = 0$ because $f$ is affine and the result follows.
\end{proof}

For the following lemma, recall the notation $\cal{E}_{\le j_0}(s)$ from~\eqref{eq:notationlej}. Notice that the result below is valid on the smaller time interval $\abs{t} \le \frac{1}{2} T_E$.

\begin{lemma} \label{lem:continuityHcal_f}
Let $g$ and $p$ satisfy Assumptions~\ref{assum:mandatory} and~\ref{assum:p}. Fix $N_0 \le 1$ and $N \in \R$. Then for all $f_0 \in \nabla^{-1} S(\theta_g^{N_0}, g)$, writing $f_0(s) = \cal{E}_{\le j_0}(s) f_0$, we have
\begin{equation*}
\nabla^{-1} S\left(\theta_{g(t)}^N, g(t)\right)
	\xlongrightarrow{\cal{H}_{f_0(s)}} S\left(\theta_{g(t)}^{N + N_0}, g(t)\right) ,
\end{equation*}
uniformly for any $s, t \in \R$ such that $\abs{s} \le \abs{t} \le \frac{1}{2} T_E$, that is to say
\begin{multline*}
\forall \ell \in \N, \exists k \in \N, \exists C > 0 : \forall a \in \nabla^{-1} S\left(\theta_{g(t)}^N, g(t)\right), \forall f_0 \in \nabla^{-1} S(\theta_g^{N_0}, g) , \qquad \\
	\abs*{\cal{H}_{f_0(s)} a}_{S(\theta_{g(t)}^{N + N_0}, g(t))}^{(\ell)}
		\le C \abs*{\nabla f_0}_{S(\theta_g^{N_0}, g)}^{(k)} \times \abs*{\nabla a}_{S(\theta_{g(t)}^N, g(t))}^{(k)} ,
\end{multline*}
uniformly in $\abs{s} \le \abs{t} \le \frac{1}{2} T_E$.
\end{lemma}

\begin{proof}
First of all, from Corollary~\ref{cor:continuityofflowinsymbolclasses} and Proposition~\ref{prop:continuityEcal}, together with Assumption~\ref{assum:p}~\ref{it:metriccontrol}, we know that
\begin{equation*}
f_0(s)
	\in \nabla^{-1} S\left(\e^{N_0 \Upsilon \abs{s}} \theta_g^{N_0}, g(s)\right) ,
		\qquad \forall \abs{s} \le \tfrac{1}{2} T_E ,
\end{equation*}
with
\begin{equation*}
\forall \ell \in \N, \exists k \in \N, \exists C > 0 : \qquad
	\abs*{\nabla f_0(s)}_{S(\e^{N_0 \Upsilon \abs{s}} \theta_g^{N_0}, g(s))}^{(k)}
		\le C \abs*{\nabla f_0}_{S(\theta_g^{N_0}, g)}^{(k)} .
\end{equation*}
We apply the pseudo-differential calculus (Proposition~\ref{prop:pseudocalcsymb}) with $f_0(s)$ in the above class and $a \in \nabla^{-1} S(\theta_{g(t)}^N, g(t))$, using that $\abs{s} \le \abs{t} \le \tfrac{1}{2} T_E$:
\begin{equation} \label{eq:pseudocalcwithoutestimate}
\nabla^{-1} S\left(\e^{N_0 \Upsilon \abs{s}} \theta_g^{N_0}, g(s)\right) \times \nabla^{-1} S\left(\theta_{g(t)}^N, g(t)\right)
	\xlongrightarrow{\widehat{\cal{P}}_1} S\left(\e^{N_0 \Upsilon \abs{s}} \gain_{g(t)} \theta_g^{N_0} \theta_{g(t)}^N, g(t)\right) .
\end{equation}
We have
\begin{equation*}
\e^{N_0 \Upsilon \abs{s}} \gain_{g(t)} \theta_g^{N_0} \theta_{g(t)}^N
	\le \e^{N_0 \Upsilon \abs{t} + N_0 (\Lambda + 2 \Upsilon) + 2 (\Lambda + 2 \Upsilon) \abs{t}} \udl{\gain}_g \theta_{g(t)}^{N+N_0}
	\le \e^{4 (\Lambda + 2 \Upsilon) \abs{t}} \udl{\gain}_g \theta_{g(t)}^{N+N_0}
	\le \theta_{g(t)}^{N+N_0} ,
\end{equation*}
where we used that $\theta_g = \e^{(\Lambda + 2 \Upsilon) \abs{t}} \theta_{g(t)}$ and $\abs{s} \le \abs{t}$ in the first inequality, $N_0 \le 1$ in the second one and $\abs{t} \le \frac{1}{2} T_E$ in the last one.
Thus, the estimate corresponding to~\eqref{eq:pseudocalcwithoutestimate} is
\begin{align*}
\forall \ell \in \N, \exists k \in \N, \exists C > 0 : \qquad
	\abs*{\cal{H}_{f_0(s)} a}_{S(\theta_{g(t)}^{N + N_0}, g(t))}^{(\ell)}
		&\le C \abs*{\nabla f_0(s)}_{S(\e^{N_0 \Upsilon \abs{s}} \theta_g^{N_0}, g(s))}^{(k)} \abs*{\nabla a}_{S(\theta_{g(t)}^N, g(t))}^{(k)} \\
		&\le C' \abs*{\nabla f_0}_{S(\theta_g^{N_0}, g)}^{(k')} \abs*{\nabla a}_{S(\theta_{g(t)}^N, g(t))}^{(k)} ,
\end{align*}
which is the sought result.
\end{proof}

\subsection{Two commutator lemmata} \label{subsec:twocommutatorlemmata}

In this section, we describe the action of commutators $\ad_F A = \comm{F}{A}$ (operator derivatives) and anti-commutators $\ad_F^+ A = F A + A F$ (operator multiplication) on the quantum evolution of an observable.
Recall the definition of the {$k$-dimensional} simplex $t \Delta_k$ in~\eqref{eq:defsimplex}.
We introduce the following notation: given a finite tuple of symbols $(a_1, a_2, \ldots, a_{j_0})$, we denote by $\Moy_{j=1}^{j_0} a_j$ the symbol
\begin{equation*}
\Moy_{j=1}^{j_0} a_j
	:= a_1 \moyal a_2 \moyal \cdots \moyal a_{j_0} .
\end{equation*}
The order of the factors matters, since the Moyal product is not commutative in general. In case we do not care about the ordering of factors, given symbols $\{a_j\}_{j \in J}$, for $J$ a finite set with $j_0 := \abs{J}$, we shall write in a loose way
\begin{equation*}
\Moy_{j \in J}
	:= a_{k_1} \moyal a_{k_2} \moyal \cdots \moyal a_{k_{j_0}} ,
\end{equation*}
where $(k_1, k_2, \ldots, k_{j_0})$ is an enumeration of the set $J$. In case the order matter, we shall be more specific and spell out the precise enumeration that we use. We work with the convention
\begin{equation*}
\Moy_\varnothing a_j
	:= 1
\end{equation*}
(the empty product is equal to $1$). Similarly, we set
\begin{equation*}
\cal{H}_{(a_1, a_2, \ldots, a_{j_0})}
	:= \cal{H}_{a_1}\cal{H}_{a_2} \cdots \cal{H}_{a_{j_0}} ,
\end{equation*}
as well as
\begin{equation*}
\cal{H}_{\{a_j\}_{j \in J}}
	:= \cal{H}_{a_{k_1}}\cal{H}_{a_{k_2}} \cdots \cal{H}_{a_{k_{j_0}}} ,
\end{equation*}
for some enumeration of $J$, provided the order of factors does not matter.
Lastly, we work with the convention
\begin{equation*}
\int_{t \Delta_0} a(s) \dd s
	:= a(t) .
\end{equation*}

\begin{lemma} \label{lem:iteratedcommutators}
Suppose $g$ and $p$ satisfy Assumptions~\ref{assum:mandatory} and~\ref{assum:p}, and recall the number $\epsilon$ from Item~\ref{it:strongsubquad} of Assumption~\ref{assum:p}. Let $\ell \in \N$ and $N \le 1$. Then for any $a \in S(\theta_g^{-\ell N}, g)$, for any integer $j_0 \ge \ell N/ \epsilon$, and for any $\abs{t} \le \frac{1}{2} T_E$, the following holds: for any family of symbols $f_1, f_2, \ldots, f_\ell \in \nabla^{-1} S(\theta_g^N, g)$, the distribution
\begin{equation*}
\cal{H}_{f_\ell} \cdots \cal{H}_{f_2} \cal{H}_{f_1} \e^{t \cal{H}_p} a
\end{equation*}
can be written as a sum of symbols of the form
\begin{equation} \label{eq:commutatorsformterms}
\int_{t \Delta_k} \left(\Moy_{k_1 \in K_1} \e^{s_{k_1} \cal{H}_p} c_{\pi(k_1)}(s_{k_1})\right) \moyal \e^{t \cal{H}_p} a_{\tilde \pi}(t) \moyal \left(\Moy_{k_2 \in K_2} \e^{s_{k_2} \cal{H}_p} c_{\pi(k_2)}(s_{k_2})\right) \dd \bf{s} ,
\end{equation}
for $k \in \{0, \ldots, \ell\}$. The number of terms in the sum depends only on $\ell$, $N$ and $j_0$. The indices split into disjoint families of indices as follows:
\begin{equation} \label{eq:partitionpi}
\{1, 2, \ldots, k\}
	= K_1 \cup K_2
		\qquad \rm{and} \qquad
\{1, 2, \ldots, \ell\}
	= \tilde \pi \cup \bigcup_{k_1 \in K_1} \pi(k_1) \cup \bigcup_{k_2 \in K_2} \pi(k_2) .
\end{equation}
The sets $\pi(k_j)$, $j = 1, 2$, and $\tilde \pi$ are non-empty.
Writing for short
\begin{equation} \label{eq:notationf(s)andhatf(s)}
f_j(s)
	:= \cal{E}_{\le j_0}(- s) f_j
		\qquad \rm{and} \qquad
\hat{f}_j(s)
	:= \cal{H}_p^{(3)} \cal{E}_{j_0}(-s) f_j ,
\end{equation}
the symbols $c_\pi$ and $a_{\tilde \pi}$ in~\eqref{eq:commutatorsformterms} are of the form
\begin{equation} \label{eq:defcpiapi}
\begin{split}
c_\pi(s)
	&= \cal{H}_{\{f_j(s)\}_{j \in \pi \setminus \{j_\ast\}}} \hat{f}_{j_\ast}(s)
			\quad \textrm{for some $j_\ast \in \pi$} , \\
a_{\tilde \pi}(s)
	&= \cal{H}_{\{f_j(s)\}_{j \in \tilde \pi}} a ,
		\qquad s \in \R .
\end{split}
\end{equation}
\end{lemma}

\begin{remark}
All the symbols involved in~\eqref{eq:defcpiapi} and hence in~\eqref{eq:commutatorsformterms} belong to $S(1, g)$. This is a consequence of Lemma~\ref{lem:continuityHcal_f} for symbols of the form $a_{\tilde \pi}(s)$ and of Proposition~\ref{prop:continuityEcal} together with Lemma~\ref{lem:continuityHcal_f} for those of the form $c_\pi(s)$. We need this fact to make sure that the Moyal products in~\eqref{eq:commutatorsformterms} make sense. This is justified in the corollary below.
\end{remark}

In the following corollary, we give precise bounds for the operator norm in the special case where the {$f_j$'s} are affine functions.

\begin{corollary} \label{cor:iteratedcommutatorsaff}
Suppose $g$ and $p$ satisfy Assumptions~\ref{assum:mandatory} and~\ref{assum:p}, and recall the number $\epsilon$ from Item~\ref{it:strongsubquad} of Assumption~\ref{assum:p}. Let $L \in \N$. There exist $k \in \N$ and $C > 0$ such that for any $t, s \in \R$ with $\abs{s} \le \abs{t} \le \frac{1}{2} T_E$, the following holds: for any $f_1, f_2, \ldots, f_\ell \in \aff(T^\star \mfd; \R)$, we have for all $a \in S(\theta_{g(s)}^{-\ell}, g(s))$:
\begin{equation*}
\forall \rho_0 \in T^\star \mfd , \qquad
	\norm*{\Opw{\cal{H}_{f_\ell} \cdots \cal{H}_{f_2} \cal{H}_{f_1} \e^{t \cal{H}_p} a}}_{\Bop(L^2(\mfd))}
		\le C \abs*{a}_{S(\theta_{g(s)}^{-\ell}, g(s))}^{(k)} \theta_{g(t)}^\ell(\rho_0) \prod_{j = 1}^\ell \abs*{H_{f_j}}_{g_{\rho_0}(t)} .
\end{equation*}
\end{corollary}

\begin{proof}[Proof of Corollary~\ref{cor:iteratedcommutatorsaff}]
We first observe that it suffices to show the statement for $s = t$, since
\begin{equation*}
\abs*{a}_{S(\theta_{g(t)}^{-\ell}, g(t))}^{(k)}
	\le \abs*{a}_{S(\theta_{g(s)}^{-\ell}, g(s))}^{(k)} ,
\end{equation*}
in view of the fact that the maps $\tau \mapsto g(\tau)$ and $\tau \mapsto \theta_{g(\tau)}^{-1}$ are non-decreasing.
We fix $j_0 \ge \frac{\ell + 1}{\epsilon}$. By Lemma~\ref{lem:seminormaff}, we know that $f_j \in \nabla^{-1} S(\theta_g, g)$. Moreover, the symbol class $S(\theta_{g(s)}^{-\ell}, g(s))$ is contained in $S(\theta_g^{-\ell}, g)$ (although the embedding is not uniform in time), so that Lemma~\ref{lem:iteratedcommutators} applies, with $N = 1$. In view of the Calder\'{o}n--Vaillancourt theorem (Proposition~\ref{prop:CV}), and given that
\begin{equation*}
\norm*{\Opw{\e^{\tau \cal{H}_p} b}}_{\Bop(L^2(\mfd))}
	= \norm*{\e^{\ii \tau P} \Opw{b} \e^{- \ii \tau P}}_{\Bop(L^2(\mfd))}
	= \norm*{\Opw{b}}_{\Bop(L^2(\mfd))} ,
		\qquad \forall \tau \in \R ,
\end{equation*}
it is sufficient to estimate the $S(1, g(t))$ seminorms of the symbols $c_\pi(\tau)$ and $a_{\tilde \pi}(\tau)$ arising in~\eqref{eq:commutatorsformterms} in order to prove the statement. In this proof, we write $\abs{\pi}$ for the cardinality of $\pi$. On the one hand, applying $\abs{\tilde \pi}$ times Lemma~\ref{lem:continuityHcal_f} to $a_{\tilde \pi}(\tau)$ defined in~\eqref{eq:defcpiapi} with $N_0 = 1$ gives
\begin{equation*}
\forall k_0 \in \N, \exists \tilde k \in \N, \exists C > 0 : \qquad
	\abs*{a_{\tilde \pi}(\tau)}_{S(1, g(t))}^{(k_0)}
		\le C \abs*{a}_{S(\theta_{g(t)}^{-\abs{\tilde \pi}}, g(t))}^{(\tilde k)} \prod_{j \in \tilde \pi} \abs*{\nabla f_j}_{S(\theta_g, g)}^{(\tilde k)} ,
			\qquad \abs{\tau} \le \abs{t} \le \tfrac{1}{2} T_E .
\end{equation*}
Now we remark that the seminorm of $a$ in $S(\theta_{g(t)}^{-\abs{\tilde \pi}}, g(t))$ is smaller than the seminorm in $S(\theta_{g(t)}^{-L}, g(t))$ because $\theta_{g(t)} \ge 1$ (Proposition~\ref{prop:thetag}) and $\abs{\tilde \pi} \le \ell$ (see~\eqref{eq:partitionpi}). We use Lemma~\ref{lem:seminormaff} to handle the seminorms of $\nabla f_j$:
\begin{equation} \label{eq:S(1)api}
\abs*{a_{\tilde \pi}(\tau)}_{S(1, g(t))}^{(k_0)}
		\le C \abs*{a}_{S(\theta_{g(t)}^{-\ell}, g(t))}^{(\tilde k)} \theta_{g(t)}^{\abs{\tilde \pi}}(\rho_0) \prod_{j \in \tilde \pi} \abs*{H_{f_j}}_{g_{\rho_0}(t)} .
\end{equation}
Notice that the constant $C$ is independent of $\rho_0$ and $\abs{\tau} \le \abs{t} \le \frac{1}{2} T_E$.

We proceed similarly for the symbols of the form $c_\pi(\tau)$ in~\eqref{eq:commutatorsformterms}. We apply~\eqref{eq:S(1)api} with $\hat f_{j_\ast}$ in place of~$a$: for all $k_0 \in \N$, there exist $\tilde k \in \N$ and $C > 0$ such that
\begin{equation} \label{eq:beforeapplem}
\forall \rho_0 \in T^\star \mfd , \qquad
	\abs*{c_\pi(\tau)}_{S(1, g(t))}^{(k_0)}
		\le C \abs*{\hat f_{j_\ast}(\tau)}_{S(\theta_{g(t)}^{-L}, g(t))}^{(\tilde k)} \theta_{g(t)}^{\abs{\pi} - 1}(\rho_0) \prod_{j \in \pi \setminus \{j_\ast\}} \abs*{H_{f_j}}_{g_{\rho_0}(t)} .
\end{equation}
Now from Lemma~\ref{lem:continuityHp3} and Proposition~\ref{prop:continuityEcal} (with $m = \theta_g$ and $\kappa = \Upsilon$), we have
\begin{equation} \label{eq:arrowshat}
\begin{multlined}
f_{j_\ast}
	\in \nabla^{-1} S\left(\theta_g, g\right)
	\xlongrightarrow{\cal{E}_{j_0}(\tau)} S\left(\e^{\Upsilon \abs{\tau}} \theta_g \theta_{g(\tau)}^{-j_0 \epsilon}, g(\tau)\right) \\
	\qquad\qquad \xlongrightarrow{\cal{H}_p^{(3)}} S\left(\udl{\gain}_g^{3/2} \e^{\Upsilon \abs{\tau}} \theta_g \theta_{g(\tau)}^{-(j_0 + 1) \epsilon}, g(\tau)\right)
	\subset S\left(\udl{\gain}_g^{3/2} \e^{2 (\Lambda + 2 \Upsilon) \abs{\tau}} \theta_{g(\tau)}^{-(j_0 + 1) \epsilon + 1}, g(\tau)\right) .
\end{multlined}
\end{equation}
The last inclusion follows from the fact that $\theta_g = \e^{(\Lambda + 2 \Upsilon) \abs{\tau}} \theta_{g(\tau)}$. Under our assumption on $j_0$, we have $(j_0 + 1) \epsilon - 1 \ge \ell$ and moreover, in the time interval $\abs{\tau} \le T_E$, we have
\begin{equation*}
\udl{\gain}_g^{3/2} \e^{2 (\Lambda + 2 \Upsilon) \abs{\tau}} 
	\le \udl{\gain}_g^{1/2} .
\end{equation*}
In addition, the maps $\tau \mapsto g(\tau)$ and $\tau \mapsto \theta_{g(\tau)}^{-1}$ are non-decreasing, so that~\eqref{eq:arrowshat} yields
\begin{equation} \label{eq:arrowshatbis}
f_{j_\ast}
	\in \nabla^{-1} S\left(\theta_g, g\right)
	\xlongrightarrow{\cal{H}_p^{(3)} \cal{E}_{j_0}(\tau)} S\left(\udl{\gain}_g^{1/2} \theta_{g(t)}^{-\ell}, g(t)\right) .
\end{equation}
Therefore we obtain the estimate
\begin{equation*}
\abs*{\hat f_{j_\ast}(\tau)}_{S(\theta_{g(t)}^{-\ell}, g(t))}^{(\tilde k)}
	\le C \udl{\gain}_g^{1/2} \abs*{\nabla f_{j_\ast}}_{S(\theta_g, g)}^{(\tilde k')}
	\le C \udl{\gain}_g^{1/2} \theta_{g(t)}(\rho_0) \abs*{H_{f_{j_\ast}}}_{g_{\rho_0}(t)} ,
\end{equation*}
in virtue of Lemma~\ref{lem:seminormaff}.
Plugging this into~\eqref{eq:beforeapplem}, we arrive at
\begin{equation} \label{eq:S(1)cpi}
\abs*{c_\pi(\tau)}_{S(1, g(t))}^{(k_0)}
	\le C \udl{\gain}_g^{1/2} \theta_{g(t)}^{\abs{\pi}}(\rho_0) \prod_{j \in \pi} \abs*{H_{f_j}}_{g_{\rho_0}(t)} .
\end{equation}

The estimates~\eqref{eq:S(1)api} and~\eqref{eq:S(1)cpi} apply to each factor in~\eqref{eq:commutatorsformterms}, with $\tau = t$ or $\tau = s_{k_j}$, which verify in any case $\abs{\tau} \le \abs{t}$. We conclude by the Calder\'{o}n--Vaillancourt theorem (Proposition~\ref{prop:CV}) for the metric~$g(t)$ that the operator norm of the quantization of~\eqref{eq:commutatorsformterms} is bounded by
\begin{equation*}
C \udl{\gain}_g^{k/2} \abs*{t \Delta_k} \abs*{a}_{S(\theta_{g(s)}^{-\ell}, g(s))}^{(\tilde k)} \theta_{g(t)}^\ell(\rho_0) \prod_{j = 1}^\ell \abs*{H_{f_j}}_{g_{\rho_0}(t)}
\end{equation*}
(recall also~\eqref{eq:partitionpi} so that each $f_j$ appears exactly once). Here $k \in \{0, 1, \ldots, \ell\}$ is the dimension of the simplex in the term~\eqref{eq:commutatorsformterms}. We obtain using the fact that $\abs{t} \le T_E$:
\begin{equation*}
\udl{\gain}_g^{k/2} \abs*{t \Delta_k}
	\le \dfrac{1}{k!} \left(\dfrac{\udl{\gain}_g^{1 + 1/2}}{\Lambda} \log\left( \dfrac{1}{\udl{\gain}_g} \right)\right)^k
	\le \dfrac{1}{k!} \left(\dfrac{C_{1/2}}{c}\right)^k ,
\end{equation*}
where $c$ is the constant in~\eqref{eq:assumLambda} and $C_{1/2} = \sup_{\tau \in [0, 1]} \tau^{1/2} \abs{\log \tau}$. This gives the desired estimate.
\end{proof}

\begin{proof}[Proof of Lemma~\ref{lem:iteratedcommutators}]
In this proof, we do not care about the time dependence of seminorms. We only need to check that the different symbols arising in the computations give rise to bounded operators, for any fixed time $t$. We use the notation $\Ad_F G = F G F^{-1}$.

We prove the lemma by induction on $\ell$ (at each step of the induction, $N$ and the symbol $a$ are arbitrary). The case $\ell = 0$ is a matter of convention. Let us check the basis step $\ell = 1$. We write $f = f_1$ and $F = \Opw{f}$. In view of Corollary~\ref{cor:Egorovwithnabla-1symbols} together with~\eqref{eq:recurrencerelation}, and using the notation of~\eqref{eq:notationf(s)andhatf(s)}, we have
\begin{align} \label{eq:onecomm}
\comm*{F}{\Ad_{\e^{\ii t P}} \Opw{a}}
	&= \Ad_{\e^{\ii t P}} \comm*{\Ad_{\e^{- \ii t P}} F}{\Opw{a}}
	= \Ad_{\e^{\ii t P}} \comm*{\Opw{\e^{- t \cal{H}_p} f}}{\Opw{a}} \nonumber\\
	&\hspace*{-1cm}= \Ad_{\e^{\ii t P}} \comm*{\Opw{\cal{E}_{\le j_0}(-t) f + \int_0^{-t} \e^{(-t - s) \cal{H}_p} \cal{H}_p^{(3)} \cal{E}_{j_0}(s) f \dd s}}{\Opw{a}} \nonumber\\
	&\hspace*{-1cm}= \Ad_{\e^{\ii t P}} \comm*{\Opw{\cal{E}_{\le j_0}(-t) f} - \int_0^{t} \Ad_{\e^{\ii (- t + s) P}} \Opw{\cal{H}_p^{(3)} \cal{E}_{j_0}(-s) f} \dd s}{\Opw{a}} \nonumber\\
	&\hspace*{-1cm}= \Ad_{\e^{\ii t P}} \Opw{\cal{H}_{f(t)} a} - \int_0^{t} \comm*{\Ad_{\e^{\ii s P}} \Opw{\hat f(s)}}{\Ad_{\e^{\ii t P}} \Opw{a}} \dd s .
\end{align}
Recall here that $j_0 \ge N/\epsilon$ from the statement. The symbol of the left-hand side is $\cal{H}_f \e^{t \cal{H}_p} a$. On the right-hand side, expanding the commutator, we have three terms: the first one has symbol
\begin{equation} \label{eq:1stform}
\e^{t \cal{H}_p} \cal{H}_{f(t)} a ,
\end{equation}
the second is the form
\begin{equation} \label{eq:2ndform}
\int_0^t \left(\e^{s \cal{H}_p} \hat f(s)\right) \moyal \left(\e^{t \cal{H}_p} a\right) \dd s
\end{equation}
and the third one is the same with the Moyal product in the other way around.
Those three terms are of the expected form~\eqref{eq:commutatorsformterms} ($k = 0$ for the first term and $k = 1$ for the latter two). Notice that the expression~\eqref{eq:onecomm} makes sense in the sense of bounded operator on $L^2(\mfd)$. Indeed, $\cal{H}_{f(s)} a \in S(1, g(s))$ by Lemma~\ref{lem:continuityHcal_f} since $a \in S(\theta_g^{- N}, g) \subset S(\theta_{g(t)}^{- N}, g(t))$ by assumption. In addition, given that $\hat f(s) = \cal{H}_p^{(3)} \cal{E}_{j_0}(- s) f$, one sees from Proposition~\ref{prop:continuityEcal} and Lemma~\ref{lem:continuityHp3}, together with $j_0 \ge N/\epsilon$ and $\theta_{g(t)} \ge 1$ (Proposition~\ref{prop:thetag}), that
\begin{equation}
\hat f(s)
	\in S\left(\e^{c \abs{s}} \theta_{g(s)}^{N - j_0 \epsilon}, g(s)\right)
	\subset S\left(\e^{c \abs{s}}, g(s)\right) ,
\end{equation}
for\footnote{The reasoning is the same as in~\eqref{eq:dontcareexpfactor} in the proof of Corollary~\ref{cor:Egorovwithnabla-1symbols}.} some constant $c > 0$. Therefore the symbol $\hat f(s)$ gives rise to a bounded operator by the Calder\'{o}n--Vaillancourt Theorem (Proposition~\ref{prop:CV}). This concludes the proof of the lemma for $\ell = 1$.

Now assume that the lemma holds for $\ell \in \N^\ast$. We fix $N \in \R$ and a family of $\ell + 1$ symbols $f_j \in \nabla^{-1} S(\theta_g^N, g)$, $j \in \{1, \ldots, \ell, \ell+1\}$. We consider $a \in S(\theta_g^{-(\ell+1) N}, g)$. We apply the induction hypothesis to handle the $\ell$ first derivations $\cal{H}_{f_j}$, with $j_0 \ge N (\ell+1) /\epsilon$, and we deduce that $\cal{H}_{f_\ell} \cdots \cal{H}_{f_2} \cal{H}_{f_1} \e^{t \cal{H}_p} a$ is a sum of terms of the form~\eqref{eq:commutatorsformterms}. Then the last derivation $\cal{H}_{f_{\ell+1}}$ can land on any of the factors in~\eqref{eq:commutatorsformterms} in virtue of the product (or Leibniz) rule~\eqref{eq:productrules}.

Thus we are lead to study factors of the form $\cal{H}_{f_{\ell+1}} \e^{s \cal{H}_p} b(s)$, with $b(s) = a_{\tilde \pi}(s)$ and $s = t$, or $b(s) = c_\pi(s)$ and $s$ is one of the variables $s_1, s_2, \ldots s_k$. Then from the calculations of step $\ell = 1$, the resulting symbol can be written as a sum of terms of the form $\e^{s \cal{H}_p} \cal{H}_{f_{\ell+1}(s)} b(s)$ as in~\eqref{eq:1stform}, or
\begin{equation} \label{eq:termto}
\int_0^s \left(\e^{\tau \cal{H}_p} \hat f_{\ell+1}(\tau)\right) \moyal \left(\e^{s \cal{H}_p} b(s)\right) \dd \tau
\end{equation}
as in~\eqref{eq:2ndform}, or the same kind of term with the Moyal product in the reversed order. It remains to check that these terms can be written indeed as a sum of terms of the form~\eqref{eq:commutatorsformterms}.
\begin{itemize}[label=\textbullet]
\item In the first case where the resulting term is of the form $\e^{s \cal{H}_p} \cal{H}_{f_{\ell+1}(s)} b(s)$, using the notation of~\eqref{eq:defcpiapi}, according to whether $b(s) = a_{\tilde \pi}(s)$ or $b(s) = c_\pi(s)$, we can write either $\cal{H}_{f_{\ell+1}(s)} b(s) = a_{\tilde \pi'}$ with $\tilde \pi' = \tilde \pi \cup \{\ell+1\}$, or $\cal{H}_{f_{\ell+1}(s)} b(s) = c_{\pi'}(s)$ with $\pi' = \pi \cup \{\ell+1\}$. Going back to~\eqref{eq:termto}, we obtain indeed a symbol of the form~\eqref{eq:commutatorsformterms}.
\item For the second type of terms, namely those of the form~\eqref{eq:termto}, we can write it as
\begin{equation} \label{eq:thetypeofterm}
\int_0^s \left(\e^{\tau \cal{H}_p} c_{\{\ell+1\}}(\tau)\right) \moyal \left(\e^{s \cal{H}_p} b(s)\right) \dd \tau ,
\end{equation}
with the notation~\eqref{eq:commutatorsformterms} for $c_{\{\ell+1\}}(\tau)$ again.
Depending on which factor of~\eqref{eq:commutatorsformterms} was hit by $\cal{H}_{f_{\ell+1}}$, we should take $s$ in the above~\eqref{eq:thetypeofterm} to be one of the $\bf{s}$ variables in~\eqref{eq:commutatorsformterms} or the $t$ variable (e.g.\ if $\cal{H}_{f_{\ell+1}(s)}$ landed on some $c_{\pi(k_1)}$, we should take $s=s_{k_1}$).
Plugging this into~\eqref{eq:commutatorsformterms} yields symbols of the form
\begin{equation*}
\int_{t \Delta_k} \int_0^{s} \left(\Moy_{k_1' \in K_1'} \e^{s_{k_1'} \cal{H}_p} c_{\pi(k_1')}(s_{k_1'})\right) \moyal \e^{t \cal{H}_p} a_{\tilde \pi}(t) \moyal \left(\Moy_{k_2 \in K_2} \e^{s_{k_2} \cal{H}_p} c_{\pi(k_2)}(s_{k_2})\right) \dd \tau \dd \bf{s} ,
\end{equation*}
where $K_1' = K_1 \cup \{k+1\}$, with $s_{k+1} = \tau$ is the same integration variable as in~\eqref{eq:thetypeofterm}.
It remains to see that it is indeed a symbol of the form~\eqref{eq:commutatorsformterms}. This follows from the fact that the integration domain can be split into a union of simplices:
\begin{multline*}
\set{(\tau, \bf{s}) \in \R^{k+1}}{\bf{s} \in t \Delta_k, \tau \in s \Delta_1} \\
	= \bigcup_{k' = 0}^k \set{(\tau, \bf{s}) \in \R^{k+1}}{0 \le s_1 \le s_2 \le \cdots \le s_{k'} \le \tau \le s_{k'+1} \le \cdots \le s \le \cdots \le s_k \le t}
\end{multline*}
(recall that $s = t$ or $s$ is one of the $t \Delta_k$ simplex variables $s_1, s_2, \ldots, s_k$).
\end{itemize}
This concludes the induction.
\end{proof}

\begin{lemma} \label{lem:combinatorialmultiplication}
Suppose $g$ and $p$ satisfy Assumptions~\ref{assum:mandatory} and~\ref{assum:p}, and recall the number $\epsilon$ from Item~\ref{it:strongsubquad} of Assumption~\ref{assum:p}. Let $a \in \sch(T^\star \mfd)$ and fix $\ell \in \N$ and $N \le 1$.  Then for any $j_0 \ge (\ell + 1) N/ \epsilon$, the following holds. For any family of symbols $f_1, f_2, \ldots, f_\ell \in \nabla^{-1} S(\theta_g^N, g)$ and any $\abs{t} \le \frac{1}{2} T_E$, the distribution
\begin{equation} \label{eq:multiplyoperator}
\left(\e^{t \cal{H}_p} a\right) \moyal f_\ell \moyal \cdots \moyal f_2 \moyal f_1
\end{equation}
can be written as a sum of terms of the form
\begin{equation} \label{eq:formweexpect}
\e^{t \cal{H}_p} \left( a \moyal \left( \Moy_{k_1 \in K_1} f_{k_1}(t) \right) \right) \moyal \left( \Moy_{\pi \in \Pi} c_\pi(t) \right)  .
\end{equation}
The number of terms in the sum depends only on $\ell$ and $N$. The indices split into two disjoint families
\begin{equation*}
\{1, 2, \ldots, \ell\}
	= K_1 \cup K_2
\end{equation*}
(one of them possibly empty), and $\Pi$ is a partition of $K_2$ ($\pi \in \Pi$ are the blocks of the partition).
Writing for short
\begin{equation*}
f_j(s)
	:= \cal{E}_{\le j_0}(- s) f_j
		\qquad \rm{and} \qquad
\hat{f}_j(s)
	:= \cal{H}_p^{(3)} \cal{E}_{j_0}(-s) f_j ,
\end{equation*}
the symbols $c_\pi$ are of the form
\begin{equation*}
c_\pi(s)
	= \cal{H}_{\{f_j\}_{j \in \pi \setminus \{j_\ast\}}} \int_0^s \e^{\tau \cal{H}_p} \hat f_{j_\ast}(\tau) \dd \tau ,
		\quad s \in \R , \qquad \textrm{for some $j_\ast \in \pi$.}
\end{equation*}
\end{lemma}

\begin{remark}
We first notice that the operator whose symbol is~\eqref{eq:multiplyoperator} is well defined as a continuous linear operator $\sch(\mfd) \to \sch'(\mfd)$ and even as a bounded operator on $L^2(\mfd)$ since the operators of the form $\Opw{f_j}$ are continuous on $\sch(\mfd)$. Actually, under the assumption $j_0 \ge (\ell + 1) N/\epsilon$, we have $\hat f_j(\tau) \in S(\theta_{g(\tau)}^{-\ell N}, g(\tau))$ from Proposition~\ref{prop:continuityEcal} and Lemma~\ref{lem:continuityHp3} (see~\eqref{eq:arrowshat} and~\eqref{eq:arrowshatbis}). The latter symbol class is contained in $S(\theta_g^{-\ell N}, g)$ (the embedding is not uniform in time, but we don't care for now). Thus, Lemma~\ref{lem:iteratedcommutators} applies to $a = \hat f_j(\tau)$ and ensures that $c_\pi(t)$ gives rise to a bounded operator on $L^2(\mfd)$. Again, more precise bounds will be established in Section~\ref{subsec:endoftheproof}.
\end{remark}

\begin{proof}
We prove the claim by induction on $\ell \in \N$. The basis step $\ell = 0$ is a matter of convention. The step $\ell = 1$ is a consequence Corollary~\ref{cor:Egorovwithnabla-1symbols}:
\begin{align*}
\left(\e^{t \cal{H}_p} a\right) \moyal f_1
	&= \e^{t \cal{H}_p} \left(a \moyal \e^{-t \cal{H}_p} f_1\right)
	= \e^{t \cal{H}_p} \left(a \moyal \cal{E}_{\le j_0}(-t) f_1 - a \moyal \int_0^t \e^{-(t - s) \cal{H}_p} \hat f_1(s) \dd s\right) \\
	&= \e^{t \cal{H}_p} \left(a \moyal f_1(t)\right) + \left(\e^{t \cal{H}_p} a\right) \moyal \int_0^t \e^{s \cal{H}_p} \hat f_1(s) \dd s .
\end{align*}
These terms are exactly of the expected form~\eqref{eq:formweexpect}.

Now suppose the claim is true at step $\ell \ge 1$. Let $f_0, f_1, \ldots, f_\ell \in \nabla^{-1} S(\theta_g^N, g)$. We apply the induction hypothesis with $f_1, f_2, \ldots, f_\ell$: we obtain that the symbol
\begin{equation*}
\left(\e^{t \cal{H}_p} a\right) \moyal f_\ell \moyal \cdots \moyal f_1 \moyal f_0
\end{equation*}
can be written as a sum of terms of the form
\begin{equation} \label{eq:compositioninduction}
\e^{t \cal{H}_p} \left( a \moyal \left( \Moy_{k_1 \in K_1} f_{k_1}(t) \right) \right) \moyal \left( \Moy_{\pi \in \Pi} c_\pi(t) \right)\moyal f_0
	= A_1 - A_2 ,
\end{equation}
with
\begin{align*}
A_1
	&:= \e^{t \cal{H}_p} \left( a \moyal \left( \Moy_{k_1 \in K_1} f_{k_1}(t) \right) \right) \moyal f_0 \moyal \left( \Moy_{\pi \in \Pi} c_\pi(t) \right) , \\
A_2	
	&:= \e^{t \cal{H}_p} \left( a \moyal \left( \Moy_{k_1 \in K_1} f_{k_1}(t) \right) \right) \moyal \cal{H}_{f_0} \left( \Moy_{\pi \in \Pi} c_\pi(t) \right) .
\end{align*}
Let us deal with $A_1$ first: from Corollary~\ref{cor:Egorovwithnabla-1symbols} we can write
\begin{equation*}
\e^{-t \cal{H}_p} f_0
	= \cal{E}_{\le j_0}(-t) f_0 - \e^{- t \cal{H}_p} \int_0^t \e^{s \cal{H}_p} \hat f_0(s) \dd s ,
\end{equation*}
so that
\begin{multline*}
A_1
	= 
\e^{t \cal{H}_p} \left( a \moyal \left( \Moy_{k_1 \in K_1} f_{k_1}(t) \right) \moyal f_0(t) \right) \moyal \left( \Moy_{\pi \in \Pi} c_\pi(t) \right) \\
	- \e^{t \cal{H}_p} \left( a \moyal \left( \Moy_{k_1 \in K_1} f_{k_1}(t) \right) \right) \moyal \left( \int_0^t \e^{s \cal{H}_p} \hat f_0(s) \dd s \right) \moyal \left( \Moy_{\pi \in \Pi} c_\pi(t) \right)  .
\end{multline*}
This yields two terms of the expected form~\eqref{eq:formweexpect}.

We now deal with the term $A_2$ in~\eqref{eq:compositioninduction}: we use the product rule~\eqref{eq:productrules} to compute the action of $\cal{H}_{f_0}$ on the Moyal product of the $c_\pi(t)$ symbols. We obtain a product of the form
\begin{equation*}
c_{\pi_1}(t) \moyal \cdots \moyal c_{\pi_{l-1}}(t) \moyal \cal{H}_{f_0} c_{\pi_l}(t) \moyal c_{\pi_{l+1}}(t) \moyal \cdots \moyal c_{\pi_k}(t) ,
\end{equation*}
where $\Pi = \{\pi_1, \pi_2, \ldots, \pi_k\}$. This can be written as
\begin{equation*}
\Moy_{\pi' \in \Pi'} c_{\pi'}(t) ,
\end{equation*}
where $\Pi'$ is a partition of $\{0, 1, \ldots, \ell\}$ whose block are given by $\pi_j' = \pi_j$ except for $j = l$ where $\pi_l' = \pi_l \cup \{0\}$.
Therefore we obtain terms of the desired form~\eqref{eq:formweexpect}. This finishes the induction.
\end{proof}

\Large
\section{Quantum evolution of confined symbols: proof of Theorem~\ref{thm:partition} and applications} \label{sec:proofpartition}
\normalsize

This section is organized as follows: we first explain the main ideas of the proof in Section~\ref{subsec:strategysketch}. Then we collect technical confinement estimates in Section~\ref{subsec:confinementremainder} before presenting the core of the proof of Theorem~\ref{thm:partition} in Section~\ref{subsec:endoftheproof}. Lastly, we prove Corollary~\ref{cor:continuitypropagatorSchwartz} as a consequence of Theorem~\ref{thm:partition} in Section~\ref{subsec:continuityschclass}.

\subsection{Strategy of the proof of Theorem~\ref{thm:partition}} \label{subsec:strategysketch}

Before getting into the proof of Theorem~\ref{thm:partition}, let us explain our strategy. If $a$ is {$g$-confined} near $\rho_0 \in T^\star \mfd$, we expect $a_t := \e^{t \cal{H}_p} a$ to be {$g(t)$-confined} near $\phi^{-t}(\rho_0)$. Therefore, our goal is to show that
\begin{equation*}
\rho \longmapsto
	\jap*{\dist_{{\sf g}_t^\sympf}\left(\rho, B_{r(t)}^{g(t)}\left(\phi^{-t}(\rho_0)\right)\right)}^n a_t(\rho) ,
		\qquad {\sf g}_t^\sympf := g_{\phi^{-t}(\rho_0)}^\sympf(t) ,
\end{equation*}
is smooth and bounded, applying Beals' theorem (Proposition~\ref{prop:Beals}). Multiplication of the symbol $a_t$ by powers of the distance function could possibly be achieved at the operator level by composing $\Opw{a_t} = \e^{\ii t P} \Opw{a} \e^{- \ii t P}$ with operators of the form $\Opw{w_{g(t), \phi^{-t}(\rho_0)}}$, where $w_{g, \rho_0}$ would be a smooth version of
\begin{equation*}
\rho \longmapsto
	\jap*{\dist_{g_{\rho_0}^\sympf}\left(\rho, B_r^g(\rho_0)\right)} .
\end{equation*}
However, pseudo-differential calculus would give lower order terms
\begin{equation*}
\Opw{a_t} \Opw{w_{g(t), \phi^{-t}(\rho_0)}}
	= \Opw{a_t w_{g(t), \phi^{-t}(\rho_0)} + \widehat{\cal{P}}_1\left(a_t, w_{g(t), \phi^{-t}(\rho_0)}\right)}
\end{equation*}
making the analysis more complicated. Instead, we compose $\Opw{a_t}$ with operators of the form $\Opw{f_j}$, with
\begin{equation*}
f_j(\rho)
	= {\sf g}_t^\sympf\left(\rho - \phi^{-t}(\rho_0), e_j\right) ,
\end{equation*}
where $(e_j)_j$ is an orthonormal basis of $W$ for the metric ${\sf g}_t^\sympf$. These affine functions satisfy
\begin{equation} \label{eq:sumaffinefunctions}
\sum_j f_j^2
	= \abs{\bigcdot - \phi^{-t}(\rho_0)}_{{\sf g}_t^\sympf}^2 .
\end{equation}
The multiplication of $a_t$ by this weight can be achieved at the operator level using ``anti-commutators":
\begin{equation} \label{eq:formulaanti-comm}
\dfrac{1}{2} \left(\Opw{f_j} \Opw{a_t} + \Opw{a_t} \Opw{f_j}\right)
	= \Opw{f_j a_t} .
\end{equation}
This holds without remainder term because $f_j$ is polynomial of degree $1$, and the terms of order~$1$ in pseudo-differential calculus cancel because $\cal{P}_1$ is skew-symmetric (see Remark~\ref{rmk:polynomialpseudocalc}). By taking appropriate anti-commutators of $\Opw{a_t}$ with quantizations of affine functions, one is able to obtain the operator
\begin{equation} \label{eq:timesconfiningweight}
\Opw{\jap*{\rho - \phi^{-t}(\rho_0)}_{{\sf g}_t^\sympf}^{2k} a_t}
\end{equation}
for any $k$. Then Lemma~\ref{lem:combinatorialmultiplication} gives another way of writing its symbol, which allows to estimate the {$L^2$-norm} of the operator~\eqref{eq:timesconfiningweight}, using that $a$ is {$g$-confined} near $\rho_0$. Similarly, Lemma~\ref{lem:iteratedcommutators} allows to handle derivatives of this symbol. Ultimately, Beals' theorem implies that
\begin{equation*}
\rho \longmapsto
	\jap*{\rho - \phi^{-t}(\rho_0)}_{{\sf g}_t^\sympf}^{2k} a_t
\end{equation*}
belongs to $S(1, {\sf g}_t^\sympf)$.

However, uniformity in $\rho_0$ could fail due to the fact that
\begin{equation*}
\rho \longmapsto
	\jap*{\rho - \rho_0}_{g_{\rho_0}^\sympf}
\end{equation*}
is potentially much larger than
\begin{equation*}
\rho \longmapsto
	\jap*{\dist_{g_{\rho_0}^\sympf}\left(\rho, B_r^g(\rho_0)\right)}
\end{equation*}
in $B_r^g(\rho_0)$, since $g^\sympf \ge g$. In fact the loss that we have at this stage is of order
\begin{equation*}
\sup_{\rho \in T^\star \mfd} \dfrac{\jap*{\rho - \rho_0}_{g_{\rho_0}^\sympf}}{\jap*{\dist_{g_{\rho_0}^\sympf}\left(\rho, B_r^g(\rho_0)\right)}}
	\approx \sup_{\zeta \in W \setminus \{0\}} \dfrac{\abs{\zeta}_{g_{\rho_0}^\sympf}}{\abs{\zeta}_{g_{\rho_0}}} .
\end{equation*}
This quantity turns out to be controlled by $\theta_g^2$ (see Proposition~\ref{prop:thetag}). Thus, applying this strategy to $\widehat{\cal{E}}_j(t) a$ (for which we have some gain in terms of $\theta_g^{-1}$ for $j$ large enough---see Proposition~\ref{prop:continuityEcalconf}) instead of $\e^{t \cal{H}_p} a$ will allow to compensate this loss.

\subsection{Preliminary confinement estimates} \label{subsec:confinementremainder}

We start with a lemma on confined symbols.

\begin{lemma} \label{lem:conftosymbolclass}
Let $g$ be an admissible metric and $m$ a {$g$-admissible} weight, with a common slow variation radius $r_0$. Then any uniformly {$g$-confined} family of symbols $(\psi_{\rho_0})_{\rho_0 \in T^\star \mfd}$ with radius $r \le r_0$ satisfies:
\begin{equation*}
\psi_{\rho_0}
	\in S\left(\dfrac{m}{m(\rho_0)}, g\right) ,
\end{equation*}
uniformly in $\rho_0 \in T^\star \mfd$, or in other words
\begin{equation} \label{eq:snestconfm}
\forall \ell \in \N, \exists C_\ell > 0 : \forall \rho_0 \in T^\star \mfd , \qquad
	\abs*{\psi_{\rho_0}}_{S(m, g)}^{(\ell)}
		\le \dfrac{C_\ell}{m(\rho_0)} .
\end{equation}
\end{lemma}

\begin{proof}
For any $\rho \in T^\star \mfd$, Proposition~\ref{prop:improvedadmissibility} yields
\begin{equation} \label{eq:temperanceofmhere}
m(\rho_0)
	\le C_m m(\rho) \jap*{\dist_{g_{\rho_0}^\sympf}\left(\rho, B_r^g(\rho_0)\right)}^{N_m} ,
		\qquad \forall \rho \in T^\star \mfd .
\end{equation}
Notice that it is important that $r$ is a slow variation radius of both $g$ and $m$.
That $\psi_{\rho_0}$ is {$g$-confined} near $B_r^g(\rho_0)$ implies that for any family of vector fields $X_1, X_2, \ldots, X_\ell$ on $T^\star \mfd$:
\begin{align*}
\abs*{\nabla^\ell \psi_{\rho_0} (X_1, X_2, \ldots, X_\ell)}
	&\le \abs*{\psi_{\rho_0}}_{\Conf_r^g(\rho_0)}^{(\ell+k)} \jap*{\dist_{g_{\rho_0}^\sympf}\left(\rho, B_r^g(\rho_0)\right)}^{-k} \prod_{j=1}^\ell \abs{X_j}_g \\
	&\le \abs*{\psi_{\rho_0}}_{\Conf_r^g(\rho_0)}^{(\ell+k)} C_m \dfrac{m(\rho)}{m(\rho_0)} \jap*{\dist_{g_{\rho_0}^\sympf}\left(\rho, B_r^g(\rho_0)\right)}^{N_m - k} \prod_{j=1}^\ell \abs{X_j}_g ,
\end{align*}
where we plugged~\eqref{eq:temperanceofmhere} in the last line.
Taking $k \ge N_m$ gives the desired estimate~\eqref{eq:snestconfm} with $C_\ell = \abs*{\psi_{\rho_0}}_{\Conf_r^g(\rho_0)}^{(\ell+k)} C_m$, which is uniform in $\rho_0$ since $(\psi_{\rho_0})_{\rho_0 \in T^\star \mfd}$ is uniformly {$g$-confined}.
\end{proof}

\begin{lemma} \label{lem:multiaffconf}
Let $T \in [0, \frac{1}{2} T_E]$ and $r_0 \in (0, r_g]$ satisfy~\eqref{eq:conditionr}. Let $(\psi_{\rho_0})_{\rho_0 \in T^\star \mfd}$ be a uniformly {$g(t)$-confined} family of symbols with radius $r(t)$. Let $f \in \aff(T^\star \mfd; \R)$. Then for any $j_0 \ge 0$, the following holds:
\begin{equation*}
\psi_{\rho_0} \moyal \cal{E}_{\le j_0}(s) f
	 \in \left(\abs*{f\left(\phi^s(\rho_0)\right)} + \abs*{H_f}_{g_{\phi^s(\rho_0)}(t)} \theta_{g(t)}^4\left(\phi^s(\rho_0)\right)\right) \Conf_{r(t)}^{g(t)}\left(\rho_0\right) ,
\end{equation*}
uniformly in $\rho_0 \in T^\star \mfd$ and $\abs{s} \le \abs{t} \le T$.
\end{lemma}

\begin{proof}
In all the proof, we write $\rho_s = \phi^s(\rho_0)$ for simplicity. Firstly\footnote{At first glance, we could simplify this proof by observing that $f - f(\rho_s) \in S(\theta_g \jap{\rho - \rho_s}_{g_{\rho_s}^\sympf}, g)$. However, the weight $\theta_g \jap{\rho - \rho_s}_{g_{\rho_s}^\sympf}$ does not have uniform structure constants with respect to $\rho_0$: slow variation of $\jap{\bigcdot - \rho_s}_{g_{\rho_s}^\sympf}$ with respect to $g$ degenerates away from $\rho_s$. This is a problem to keep track of the dependence of constants on $\rho_0$ while applying the pseudo-differential calculus.}, let us write
\begin{equation*}
\cal{E}_{\le j_0}(s) f
	= \e^{s H_p} f + \tilde f(s) ,
		\qquad
\tilde f(s)
	:= \sum_{j=1}^{j_0} \cal{E}_j(s) f .
\end{equation*}
Applying pseudo-differential calculus at order~$1$, we have
\begin{equation} \label{eq:threetermstohandle}
\psi_{\rho_0} \moyal \cal{E}_{\le j_0}(s) f
	= \psi_{\rho_0} \e^{s H_p} f + \psi_{\rho_0} \tilde f(s) + \widehat{\cal{P}}_1\left(\psi_{\rho_0}, \cal{E}_{\le j_0}(s) f\right) .
\end{equation}
We handle the three terms in~\eqref{eq:threetermstohandle} separately.
\begin{itemize}[label=\textbullet]
\item We first deal with the third term in the right-hand side of~\eqref{eq:threetermstohandle}. Since $f \in \nabla^{-1} S\left( \theta_g, g\right)$ by Lemma~\ref{lem:seminormaff}, we have by Corollary~\ref{cor:continuityofflowinsymbolclasses} and Proposition~\ref{prop:continuityEcal} (applied with $m = \theta_g$, $\kappa = \Upsilon$):
\begin{align*}
f
	\in \nabla^{-1} S\left( \theta_g, g \right)
	\xlongrightarrow{\cal{E}_{\le j_0}(s)} \nabla^{-1} S\left(\e^{\Upsilon \abs{s}} \theta_g, g(s)\right)
		&\subset \nabla^{-1} S\left(\e^{\frac{3}{2} (\Lambda +  2\Upsilon) \abs{t}} \theta_{g(t)}, g(t)\right) .
\end{align*}
By the pseudo-differential calculus for confined symbols (Proposition~\ref{prop:pseudocalcconf}), we have
\begin{equation*}
\Conf_{r(t)}^{g(t)}\left(\rho_0\right) \times \nabla^{-1} S\left(\e^{\frac{3}{2} (\Lambda +  2\Upsilon) \abs{t}} \theta_{g(t)}, g(t)\right)
		\xlongrightarrow{\widehat{\cal{P}}_1} \gain_{g(t)}(\rho_0) \e^{\frac{3}{2} (\Lambda +  2\Upsilon) \abs{t}} \theta_{g(t)}(\rho_0) \Conf_{r(t)}^{g(t)}\left(\rho_0\right) .
\end{equation*}
Recalling that $\theta_{g(t)} \circ \phi^s \le C_\Upsilon \e^{\Upsilon \abs{s}} \theta_{g(t)}$ from Assumption~\ref{assum:p}~\ref{it:metriccontrol}, then, provided that $\abs{t} \le \frac{1}{2} T_E$, we have
\begin{equation*}
\gain_{g(t)}(\rho_0) \e^{\frac{3}{2} (\Lambda +  2\Upsilon) \abs{t}} \theta_{g(t)}(\rho_0)
	\le \udl{\gain}_g \e^{4 (\Lambda +  2\Upsilon) \abs{t}} \theta_{g(t)}(\rho_s)
	\le \theta_{g(t)}(\rho_s) .
\end{equation*}
Lemma~\ref{lem:seminormaff} implies
\begin{equation*}
\abs*{\nabla f}_{S(\theta_g, g)}^{(0)}
	\le \theta_{g(t)}(\rho_s) \abs{H_f}_{g_{\rho_s}(t)} ,
\end{equation*}
so that we obtain
\begin{equation} \label{eq:mult3rdterm}
\widehat{\cal{P}}_1\left(\psi_{\rho_0}, \cal{E}_{\le j_0}(s) f\right)
	\in \abs{H_f}_{g_{\rho_s}(t)} \theta_{g(t)}^2(\rho_s) \Conf_{r(t)}^{g(t)}\left(\rho_0\right) ,
\end{equation}
which handles the third term in~\eqref{eq:threetermstohandle}.
\item Similarly for the second term of the right-hand side of~\eqref{eq:threetermstohandle}, we have by Proposition~\ref{prop:continuityEcal}:
\begin{equation} \label{eq:ftildesame}
\tilde f(s)
	\in S\left(\theta_{g(t)}(\rho_s) \abs{H_f}_{g_{\rho_s}(t)} \e^{\Upsilon \abs{s}} \theta_g, g(s)\right)
	\subset S\left(\theta_{g(t)}(\rho_s) \abs{H_f}_{g_{\rho_s}(t)} \e^{\frac{3}{2} (\Lambda + 2\Upsilon) \abs{t}} \theta_{g(t)}, g(t)\right) ,
\end{equation}
and the pseudo-differential calculus for confined symbols (Proposition~\ref{prop:pseudocalcconf}) yields
\begin{equation} \label{eq:alsovalidforP0}
\begin{multlined}
\Conf_{r(t)}^{g(t)}\left(\rho_0\right) \times S\left(\theta_{g(t)}(\rho_s) \abs{H_f}_{g_{\rho_s}(t)} \e^{\frac{3}{2} (\Lambda + 2\Upsilon) \abs{t}} \theta_{g(t)}, g(t)\right) \\
		\xlongrightarrow{\cal{P}_0} \theta_{g(t)}(\rho_s) \theta_{g(t)}(\rho_0) \abs{H_f}_{g_{\rho_s}(t)} \e^{\frac{3}{2} (\Lambda + 2\Upsilon) \abs{t}} \Conf_{g(t)}^{r(t)}\left(\rho_0\right) .
\end{multlined}
\end{equation}
 From Proposition~\ref{prop:thetag}:
\begin{equation} \label{eq:getridexpthetag}
\e^{2 (\Lambda + 2 \Upsilon) \abs{t}}
	\le \e^{2 (\Lambda + 2 \Upsilon) \abs{t}} \gain_{g(t)} \theta_{g(t)}^2
	= \e^{4 (\Lambda + 2 \Upsilon) \abs{t}} \udl{\gain}_g \theta_{g(t)}^2
	\le \theta_{g(t)}^2 ,
		\qquad \abs{t} \le \tfrac{1}{2} T_E .
\end{equation}
Therefore
\begin{equation} \label{eq:mult2ndterm}
\psi_{\rho_0} \tilde f(s)
	\in \theta_{g(t)}^4(\rho_s) \abs{H_f}_{g_{\rho_s}(t)} \Conf_{r(t)}^{g(t)}\left(\rho_0\right) ,
		\qquad \abs{t} \le \tfrac{1}{2} T_E .
\end{equation}
\item Lastly, the first term in the right-hand side of~\eqref{eq:threetermstohandle} is treated separately as follows: choosing $\tilde \rho$ such that
\begin{equation} \label{eq:theargumentstarts}
\abs*{\tilde \rho - \rho_0}_{g_{\rho_0}^\sympf}
	= \dist_{g_{\rho_0}^\sympf}\left(\rho, B_{r_0}^g(\rho_0)\right) ,
\end{equation}
we have, using that $f$ is affine and by the triangle inequality:
\begin{equation} \label{eq:sincefisaffine}
\abs*{\e^{s H_p} f(\rho)}
	\le \abs*{f\left(\phi^s(\rho_0)\right)} + \abs*{\sympf\left(H_f, \phi^s(\tilde \rho) - \phi^s(\rho_0)\right)} + \abs*{\sympf\left(H_f, \phi^s(\rho) - \phi^s(\tilde \rho)\right)} .
\end{equation}
We handle the third term in~\eqref{eq:sincefisaffine} thanks to the Cauchy--Schwarz inequality $\abs{\sympf(X, Y)} \le \abs{X}_g \abs{Y}_{g^\sympf}$ and Proposition~\ref{prop:temperancepropertyflow}:
\begin{align} \label{eq:1stpieceest}
\abs*{\sympf\left(H_f, \phi^s(\rho) - \phi^s(\tilde \rho)\right)}
	&\le \abs*{H_f}_{g_{\rho_s}(s)} \abs*{\phi^s(\rho) - \phi^s(\tilde \rho)}_{g_{\rho_s}^\sympf(s)} \nonumber\\
	&\le C \abs*{H_f}_{g_{\rho_s}(t)} \abs*{\rho - \tilde \rho}_{g_{\rho_0}^\sympf} \jap*{\rho - \tilde \rho}_{g_{\rho_0}^\natural}^N \nonumber\\
	&\le C \abs*{H_f}_{g_{\rho_s}(t)} \e^{(\Lambda + 2 \Upsilon) \abs{t}} \abs*{\rho - \tilde \rho}_{g_{\rho_0}^\sympf(t)} \jap*{\rho - \tilde \rho}_{g_{\rho_0}^\natural}^N \nonumber\\
	&\le C \abs*{H_f}_{g_{\rho_s}(t)} \theta_{g(t)}(\rho_s) \jap*{\dist_{g_{\rho_0}^\sympf(t)}\left(\rho, B_{r(t)}^{g(t)}(\rho_0)\right)}^{N+1} ,
\end{align}
where we used~\eqref{eq:getridexpthetag} in the last inequality.
As for the second term in~\eqref{eq:sincefisaffine}, we have by Proposition~\ref{prop:Lipschitzproperty} and by definition of the temperance weight $\theta_g$:
\begin{align} \label{eq:2ndpieceest}
\abs*{\sympf\left(H_f, \phi^s(\tilde \rho) - \phi^s(\rho_0)\right)}
	&\le \abs*{H_f}_{g_{\rho_s}^\sympf(t)} \abs*{\phi^s(\tilde \rho) - \phi^s(\rho_0)}_{g_{\rho_s}(t)} \nonumber\\
	&\le \theta_{g(t)}^2(\rho_s) \abs*{H_f}_{g_{\rho_s}(t)} \times \abs{\tilde \rho - \rho_0}_{g_{\rho_0}(t)} \e^{(\Lambda + C_g^3 C_p \udl{\gain}_g) \abs{t}} \nonumber\\
	&\le \theta_{g(t)}^2(\rho_s) \abs*{H_f}_{g_{\rho_s}(t)} r(t) .
\end{align}
Recalling that from~\eqref{eq:conditionr(t)} we have $r(t) \le r_g \le 1$,~\eqref{eq:1stpieceest} and~\eqref{eq:2ndpieceest} show that
\begin{equation*}
\abs*{\e^{s H_p} f(\rho)}
	\le C \left( \abs*{f\left(\phi^s(\rho_0)\right)} + \theta_{g(t)}^2(\rho_s) \abs*{H_f}_{g_{\rho_s}(t)} \right) \jap*{\dist_{g_{\rho_0}^\sympf(t)}\left(\rho, B_{r(t)}^{g(t)}(\rho_0)\right)}^{N+1} .
\end{equation*}
This implies in turn that $\psi_{\rho_0} \e^{s H_p} f$ decays like a confined symbol in
\begin{equation} \label{eq:multstillinconf}
\left(\abs*{f(\rho_s)} + \theta_{g(t)}^2(\rho_s) \abs*{H_f}_{g_{\rho_s}(t)}\right) \Conf_{g(t)}^{r(t)}(\rho_0) ,
\end{equation}
namely
\begin{equation} \label{eq:decaylikeconfclass}
\abs*{\psi_{\rho_0} \e^{s H_p} f}
	\le \left(\abs*{f(\rho_s)} + \theta_{g(t)}^2(\rho_s) \abs*{H_f}_{g_{\rho_s}(t)}\right) \jap*{\dist_{g_{\rho_0}^\sympf(t)}\left(\rho, B_{r(t)}^{g(t)}(\rho_0)\right)}^{-k} ,
\end{equation}
for all $k \in \N$.
To handle derivatives, we differentiate the product $\psi_{\rho_0} \e^{s H_p} f$ using the Leibniz formula. When derivatives hit $\psi_{\rho_0}$, repeating the argument from~\eqref{eq:theargumentstarts} shows that $(\nabla^\ell \psi_{\rho_0}) \e^{s H_p} f$ still decays like a confined symbol in the class~\eqref{eq:multstillinconf}, namely~\eqref{eq:decaylikeconfclass} is true with $\nabla^\ell \psi_{\rho_0}$ in place of $\psi_{\rho_0}$. When derivatives hit $\e^{s H_p} f$, we proceed as follows: $\nabla (\e^{s H_p} f)$ satisfies the same properties as $\tilde f(s)$ in the previous step, namely~\eqref{eq:ftildesame}. As we did in~\eqref{eq:alsovalidforP0} and~\eqref{eq:getridexpthetag}, we obtain
\begin{equation*}
\psi_{\rho_0} \nabla \left(\e^{s H_p} f\right)
	 \in \abs*{H_f}_{g_{\rho_s}(t)} \theta_{g(t)}^4(\rho_s) \Conf_{g(t)}^{r(t)}\left(\rho_0\right) .
\end{equation*}
So finally we obtain that $\psi_{\rho_0} \e^{s H_p} f$ belongs to the class~\eqref{eq:multstillinconf}.
\end{itemize}
Putting this together with the estimates~\eqref{eq:mult3rdterm} and~\eqref{eq:mult2ndterm} in~\eqref{eq:threetermstohandle}, we finally obtain the sought result.
\end{proof}

\subsection{End of the proof of Theorem~\ref{thm:partition}} \label{subsec:endoftheproof}

We are now prepared for the proof of Theorem~\ref{thm:partition}. The proof is divided into three main steps.
\begin{itemize}[label=\textbullet]
\item The first one consists in describing the action of anti-commutators of the operator $\Opw{\e^{t \cal{H}_p} \psi_{\rho_0}}$ with quantizations of affine symbols, through Lemma~\ref{lem:combinatorialmultiplication}. The goal of this step is to ensure that $\e^{t \cal{H}_p} \psi_{\rho_0}$ decays like a confined symbol.
\item In the second step, we describe the action of commutators of the operator resulting from the first step, with quantizations of affine symbols again. This is to ensure that the resulting symbol is smooth.
\item Finally, we apply Beals' theorem in the last step (Proposition~\ref{prop:Beals}) so as to establish the seminorm estimates of Theorem~\ref{thm:partition}.
\end{itemize}

Let $T \in [0, \tfrac{1}{2} T_E]$ and let $(\psi_{\rho_0})_{\rho_0 \in T^\star \mfd}$ be a {$g$-uniformly} confined family of symbols with radius
\begin{equation*}
r_0
	:= r_g \e^{- (2 (\Lambda + \Upsilon) + C_g^3 C_p \udl{\gain}_g) T} .
\end{equation*}
Fix $\rho_0 \in T^\star \mfd$ and set $\rho_t = \phi^{-t}(\rho_0)$ for all $t \in [-T, T]$.

\medskip
\emph{Step~1: Multiplying the symbol through anti-commutators with affine functions.}

Write $d = \dim \mfd$. Let $(e_j)_{1 \le j \le 2d}$ be an orthonormal basis of $W$, endowed with the scalar product $g_{\rho_t}^\sympf(t)$, and introduce
\begin{equation*}
w_j(\rho)
	:= g_{\rho_t}^\sympf(t)\left(\rho - \rho_t, e_j\right) ,
\end{equation*}
so that
\begin{equation*}
\abs*{\rho - \rho_t}_{g_{\rho_t}^\sympf(t)}^2
	= \sum_{j=1}^{2d} w_j^2(\rho) ,
		\qquad \forall \rho \in T^\star \mfd .
\end{equation*}
Then one has
\begin{equation} \label{eq:conditionsw_n}
\forall j \in \{1, 2, \ldots, 2d\} , \qquad
	w_j(\rho_t) = 0
		\quad \rm{and} \quad
	\abs*{H_{w_j}}_{g_{\rho_t}(t)} = \abs{J_g^{-1} e_j}_{g_{\rho_t}(t)} = \abs{e_j}_{g_{\rho_t}^\sympf(t)} = 1 .
\end{equation}

Fix $K \ge 0$ an integer. Expanding $\jap{\bullet}^{2K}$ below, the symbol
\begin{equation} \label{eq:symbolofinterest}
a : \rho \longmapsto
	\jap*{\rho - \rho_t}_{g_{\rho_t}^\sympf(t)}^{2K} \widehat{\cal{E}}_{j_0+1}(t) \psi_{\rho_0}
\end{equation}
can be written as a sum of symbols of the form
\begin{equation*}
\prod_n w_n \times \widehat{\cal{E}}_{j_0+1}(t) \psi_{\rho_0} ,
\end{equation*}
where the number of factors in the product is less than $2K$. The factors $(w_n)_n$ are relabeled to form a family of $N$ affine functions $(w_1, w_2, \ldots, w_N)$, with $N \le 2 K$.
In view of~\eqref{eq:formulaanti-comm}, this symbol can be rewritten as a sum of anti-commutators with operators of the form $\Opw{w_n}$:
\begin{equation*}
\Opw{\prod_n w_n \times \widehat{\cal{E}}_{j_0+1}(t) \psi_{\rho_0}}
	= 2^{-N} \ad_{\Opw{w_N}}^+ \cdots \ad_{\Opw{w_2}}^+ \ad_{\Opw{w_1}}^+ \Opw{\widehat{\cal{E}}_{j_0+1}(t) \psi_{\rho_0}} ,
\end{equation*}
where $\ad_F^+ A := FA + AF$.
Expanding those anti-commutators, we are left with studying
\begin{equation*}
\left(\Moy_{n_1 \in \cal{N}_1} w_{n_1}\right) \moyal \widehat{\cal{E}}_{j_0+1}(t) \psi_{\rho_0} \moyal \left(\Moy_{n_2 \in \cal{N}_2} w_{n_2}\right)
	= \int_0^t \left(\Moy_{n_1 \in \cal{N}_1} w_{n_1}\right) \moyal \e^{s \cal{H}_p} \cal{H}_p^{(3)} \cal{E}_{j_0}(t - s) \psi_{\rho_0} \moyal \left(\Moy_{n_2 \in \cal{N}_2} w_{n_2}\right) \dd s ,
\end{equation*}
where the right-hand side comes from the recurrence relation~\eqref{eq:recurrencerelation}, and the sets of indices $\cal{N}_1 \cup \cal{N}_2 = \{1, 2, \ldots, N\}$ are disjoint.
Applying Lemma~\ref{lem:combinatorialmultiplication}, we can write this symbol as a sum of symbols of the form
\begin{equation} \label{eq:symbolleftwith}
\tilde a
	:= \int_0^t \left( \Moy_{\pi_1 \in \Pi_1} c_{\pi_1}(s) \right) \moyal \e^{s \cal{H}_p} \tilde \psi(t, s) \moyal \left( \Moy_{\pi_2 \in \Pi_2} c_{\pi_2}(s) \right) \dd s .
\end{equation}
where
\begin{equation} \label{eq:deftildepsits}
\tilde \psi(t, s)
	:= \left( \Moy_{l_1 \in \cal{L}_1} w_{l_1}(s) \right) \moyal \cal{H}_p^{(3)} \cal{E}_{j_0}(t - s) \psi_{\rho_0} \moyal \left( \Moy_{l_2 \in \cal{L}_2} w_{l_2}(s) \right) .
\end{equation}
Here we have the partitions of indices
\begin{equation} \label{eq:partitionkjljnj}
\cal{K}_j \cup \cal{L}_j
	= \cal{N}_j ,
		\qquad j \in \{1, 2\} ,
\end{equation}
the set $\Pi_j$ is a partition of $\cal{K}_j$ and 
\begin{equation*}
w_{l}(s)
	= \cal{E}_{\le j_0}(-s) w_{l}(s) ,
		\qquad
\hat w_{l}(s)
	= \cal{H}_p^{(3)} \cal{E}_{j_0}(-s) w_{l} ;
\end{equation*}
and given $\pi \subset \N$, we have
\begin{equation} \label{eq:recalldefcpi}
c_\pi(s)
	= \cal{H}_{\{w_ j\}_{j \in \pi \setminus \{j_\ast\}}} \int_0^s \e^{\tau \cal{H}_p} \hat w_{j_\ast}(\tau) \dd \tau ,
		\quad \textrm{for some $j_\ast \in \pi$.}
\end{equation}
We claim that
\begin{equation} \label{eq:claimtildepsits}
\tilde \psi(t, s)
	\in \udl{\gain}_g^{3/2} \theta_{g(t)}^{4(L_1 + L_2) - (j_0+1) \epsilon}\left(\rho_{t}\right) \Conf_{r(t)}^{g(t)}\left(\rho_{t - s}\right) ,
\end{equation}
where $L_j := \# \cal{L}_j$.
Indeed, we have
\begin{equation*}
\cal{H}_p^{(3)} \cal{E}_{j_0}(t - s) \psi_{\rho_0}
	\in \udl{\gain}_g^{3/2} \theta_{g(t-s)}^{- (j_0+1) \epsilon}\left(\rho_{t-s}\right) \Conf_{r(t-s)}^{g(t-s)}\left(\rho_{t - s}\right) ,
\end{equation*}
since by Proposition~\ref{prop:continuityEcalconf} and Proposition~\ref{prop:mappingremainderConf}, we have
\begin{equation} \label{eq:confinconfII}
\Conf_{r_0}^g(\rho_0)
	\xlongrightarrow{\cal{E}_{j_0}(t - s)} \theta_{g(t-s)}^{- j_0 \epsilon}\left(\rho_{t-s}\right) \Conf_{r(t-s)}^{g(t-s)}\left(\rho_{t - s}\right)
	\xlongrightarrow{\cal{H}_p^{(3)}} \udl{\gain}_g^{3/2} \theta_{g(t-s)}^{- (j_0+1) \epsilon}\left(\rho_{t-s}\right) \Conf_{r(t-s)}^{g(t-s)}\left(\rho_{t - s}\right) .
\end{equation}
Recalling~\eqref{eq:monotonicityConf}, we have
\begin{equation*}
\Conf_{g(t-s)}^{r(t-s)}(\rho_{t-s})
	\subset \Conf_{g(t)}^{r(t)}(\rho_{t-s}) ,
\end{equation*}
and from the definition of $\theta_{g(t)}$ (Definition~\ref{def:thetag}) together with Assumption~\ref{assum:p}~\ref{it:metriccontrol}, we deduce that
\begin{equation*}
\theta_{g(t-s)}(\rho_{t-s})
	= \e^{(\Lambda + 2 \Upsilon) \abs{s}} \theta_{g(t)}(\rho_{t-s})
	\ge \dfrac{1}{C_\Upsilon} \theta_{g(t)}(\rho_t) .
\end{equation*}
Therefore we have
\begin{equation} \label{eq:confhcale}
\cal{H}_p^{(3)} \cal{E}_{j_0}(t - s) \psi_{\rho_0}
	\in \udl{\gain}_g^{3/2} \theta_{g(t)}^{- (j_0+1) \epsilon}\left(\rho_{t}\right) \Conf_{r(t)}^{g(t)}\left(\rho_{t - s}\right) ,
\end{equation}
uniformly with respect to $\rho_0$ and $\abs{s} \le \abs{t} \le \frac{1}{2} T_E$.

Then we apply Lemma~\ref{lem:multiaffconf} to $\cal{H}_p^{(3)} \cal{E}_{j_0}(t - s) \psi_{\rho_0}$ satisfying~\eqref{eq:confhcale} and $w_l(s)$ satisfying~\eqref{eq:conditionsw_n}, $l \in \cal{L}_1 \cup \cal{L}_2$. Given that $\phi^{-s}(\rho_{t-s}) = \phi^{-t}(\rho_0) = \rho_t$, we have
\begin{equation*}
\abs*{w_n(\rho_t)} + \abs{H_{w_n}}_{g_{\rho_t}(t)} \theta_{g(t)}^4(\rho_t)
	= \theta_{g(t)}^4(\rho_t) ,
\end{equation*}
so that~\eqref{eq:claimtildepsits} holds.

\medskip
\emph{Step~2: Differentiation of the symbol through commutators with affine functions.}

We fix a family $f_1, f_2, \ldots, f_L \in \aff(T^\star \mfd; \R)$. From the previous computations, we know that the symbol $\cal{H}_{f_L} \cdots \cal{H}_{f_2} \cal{H}_{f_1} a$ (where $a$ is defined in~\eqref{eq:symbolofinterest}) can be written as a sum of terms of the form $\cal{H}_{f_L} \cdots \cal{H}_{f_2} \cal{H}_{f_1} \tilde a$ (defined in~\eqref{eq:symbolleftwith}). Applying the product rule~\eqref{eq:productrules}, derivatives distribute on the various Moyal factors in~\eqref{eq:symbolleftwith}. Therefore we have two cases to study: either derivatives hit $\e^{s \cal{H}_p} \tilde \psi(t, s)$, or they land on a symbol of the form $c_\pi(s)$. In the sequel, $\{f_1, f_2, \ldots, f_\ell\}$ is a relabeled sub-family of $\{f_1, f_2, \ldots, f_L\}$.
\begin{itemize} [label=\textbullet]
\item If the derivatives $\cal{H}_{f_1}, \cal{H}_{f_2}, \ldots, \cal{H}_{f_\ell}$ land on $\e^{s \cal{H}_p} \tilde \psi(t, s)$, from~\eqref{eq:claimtildepsits}, we have thanks to Lemma~\ref{lem:conftosymbolclass}
\begin{equation*}
\tilde \psi(t, s)
	\in \udl{\gain}_g^{3/2} \theta_{g(t)}^{4(L_1 + L_2) - (j_0+1) \epsilon}\left(\rho_{t}\right) \theta_{g(t)}^\ell\left(\rho_{t-s}\right) S\left( \theta_{g(t)}^{-\ell}, g(t) \right) .
\end{equation*}
Using Assumption~\ref{assum:p}~\ref{it:metriccontrol} and Proposition~\ref{prop:thetag}, we have
\begin{align*}
\theta_{g(t)}(\rho_{t-s})
	&\le C_\Upsilon \e^{\Upsilon \abs{s}} \theta_{g(t)}(\rho_t)
	\le C_\Upsilon \e^{\Upsilon \abs{t}} \theta_{g(t)}(\rho_t) \left( \gain_{g(t)}(\rho_t) \theta_{g(t)}^2(\rho_t) \right)^{1/4} \\
	&\le C_\Upsilon \theta_{g(t)}(\rho_t) \left(\e^{4 (\Lambda + 2 \Upsilon) \abs{t}} \udl{\gain}_g \theta_{g(t)}^2(\rho_t) \right)^{1/4}
	\le C_\Upsilon \theta_{g(t)}^2(\rho_t) ,
\end{align*}
whenever $\abs{t} \le \frac{1}{2} T_E$.
Therefore Corollary~\ref{cor:iteratedcommutatorsaff} gives
\begin{equation} \label{eq:normoptildepsi}
\norm*{\Opw{\cal{H}_{f_\ell} \cdots \cal{H}_{f_2} \cal{H}_{f_1} \e^{s \cal{H}_p} \tilde \psi(t, s)}}_{L^2 \to L^2}
	\le C \udl{\gain}_g^{3/2} \theta_{g(t)}^{4(L_1 + L_2) + 3 \ell - (j_0+1) \epsilon}\left(\rho_{t}\right) \prod_{j = 1}^\ell \abs*{H_{f_j}}_{g_{\rho_t}(t)} ,
\end{equation}
with a constant $C$ independent of $\rho_0$ and $\abs{s} \le \abs{t} \le \frac{1}{2} T_E$.
\item If the derivatives $\cal{H}_{f_1}, \cal{H}_{f_2}, \ldots, \cal{H}_{f_\ell}$ land on $c_\pi(s)$, recalling~\eqref{eq:recalldefcpi}, we have to consider a symbol of the form
\begin{equation*}
\cal{H}_{f_\ell} \cdots \cal{H}_{f_2} \cal{H}_{f_1} c_\pi(s)
	= \cal{H}_{f_\ell} \cdots \cal{H}_{f_2} \cal{H}_{f_1} \cal{H}_{\{w_j\}_{j \in \pi \setminus \{j_\ast\}}} \int_0^s \e^{\tau \cal{H}_p} \hat w_{j_\ast}(\tau) \dd \tau .
\end{equation*}
From Proposition~\ref{prop:continuityEcal} and Lemma~\ref{lem:continuityHp3}, as we did in~\eqref{eq:arrowshat}:
\begin{multline*}
\nabla^{-1} S(\theta_g, g)
	\xlongrightarrow{\cal{E}_{j_0}(\tau)} S\left( e^{\Upsilon \abs{\tau}} \theta_g \theta_{g(\tau)}^{-j_0 \epsilon}, g(\tau) \right)
	\xlongrightarrow{\cal{H}_p^{(3)}} S\left( \udl{\gain}_g^{3/2} \e^{2 (\Lambda + 2 \Upsilon) \abs{\tau}} \theta_{g(\tau)}^{-(j_0+1) \epsilon + 1}, g(\tau) \right) \\
	\subset S\left( \udl{\gain}_g^{1/2} \theta_{g(\tau)}^{-(j_0+1) \epsilon + 1}, g(\tau) \right) ,
\end{multline*}
where the last inclusion is valid in the time range $\abs{\tau} \le T_E$.
Since we have
\begin{equation*}
w_{j_\ast}
	\in \theta_{g(t)}(\rho_t) \abs{H_{w_{j_\ast}}}_{g_{\rho_t}(t)} \nabla^{-1} S(\theta_g, g)
\end{equation*}
by Lemma~\ref{lem:seminormaff} and~\eqref{eq:conditionsw_n}, we deduce that
\begin{equation*}
\hat w_{j_\ast}(\tau)
	= \cal{H}_p^{(3)} \cal{E}_{j_0}(s) w_{j_\ast}
	\in \theta_{g(t)}(\rho_t) S\left( \udl{\gain}_g^{1/2} \theta_{g(\tau)}^{-L}, g(\tau) \right) ,
\end{equation*}
provided $j_0 \ge (L + 1)/\epsilon$. Then we can apply Corollary~\ref{cor:iteratedcommutatorsaff} to obtain
\begin{equation*}
\norm*{\Opw{\cal{H}_{f_\ell} \cdots \cal{H}_{f_2} \cal{H}_{f_1} c_\pi(s)}}_{L^2 \to L^2}
	\le C \abs{s} \udl{\gain}_g^{1/2} \theta_{g(t)}^{\abs{\pi} + \ell}\left(\rho_{t}\right) \prod_j \abs*{H_{w_{j}}}_{g_{\rho_t}(t)} \times \prod_{j = 1}^\ell \abs*{H_{f_j}}_{g_{\rho_t}(t)} .
\end{equation*}
In the right-hand side, the norm of the vectors $H_{w_{j}}$ is one due to~\eqref{eq:conditionsw_n}, and the factor $\udl{\gain}_g^{1/2} \abs{s}$ is bounded by a constant independent of $\udl{\gain}_g$ since $\abs{s} \le T_E$. Therefore
\begin{equation} \label{eq:normopcpi}
\norm*{\Opw{\cal{H}_{f_\ell} \cdots \cal{H}_{f_2} \cal{H}_{f_1} c_\pi(s)}}_{L^2 \to L^2}
	\le C' \theta_{g(t)}^{\abs{\pi} + \ell}\left(\rho_{t}\right) \prod_{j = 1}^\ell \abs*{H_{f_j}}_{g_{\rho_t}(t)} .
\end{equation}
\end{itemize}

On the whole, plugging the estimates~\eqref{eq:normoptildepsi} and~\eqref{eq:normopcpi} into~\eqref{eq:symbolleftwith}, we deduce that
\begin{equation} \label{eq:intermediateopineq}
\norm*{\Opw{\cal{H}_{f_L} \cdots \cal{H}_{f_2} \cal{H}_{f_1} \tilde a}}_{L^2 \to L^2}
	\le C \abs{t} \udl{\gain}_g^{3/2} \theta_{g(t)}^{4 N + 3 L - (j_0+1) \epsilon}\left(\rho_{t}\right) \prod_{j=1}^L \abs*{H_{f_j}}_{g_{\rho_t}(t)} .
\end{equation}
The exact power of $\theta_{g(t)}$ is due to~\eqref{eq:partitionkjljnj}. Indeed, the power $4N$ comes from Step~1: the indices $\cal{L}_1$ and $\cal{L}_2$ are counted in the estimate~\eqref{eq:normoptildepsi}, while the remaining indices $\cal{K}_1$ and $\cal{K}_2$ are partitioned by $\Pi_1$ and $\Pi_2$ respectively, and the blocks of these partitions appear in~\eqref{eq:normopcpi}. The contribution of each is at most $\theta_{g(t)}^4$. As for the power $3L$, we simply observe in Step 2 that all the {$f_j$'s} appear exactly once in one of the Moyal factors of $\tilde a$ in~\eqref{eq:symbolleftwith} while applying the product rule. The contribution of each is at most~$\theta_{g(t)}^3$.

Now in~\eqref{eq:intermediateopineq}, we recall that $\abs{t} \udl{\gain}_g^{3/2}$ is bounded by a constant for $\abs{t} \le T_E$, hence
\begin{equation} \label{eq:estimateforBeals}
\norm*{\Opw{\cal{H}_{f_\ell} \cdots \cal{H}_{f_2} \cal{H}_{f_1} a}}_{L^2 \to L^2}
	\le C' \theta_{g(t)}^{4 N + 3 L - (j_0+1) \epsilon}\left(\rho_{t}\right) \prod_j \abs*{H_{f_j}}_{g_{\rho_t}(t)} .
\end{equation}

This holds for any $j_0$ large enough, with a constant $C'$ possibly depending on $j_0$, but not on $\abs{t} \le \frac{1}{2} T_E$ nor on $\rho_0$.

\medskip
\emph{Step~3: Applying Beals' theorem.}

We finish by applying Beals' theorem (Proposition~\ref{prop:Beals}). We take $\ell \in \N$ and we consider the integer $k_\ell$ of Proposition~\ref{prop:Beals}. Then the following holds:
\begin{equation*}
\abs*{a}_{S(1, g_{\rho_t}(t))}^{(\ell)}
	\le C_\ell \theta_{g(t)}^{4 N + 3 k_\ell - (j_0+1) \epsilon} .
\end{equation*}
Going back to~\eqref{eq:symbolofinterest}, this implies that if $j_0 \ge 8 (K + k_\ell + \ell)/\epsilon$, then
\begin{equation} \label{eq:firstineqbeforeresult}
\abs*{\nabla^\ell \jap{\rho - \rho_t}_{g_{\rho_t}^\sympf(t)}^{2K} \widehat{\cal{E}}_{j_0+1}(t) \psi_{\rho_0}(\rho)}_{g_{\rho_t}(t)}
	\le C ,
\end{equation}
for some constant $C$ independent of $\rho_0$ and $\abs{t} \le \frac{1}{2} T_E$. We deduce that
\begin{equation} \label{eq:secondestimatebeforeresult}
\abs*{\jap{\rho - \rho_t}_{g_{\rho_t}^\sympf(t)}^{2K} \nabla^\ell \widehat{\cal{E}}_{j_0+1}(t) \psi_{\rho_0}(\rho)}_{g_{\rho_t}(t)}
	\le C ,
\end{equation}
with a possibly different constant. Indeed, for $\ell = 0$, this is exactly~\eqref{eq:firstineqbeforeresult}. Then by induction, if this is true for some $0 \le \ell' \le \ell - 1$, then we can show~\eqref{eq:secondestimatebeforeresult} for $\ell'+1$ by applying the Leibniz formula:
\begin{equation*}
\dfrac{1}{(\ell' + 1)!} \left(\nabla^{\ell'+1} \jap{\rho - \rho_t}_{g_{\rho_t}^\sympf(t)}^{2K} \widehat{\cal{E}}_{j_0+1}(t) \psi_{\rho_0}(\rho)\right)
	= \sum_{j = 0}^{\ell'+1} \dfrac{1}{j!} \nabla^j \jap{\rho - \rho_t}_{g_{\rho_t}^\sympf(t)}^{2K} \dfrac{1}{(\ell' + 1 - j)!} \nabla^{\ell'+1-j} \widehat{\cal{E}}_{j_0+1}(t) \psi_{\rho_0}(\rho) .
\end{equation*}
On the one hand, the left-hand side is bounded by~\eqref{eq:firstineqbeforeresult}. On the other hand, all the terms in the sum, except the term $j = 0$, can be controlled using the induction hypothesis~\eqref{eq:secondestimatebeforeresult} and the fact that for any vector field $X$,
\begin{equation*}
\abs*{\nabla_{X^j}^j \jap{\rho - \rho_t}_{g_{\rho_t}^\sympf(t)}^{2K}}
	\lesssim \jap{\rho - \rho_t}_{g_{\rho_t}^\sympf(t)}^{2 K - j} \abs{X}_{g_{\rho_t}^\sympf(t)}^j
	\le \jap{\rho - \rho_t}_{g_{\rho_t}^\sympf(t)}^{2 K} \abs{X}_{g_{\rho_t}(t)}^j \theta_{g(t)}^{2j}(\rho_t) .
\end{equation*}
This yields
\begin{equation*}
\abs*{\nabla^j \jap{\rho - \rho_t}_{g_{\rho_t}^\sympf(t)}^{2K} \nabla^{\ell'+1-j} \widehat{\cal{E}}_{j_0+1}(t) \psi_{\rho_0}(\rho)}_{g_{\rho_t}(t)}
	\le C \theta_{g(t)}^{4 N + 3 k_\ell + 2 \ell - (j_0+1) \epsilon}
	\le C' ,
\end{equation*}
hence isolating the $j=0$ term in the sum, we obtain~\eqref{eq:secondestimatebeforeresult} for $\ell'+1$.

\medskip
\emph{Final conclusion.}

We conclude that for any $\ell$, there exists $j_0$ sufficiently large and $C > 0$ (uniform with respect to $\abs{t} \le \frac{1}{2} T_E$ and $\rho_0$) such that
\begin{equation*}
\abs*{\nabla^\ell \widehat{\cal{E}}_{j_0+1}(t) \psi_{\rho_0}(\rho)}_{g_{\rho_t}(t)}
	\le C \jap{\rho - \rho_t}_{g_{\rho_t}^\sympf(t)}^{-2 \ell}
	\le C \jap*{\dist_{g_{\rho_t}^\sympf(t)}\left(\rho, B_{r(t)}^{g(t)}(\rho_t)\right)}^{-2 \ell} .
\end{equation*}
(Recall that the constant $C$ involves some seminorm $\abs*{\psi_{\rho_0}}_{\Conf_{r_0}^g(\rho_0)}^{(k)}$ due to~\eqref{eq:confinconfII}.)
By Proposition~\ref{prop:continuityEcalconf}, this estimate is equally true for $\cal{E}_{\le j_0}(t) \psi_{\rho_0}$, for any $j_0 \in \N$. We deduce that
\begin{equation*}
\exists C > 0 : \forall \rho_0 \in T^\star \mfd, \forall \abs{t} \le \tfrac{1}{2} T_E , \forall \ell \in \N, \exists k \in \N : \;\,
	\abs*{\e^{t \cal{H}_p} \psi_{\rho_0}}_{\Conf_{r(t)}^{g(t)}(\rho_0)}^{(\ell)}
		\le C \abs*{\psi_{\rho_0}}_{\Conf_{r_0}^g(\rho_0)}^{(k)} ,
\end{equation*}
which is the sought result.
\qed

\subsection{Continuity on the Schwartz class: proof of Corollary~\ref{cor:continuitypropagatorSchwartz}} \label{subsec:continuityschclass}

In this proof, we fix a family of seminorms on $\sch(\mfd)$. We identify $\mfd$ with $\R^d$, $d = \dim \mfd$, by selecting an arbitrary Euclidean structure on $\mfd$, and we set
\begin{equation*}
\abs*{u}_{\cal{S}(\mfd)}^{(k)}
	:= \max_{\abs{\alpha} \le k} \norm*{\jap*{\bullet}^k \partial^\alpha u}_{\infty} ,
		\qquad k \in \N .
\end{equation*}
We proceed similarly on $T^\star \mfd$.

First of all, we claim the following:
\begin{empheq}[left={\forall k \in \N, \exists C > 0, \exists k_0 \in \N : \forall u \in \cal{S}(\mfd) , \qquad} \empheqlbrace]{align}
	\left( \abs*{u}_{\cal{S}(\mfd)}^{(k)} \right)^2
		&\le C \abs*{u \ovee u}_{\cal{S}(T^\star \mfd)}^{(k+k_0)} , \label{eq:elementary1}\\
	\abs*{u \ovee u}_{\cal{S}(T^\star \mfd)}^{(k)}
		&\le C \left( \abs*{u}_{\cal{S}(\mfd)}^{(k+k_0)} \right)^2 , \label{eq:elementary2}
\end{empheq}
where $u \ovee u$ is the Wigner distribution defined in~\eqref{eq:defWigner}.
We start by showing these estimates with $u \otimes u$ in place of $u \ovee u$. On the one-hand, we have for any $k \in \N$ and $\alpha_1, \alpha_2 \in \N^d$ with $\abs{\alpha_1}, \abs{\alpha_2} \le k$:
\begin{equation*}
\abs*{\left(\jap*{\bullet}^k \partial^{\alpha_1} u\right) \otimes \left(\jap*{\bullet}^k \partial^{\alpha_2} u\right)}
	\le \abs*{\jap*{(x_1, x_2)}^{2k} \partial_{x_1}^{\alpha_1} \partial_{x_2}^{\alpha_2} (u \otimes u)}
	\le \abs*{u \otimes u}_{\cal{S}(\mfd \times \mfd)}^{(2k)} ,
\end{equation*}
which leads to~\eqref{eq:elementary1} for $u \otimes u$ with $C = 1$ and $k_0 = k$.
On the other hand, for any $k \in \N$ and any $\alpha \in \N^{2d}$ with $\abs{\alpha} \le k$, we can split $\alpha$ as $\alpha = (\alpha_1, \alpha_2)$, $\alpha_1, \alpha_2 \in \N^d$, and we obtain:
\begin{equation*}
\abs*{\jap*{(x_1, x_2)}^{k} \partial_{(x_1, x_2)}^{\alpha} (u \otimes u)}
	\le \abs*{\left(\jap*{\bullet}^k \partial^{\alpha_1} u\right) \otimes \left(\jap*{\bullet}^k \partial^{\alpha_2} u\right)}
	\le \left( \abs*{u}_{\cal{S}(\mfd)}^{(k)} \right)^2 ,
\end{equation*}
which leads to~\eqref{eq:elementary2} for $u \otimes u$ with $C = 1$ and $k_0 = 0$. To deduce~\eqref{eq:elementary1} and~\eqref{eq:elementary2} with $u \ovee u$, we use the fact that the latter can be computed from $u \otimes u$ through a rotation in $\mfd \times \mfd$ and a partial Fourier transform:
\begin{equation*}
(u \ovee u)(x, \xi)
	= \ft_{y \to \xi} \left( (u \otimes u)\left(x + \dfrac{y}{2}, x - \dfrac{y}{2}\right) \right)
\end{equation*}
(recall the definition of the Wigner transform in~\eqref{eq:defWigner}).
These two operations are continuous on the Schwartz space, hence
\begin{empheq}[left={\forall k \in \N, \exists C > 0, \exists k_0 \in \N : \forall u \in \cal{S}(\mfd) , \qquad} \empheqlbrace]{align}
	\abs*{u \otimes u}_{\cal{S}(\mfd \times \mfd)}^{(k)}
		&\le C \abs*{u \ovee u}_{\cal{S}(T^\star \mfd)}^{(k+k_0)} \\
	\abs*{u \ovee u}_{\cal{S}(T^\star \mfd)}^{(k)}
		&\le C \abs*{u \otimes u}_{\cal{S}(\mfd \times \mfd)}^{(k+k_0)} .
\end{empheq}
This justifies~\eqref{eq:elementary1} and~\eqref{eq:elementary2}.

Second, for any $u_0 \in \sch(\mfd)$, writing $u(t) = \e^{- \ii t P} u_0$, we know by definition of the Weyl quantization~\eqref{eq:defWeylquantization} and Proposition~\ref{prop:quantizationisometry} that for all $a \in \sch(T^\star \mfd)$ and $t \in \R$:
\begin{align*}
\inp*{a}{u(t) \ovee u(t)}_{L^2(T^\star \mfd)}
	&= \inp*{u(t)}{\Opw{a} u(t)}_{L^2(\mfd)}
	= \inp*{u_0}{\e^{\ii t P} \Opw{a} \e^{- \ii t P} u_0}_{L^2(\mfd)} \\
	&= \inp*{\e^{t \cal{H}_p} a}{u_0 \ovee u_0}_{L^2(T^\star \mfd)}
	= \inp*{a}{\e^{- t \cal{H}_p} (u_0 \ovee u_0)}_{L^2(T^\star \mfd)} ,
\end{align*}
where we used the fact that $\e^{t \cal{H}_p}$ is a group of isometries in the last equality (Proposition~\ref{prop:isometrycalHp}).

Now fix $T_0 > 0$, $r_0 > 0$ such that $r(T_0) \le r_g$ ($r(\bullet)$ defined in~\eqref{eq:defr(t)}) and fix also $\rho_0 \in T^\star \mfd$. Recall that the space of confined symbols $\Conf_{r(t)}^{g(t)}\left(\phi^{t}(\rho_0)\right)$, $\abs{t} \le T_0$, has seminorms equivalent to those of $\sch(T^\star \mfd)$ (see Remark~\ref{rmk:equivalenceSchConf}). We do not care about the time dependence of the constants involved in the corresponding estimates.

We apply Theorem~\ref{thm:partition} to the symbols $u_0 \ovee u_0 \in \Conf_{r_0}^g(\rho_0)$ with $u_0 \in \sch(\mfd)$: for every $\ell \in \N$, there exist $k \in \N$ and $C_\ell > 0$ such that
\begin{equation*}
\forall u_0 \in \sch(\mfd), \forall t \in [-T_0, T_0] , \qquad
	\abs*{u(t) \ovee u(t)}_{\Conf_{r(t)}^{g(t)}(\phi^t(\rho_0))}^{(\ell)}
		\le C_\ell \abs*{u_0 \ovee u_0}_{\Conf_{r_0}^g(\rho_0)}^{(k)} .
\end{equation*}
Then we use the equivalence of seminorms to deduce that for every $t \in [-T_0, T_0]$, we have\footnote{Notice here that we do not claim a uniform control of the constants $C_\ell$ over time. One would need to quantify the equivalence of seminorms of $\Conf_{r(t)}^{g(t)}(\phi^t(\rho_0))$ and $\sch(T^\star \mfd)$ to handle this issue, but this is not needed for Corollary~\ref{cor:continuitypropagatorSchwartz}.}
\begin{equation*}
\forall \ell \in \N, \exists k \in \N, \exists C_\ell > 0 : \forall u_0 \in \sch(\mfd),  , \qquad
	\abs*{u(t) \ovee u(t)}_{\sch(T^\star \mfd)}^{(\ell)}
		\le C_\ell \abs*{u_0 \ovee u_0}_{\sch(T^\star \mfd)}^{(k)}
\end{equation*}
It remains to combine this with~\eqref{eq:elementary1} and~\eqref{eq:elementary2} to conclude the proof.
\qed

\Large
\section{Quantum evolution in symbol classes: proof of Theorem~\ref{thm:main}} \label{sec:consequences}
\normalsize

As explained in the introduction, in order to go from Theorem~\ref{thm:partition} to Theorem~\ref{thm:main}, we will use the following characterization of symbols in $S(m, g)$. Recall that $\vol_g$ refers to the Riemannian volume associated with the metric $g$.

\begin{proposition}[{\cite[Proposition 2.3.16]{Lerner:10}}] \label{prop:reconstructsymbol}
Let $g$ be an admissible metric and $m$ be a {$g$-admissible} weight, with a common slow variation radius $r_0 \in (0, 1]$, and $(\psi_{\rho_0})_{\rho_0 \in T^\star \mfd}$ be a {$g$-uniformly} confined family of symbols with radius $r \le r_0$. Then the symbol $a$ defined by
\begin{equation} \label{eq:reconstructa}
a(\rho)
	:= \int_{T^\star \mfd} m(\rho_0) \psi_{\rho_0}(\rho) \dd \vol_g(\rho_0) ,
		\qquad \forall \rho \in T^\star \mfd ,
\end{equation}
belongs to $S(m, g)$, with the estimate:
\begin{equation*}
\forall \ell \in \N, \exists k \in \N, \exists C_\ell > 0 : \qquad
	\abs*{a}_{S(m, g)}^{(\ell)}
		\le C_\ell \sup_{\rho_0 \in T^\star \mfd} \abs*{\psi_{\rho_0}}_{\Conf_r^g(\rho_0)}^{(k)} .
\end{equation*}
Conversely, if $a \in S(m, g)$, then the family of symbols
\begin{equation*}
\varphi_{\rho_0}
	:= \dfrac{1}{m(\rho_0)} \psi_{\rho_0} a ,
		\qquad \rho_0 \in T^\star \mfd ,
\end{equation*}
is {$g$-uniformly} confined with the estimates:
\begin{equation*}
\forall \ell \in \N, \exists k \in \N, \exists C_\ell > 0 : \qquad
	\sup_{\rho_0 \in T^\star \mfd} \abs*{\varphi_{\rho_0}}_{\Conf_r^g(\rho_0)}^{(\ell)}
		\le C_\ell \abs{a}_{S(m, g)}^{(k)} \sup_{\rho_0 \in T^\star \mfd} \abs*{\psi_{\rho_0}}_{\Conf_r^g(\rho_0)}^{(k)} .
\end{equation*}
The constants $C_\ell$ and integers $k$ depend only on structure constants of $g$ and $m$, and are uniform in $r \le r_0$.
\end{proposition}

\begin{remark}
The precise estimates are not stated explicitly in Lerner's book~\cite{Lerner:10}, though they follow directly from the proof.
\end{remark}

We start with a technical lemma.

\begin{lemma} \label{lem:absconvintegral}
Let $g$ be an admissible metric and $m$ be a {$g$-admissible} weight, with a common slow variation radius $r_0 \in (0, 1]$, and $(\psi_{\rho_0})_{\rho_0 \in T^\star \mfd}$ be a {$g$-uniformly} confined family of symbols with radius $r \le r_0$. Define the symbol $a$ by
\begin{equation*}
a(\rho)
	:= \int_{T^\star \mfd} m(\rho_0) \psi_{\rho_0}(\rho) \dd \vol_g(\rho_0) ,
		\qquad \forall \rho \in T^\star \mfd .
\end{equation*}
Then for any $u, v \in \sch(\mfd)$ and any $t \in \R$, we have
\begin{align*}
	\inp*{v}{\e^{\ii t P} \Opw{a} \e^{- \ii t P} u}_{L^2}
		&= \inp*{v}{\Opw{\e^{t \cal{H}_p} a} u}_{L^2} \\
		&= \int_{T^\star \mfd} m(\rho_0) \inp*{v}{\Opw{\e^{t \cal{H}_p} \psi_{\rho_0}} u}_{L^2} \dd \vol_g(\rho_0) ,
\end{align*}
where the integral in the right-hand side is absolutely convergent.
\end{lemma}

\begin{proof}
We first reduce to the case where $t = 0$ since the Schr\"{o}dinger propagator preserves the Schwartz class (Corollary~\ref{cor:continuitypropagatorSchwartz}).
According to~\eqref{eq:defWeylquantization} (see also Proposition~\ref{prop:quantizationisometry}), we have
\begin{equation*}
\inp*{v}{\Opw{\varphi_{\rho_0}} u}_{L^2(\mfd)}
	= \inp*{\varphi_{\rho_0}}{u \ovee v}_{L^2(T^\star \mfd)} ,
\end{equation*}
so the result is just a matter of justifying the absolute convergence of the double integral
\begin{equation*}
\int_{T^\star \mfd} \int_{T^\star \mfd} m(\rho_0) \ovl{\psi_{\rho_0}(\rho)} (u \ovee v)(\rho) \dd \rho \dd \vol_g(\rho_0)
\end{equation*}
and applying Fubini's theorem. Write $\dd \vol_g = \abs{g}^{1/2} \dd \rho$, where $\abs{g}$ corresponds to the determinant of $g$ with respect to a fixed Euclidean structure on $\mfd$. Also set $\varphi := \abs{g_\rho}^{1/2} u \ovee v$, in such a way that Since $g$ is admissible, we infer that $\abs{g}$ is a {$g$-admissible} weight, just like $m$. In particular, if we fix a point $\rho_\ast \in T^\star \mfd$ as an origin, the improved admissibility property (Proposition~\ref{prop:improvedadmissibility}) implies
\begin{align}
\abs{g_{\rho_\ast}}^{1/2}
		&\le C \jap*{\dist_{g_{\rho_\ast}^\sympf}\left( \rho, B_r^g(\rho_\ast) \right)}^{N} \abs{g_{\rho}}^{1/2} , \label{eq:1stgmod}\\
m(\rho_0)
		&\le C \jap*{\dist_{(g_{\rho_0} + g_{\rho_\ast})^\sympf}\left( B_{r}^g(\rho_0), B_{r}^g(\rho_\ast) \right)}^{N} m(\rho_\ast) \label{eq:2ndmmod}.
\end{align}
In addition, we apply the bi-confinement estimate of Proposition~\ref{prop:bi-confinement} (with $s = 0$) to obtain
\begin{align*}
\abs*{\psi_{\rho_0}(\rho) (u \ovee v)(\rho)}
	&\le C_k \dfrac{\abs*{\psi_{\rho_0}}_{\Conf_{r}^{g}(\rho_0)}^{(k')} \abs*{u \ovee v}_{\Conf_{r}^{g}(\rho_\ast)}^{(k')}}{\jap*{\dist_{(g_{\rho_0} + g_{\rho_\ast})^\sympf}(\rho, B_r^g(\rho_0)) + \dist_{(g_{\rho_0} + g_{\rho_\ast})^\sympf}(\rho, B_r^g(\rho_\ast))}^k} \\
	&\le C_k' \dfrac{\abs*{\psi_{\rho_0}}_{\Conf_{r}^{g}(\rho_0)}^{(k')} \abs*{u \ovee v}_{\Conf_{r}^{g}(\rho_\ast)}^{(k')}}{\jap*{\dist_{g_{\rho_\ast}^\sympf}(\rho, B_r^g(\rho_\ast))}^{n_0} \jap*{\dist_{(g_{\rho_0} + g_{\rho_\ast})^\sympf}(B_r^g(\rho_0), B_r^g(\rho_\ast))}^{n_0}} ,
\end{align*}
where we used the triangle inequality and Lemma~\ref{lem:integr} Item~\ref{it:minorationdowns} to bound the denominator from below. The integer $n_0$ can be arbitrarily large, provided $k$ is large enough. Putting this together with~\eqref{eq:1stgmod} and~\eqref{eq:2ndmmod}, we obtain
\begin{multline*}
\int_{T^\star \mfd} \int_{T^\star \mfd} m(\rho_0) \abs*{\ovl{\psi_{\rho_0}(\rho)} (u \ovee v)(\rho)} \dd \rho \dd \vol_g(\rho_0) \\
	\le C^2 C_k' \dfrac{m(\rho_\ast)}{\abs{g_{\rho_\ast}}^{1/2}} \int_{T^\star \mfd} \int_{T^\star \mfd} \dfrac{\abs*{\psi_{\rho_0}}_{\Conf_{r}^{g}(\rho_0)}^{(k')} \abs*{u \ovee v}_{\Conf_{r}^{g}(\rho_\ast)}^{(k')} \dd \vol_g(\rho) \dd \vol_g(\rho_0)}{\jap*{\dist_{g_{\rho_\ast}^\sympf}(\rho, B_r^g(\rho_\ast))}^{n_0'} \jap*{\dist_{(g_{\rho_0} + g_{\rho_\ast})^\sympf}(B_r^g(\rho_0), B_r^g(\rho_\ast))}^{n_0'}} ,
\end{multline*}
with an integer $n_0'$ as large as we want. The resulting integral is therefore convergent by Lemma~\ref{lem:integr} Item~\ref{it:integrabilityg} (recall that $g_{\rho_\ast}^\sympf \ge (g_{\rho_\ast} + g_{\rho_\ast})^\sympf$ since the {$\sympf$-duality} is non-increasing~\eqref{eq:sympfduality}). This concludes the proof.
\end{proof}

With this lemma at hand, we can move to the proof of Theorem~\ref{thm:main}.

\begin{proof}[Proof of Theorem~\ref{thm:main}]
\medskip
\emph{Step~0: Reduction to $T_0 = 0$.}

Suppose Theorem~\ref{thm:main} is true for $T_0 = 0$. Let $T' > 0$ be a fixed positive time and let us introduce the metric $\tilde g := \e^{-2T'} g$. As explained in Remark~\ref{rmk:positiveTE}, this new metric~$\tilde g$ is also admissible, with the same structure constants as~$g$ (except the slow variation radius, which reads $r_{\tilde g} = \e^{-T'} r_g$). Moreover, from the definition of the gain function (Definition~\ref{def:gaing}), one can check that $\gain_{\tilde g} = \e^{-2T'} \gain_g$. Assumption~\ref{assum:p} is verified for~$\tilde g$ with the same values of $\Lambda$ and $\Upsilon$, since both quantities are {$0$-homogeneous} with respect to $g$. In particular, we have
\begin{equation*}
T_E(\tilde g)
	= T_E(g) + 2 T_0 , \qquad
		T_0 := \dfrac{T'/2}{\Lambda + 2 \Upsilon} .
\end{equation*}
Scaling the metric results in a multiplication of seminorms of $p$ by powers of $\e^{T'}$:
\begin{equation*}
\abs*{\nabla^\ell p}_{\tilde g} = \e^{\ell T'} \abs*{\nabla^\ell p}_g ,
	\qquad \forall \ell \in \N ,
\end{equation*}
so that the identity maps
\begin{equation} \label{eq:inclusionstilde}
S\left(m, g(t)\right) \longrightarrow S\left(m, \tilde g(t)\right)
	\qquad {\rm and} \qquad
S\left(m, \tilde g(t)\right) \longrightarrow S\left(m, g(t)\right)
\end{equation}
are continuous, and the constants involved are independent of~$m$ and~$t$. In addition, any {$g$-admissible} weight $m$ is also {$\tilde g$-admissible}, so we know from Proposition~\ref{prop:uniformm(t)} that $m(t) = \e^{t H_p} m$ is uniformly {$\tilde g(t)$-admissible} for $\abs{t} \le \frac{1}{2} T_E(\tilde g)$.
Thus, Theorem~\ref{thm:main} with the metric $\tilde g$ in time~$T_E(\tilde g)$ gives the estimates~\eqref{eq:continuityquantumevolsymb}, \eqref{eq:estimateethpclassical}, \eqref{eq:estimatecalEj0} and~\eqref{eq:estimateremainder} in symbol classes attached to~$\tilde g$. These estimates are also valid in the corresponding symbol classes attached to~$g$ in virtue of~\eqref{eq:inclusionstilde}.

From now on we assume that $T_0 = 0$. We proceed in two steps.

\medskip
\emph{Step 1 \--- Crude estimate for $\e^{t \cal{H}_p}$.}
Let $s, \tau \in \R$ be such that $\abs{s} + \abs{\tau} \le T \le \frac{1}{2} T_E$. Notice that we can apply Theorem~\ref{thm:partition} to $g(s)$ in place of $g$ and $\tau$ in place of $t$, since the metrics $g(t)$ are uniformly admissible for $\abs{t} \le T_E$ by Proposition~\ref{prop:improvedadmissibility}, and Assumption~\ref{assum:p} on $p$ is verified for $g(t)$ as well (see Remark~\ref{rmk:homogeneity}). Let $\tilde m$ be a {$g(s)$-admissible} weight (for instance $\tilde m = m(s) = \e^{s H_p} m$), so that $\tilde m(\tau) = \e^{\tau H_p} \tilde m$ is uniformly {$g(\abs{s} + \abs{\tau})$-admissible} for $\abs{\tau} + \abs{s} \le T$ (apply Proposition~\ref{prop:uniformm(t)} to $\tilde m, g(s)$). Let $b \in S(\tilde m, g(s))$ and let us first prove that $\e^{\tau \cal{H}_p} b \in S(\tilde m(\tau), g(t))$, with $t = \abs{s} + \abs{\tau} \le T$, with seminorm estimates controlled by some negative power of $\udl{\gain}_g$.

We consider
\begin{equation*}
r_0
	:= R \e^{-(2 (\Lambda + \Upsilon) + C_g^3 C_p \udl{\gain}_g) T_E} ,
\end{equation*}
so that $r(\bigcdot)$ defined in~\eqref{eq:defr(t)} satisfies $r(T) \le R$, where $R \le r_g \le 1$ is a common slow variation radius of $g(s)$ and $\tilde m$. Notice that, in view of the definition of $T_E$ in~\eqref{eq:defTE}, the dependence of $r_0$ on $\udl{\gain}_g$ is of the form $r_0 = O(\udl{\gain}_g^{\beta})$ for some $\beta > 0$. Now, we fix a partition of unity $(\psi_{\rho_0})_{\rho_0 \in T^\star \mfd}$ adapted to $g(s)$ with radius $r_0$, given by Proposition~\ref{prop:existencepartitionofunity}. Recall that we have
\begin{equation} \label{eq:losspowerofh}
\sup_{\rho_0 \in T^\star \mfd} \abs*{\psi_{\rho_0}}_{\Conf_{r_0}^{g(s)}(\rho_0)}^{(\ell)}
	= O(r_0^{- (\ell+1) (1 + 2 \dim \mfd)})
	= O(\udl{\gain}_g^{-(\ell + 1) \beta}) .
\end{equation}

Since $(\psi_{\rho_0})_{\rho_0 \in T^\star \mfd}$ is a {$g(s)$-partition} of unity, we can write
\begin{equation*}
b
	= \int_{\rho_0 \in T^\star \mfd} \tilde m(\rho_0) \varphi_{\rho_0} \dd \vol_{g(s)}(\rho_0) ,
\end{equation*}
where $\varphi_{\rho_0} = \tilde m^{-1}(\rho_0) \psi_{\rho_0} b$ is a {$g(s)$-uniformly} confined family of symbols with radius $r_0$ by Proposition~\ref{prop:reconstructsymbol}, with estimates:
\begin{equation} \label{eq:varphirho0lossh}
\forall \ell \in \N, \exists C_\ell > 0, \exists k \in \N : \qquad
	\sup_{\rho_0 \in T^\star \mfd} \abs*{\varphi_{\rho_0}}_{\Conf_{r_0}^{g(s)}(\rho_0)}^{(\ell)}
		\le C_\ell \abs*{b}_{S(\tilde m, g)}^{(k)} \udl{\gain}_g^{- \beta (k+1)}
\end{equation}
in view of~\eqref{eq:losspowerofh}.
We apply the quantum evolution to the symbol~$b$:
\begin{equation} \label{eq:commuteintquantumevol}
\e^{\tau \cal{H}_p} b
	= \int_{T^\star \mfd} \tilde m(\rho_0) \e^{\tau \cal{H}_p} \varphi_{\rho_0} \dd \vol_{g(s)}(\rho_0)
	= \int_{T^\star \mfd} \tilde m(\rho_0) \varphi_{\phi^{-\tau}(\rho_0)}^\tau \dd \vol_{g(s)}(\rho_0) ,
\end{equation}
where $\varphi_{\rho_0}^\tau$ is defined by $\varphi_{\rho_0}^\tau := \e^{t \cal{H}_p} \varphi_{\phi^\tau(\rho_0)}$. The fact that $\e^{\tau \cal{H}_p}$ and the integral over $\rho_0$ commute follows from the fact that, given two Schwartz functions $u$ and $v$, the double integral
\begin{equation*}
\int_{T^\star \mfd} \int_{T^\star \mfd} \tilde m(\rho_0) \ovl{\e^{\tau \cal{H}_p} \varphi_{\rho_0}(\rho)} (u \ovee v)(\rho) \dd \rho \dd \vol_{g(s)}(\rho_0)
\end{equation*}
is absolutely convergent by Lemma~\ref{lem:absconvintegral} (recall the definition of the Wigner transform $u \ovee v$ in~\eqref{eq:defWigner}), so that Fubini's theorem gives indeed
\begin{equation*}
\inp*{b}{\e^{- \tau \cal{H}_p} u \ovee v}_{L^2(T^\star \mfd)}
	= \int_{T^\star \mfd} \tilde m(\rho_0) \inp*{\varphi_{\rho_0}}{\e^{-\tau \cal{H}_p} (u \ovee v)}_{L^2(T^\star \mfd)} \dd \vol_{g(s)}(\rho_0) ,
\end{equation*}
which is exactly~\eqref{eq:commuteintquantumevol} by Proposition~\ref{prop:isometrycalHp} and Proposition~\ref{prop:quantizationisometry}.
Now we make a change of variables in~\eqref{eq:commuteintquantumevol} to obtain
\begin{equation*}
\e^{\tau \cal{H}_p} b
	= \int_{T^\star \mfd} (\e^{\tau H_p} \tilde m)(\rho_0) \varphi_{\rho_0}^\tau \dd \vol_{(\phi^\tau)^\pullb g(s)}(\rho_0) ,
\end{equation*}
since $(\phi^{-\tau})_\pushf \vol_{g(s)} = \vol_{(\phi^\tau)^\pullb g(s)}$ (see for instance~\cite[Exercise 2.42]{Lee:18}). Now in view of Proposition~\ref{prop:aprioriexpgrowth}~\eqref{eq:diffflowexp}, we have $(\phi^\tau)^\pullb g(s) \le g(t)$ (recall $t = \abs{s} + \abs{\tau}$) so that the Radon--Nikodym theorem gives a measurable function $f_\tau : T^\star \mfd \to [0, 1]$ such that $\vol_{(\phi^\tau)^\pullb g(s)} = f_\tau \vol_{g(t)}$.
By Theorem~\ref{thm:partition}, the family of symbols $(f_\tau(\rho_0) \varphi_{\rho_0}^\tau)_{\rho_0 \in T^\star \mfd}$ is a {$g(t)$-uniformly} confined family of symbols with radius $r(\tau) \le r(t)$, with
\begin{align*}
\forall \ell \in \N, \exists C_\ell > 0, \exists k \in \N : \qquad
	\sup_{\rho_0 \in T^\star \mfd} \abs*{f_\tau(\rho_0) \varphi_{\rho_0}^\tau}_{\Conf_{r(t)}^{g(t)}(\rho_0)}^{(\ell)}
		\le C_\ell \sup_{\rho_0 \in T^\star \mfd} \abs*{\varphi_{\rho_0}}_{\Conf_{r_0}^{g(s)}(\rho_0)}^{(k)} .
\end{align*}
(We applied Theorem~\ref{thm:partition} in time $\tau$, which satisfies $\abs{\tau} \le T - \abs{s} \le \frac{1}{2} (T_E(g) - \abs{s}) = \frac{1}{2} T_E(g(s))$.)
Combining this with~\eqref{eq:varphirho0lossh}, Proposition~\ref{prop:reconstructsymbol} yields
\begin{equation} \label{eq:estquantumdynwithlossh}
\forall \ell \in \N, \exists C_\ell > 0, \exists k \in \N : \qquad
	\abs*{\e^{\tau \cal{H}_p} b}_{S(\tilde m(\tau), g(t))}^{(\ell)}
		\le C_\ell \udl{\gain}_g^{-\beta (k+1)} \abs*{b}_{S(\tilde m, g(s))}^{(k)} .
\end{equation}
The property of uniform admissibility of $\tilde m$ is important here: it ensures that the constants in the estimates given by Proposition~\ref{prop:reconstructsymbol} are uniform in time.

\medskip
\emph{Step 2 \--- Asymptotic expansion.}
Now we prove the continuity estimates on symbol classes for $\cal{E}_{j_0}(t)$ and for the remainder in the asymptotic expansion~\eqref{eq:asymptoticexpansion}. We will beat the loss of powers of $\udl{\gain}_g$ in~\eqref{eq:estquantumdynwithlossh} by considering the Dyson expansion at a sufficiently high order $j_0$. Notice that for $\cal{E}_0(t)$, the estimate is provided by Corollary~\ref{cor:continuityofflowinsymbolclasses}. For $\cal{E}_{j_0}(t)$ with $j_0 \ge 1$, Proposition~\ref{prop:continuityEcal} gives
\begin{equation*}
\nabla^{-1} S(m, g)
	\xrightarrow{\cal{E}_{j_0}(t)} S\left( \gain_{g(t)}^{2 j_0} m(t), g(t) \right) ,
\end{equation*}
which is the desired estimate~\eqref{eq:estimatecalEj0}.
For the remainder~\eqref{eq:estimateremainder} in the asymptotic expansion~\eqref{eq:asymptoticexpansion}, we use the above estimate for $\cal{E}_{j_0}(t)$ and recall the recurrence relation~\eqref{eq:recurrencerelation}: Lemma~\ref{lem:continuityHp3} yields
\begin{equation*}
\nabla^{-1} S(m, g)
	\xrightarrow{\cal{E}_{j_0}(s)} S\left( \gain_{g(s)}^{2 j_0} m(s), g(s) \right)
	\xrightarrow{\cal{H}_p^{(3)}} S\left( \gain_{g(s)}^{2 (j_0+1)} m(s) \e^{-(\Lambda + 2 \Upsilon) \abs{s}} \udl{\gain}_g, g(s) \right) .
\end{equation*}
Now we apply~\eqref{eq:estquantumdynwithlossh} with $b = \cal{H}_p^{(3)} \cal{E}_{j_0}(s) a$ for some $a \in S(m, g)$, $M = \gain_{g(s)}^{2(j_0+1)} m(s)$ and $\tau = t - s$. Using Lemma~\ref{lem:gain(t)}, we obtain
\begin{multline*}
\forall \ell \in \N, \exists C_\ell > 0, \exists k \in \N : \qquad \\
	\abs*{\e^{(t - s) \cal{H}_p} \cal{H}_p^{(3)} \cal{E}_{j_0}(s) a}_{S(\gain_{g(t)}^{2 (j_0+1)} m(t), g(t))}^{(\ell)}
		\le C_\ell \udl{\gain}_g^{-\beta (k+1)} \udl{\gain}_g \e^{- (\Lambda + 2 \Upsilon) \abs{s}} \abs*{a}_{S(m, g)}^{(k)} .
\end{multline*}
Now we sacrifice half of the gain $\gain_{g(t)}^{2(j_0+1)}$, using that for $\abs{t} \le \frac{1}{2} T_E$, we have $\gain_{g(t)} \le \udl{\gain}_g^{1/2}$. This yields:
\begin{equation*}
\abs*{\e^{(t - s) \cal{H}_p} \cal{H}_p^{(3)} \cal{E}_{j_0}(s) a}_{S(\gain_{g(t)}^{j_0+1} m(t), g(t))}^{(\ell)}
	\le C_\ell \udl{\gain}_g^{-\beta (k+1) + \frac{1}{2} (j_0+1)} \udl{\gain}_g \e^{- (\Lambda + 2 \Upsilon) \abs{s}} \abs*{a}_{S(m, g)}^{(k)} .
\end{equation*}
So for fixed $\ell$, choosing $j_0$ large enough gives
\begin{equation*}
\abs*{\int_{t \Delta_1} \e^{(t - s) \cal{H}_p} \cal{H}_p^{(3)} \cal{E}_{j_0}(s) a \dd s}_{S(\gain_{g(t)}^{j_0+1} m(t), g(t))}^{(\ell)}
	\le C_\ell' \int_{t \Delta_1} \udl{\gain}_g \e^{- (\Lambda + 2 \Upsilon) \abs{s}} \abs*{a}_{S(m, g)}^{(k)} \dd s
	\le \dfrac{C_\ell'}{c} \abs*{a}_{S(m, g)}^{(k)}
\end{equation*}
(recall~\eqref{eq:assumLambda} in the last step). This finishes the proof of the asymptotic expansion. The estimate~\eqref{eq:estimateremainder} for the remainder $\widehat{\cal{E}}_{j_0+1}(t)$, for any $j_0$ and with the correct power $h_{g(t)}^{2 (j_0+1)}$, follows by using the asymptotic expansion at a sufficiently high order $j_1 \ge j_0$.
With the estimates~\eqref{eq:estimatecalEj0} for $\cal{E}_j(t)$ at hand, we also deduce that
\begin{equation*}
\forall \ell \in \N, C_\ell > 0, \exists k \in \N : \qquad
	\abs*{\e^{t \cal{H}_p} a}_{S(m(t), g(t))}^{(\ell)}
		\le C_\ell \abs*{a}_{S(m, g)}^{(k)} .
\end{equation*}
This concludes the proof. 
\end{proof}

\Large
\section{Examples of application: proofs of statements in Section~\ref{subsec:examples}} \label{sec:proofsex}
\normalsize

The splitting~\eqref{eq:splitting} can be understood as follows: introducing
\begin{gather*}
\begin{array}{rcl}
\pi : T^\star \mfd &\longrightarrow& \mfd \\
(x, \xi) &\longmapsto& x
\end{array}
	\qquad \rm{and} \qquad
\begin{array}{rcl}
\iota : \mfd &\longrightarrow& T^\star \mfd \\
x &\longmapsto& (x, 0)
\end{array}
\end{gather*}
the cotangent bundle projection and the null section, we know that $\pi \circ \iota = \id_\mfd$, and
\begin{equation*}
\pi_0 := \iota \circ \pi : T^\star \mfd \longrightarrow T^\star \mfd
\end{equation*}
is the projection onto the null section of $T^\star \mfd$. The kernel and range of the projection have same dimension and we have
\begin{equation*}
T(T^\star \mfd)
	= \ran \dd \pi_0 \oplus \ker \dd \pi_0 .
\end{equation*}
The first component can be naturally identified with $T \mfd$ through the map $\dd \iota$. As for the second component, we identify the fibers $T_x^\star \mfd$ with their tangent space, which turns out to project onto the zero vector under $\dd \pi_0$.

\subsection{Schrödinger operator: proofs of statements in Section~\ref{subsubsec:Schhbar}} \label{subsec:proofsSchhbar}

The goal of this section is to prove Proposition~\ref{prop:Schhbar} concerning semiclassical Schrödinger operators of the form~\eqref{eq:Schop}.

\begin{proof}[Proof of Proposition~\ref{prop:Schhbar}]
Let us compute the Weyl symbol of the operator $P$ in~\eqref{eq:Schop}. The symbol of $- \frac{\hslash^2}{2} \Delta + V$ is $\tfrac{\hslash^2}{2} \abs{\xi}_{\upgamma^{-1}} + V(x)$. As for the rest of the operator, we can take Euclidean coordinates in which $\beta \cdot \partial = \sum_{j=1}^d \beta^j(x) \partial_j$ with $d = \dim \mfd$, and find
\begin{align*}
\tfrac{1}{\ii} \beta \cdot \partial
	&= \sum_{j=1}^d \Opw{\beta^j} \Opw{\xi_j}
	= \sum_{j=1}^d \Opw{\beta^j \xi_j + \tfrac{1}{2 \ii} \poiss*{\beta_j}{\xi_j}}
	= \Opw{\xi. \beta - \tfrac{1}{2 \ii} \sum_{j=1}^d \partial_j \beta^j} \\
	&= \Opw{\xi. \beta - \tfrac{1}{2 \ii} \dvg \beta} .
\end{align*}
This results from the pseudo-differential calculus (Proposition~\ref{prop:pseudocalcsymb}), which is exact at order $1$ (see Remark~\ref{rmk:polynomialpseudocalc}). We deduce that
\begin{equation*}
P = \Opw{p}
	\qquad \rm{with} \qquad
p(x, \xi) = \tfrac{\hslash^2}{2} \abs*{\xi}_{\upgamma^{-1}}^2 - \hslash \xi. \beta + \tfrac{1}{2} \abs{\beta}_{\upgamma}^2 + V(x) .
\end{equation*}
This symbol is semi-bounded, so that Proposition~\ref{prop:classicalwellposedness} applies and ensures that Assumption~\ref{assum:mandatory} is satisfied.

Next we compute the Hessian of $p$:
\begin{equation} \label{eq:hessp}
\nabla^2 p =
	\begin{pmatrix}
	{\upgamma}(\nabla \beta, \nabla \beta) + \nabla^2 V  & - \hslash (\nabla \beta)^\star \\ - \hslash \nabla \beta & \hslash^2 {\upgamma}^{-1}
	\end{pmatrix} .
\end{equation}
In this expression, $\nabla \beta$ is seen as a section of the bundle $\Lop(T \mfd) \to \mfd$ which maps vector fields on~$\mfd$ to vector fields on~$\mfd$. The map $(\nabla \beta)^\star$ refers to the dual map acting on covectors, or differential {$1$-forms}.
We crucially used the fact that $\nabla^2 \beta = 0$ here (this ensures terms of the form $\xi. \nabla^2 \beta$ disappear in the second derivative with respect to $x$).
We observe that under the assumption~\eqref{eq:assumV2} and the fact that $\beta$ is affine (in particular $\abs{\nabla \beta}$ is bounded), we deduce that
\begin{equation*}
\sup_{\hslash \in (0, 1]} \sup_{T^\star \mfd} \abs*{\nabla^2 p}_{g_\hslash} < \infty .
\end{equation*}
Since the metric $\upgamma$ is flat, we clearly have $\nabla_{H_p} g = 0$. It also follows that $g$ is slowly varying and temperate, and its gain function is classically $\gain_{g_\hslash} = \hslash \le 1$, hence admissibility of $g_\hslash$, uniformly with respect to $\hslash$. Regarding the temperance weight $\theta_{g_\hslash}$, we should compute it using a well-chosen\footnote{Dependence of the background Euclidean metric on $\hslash$ is not an issue since it does not affect structure constants of $\theta_{g_\hslash}$ as a {$g_\hslash$-admissible} weight.} background Euclidean metric $\sf{g} = \sf{g}_\hslash$ (for instance the symplectic metric $\hslash^{-1} \dd x^2 + \hslash \dd \xi^2$) in order to find $\theta_{g_\hslash} = \udl{\gain}_{g_\hslash}^{-1/2} = \hslash^{-1/2}$.

Then we estimate higher order derivatives of $p$.
We observe that all mixed derivatives in $x$ and $\xi$ of order $\ge 3$ vanish in view of~\eqref{eq:hessp}. We deduce that we have
\begin{equation*}
\abs*{\nabla^\ell p}_g
	= \abs*{\nabla^\ell V}_\upgamma ,
		\qquad \forall \ell \ge 3 ,
\end{equation*}
that is to say
\begin{equation*}
\nabla^3 p \in S\left(1, g\right) .
\end{equation*}
As a consequence, both sub-quadraticity and strong sub-quadraticity assumptions (see Assumptions~\ref{assum:p}) are verified.
\end{proof}

\subsection{Half-wave operator: proofs of statements in Section~\ref{subsubsec:halfwaves}} \label{subsec:proofshalfwaves}

Here we place ourselves in the setting of Section~\ref{subsubsec:halfwaves} and prove in particular Proposition~\ref{prop:halfwaves}.

\begin{remark} \label{rem:comparisonwithsfg}
The first assumption in~\eqref{eq:assumgamma} is equivalent to saying that there exists a constant $c > 0$ such that
\begin{equation} \label{eq:comparegammasfg}
c^{-2} I \le \gamma \le c^2 I ,
\end{equation}
where $I$ is the standard Euclidean metric on~$\R^d$.
Therefore this is a boundedness/ellipticity assumption on the metric $\gamma$. This can also be seen as a slow variation property with infinite radius. The second requirement in~\eqref{eq:assumgamma} can then be understood as saying that the metric is $\cont_b^\infty$.
\end{remark}

In the following lemma, we first notice that if $\gamma$ satisfies~\eqref{eq:assumgamma}, then so does $\gamma^{-1}$ (which is defined over $\mfd$ similarly to $g^{-1}$ over $T^\star \mfd$ in~\eqref{eq:defg-1}).

\begin{lemma} \label{lem:assumgamma-1}
Under assumption~\eqref{eq:assumgamma}, we have
\begin{equation} \label{eq:assumgamma-1}
\forall x_1, x_2 \in \mfd , \quad
	\gamma_{x_1}^{-1} \le C_\gamma^2 \gamma_{x_2}^{-1}
		\qquad \rm{and} \qquad
\forall k \in \N , \quad
	\sup_\mfd \left(\abs*{\nabla^k \gamma}_{I} + \abs*{\nabla^k \gamma^{-1}}_{I}\right) < \infty .
\end{equation}
\end{lemma}

\begin{proof}
The first assertion in~\eqref{eq:assumgamma-1} follows directly from the corresponding property in~\eqref{eq:assumgamma} on $\gamma$ and the fact that
\begin{equation*}
\abs*{\xi}_{\gamma^{-1}}
	= \sup_{x \in V \setminus \{0\}} \dfrac{\abs{\brak{\xi}{v}_{V^\star, V}}}{\abs{v}_\gamma} .
\end{equation*}
Boundedness of derivatives of $\gamma$ with respect to $I$ follows from Remark~\ref{rem:comparisonwithsfg}.

We prove the boundedness of derivatives of $\gamma^{-1}$ by induction on $k$.
For $k = 0$, we know that $\abs{\gamma^{-1}}_{I}$ is bounded on the whole $\mfd$ thanks to the first assertion in~\eqref{eq:assumgamma-1}. Now assume $\abs{\nabla^\ell \gamma^{-1}}_{I}$ is uniformly bounded on $\mfd$ for all $\ell \le k$, for some $k \ge 0$. By the Leibniz formula, we have
\begin{equation*}
0
	= \dfrac{1}{(k+1)!} \nabla^{k+1} (\gamma \gamma^{-1})
	= \sum_{\ell = 0}^{k+1} \dfrac{1}{(k+1-\ell) !} \nabla^{k+1-\ell} \gamma \dfrac{1}{\ell !} \nabla^\ell \gamma^{-1} .
\end{equation*}
Extracting the term for $\ell = k+1$, we obtain
\begin{equation*}
\dfrac{1}{(k+1)!} \nabla^{k+1} \gamma^{-1}
	= - \gamma^{-1} \sum_{\ell=0}^k \dfrac{1}{(k+1-\ell) !} \nabla^{k+1-\ell} \gamma \dfrac{1}{\ell !} \nabla^\ell \gamma^{-1} .
\end{equation*}
Each term in the right-hand side is bounded uniformly on $\mfd$ using the induction hypothesis since for all $\ell \le k$:
\begin{equation*}
\abs*{\gamma^{-1} \nabla^{k+1-\ell} \gamma \nabla^\ell \gamma^{-1}}_{I}
	\le \abs*{\gamma^{-1}}_{I} \abs*{\nabla^{k+1-\ell} \gamma}_{I} \abs*{\nabla^\ell \gamma^{-1}}_{I} .
\end{equation*}
We deduce that $\abs{\nabla^{k+1} \gamma^{-1}}_{I}$ is bounded, which concludes the induction.
\end{proof}

We compute the full Weyl symbol of the operator $- \widetilde{\Delta}_\gamma$ defined in~\eqref{eq:deftildeDeltagamma}. In the following lemma, we denote by $\gamma^{ij}$ the components of $\gamma^{-1}$ in a system of Euclidean coordinates, and by $\abs{\gamma}$ the determinant of $\gamma$ with respect to this Euclidean structure. Repeated indices are summed according to Einstein summation convention.

\begin{lemma} \label{lem:symbolDeltatilde}
Under assumption~\eqref{eq:assumgamma}, we have
\begin{equation} \label{eq:symbolDeltatilde}
- \widetilde{\Delta}_\gamma
	= \Opw{\abs*{\xi}_{\gamma^{-1}}^2 + \dfrac{1}{4} \partial_{ij}^2 \gamma^{ij} + \dfrac{1}{4} \gamma^{ij} w_i w_j + \dfrac{1}{2} \partial_i (w_j \gamma^{ij})} ,
\end{equation}
where $w_j = \partial_j \log \sqrt{\abs{\gamma}}$.
\end{lemma}

One can compute $w_j$ as derivatives of the metric with the Jacobi formula:
\begin{equation*}
w_j
	= \dfrac{1}{2} \tr\left(\gamma^{-1} \partial_j \gamma\right) .
\end{equation*}
We deduce that the sub-principal terms in~\eqref{eq:symbolDeltatilde} are controlled by first and second order derivatives of $\gamma$. Therefore, under the assumption~\eqref{eq:assumgamma}, Lemma~\ref{lem:assumgamma-1} implies that the symbol of $-\widetilde{\Delta}_\gamma$ is $\abs{\xi}_{\gamma^{-1}}^2$ modulo a $\cont_b^\infty$ symbol.

\begin{proof}[Proof of Lemma~\ref{lem:symbolDeltatilde}]
We factor $- \widetilde{\Delta}_\gamma$ into
\begin{equation} \label{eq:Deltatilde}
- \widetilde{\Delta}_\gamma
	= A_j^\ast \gamma^{ij} A_i ,
\end{equation}
where
\begin{equation*}
A_j
	= \abs*{\gamma}^{1/4} \tfrac{1}{\ii} \partial_j \abs*{\gamma}^{-1/4}
	= \tfrac{1}{\ii} \partial_j + \tfrac{1}{\ii} \abs*{\gamma}^{1/4} \left(\partial_j \abs*{\gamma}^{-1/4}\right)
	= \tfrac{1}{\ii} \partial_j + \tfrac{1}{\ii} \left(\partial_j \log \abs*{\gamma}^{-1/4}\right) .
\end{equation*}
With the notation $w_j$ of the statement, we have
\begin{equation*}
A_j
	= \tfrac{1}{\ii} \partial_j - \tfrac{1}{2 \ii} w_j
	= \Opw{\xi_j - \tfrac{1}{2 \ii} w_j} .
\end{equation*}
Plugging this into~\eqref{eq:Deltatilde} and using pseudo-differential calculus leads to
\begin{align*}
- \widetilde{\Delta}_\gamma
	&= \Opw{\xi_i + \tfrac{1}{2 \ii} w_i} \Opw{\gamma^{ij}} \Opw{\xi_j - \tfrac{1}{2 \ii} w_j} \\
	&= \Opw{\xi_i + \tfrac{1}{2 \ii} w_i} \Opw{\gamma^{ij} \left(\xi_j - \tfrac{1}{2 \ii} w_j\right) + \tfrac{1}{2 \ii} \poiss{\gamma^{ij}}{\xi_j}} \\
	&= \Opw{\left( \xi_i + \tfrac{1}{2 \ii} w_i \right) \left( \gamma^{ij} \xi_j - \tfrac{1}{2 \ii} \gamma^{ij} w_j - \tfrac{1}{2 \ii} \partial_j \gamma^{ij} \right) + \tfrac{1}{2 \ii} \poiss*{\xi_i + \tfrac{1}{2 \ii} w_i }{\gamma^{ij} \xi_j - \tfrac{1}{2 \ii} \gamma^{ij} w_j - \tfrac{1}{2 \ii} \partial_j \gamma^{ij}}} \\
	&= \Opw{\gamma^{ij} \xi_i \xi_j + \dfrac{1}{4} \gamma^{ij} w_i w_j + \dfrac{1}{4} w_i \partial _j \gamma^{ij} + \dfrac{1}{4} \partial_i (w_j \gamma^{ij}) + \dfrac{1}{4} \partial_{ij}^2 \gamma^{ij} + \dfrac{1}{4} \gamma^{ij} \partial_j w_i} ,
\end{align*}
hence the sought result by grouping terms according to the Leibniz rule and suitably permuting indices.
\end{proof}

We now discuss admissibility of $g$.

\begin{lemma} \label{lem:gainthetahalfwaves}
The metric $g$ defined in~\eqref{eq:ghalfwave} is admissible for $\alpha \in [0, 1]$. Its gain function is
\begin{equation*}
\gain_g(x, \xi)
	= \dfrac{1}{\jap*{\xi}_{\gamma_x^{-1}}^\alpha}
	= \dfrac{1}{p^\alpha} ,
		\qquad (x, \xi) \in T^\star \mfd .
\end{equation*}
For all $\alpha \in [0, 1]$, there exists $C > 0$ such that
\begin{equation} \label{eq:controlthetaghalfwaves}
\forall (x, \xi) \in T^\star \mfd , \qquad
	C^{-1} \jap*{\xi}_{\gamma_x^{-1}}^{\frac{1 + \alpha}{2}}
		\le \theta_g(x, \xi)
			\le C \jap*{\xi}_{\gamma_x^{-1}}^{\frac{1 + \alpha}{2}} .
\end{equation}
\end{lemma}

\begin{proof}
We introduce the metric
\begin{equation*}
\tilde g_{(x, \xi)}
	= \jap*{\xi}_{I}^{1-\alpha} \dd x^2 + \jap*{\xi}_{I}^{-(1+\alpha)} \dd \xi^2 ,
\end{equation*}
where $\dd x^2 + \dd \xi^2$ is the standard Euclidean metric on $T^\star \mfd \simeq \R^{2d}$.
From~\eqref{eq:assumgamma} and~\eqref{eq:assumgamma-1}, together with {$\sympf$-duality}~\eqref{eq:sympfduality} as well, we have for some sufficiently large constant $c > 0$:
\begin{equation} \label{eq:ggtilde}
c^{-2} \tilde g
	\le g
	\le c^2 \tilde g
		\qquad \rm{and} \qquad
c^{-2} \tilde g^\sympf
	\le g^\sympf
	\le c^2 \tilde g^\sympf .
\end{equation}
From~\cite[Lemma 2.2.18]{Lerner:10}, the metric $\tilde g$ is admissible, since it can be written as
\begin{equation*}
\tilde g
	= \jap*{\xi}^{2a} \dd x^2 + \jap*{\xi}^{-2b} \dd \xi^2 ,
\end{equation*}
where $a = \frac{1-\alpha}{2}$ and $b = \frac{1+\alpha}{2}$ satisfy
\begin{equation*}
0 \le a \le b \le 1
	\qquad \rm{and} \qquad
a < 1 .
\end{equation*}
Admissibility of $g$ then follows from~\eqref{eq:ggtilde}.
Finally, passing to matrix representation, we have
\begin{align*}
g^\sympf
	&= \sympf^\star \left(\jap*{\xi}_{\gamma^{-1}}^{-2 a} \gamma^{-1} \oplus \jap*{\xi}_{\gamma^{-1}}^{2 b} \gamma\right) \sympf
	= \begin{pmatrix}
	0 & -1 \\ 1 & 0
	\end{pmatrix}
	\begin{pmatrix}
	\jap*{\xi}_{\gamma^{-1}}^{-2 a} \gamma^{-1} & 0 \\ 0 &  \jap*{\xi}_{\gamma^{-1}}^{2 b} \gamma
	\end{pmatrix}
	\begin{pmatrix}
	0 & 1 \\ -1 & 0
	\end{pmatrix} \\
	&=
	\begin{pmatrix}
	 \jap*{\xi}_{\gamma^{-1}}^{2 b} \gamma & 0 \\ 0 & \jap*{\xi}_{\gamma^{-1}}^{-2 a} \gamma^{-1}
	\end{pmatrix} ,
\end{align*}
from which we obtain $\gain_g(x, \xi) = \jap{\xi}_{\gamma^{-1}}^{a - b} = \jap{\xi}_{\gamma^{-1}}^{-\alpha}$.

Now for the temperance weight, we can see that $c^{-1} \theta_{\tilde g} \le \theta_g \le c \theta_{\tilde g}$ as a consequence of~\eqref{eq:ggtilde}, and we can compute $\theta_{\tilde g}$ with respect to the background metric $\dd x^2 + \dd \xi^2$, and we find that $\theta_{\tilde g}(x, \xi) = \jap{\xi}_{I}^{\frac{1 + \alpha}{2}}$. Using~\eqref{eq:comparegammasfg}, we obtain~\eqref{eq:controlthetaghalfwaves}.
\end{proof}

We estimate successive derivatives of $p$ thanks to the Faà di Bruno formula (Appendix~\ref{app:faadibruno}).

\begin{lemma} \label{lem:derphalfwaves}
Assume $\gamma$ satisfies~\eqref{eq:assumgamma} and consider $p$ and $g$ introduced in Proposition~\ref{prop:halfwaves}. Then the following holds:
\begin{equation*}
\abs*{\nabla^k p}_g
	\le C_{\gamma, k} p^{1 - \frac{1-\alpha}{2} k}
\end{equation*}
\end{lemma}

\begin{proof}
According to Remark~\ref{rmk:polarintro}, it suffices to pick a vector field $X$ on $T^\star \mfd$ and establish the bound
\begin{equation*}
\abs*{\nabla^k p. X^k}
	\le C_{\gamma, k} p^{1 - (\frac{1 - \alpha}{2}) k} \abs*{X}_g^k .
\end{equation*}
We view $p$ as the composition of $f : s \mapsto \sqrt{1 + s}$ with $q : (x, \xi) \mapsto \gamma^{-1}(\xi, \xi)$. The the Faà di Bruno formula (see Appendix~\ref{app:faadibruno}) yields
\begin{equation} \label{eq:appfdb}
\dfrac{1}{k!} \nabla^k p. X^k
	= \sum_{j = 1}^k \sum_{\substack{{\bf{n}} \in (\N^\ast)^j \\ \abs{{\bf{n}}} = k}} \dfrac{1}{j!} (f^{(j)} \circ q). \left( \dfrac{1}{{\bf{n}}!} \nabla^{\bf{n}} q. X^{\bf{n}} \right) .
\end{equation}
On the one hand, we notice that $\abs{f^{(j)}(s)} \le C_j (1 + s)^{1/2 - j}$, so that
\begin{equation} \label{eq:derivativesfq}
\abs*{f^{(j)} \circ q}
	\le C_j (p^2)^{1/2 - j}
	= C_j p^{1 - 2 j} .
\end{equation}
On the other hand, we shall write the components of $X$ according to the splitting~\eqref{eq:splitting}:
\begin{equation*}
X
	= X_\hor + X_\ver
	\in T \mfd \oplus T^\star \mfd .
\end{equation*}
Expanding $\nabla_{X^k}^k$ according to this decomposition, we are lead to estimate
\begin{equation} \label{eq:derqsfg}
\nabla^k q. \left(X_\hor^{k_1}, X_\ver^{k_2}\right)
	= \nabla_{X_\ver^{k_2}}^{k_2} (\nabla_{X_\hor^{k_1}}^{k_1} \gamma^{-1})(\xi, \xi) ,
		\qquad \rm{with} \;\, k_1 + k_2 = k .
\end{equation}
Indeed, derivatives with $X_\hor$ or $X_\ver$ affect only the $x$ variable or the $\xi$ variable respectively (that is to say they correspond to $\partial_x$ and $\partial_\xi$). Since~$q$ is polynomial of order $2$ in $\xi$, the derivative vanishes as soon as $k_2 \ge 3$, while for $k_2 = 1$ and $k_2 = 2$, we have whenever $k_1 \in \N$:
\begin{equation*}
\nabla^k q. \left(X_\hor^{k_1}, X_\ver\right)
	= 2 (\nabla_{X_\hor^{k_1}}^{k_1} \gamma^{-1})(X_\ver, \xi) ,
		\qquad
\nabla^k q. \left(X_\hor^{k_1}, X_\ver, X_\ver\right)
	= 2 (\nabla_{X_\hor^{k_1}}^{k_1} \gamma^{-1})(X_\ver, X_\ver) .
\end{equation*}
In particular, by~\eqref{eq:assumgamma} and the Cauchy--Schwarz inequality, we have the estimate
\begin{equation*}
\abs*{\nabla^k q. \left(X_\hor^{k_1}, X_\ver^{k_2}\right)}
	\le C_\gamma p^{2 - k_2} \abs*{X_\hor}_{\gamma}^{k_1} \abs*{X_\ver}_{\gamma^{-1}}^{k_2} ,
\end{equation*}
still valid for $k_2 = 0$.
This is an estimate with respect to the metric $\gamma \oplus \gamma^{-1}$. In terms of the metric $g$, we have
\begin{equation*}
\abs*{\nabla^k q. \left(X_\hor^{k_1}, X_\ver^{k_2}\right)}
	\le C_\gamma p^{2 - k_2 - \frac{1-\alpha}{2} k_1 + \frac{1+\alpha}{2} k_2} \abs*{X_\hor}_g^{k_1} \abs*{X_\ver}_g^{k_2}
	\le C_\gamma p^{2 - \frac{1-\alpha}{2} k} \abs*{X}_g^k .
\end{equation*}
Combining this with~\eqref{eq:derivativesfq}, it implies that all the terms arising in the Faà di Bruno formula~\eqref{eq:appfdb} are of order
\begin{equation*}
p^{1 - 2 j} \times p^{2j - \frac{1 - \alpha}{2} k}
	= p^{1 - \frac{1 - \alpha}{2} k} ,
\end{equation*}
hence the sought result.
\end{proof}

We are now in a position to prove the main result of Section~\ref{subsubsec:halfwaves}.

\begin{proof}[Proof of Proposition~\ref{prop:halfwaves}]
The metric $g$ is admissible by Lemma~\ref{lem:gainthetahalfwaves}. Moreover, in view of the assumption~\eqref{eq:assumgamma} on $\gamma$ (or in fact~\eqref{eq:comparegammasfg}), we clearly have $\gamma_{\pi \circ \phi^t(\rho)} \le C_\Upsilon^2 \e^{2 \Upsilon \abs{t}} \gamma_{\pi(\rho)}$ for any $\rho \in T^\star \mfd$ with $\Upsilon = 0$ and $C_\Upsilon = C_\gamma$. Since $p$ is invariant by the Hamiltonian flow, we deduce that the metric $g$ satisfies Item~\ref{it:metriccontrol} of Assumption~\ref{assum:p} with $\Upsilon = 0$.

Let us check the boundedness of the Lie derivative $\lieder_{H_p} g$ in~\eqref{eq:defLambda} (Assumption~\ref{assum:p}). To do that we use Remark~\ref{rmk:subquad}.
On the one hand, we can estimate $\abs{\nabla_{H_p} g}_g$ by remarking that by definition of $g$, the coefficients of the metric depend only on $p$ which is invariant by the flow, so that the derivative $\nabla_{H_p}$ affects only $\gamma_x$ and $\gamma_x^{-1}$. The latter depend only on $x$. The $x$ component of $H_p$ is
\begin{equation*}
\dfrac{\gamma^{-1} \xi}{\jap{\xi}_{\gamma^{-1}}} ,
\end{equation*}
whose $\gamma$ norm remains bounded on the whole phase space, so indeed $\abs*{\nabla_{H_p} g}_g$ is uniformly bounded.

Now we estimate the Hessian and the higher-order derivatives of $p$. Lemma~\ref{lem:derphalfwaves} gives $\nabla^2 p \in S(m, g)$ where
\begin{equation*}
m
	= p^{1 - 2 \frac{1 - \alpha}{2}}
	= p^\alpha
	= \gain_g^{-1} ,
\end{equation*}
so sub-quadraticity holds (namely the first condition in~\eqref{eq:assumsubquad}). So Item~\ref{it:Lambda} of Assumption~\ref{assum:p} is proved. Now for derivatives of order $k \ge 3$, Lemma~\ref{lem:derphalfwaves} gives $\nabla^3 p \in S(m, g)$ with
\begin{equation*}
m
	= p^{1 - 3 \frac{1 - \alpha}{2}}
	= p^{3 \alpha} p^{- \frac{1}{2} - \alpha (3 - \frac{3}{2})}
	\le \gain_g^{-3} p^{-1/2} ,
\end{equation*}
hence strong sub-quadraticity, that is to say Item~\ref{it:strongsubquad} of Assumption~\ref{assum:p}, holds too (recall that the temperance weight is a power of $p$ by Lemma~\ref{lem:gainthetahalfwaves}). This proves that Assumption~\ref{assum:p} is fulfilled. Lastly the symbol $p$ is semi-bounded, which ensures that Assumption~\ref{assum:mandatory} is fulfilled too by Proposition~\ref{prop:classicalwellposedness} (and Remark~\ref{rmk:inpracticecheckassum}), and this finishes the proof.
\end{proof}

We end this section by explaining the link between the operator $P$ defined in~\eqref{eq:defPhalfwaves} and $\sqrt{\widetilde{\Delta}_\gamma}$.

\begin{lemma} \label{lem:sqarerootclosetoP}
The operator $- \widetilde \Delta_\gamma - P^2$ extends to a bounded operator on $L^2(\mfd)$.
\end{lemma}

\begin{proof}
We use pseudo-differential calculus with $p$ and itself, using $\poiss*{p}{p} = 0$, to obtain
\begin{equation*}
P^2
	= \Opw{p^2 + \widehat{\cal{P}}_2(p, p)} .
\end{equation*}
Lemma~\ref{lem:derphalfwaves} together with $\gain_g = p^{-\alpha}$ (Lemma~\ref{lem:gainthetahalfwaves}) imply that $p \in \nabla^{-2} S(\gain_g^{-1}, g)$, so that
\begin{equation*}
\widehat{\cal{P}}_2(p, p)
	\in S(1, g) .
\end{equation*}
Setting
\begin{equation*}
a = - 1 + \widehat{\cal{P}}_2(p, p) + \dfrac{1}{4} \partial_{ij}^2 \gamma^{ij} + \dfrac{1}{4} \gamma^{ij} w_i w_j + \dfrac{1}{2} \partial_i (w_j \gamma^{ij}) ,
\end{equation*}
we have $- \widetilde \Delta_\gamma = P^2 + \Opw{a}$ in view of Lemma~\ref{lem:symbolDeltatilde}, and under the assumptions~\eqref{eq:assumgamma}, we have indeed
\begin{equation*}
a \in S(1, g) .
\end{equation*}
(Sub-principal terms involving derivatives of $\gamma$ belong to $S(1, g)$ since they depend only on the $x$ variable.)
The result follows from the Calder\'{o}n--Vaillancourt theorem (Proposition~\ref{prop:CV}).
\end{proof}

\subsection{Vector fields: proofs of statements in Section~\ref{subsubsec:X}} \label{subsec:proofsX}

We start with the computation of the logarithmic derivative in~\eqref{eq:conjX}.

\begin{lemma} \label{lem:logderivativemu}
Let $X$ be a non-vanishing vector field preserving a smooth density $\mu = \abs{\mu} \dd x$, where $\dd x$ is (a fixed normalization of) the Lebesgue measure on $\mfd$. Then we have $X \log \abs{\mu} = - \dvg X$, where\footnote{The $\dvg$ operator coincides with the divergence $\dvg_{\dd x}$ with respect to the Lebesgue measure $\dd x$. One can check that it does not depend on the normalization of the latter.} $\dvg = \tr \nabla$.
\end{lemma}

\begin{proof}
For any test function $\cont_\comp^\infty(\mfd)$, we have by integration by parts
\begin{equation*}
\int_\mfd \varphi \dvg X \dd x
	= - \int_\mfd \nabla_X \varphi \dd x
	= - \int_\mfd (X \varphi) \abs*{\mu}^{-1} \dd \mu
	= \int_\mfd \varphi (X \abs*{\mu}^{-1}) \dd \mu
	= - \int_\mfd \varphi (X \log \abs*{\mu}) \dd x .
\end{equation*}
The two crucial steps consist in an integration by parts with respect to the measure $\dd x$ in the first equality, and using the anti-symmetry of $X$ with respect to the measure $\dd \mu$ in the third equality.
\end{proof}

Now we proceed with the computation of the Weyl symbol of $P$ defined in~\eqref{eq:defPX}. Writing $X = \sum_{j=1}^d X^j \partial_j$ in global Euclidean coordinates, we have
\begin{equation*}
\Opw{\tfrac{1}{\ii} X}
	= \sum_{j=1}^d \Opw{X^j} \Opw{\xi_j}
	= \sum_{j=1}^d \Opw{X_j \xi_j - \dfrac{1}{2 \ii} \partial_j X^j}
	= \Opw{\xi. X - \dfrac{1}{2 \ii} \dvg X} ,
\end{equation*}
according to pseudo-differential calculus. Hence the Weyl symbol of $P$ in~\eqref{eq:defPX} is given by $P = \Opw{\xi.X}$.
As a preparation for the proofs below, let us compute the Hamiltonian vector field and the Hessian of $p$:
\begin{equation} \label{eq:Hpvectorfield}
H_p
	= \begin{pmatrix}
	X \\ - (\nabla X)^\star \xi
	\end{pmatrix} ,
		\qquad\qquad
\nabla^2 p
	= \begin{pmatrix}
	\xi. \nabla^2 X & \nabla X \\ \nabla X & 0
	\end{pmatrix} .
\end{equation}
Here $(\nabla X)^\star \in \Lop(T^\star \mfd)$ refers to the map dual to $\nabla X$, viewed as a linear map acting on $T \mfd$. We observe that derivatives of any order with respect to the $x$ variable can possibly blow up at fiber infinity ($\xi \to \infty$), in a situation similar to the previous case of waves with a non-flat metric, while derivatives of order larger than $2$ with respect to $\xi$ vanish.

We discuss admissibility of the metric~\eqref{eq:gvectorfield}.

\begin{lemma} \label{lem:gainthetaX}
The metric $g$ defined in~\eqref{eq:gvectorfield} subject to~\eqref{eq:conditionalphanu} is admissible. Its gain function and temperance rate are given by
\begin{equation*}
\gain_g(x, \xi)
	= \dfrac{1}{\jap*{\xi}_{I}^{\alpha_2 - \alpha_1}}
		\qquad \rm{and} \qquad
\exists C > 0 : \quad
	C^{-1} \jap*{\xi}_{I}^{\alpha_2}
		\le \theta_g(x, \xi)
		\le C \jap*{\xi}_{I}^{\alpha_2} .
\end{equation*}
\end{lemma}

\begin{proof}
Admissibility of $g$ comes from \cite[Lemma 2.2.18]{Lerner:10} under conditions~\eqref{eq:conditionalphanu}.
The expressions of the gain function and of the temperance weight follow from the definitions.
\end{proof}

Essential self-adjointness of $P$ has to be verified separately, since the condition~\eqref{eq:assumeessentialselfadjointness} is not true, except in very degenerate cases (such as a constant vector field $X$), so that Proposition~\ref{prop:classicalwellposedness} does not apply.

\begin{lemma} \label{lem:selfadjX}
Under the assumption~\eqref{eq:assumX} on $X$, the operator $P$ acting on $\cont_\comp^\infty(\mfd)$ is essentially self-adjoint. The Hamiltonian flow of $p$ is globally well-defined.
\end{lemma}

\begin{proof}
We first show that the Hamiltonian flow is well-defined for all times. In view of the expression~\eqref{eq:Hpvectorfield} of the Hamiltonian vector field and~\eqref{eq:assumX}, we see that $H_p$ is a locally Lipschitz vector field. Then the Picard--Lindelöf theorem ensures that the Cauchy problem
\begin{equation*}
\dfrac{\dd}{\dd t} \rho(t)
	= H_p\left(\rho(t)\right) ,
		\qquad
\rho(0) = \rho_0 \in T^\star \mfd
\end{equation*}
admits a solution on some time interval $(-t_0, t_0)$, $t_0 > 0$. Then for any $t$ in this interval, the expression~\eqref{eq:Hpvectorfield} of $H_p$ and the Cauchy--Schwarz inequality yield
\begin{align*}
\dfrac{\dd}{\dd t} \abs*{\rho(t) - \rho_0}_{I}^2
	&\le 2  \abs*{H_p\left(\rho(t)\right)}_{I} \abs*{\rho(t) - \rho_0}_{I}
	\le 2 \left(\abs*{X}_{I, \infty} + \abs*{\nabla X}_{I, \infty} \abs*{\rho(t)}_{I}\right) \abs*{\rho(t) - \rho_0}_{I} \\
	&\le 4 \left(\abs*{X}_{I, \infty} + \abs*{\nabla X}_{I, \infty}\right) \jap*{\rho_0}_{I} \jap*{\rho(t) - \rho_0}_{I}^2 .
\end{align*}
Therefore, Grönwall's lemma implies that $\jap*{\rho(t)} \le C \e^{C \abs{t}}$ for some constant $C$, which implies that $t_0 = + \infty$.

Denote by $(\phi_X^t)_{t \in \R}$ the flow generated by $X$, that is to say the restriction of $H_p$ to the null section of $T^\star \mfd$.
It remains to show that $P$ is essentially self-adjoint. To check this, we let $u \in \ker(P^\ast \pm \ii)$, and prove that $u = 0$. Then for any $v_0 \in \cont_\comp^\infty(\mfd)$, set
\begin{equation*}
v(t, x)
	:= \exp\left(\int_0^t \frac{1}{2} (\dvg X)(\phi_X^s(x)) \dd s\right) v_0\left(\phi_X^t(x)\right) ,
		\qquad x \in \mfd, t \in \R .
\end{equation*}
One checks that $v(t ,\bigcdot) \in \cont_\comp^\infty(\mfd)$ for all times and that it solves
\begin{equation*}
\partial_t v(t, x)
	= \ii P v(t, x) .
\end{equation*}
Then we proceed as follows: by dominated convergence,
\begin{equation*}
\dfrac{\dd}{\dd t} \inp*{u}{v(t)}_{L^2}
	= \ii \inp*{u}{P v(t)}_{L^2}
	= \ii \inp*{P^\ast u}{v(t)}_{L^2}
	= \mp \inp*{u}{v(t)}_{L^2} .
\end{equation*}
The second equality relies on the fact that $v(t)$ remains in $\cont_\comp^\infty(\mfd)$ while the last one comes from the choice of $u$. We deduce that
\begin{equation*}
\inp*{u}{v(t)}_{L^2}
	= \e^{\mp t} \inp*{u}{v_0}_{L^2} ,
		\qquad \forall t \in \R.
\end{equation*}
Since $\norm{v(t)}_{L^2(\mfd)} = \norm{v_0}_{L^2(\mfd)}$ for all times, letting $t \to \pm \infty$ leads to $\inp*{u}{v_0}_{L^2} = 0$, for all $v_0 \in \cont_\comp^\infty(\mfd)$, hence $u = 0$.
\end{proof}

Next we estimate the derivatives of $p$.

\begin{lemma} \label{lem:derpX}
Suppose $g$ defined in~\eqref{eq:gvectorfield} is subject to~\eqref{eq:conditionalphanu}. Then we have
\begin{equation*}
\forall k \in \N , \qquad
	\abs*{\nabla^k p}_g
		\le C_k \jap*{\xi}_{I}^{\alpha_2 - \alpha_1 - (k-2) \alpha_1} .
\end{equation*}
\end{lemma}

\begin{proof}
Similarly to~\eqref{eq:derqsfg}, we have for any vector field $Y = Y_\hor + Y_\ver$:
\begin{align} \label{eq:checkthepower}
\abs*{\nabla^k p. \left(Y_\hor^{k_1}, Y_\ver^{k_2}\right)}
	&\le C \jap{\xi}_I^{1 - k_2} \abs*{Y_\hor}_I^{k_1} \abs*{Y_\ver}_I^{k_2}
	= C \jap{\xi}_I^{1 - k_2 - \alpha_1 k_1 + \alpha_2 k_2} \abs*{Y_\hor}_g^{k_1} \abs*{Y_\ver}_g^{k_2} \nonumber\\
	&= C \jap{\xi}_I^{\alpha_2 - \alpha_1 - \alpha_1 (k_1 - 1) - (1 - \alpha_2) (k_2 - 1)} \abs*{Y}_g^k .
\end{align}
Here $k = k_1 + k_2$, and the constant $C$ depends on norms of derivatives of $X$ in~\eqref{eq:assumX}. Now, we need to check only the cases $k_2 = 0$ and~$1$, since $p$ is polynomial of order $1$ in $\xi$, namely $\nabla^k p = 0$ as soon as $k_2 \ge 2$. For $k_2 = 0$, the power in the right-hand side of~\eqref{eq:checkthepower} reads
\begin{equation*}
\alpha_2 - \alpha_1 - \alpha_1 (k_1 - 1) - (1 - \alpha_2) (k_2 - 1)
	= \alpha_2 - \alpha_1 - \alpha_1 (k - 1) + (1 - \alpha_2)
	\le \alpha_2 - \alpha_1 - \alpha_1 (k - 2) ,
\end{equation*}
using the fact that $1 - \alpha_2 \le \alpha_1$ (see~\eqref{eq:conditionalphanu}). For $k_2 = 1$, the power in the right-hand side of~\eqref{eq:checkthepower} reads
\begin{equation*}
\alpha_2 - \alpha_1 - \alpha_1 (k_1 - 1) - (1 - \alpha_2) (k_2 - 1)
	=  \alpha_2 - \alpha_1 - \alpha_1 (k - 2) ,
\end{equation*}
which concludes the proof of the lemma.
\end{proof}

We finally prove the main result of Section~\ref{subsubsec:X}.

\begin{proof}[Proof of Proposition~\ref{prop:exX}]
Assumption~\ref{assum:mandatory} is satisfied by Lemma~\ref{lem:selfadjX}, and the metric $g$ is admissible by Lemma~\ref{lem:gainthetaX}.
We check sub-quadraticity: from Lemma~\ref{lem:derpX}, we have $\nabla^2 p \in S(m, g)$ with
\begin{equation} \label{eq:showsubqu}
m
	= \jap*{\xi}_{I}^{\alpha_2 - \alpha_1}
	= \gain_g^{-1}
\end{equation}
according to Lemma~\ref{lem:gainthetaX}. As for strong sub-quadraticity, we have $\nabla^3 p \in S(m, g)$ with
\begin{equation*}
m
	= \jap*{\xi}_{I}^{\alpha_2 - \alpha_1 - \alpha_1}
	= \gain_g^{-3} \jap*{\xi}_{I}^{- 2 (\alpha_2 - \alpha_1) - \alpha_1} .
\end{equation*}
Recalling that the temperance weight is a power of $\jap*{\xi}_{I}$ (Lemma~\ref{lem:gainthetaX}), this shows that Item~\ref{it:strongsubquad} of Assumption~\ref{assum:p} is satisfied. Finally, recalling~\eqref{eq:Hpvectorfield}, we have
\begin{equation*}
H_p \jap*{\xi}_{I}
	= O\left(\jap*{\xi}_{I}\right) .
\end{equation*}
Indeed, the action of $H_p$ with respect to the position variable is the differential operator $X$, which does not affect $\jap*{\xi}_{I}$, since the latter quantity does not depend on $x$. As for the momentum variable, we have
\begin{equation*}
\abs*{\left[ (\nabla X)^\star \xi \right] \jap*{\xi}_{I}}
	\lesssim \jap*{\xi}_{I}^{-1} \jap*{I\left(\xi, (\nabla X)^\star \xi\right)}
	\lesssim \jap*{\xi}_{I} .
\end{equation*}
Then one deduces that $\abs{\nabla_{H_p} g}_g \le C$ for some constant $C$ depending on the derivatives of $X$ in~\eqref{eq:assumX}, thereby proving Item~\ref{it:metriccontrol} of Assumption~\ref{assum:p} by a Gr{\"o}nwall argument. It shows also that Item~\ref{it:Lambda} is true in view of~\eqref{eq:showsubqu} and Remark~\ref{rmk:subquad}, and the proof is complete.
\end{proof}

\appendix

\Large
\section{Refined estimates for the Weyl--Hörmander calculus} \label{app:Moyal}
\normalsize

Here we collect proofs of the statements of Section~\ref{sec:WHcalculus}. We start with some definitions.

\subsection{Symplectic Fourier transform}

Recall the vector spaces $V$ and $W = V \oplus V^\star$, that are identified to $T_x \mfd$ and $T_\rho (T^\star \mfd)$ respectively for any $x \in \mfd$ and $\rho \in T^\star \mfd$ through the affine structure of $\mfd$ in~\eqref{eq:affinestructureM} and of $T^\star \mfd$ in~\eqref{eq:affinestructurephasespace}.
Recall the symplectic form $\sympf = \dd \xi \wedge \dd x$ on ~$W$ (the same as on $T^\star \mfd$). We introduce $\ft_\sympf$ to be the symplectic Fourier transform acting on functions on the symplectic vector space~$W$:
\begin{equation*}
\ft_\sympf f(\zeta)
	= \int_W f(\zeta') \e^{- \ii \sympf(\zeta, \zeta')} \dd \zeta' ,
\end{equation*}
where the Lebesgue measure $\dd \zeta$ is normalized\footnote{Usually, one has $\dd \zeta = (2 \pi)^{- \dim \mfd} \dd v \dd \eta$, where $\dd v$ and $\dd \eta$ are normalized with respect to a fixed Euclidean structure on $V$.} in such a way that $\ft_\sympf^2 = \id$, that is to say $\ft_\sympf^{-1} = \ft_\sympf$. The symplectic Fourier transform is continuous on $\sch(W)$ and can be extended as a continuous operator on $\sch'(W)$.

When working on the sum of two copies of $W$, we shall use the symplectic Fourier transform $\ft_{\sympf \oplus \sympf} = \ft_\sympf \otimes \ft_\sympf$, acting on Schwartz functions (or tempered distributions) on $W \oplus W$.

Once equipped with these Fourier transforms, we can build Fourier multipliers: given a tempered distribution $a \in \sch(T^\star \mfd)$, we define the operator $a(D) : \sch(T^\star \mfd) \to \sch'(T^\star \mfd)$ as
\begin{equation*}
a(D) f(\rho)
	:= \int_W a(\zeta) \left( \int_W f(\rho_0 + \zeta') \e^{- \ii \sympf(\zeta, \zeta')} \dd \zeta' \right) \e^{- \ii \sympf(\rho - \rho_0, \zeta)} \dd \zeta ,
\end{equation*}
where $\rho_0 \in T^\star \mfd$ is a fixed origin. The operator $a(D)$ does not depend on the choice of $\rho_0$, essentially because the symplectic form is invariant by translation. Therefore we shall write
\begin{equation*}
a(D)
	= \ft_\sympf^{-1} a \ft_\sympf .
\end{equation*}
Moreover $a(D)$ commutes with $\nabla$ since $\nabla \sympf = 0$. The same applies for Fourier multipliers on $T^\star \mfd \oplus T^\star \mfd = T^\star \mfd^{\oplus 2}$ using $\ft_{\sympf \oplus \sympf}$.

\subsection{Poisson operator} \label{subsec:Poissonop}

Let us define the following Fourier multiplier
\begin{equation*}
\frak{P} = - \ft_{\sympf \oplus \sympf}^{-1} \sympf \ft_{\sympf \oplus \sympf} ,
\end{equation*}
acting on $\sch(T^\star \mfd^{\oplus 2})$, where $\sympf$ is understood as the multiplication by the function $(\zeta_1, \zeta_2) \mapsto \sympf(\zeta_1, \zeta_2)$. More explicitly, this operator acts on functions $f \in \sch(W \oplus W)$ as
\begin{equation} \label{eq:Poissonoperatorintegralformula}
\left(\frak{P} f\right)(\zeta_1, \zeta_2)
	= - \int_{W \oplus W} \sympf(\zeta_1', \zeta_2') \left(\ft_{\sympf \oplus \sympf} f\right)(\zeta_1', \zeta_2') \e^{- \ii \sympf(\zeta_1, \zeta_1') - \ii \sympf(\zeta_2, \zeta_2')} \dd \zeta_1' \dd \zeta_2' .
\end{equation}
One can check that it is continuous on $\sch(T^\star \mfd^{\oplus 2})$, so that it extends to tempered distributions. To sum up, we have that
\begin{equation*}
\frak{P} : \sch(T^\star \mfd^{\oplus 2}) \longrightarrow \sch(T^\star \mfd^{\oplus 2})
	\qquad \rm{and} \qquad
\frak{P} : \sch'(T^\star \mfd^{\oplus 2}) \longrightarrow \sch'(T^\star \mfd^{\oplus 2})
\end{equation*}
are continuous.
Since $\sympf$ is alternating, the map $\sf{s} : (\rho_1, \rho_2) \mapsto (\rho_2, \rho_1)$ on $T^\star \mfd^{\oplus 2}$ (or rather the composition by this map) acts as follows on $\frak{P}$:
\begin{equation} \label{eq:symmetryfrakP}
\frak{P} \sf{s} = - \sf{s} \frak{P} .
\end{equation}
In addition, one can check that the complex conjugation commutes with $\frak{P}$.
Formally, one can think of this operator as
\begin{equation*}
\frak{P}
	= \left\{\nabla^{\rho_1}, \nabla^{\rho_2}\right\}
	= \sympf\left(H_{\bigcdot}^{\rho_1}, H_{\bigcdot}^{\rho_2}\right) .
\end{equation*}
In the right-hand side, given $a = a(\rho_1, \rho_2)$ a smooth function on $T^\star \mfd^{\oplus 2}$, we understand $H_a^{\rho_1}$ as the Hamiltonian vector field associated with $\rho_1 \mapsto a(\rho_1, \rho_2)$, with the $\rho_2$ variable frozen, and similarly for $H_a^{\rho_2}$ with the $\rho_1$ variable frozen (we refer to~\eqref{eq:defHp} for the definition of the Hamiltonian vector field). We provide with an alternative expression of $\frak{P}$.

\begin{lemma} \label{lem:computationPoissonop}
For any $f \in \sch(T^\star \mfd^{\oplus 2})$, the following holds:
\begin{equation*}
\frak{P} f(\rho_1, \rho_2)
	= \dvg_{\rho_2} \left( H_{f(\rho_1, \rho_2)}^{\rho_1} \right)
	= - \dvg_{\rho_1} \left( H_{f(\rho_1, \rho_2)}^{\rho_2} \right) ,
		\qquad \rho_1, \rho_2 \in T^\star \mfd .
\end{equation*}
\end{lemma}

\begin{proof}
We deal with functions $f \in \sch(W^{\oplus 2})$. Going to functions on $T^\star \mfd^{\oplus 2}$ follows by choosing an arbitrary origin $\rho_0$. By integration by parts, we have
\begin{equation*}
\ft_\sympf  \left(\tfrac{1}{\ii} \nabla_{\zeta_0} f\right)(\zeta)
	= \sympf(\zeta, \zeta_0) \ft_\sympf f(\zeta) ,
		\qquad \forall f \in \sch(W), \forall \zeta, \zeta_0 \in W
\end{equation*}
(recall the definition of $\nabla_{\zeta_0}$ in Section~\ref{subsub:tensorsandnorms}),
as well as for any compactly supported vector field $X$ on $W$:
\begin{equation*}
\ft_\sympf \left(\dvg X\right)(\zeta)
	= - \int_W \nabla_X^{\zeta'} \e^{- \ii \sympf(\zeta, \zeta')} \dd \zeta'
	= \int_W \ii \sympf\left(\zeta, X(\zeta')\right) \e^{- \ii \sympf(\zeta, \zeta')} \dd \zeta'
	= \ii \ft_\sympf\left( \sympf(\zeta, X(\bigcdot) \right)(\zeta) .
\end{equation*}
From these two facts, together with the definition of the Hamiltonian vector field~\eqref{eq:defHp}, we deduce the following:
\begin{align*}
\ft_{\sympf \oplus \sympf} \frak{P} f(\zeta_1, \zeta_2)
	&= - (\ft_\sympf)_{\zeta_2' \to \zeta_2} \left(\sympf(\zeta_1, \zeta_2) (\ft_\sympf)_{\zeta_1' \to \zeta_1} a(\zeta_1', \zeta_2')\right) \\
	&= - (\ft_\sympf)_{\zeta_2' \to \zeta_2} \left((\ft_\sympf)_{\zeta_1' \to \zeta_1}  \tfrac{1}{\ii} \nabla_{\zeta_2}^{\zeta_1'} a(\zeta_1', \zeta_2')\right) \\
	&= - (\ft_\sympf)_{\zeta_1' \to \zeta_1} (\ft_\sympf)_{\zeta_2' \to \zeta_2} \tfrac{1}{\ii} \sympf\left(\zeta_2, H_{a(\zeta_1', \zeta_2')}^{\zeta_1'}\right) \\
	&= (\ft_\sympf)_{\zeta_1' \to \zeta_1} (\ft_\sympf)_{\zeta_2' \to \zeta_2} \dvg_{\zeta_2'} H_{a(\zeta_1', \zeta_2')}^{\zeta_1'} ,
\end{align*}
which gives the desired result by taking the inverse Fourier transform. To have the expression in the other way around, we simply use~\eqref{eq:symmetryfrakP}. This finishes the proof of the lemma.
\end{proof}

Next we give bounds on the operator $\frak{P}$ when derivatives are measures with respect to a certain metric on $T^\star \mfd^{\oplus 2}$.

\begin{lemma} \label{lem:continuityfrakP}
Let $g_1$ and $g_2$ be two Riemmanian metrics on $T^\star \mfd$. Then the following holds:
\begin{equation*}
\forall j \in \N, \forall \ell \in \N , \qquad
	\abs*{\nabla^\ell \frak{P}^j f}_{g_1 \oplus g_2}
		\le (2 \dim \mfd)^j \gain_{g_1 \oplus g_2}^j \abs*{\nabla^{j + \ell} f}_{g_1 \oplus g_2} ,
\end{equation*}
where $\gain_{g_1 \oplus g_2}(\rho_1, \rho_2)$ is defined by\footnote{Notice that the restriction of $\gain_{g_1 \oplus g_2}$ to the diagonal is the joint gain function $h_{g_1, g_2}$ defined in~\eqref{eq:defjointgain}. See Lemma~\ref{lem:justificationequalities} for a proof of the right-hand side equality in~\eqref{eq:gaing1oplusg2}.}
\begin{equation} \label{eq:gaing1oplusg2}
\gain_{g_1 \oplus g_2}(\rho_1, \rho_2)
	= \sup_{\zeta \in W \setminus \{0\}} \dfrac{\abs*{\zeta}_{(g_1)_{\rho_1}}}{\abs*{\zeta}_{(g_2^\sympf)_{\rho_2}}}
	= \sup_{\zeta \in W \setminus \{0\}} \dfrac{\abs*{\zeta}_{(g_2)_{\rho_2}}}{\abs*{\zeta}_{(g_1^\sympf)_{\rho_1}}} ,
		\qquad \rho_1, \rho_2 \in T^\star \mfd .
\end{equation}
\end{lemma}

\begin{proof}
We recall that $\dvg X = \tr \nabla X$ for any vector field $X$ on $T^\star \mfd$, and that for any norm $\norm{\bullet}$ on $T_\rho (T^\star \mfd)$, we have $\abs{\tr \nabla X} \le \dim(T^\star \mfd) \norm*{\nabla X}$. We fix two points $\rho_1, \rho_2 \in T^\star \mfd$ and write $g_j = (g_j)_{\rho_j}$, $j = 1, 2$, for simplicity. From Lemma~\ref{lem:computationPoissonop}, we deduce that
\begin{align*}
\abs*{\frak{P} f(\rho_1, \rho_2)}
	&\le \dim(T^\star \mfd) \abs*{\nabla^{\rho_2} H_f^{\rho_1}}_{g_1^\sympf}
	= \dim(T^\star \mfd) \sup_{\zeta_2 \in W \setminus \{0\}} \dfrac{\abs{\nabla_{\zeta_2}^{\rho_2} H_f^{\rho_1}}_{g_1^\sympf}}{\abs{\zeta_2}_{g_1^\sympf}} \\
	&= \dim(T^\star \mfd) \sup_{\zeta_2 \in W \setminus \{0\}} \dfrac{\abs{\zeta_2}_{g_2}}{\abs{\zeta_2}_{g_1^\sympf}} \dfrac{\abs*{H_{\nabla_{\zeta_2}^{\rho_2} f}^{\rho_1}}_{g_1^\sympf}}{\abs{\zeta_2}_{g_2}} \\
	&= \dim(T^\star \mfd) \sup_{\zeta_2 \in W \setminus \{0\}} \sup_{\zeta_ 1\in W \setminus \{0\}} \dfrac{\abs{\zeta_2}_{g_2}}{\abs{\zeta_2}_{g_1^\sympf}} \dfrac{\abs{\nabla_{\zeta_1}^{\rho_1} \nabla_{\zeta_2}^{\rho_2} f}}{\abs{\zeta_1}_{g_1} \abs*{\zeta_2}_{g_2}} \\
	&\le \dim(T^\star \mfd) \gain_g(\rho_1, \rho_2) \abs*{\nabla f(\rho_1, \rho_2)}_{g_1 \oplus g_2} .
\end{align*}
The third equality comes from the fact that
\begin{equation*}
\abs*{H_a}_{g^\sympf}
	= \sup_{\zeta \in W \setminus \{0\}} \dfrac{\abs{\sympf(H_a, g^{-1} \sympf \zeta)}}{\abs{\zeta}_{g^\sympf}}
	= \sup_{\zeta \in W \setminus \{0\}} \dfrac{\abs{\sympf(H_a, \zeta)}}{\abs{\zeta}_g}
	= \sup_{\zeta \in W \setminus \{0\}} \dfrac{\abs{\nabla_\zeta a}}{\abs{\zeta}_g}
	= \abs*{\nabla a}_g .
\end{equation*}
To handle derivatives, we remark that $\nabla$ commutes with $\frak{P}$ since the latter is a Fourier multiplier, hence the result for $j = 1$. Iterating this result, we obtain the statement for any $j \ge 1$ (the result for $j = 0$ is direct).
\end{proof}

\subsection{Moyal operator} \label{subsec:Moyalop}

Let us also introduce for any $s \in \R$ the operator defined by
\begin{equation*}
\e^{- \ii \frac{s}{2} \frak{P}}
	= \ft_{\sympf \oplus \sympf}^{-1} \e^{\ii \frac{s}{2} \sympf} \ft_{\sympf \oplus \sympf} ,
\end{equation*}
where $\e^{\ii \frac{s}{2} \sympf}$ is understood as the multiplication by the function $(\zeta_1, \zeta_2) \mapsto \e^{\ii \frac{s}{2} \sympf(\zeta_1, \zeta_2)}$. Notice that a more explicit formula as~\eqref{eq:Poissonoperatorintegralformula} is available, with $\e^{\ii \frac{s}{2} \sympf}$ in place of $\sympf$.
Similarly to $\frak{P}$, one can show that the operators
\begin{equation*}
\e^{- \ii \frac{s}{2} \frak{P}} : \sch(T^\star \mfd^{\oplus 2}) \longrightarrow \sch(T^\star \mfd^{\oplus 2})
	\qquad \rm{and} \qquad
\e^{- \ii \frac{s}{2} \frak{P}} : \sch'(T^\star \mfd^{\oplus 2}) \longrightarrow \sch'(T^\star \mfd^{\oplus 2})
\end{equation*}
are continuous. One can see from the definition that
\begin{equation} \label{eq:Moyalcomplexconjugation}
\e^{- \ii \frac{s}{2} \frak{P}} f
	= \ovl{\e^{- \ii \frac{s}{2} \frak{P}} (\bar f \circ \sf{s})} \circ \sf{s} ,
\end{equation}
where $\sf{s}$ is the map $(\rho_1, \rho_2) \mapsto (\rho_2, \rho_1)$. Notice that this is consistent with~\eqref{eq:symmetryfrakP}.

\subsection{Bi-confinement estimate}

The operator $\e^{- \frac{\ii}{2} \frak{P}}$ is related to the Moyal product by just taking the restriction to the diagonal; see formula~\eqref{eq:exprMoyal}. We quickly reproduce the so-called bi-confinement estimate. Proofs are mainly based on~\cite{Lerner:10}; see also~\cite{Hoermander:V3}.

Let us start with a technical lemma that will be useful in several places in the sequel. It is partly contained in~\cite[Lemma 2.2.24]{Lerner:10}.

\begin{lemma} \label{lem:integr}
Let $g$ be an admissible metric with structure constants $C_g$, $r_g$ and $N_g$ (see Proposition~\ref{prop:improvedadmissibility}). Then the following holds.
\begin{enumerate}
\item\label{it:integrabilityg} There exists an integer $n$, depending only on structure constants of $g$, such that
\begin{equation*}
\sup_{\rho \in T^\star \mfd} \int_{T^\star \mfd} \dfrac{\dd \vol_g(\rho_0)}{\jap{\dist_{(g_\rho + g_{\rho_0})^\sympf}\left( B_{r_g}^g(\rho), B_{r_g}^g(\rho_0) \right)}^n} < \infty .
\end{equation*}
\item\label{it:minorationdowns} For any $\rho, \rho_1, \rho_2 \in T^\star \mfd$ and $r_1, r_2 \in (0, r_g]$, we have for $j \in \{1, 2\}$:
\begin{equation*}
\jap*{\dist_{g_{\rho_j}^\sympf}\left( \rho, B_{r_j}^g(\rho_j) \right)}
	\le \sqrt{2} C_g \jap*{\dist_{(g_{\rho_1} + g_{\rho_2})^\sympf}\left( \rho, B_{r_1}^g(\rho_1) \right) + \dist_{(g_{\rho_1} + g_{\rho_2})^\sympf}\left( \rho, B_{r_2}^g(\rho_2) \right)}^{1+N_g} .
\end{equation*}
\end{enumerate}
\end{lemma}

\begin{proof}
The first item of the lemma comes from~\cite[Lemma 2.2.24 (2.2.29)]{Lerner:10}. Let us prove the second one. For simplicity, let us write $g_j$ and $B_j$ in place of $g_{\rho_j}$ and $B_{r_j}^g(\rho_j)$ respectively, $j = 1, 2$. Let $\{j, j'\} = \{1, 2\}$. Then from Proposition~\ref{prop:improvedadmissibility} combined with~\eqref{eq:chainineq}, we deduce that
\begin{equation*}
g_{j'}
	\le C_g^2 g_j \jap*{\dist_{(g_1 + g_2)^\sympf}\left( B_1, B_2 \right)}^{2N_g} ,
\end{equation*}
which, by adding $g_j$ on both sides and passing to the {$\sympf$-dual} metrics~\eqref{eq:sympfduality}, leads to
\begin{equation*}
g_j^\sympf
	\le 2 C_g^2 (g_1 + g_2)^\sympf \jap*{\dist_{(g_1 + g_2)^\sympf}\left( B_1, B_2 \right)}^{2N_g} .
\end{equation*}
We infer that
\begin{equation*}
\dist_{g_j^\sympf}\left( \rho, B_j \right)
	\le \sqrt{2} C_g \dist_{(g_1 + g_2)^\sympf}\left( \rho, B_j \right) \jap*{\dist_{(g_1 + g_2)^\sympf}\left( B_1, B_2 \right)}^{N_g} .
\end{equation*}
The last factor in the right-hand side can be bounded from above thanks to the triangle inequality:
\begin{equation*}
\dist_{(g_1 + g_2)^\sympf}\left( B_1, B_2 \right)
	\le \dist_{(g_{\rho_1} + g_{\rho_2})^\sympf}\left( \rho, B_1 \right) + \dist_{(g_{\rho_1} + g_{\rho_2})^\sympf}\left( \rho, B_2 \right) ,
\end{equation*}
and the sought inequality follows.
\end{proof}

The proposition below is the cornerstone of the Weyl--Hörmander pseudo-differential calculus. A classical computation involving Fourier transforms of imaginary Gaussian functions gives for all $f \in \sch(T^\star \mfd^{\oplus 2})$:
\begin{equation*}
\forall s \in \R, \forall \rho_1, \rho_2 \in T^\star \mfd , \quad
	\e^{- \ii \frac{s}{2} \frak{P}} f(\rho_1, \rho_2)
		= \int_{W \oplus W} f\left(\rho_1 + \sqrt{\tfrac{s}{2}} \zeta_1, \rho_2 + \sqrt{\tfrac{s}{2}} \zeta_2\right) \e^{- \ii \sympf(\zeta_1, \zeta_2)} \dd \zeta_1 \dd \zeta_2 .
\end{equation*}

\begin{proposition}[Bi-confinement estimate] \label{prop:bi-confinement}
Let $g$ be an admissible metric on $T^\star \mfd$. There exists an integer $k_0$ depending only on the dimension of $\mfd$ and the structure constants of $g$ such that the following holds. For any pair of points $\rho_1, \rho_2 \in T^\star \mfd$, introducing for $j = 1, 2$ the balls $B_j = B_{r_j}^{g_j}(\rho_j)$, for any $r_j \in (0, r_g]$, we have for all $f \in \sch(T^\star \mfd^{\oplus 2})$ and all $\rho \in T^\star \mfd$:
\begin{equation*}
\forall k, \ell \in \N , \exists C_{k, \ell} > 0 : \qquad
	\abs*{\nabla^\ell \left(\e^{- \ii \frac{s}{2} \frak{P}} f(\rho, \rho)\right)}_g
		\le C_{k, \ell} \dfrac{\cal{N}_{k+\ell}(f; \rho_1, \rho_2)}{\jap*{\dist_{(g_{\rho_1} + g_{\rho_2})^\sympf}(\rho, B_1) + \dist_{(g_{\rho_1} + g_{\rho_2})^\sympf}(\rho, B_2)}^k} ,
\end{equation*}
where
\begin{multline*}
\cal{N}_j(f; \rho_1, \rho_2)
	= \max_{0 \le l \le (1 + j) k_0} \sup_{\rho_1', \rho_2' \in T^\star \mfd} \jap*{\dist_{g_{\rho_1}^\sympf}(\rho_1', B_1)}^{(1 + j) k_0} \\
		\times \jap*{\dist_{g_{\rho_2}^\sympf}(\rho_2', B_2)}^{(1 + j) k_0} \abs*{\nabla^l f(\rho_1', \rho_2')}_{g \oplus g} .
\end{multline*}
The constants $C_{k, \ell}$ depend only on $k, \ell$ and the structure constants of $g$ (but not on~$r_1, r_2$).
\end{proposition}

\begin{remark} \label{rmk:calNConf}
Notice that when $f = \psi_1 \otimes \psi_2$, we have by definition of the seminorms on spaces of confined symbols (Definition~\ref{def:Conf}):
\begin{equation*}
\cal{N}_k(f; \rho_1, \rho_2)
	\le \abs*{\psi_1}_{\Conf_{r_1}^{g}(\rho_1)}^{((1 + k) k_0)} \abs*{\psi_2}_{\Conf_{r_2}^{g}(\rho_2)}^{((1 + k) k_0)} .
\end{equation*}
\end{remark}

\begin{proof}
The proof follows essentially that of~\cite[Theorem 2.3.2]{Lerner:10}; see also~\cite[Section 18.5]{Hoermander:V3}. Consider the case $\ell = 0$ first. For simplicity, write ${\sf g}_j = g_{\rho_j}$, $j = 1, 2$. Given $\zeta_1 \in W$, we set
\begin{equation*}
Z
	= Z(\zeta_1)
	= \dfrac{\sympf^{-1} {\sf g}_2^\sympf \zeta_1}{\abs{\zeta_1}_{{\sf g}_2^\sympf}}
	= \dfrac{{\sf g}_2^{-1} \sympf \zeta_1}{\abs{\zeta_1}_{{\sf g}_2^\sympf}} ,
\end{equation*}
and we perform integration by parts in the integral over $\zeta_2$ using the fact that
\begin{equation*}
\left(1 + \tfrac{1}{\ii} \sqrt{\tfrac{2}{s}} \nabla_Z^{\zeta_2}\right) \e^{- \ii \sympf(\zeta_1, \zeta_2)}
	= \left(1 - \sqrt{\tfrac{2}{s}} \sympf(\zeta_1, Z)\right) \e^{- \ii \sympf(\zeta_1, \zeta_2)}
	= \left(1 + \sqrt{\tfrac{2}{s}} \abs*{\zeta_1}_{{\sf g}_2^\sympf}\right) \e^{- \ii \sympf(\zeta_1, \zeta_2)} .
\end{equation*}
(The minus sign appearing in the second equality is due to the fact that $g^\sympf = - \sympf g^{-1} \sympf$; recall~\eqref{eq:sympfstar} and Definition~\ref{def:sympfdual}.)
We obtain
\begin{equation} \label{eq:ibpMoyal}
\e^{- \ii \frac{s}{2} \frak{P}} f(\rho, \rho)
	= \int_{W \oplus W} \left( 1 - \tfrac{1}{\ii} \sqrt{\tfrac{2}{s}} \nabla_Z^{\zeta_2} \right)^{l_0} f\left(\rho + \sqrt{\tfrac{s}{2}} \zeta_1, \rho + \sqrt{\tfrac{s}{2}} \zeta_2\right) \times \dfrac{\e^{- \ii \sympf(\zeta_1, \zeta_2)}}{(1 + \sqrt{\tfrac{2}{s}} \abs{\zeta_1}_{{\sf g}_2^\sympf})^{l_0}} \dd \zeta_1 \dd \zeta_2 .
\end{equation}
Write for $k, l_0, n \in \N$:
\begin{equation} \label{eq:defuprightN}
N_{l_0, k, n}(f; \rho_1, \rho_2)
	= N_{l_0, k, n}(f)
	= \max_{0 \le l \le l_0} \sup_{\rho_1', \rho_2' \in T^\star \mfd} \jap*{\dist_{{\sf g}_1^\sympf}\left(\rho_1', B_1\right)}^{k} \jap*{\dist_{{\sf g}_2^\sympf}\left(\rho_2', B_2\right)}^{n} \abs*{\nabla^l f(\rho_1', \rho_2')}_{g \oplus g} .
\end{equation}
Since $g_2$ is admissible (and $r_2 \le r_{g}$), we have
\begin{equation} \label{eq:normofZg2}
\abs*{Z}_{g_2}(\rho_2')
	\le \abs*{Z}_{g_2}(\rho_2) C(g_2) \jap*{\dist_{{\sf g}_2^\sympf}\left(\rho_2', B_2\right)}^{N(g_2)}
	= C(g_2) \jap*{\dist_{{\sf g}_2^\sympf}\left(\rho_2', B_2\right)}^{N(g_2)} ,
\end{equation}
and we obtain
\begin{equation} \label{eq:preestimateN}
\begin{multlined}
\abs*{\e^{- \ii \frac{s}{2} \frak{P}} f(\rho, \rho)}
	\le C(g_2) 2^{l_0} N_{l_0, k, n}(f) \int_{W \oplus W} \jap*{\dist_{{\sf g}_1^\sympf}\left( \rho + \sqrt{\tfrac{s}{2}} \zeta_1, B_1\right)}^{-k} \\
		\times \jap*{\dist_{{\sf g}_2^\sympf}\left(\rho + \sqrt{\tfrac{s}{2}} \zeta_2, B_2\right)}^{l_0 N(g_2) - n} \dfrac{\dd \zeta_1 \dd \zeta_2}{(1 + \sqrt{\tfrac{2}{s}} \abs{\zeta_1}_{{\sf g}_2^\sympf})^{l_0}}
\end{multlined}
\end{equation}
(we used~\eqref{eq:normofZg2} with $\rho_2' = \rho + \sqrt{\frac{s}{2}} \zeta_2$).
Now we make two observations. On the one hand, we have by the triangle inequality:
\begin{align*}
\jap*{\dist_{({\sf g}_1 + {\sf g}_2)^\sympf}(\rho, B_1)}
	&\le 1 + \abs*{\sqrt{\tfrac{s}{2}} \zeta_1}_{{\sf g}_2^\sympf} + \dist_{{\sf g}_1^\sympf}\left( \rho + \sqrt{\tfrac{s}{2}} \zeta_1, B_1\right) \\
	&\le 2 \left(1 + \sqrt{\tfrac{2}{s}} \abs{\zeta_1}_{{\sf g}_2^\sympf}\right) \jap*{\dist_{{\sf g}_1^\sympf}\left( \rho + \sqrt{\tfrac{s}{2}} \zeta_1, B_1\right)} .
\end{align*}
We used the fact that $({\sf g}_1 + {\sf g}_2)^\sympf \le {\sf g}_j^\sympf$ (recall that {$\sympf$-duality} is non-increasing~\eqref{eq:sympfduality}) in the first inequality and $s \in [0, 1]$ in the second one. That yields
\begin{equation} \label{eq:oneoftheterms}
\dfrac{1}{\jap*{\dist_{{\sf g}_1^\sympf}\left( \rho + \sqrt{\tfrac{s}{2}} \zeta_1, B_1\right)}}
	\le 2 \dfrac{1 + \sqrt{\tfrac{2}{s}} \abs{\zeta_1}_{{\sf g}_2^\sympf}}{\jap*{\dist_{({\sf g}_1 + {\sf g}_2)^\sympf}(\rho, B_1)}} .
\end{equation}
On the other hand, we have
\begin{equation*}
1 + \abs*{\rho + \sqrt{\tfrac{s}{2}} \zeta_2 - \rho_2}_{{\sf g}_2}
	\le 1 + r_2 + \dist_{{\sf g}_2}\left( \rho + \sqrt{\tfrac{s}{2}} \zeta_2, B_2 \right)
	\le 4 \jap*{\dist_{{\sf g}_2^\sympf}\left( \rho + \sqrt{\tfrac{s}{2}} \zeta_2, B_2 \right)}
\end{equation*}
(recall that the slow variation radius is assumed to be smaller than~$1$ and~$g$ is admissible so that $g \le g^\sympf$), that is to say
\begin{equation} \label{eq:twooftheterms}
\dfrac{1}{\jap*{\dist_{{\sf g}_2^\sympf}\left( \rho + \sqrt{\tfrac{s}{2}} \zeta_2, B_2 \right)}}
	\le \dfrac{4}{1 + \abs*{\rho + \sqrt{\tfrac{s}{2}} \zeta_2 - \rho_2}_{{\sf g}_2}} .
\end{equation}
If~$n$ is large enough so that $l_0 N(g_2) - n < 0$ in~\eqref{eq:preestimateN}, we deduce from~\eqref{eq:oneoftheterms} and~\eqref{eq:twooftheterms} that
\begin{multline*}
\abs*{\e^{- \ii \frac{s}{2} \frak{P}} f(\rho, \rho)}
	\le N_{l_0, k, n}(f) \dfrac{C(g_2) 2^{l_0 + 2 k + 2 n}}{\jap*{\dist_{({\sf g}_1 + {\sf g}_2)^\sympf}(\rho, B_1)}^{k}} \\
		\times \int_{W \oplus W} \left(1 + \abs*{\sqrt{\tfrac{s}{2}} \zeta_2 + \rho - \rho_2}_{{\sf g}_2}\right)^{l_0 N(g_2) - n} \dfrac{\dd \zeta_1 \dd \zeta_2}{(1 + \sqrt{\tfrac{2}{s}} \abs{\zeta_1}_{{\sf g}_2^\sympf})^{l_0-k}} .
\end{multline*}
Now taking $n_1$ arbitrary, $l_0 = k + 2 \dim \mfd + 1$ and $n = l_0 \lceil N(g_2) \rceil + 2 \dim \mfd + 1$, we observe that the integral in the right-hand side is bounded by
\begin{equation*}
\int_{W \oplus W} \dfrac{\dd \zeta_1 \dd \zeta_2}{\left((1 + \abs{\zeta_2}_{{\sf g}_2}) (1 + \abs{\zeta_1}_{{\sf g}_2^\sympf})\right)^{2 \dim \mfd + 1}}
\end{equation*}
(we changed variables $(\sqrt{\frac{2}{s}} \zeta_1, \sqrt{\frac{s}{2}} \zeta_2) \mapsto (\zeta_1, \zeta_2)$). This is independent of ${\sf g}_2$. Indeed, denoting by $G_2$ and $G_2^\sympf$ the matrices of ${\sf g}_2$ and ${\sf g}_2^\sympf$ in a symplectic basis, we have $G_2^\sympf = - J G_2^{-1} J$ where $J$ is the symplectic matrix, hence $\abs{\det G_2^\sympf} \abs{\det G_2} = 1$. Therefore we have
\begin{equation} \label{eq:inequalityB1}
\abs*{\e^{- \ii \frac{s}{2} \frak{P}} f(\rho, \rho)}
	\le C N_{l_0, k, n}(f) \dfrac{2^{l_0 + 2 k + 2 n}}{\jap*{\dist_{({\sf g}_1 + {\sf g}_2)^\sympf}(\rho, B_1)}^{k}} ,
\end{equation}
where the constant $C$ depends only on the dimension of $\mfd$ and the structure constants of $g$.

Recalling~\eqref{eq:Moyalcomplexconjugation}, we apply the previous estimate, exchanging the roles of $\rho_1$, $\rho_2$ and $g_1$, $g_2$, and we obtain
\begin{equation} \label{eq:inequalityB2}
\abs*{\e^{- \ii \frac{s}{2} \frak{P}} f(\rho, \rho)}
	\le C N_{l_0, n, k}(f) \dfrac{2^{l_0 + 2 k + 2 n}}{\jap*{\dist_{({\sf g}_1 + {\sf g}_2)^\sympf}(\rho, B_2)}^{k}} ,
\end{equation}
with a similar choice of $l_0$, $k$ and $n$. Using that
\begin{equation*}
\max \left\{ \dist_{({\sf g}_1 + {\sf g}_2)^\sympf}(\rho, B_1), \dist_{({\sf g}_1 + {\sf g}_2)^\sympf}(\rho, B_2) \right\}
	\ge \dfrac{1}{2} \dist_{({\sf g}_1 + {\sf g}_2)^\sympf}(\rho, B_1) + \dfrac{1}{2} \dist_{({\sf g}_1 + {\sf g}_2)^\sympf}(\rho, B_2) ,
\end{equation*}
we infer from~\eqref{eq:inequalityB1} and~\eqref{eq:inequalityB2} that
\begin{equation} \label{eq:biconfell=0}
\abs*{\e^{- \ii \frac{s}{2} \frak{P}} f(\rho, \rho)}
	\le C N_{l_0, k+n, k+n}(f; \rho_1, \rho_2) \dfrac{2^{1 + l_0 + 3(k + n)}}{\jap*{\dist_{({\sf g}_1 + {\sf g}_2)^\sympf}(\rho, B_1) + \dist_{({\sf g}_1 + {\sf g}_2)^\sympf}(\rho, B_2)}^{k}} .
\end{equation}
This is the sought result for $\ell = 0$ ($l_0$ and $n$ are affine functions of $k$).

To handle the derivatives, it suffices to estimate $\nabla_{X^\ell}^\ell (\e^{- \ii \frac{s}{2} \frak{P}} f(\rho, \rho))$ for any (constant) vector field~$X$ on $T^\star \mfd$ according to Remark~\ref{rmk:polarintro}. The chain rule gives
\begin{equation*}
\abs*{\nabla_{X^\ell}^\ell \left(\e^{- \ii \frac{s}{2} \frak{P}} f(\rho, \rho)\right)}
	\le C_\ell \max_{\ell_1 + \ell_2 = \ell} \abs*{\left( \nabla_{X^{\ell_1}}^{\ell_1} \otimes \nabla_{X^{\ell2}}^{\ell_2} (\e^{- \ii \frac{s}{2} \frak{P}} f) \right)(\rho, \rho)} .
\end{equation*}
Fix $\ell_1 + \ell_2 = \ell$ and set $\tilde f := \nabla_{X^{\ell_1}}^{\ell_1} \otimes \nabla_{X^{\ell2}}^{\ell_2} (\e^{- \ii \frac{s}{2} \frak{P}} f)$. The bound~\eqref{eq:biconfell=0} applies to $\tilde f$, with $N_{l_0, k', k'}(\tilde f)$ in the right-hand side ($k' := k+n$). For any~$l$, we have
\begin{equation} \label{eq:derivativestwoells}
\abs*{\nabla^l \left(\nabla_{X^{\ell_1}}^{\ell_1} \otimes \nabla_{X^{\ell2}}^{\ell_2}\right) f(\rho_1', \rho_2')}_{g \oplus g}
	\le \abs*{X}_{g_{\rho_1'}}^{\ell_1} \abs*{X}_{g_{\rho_2'}}^{\ell_2} \abs*{\nabla^{\ell + l} f}_{g \oplus g} .
\end{equation}
Using admissibility of~$g$ (Proposition~\ref{prop:improvedadmissibility}) twice, we have for $j = 1, 2$:
\begin{equation*}
\abs*{X}_{g_{\rho_j'}}
	\le C_g \abs*{X}_{g_{\rho_j}} \jap*{\dist_{g_{\rho_j}^\sympf}\left( \rho_j', B_j \right)}^{N_g}
	\le C_g^2 \abs*{X}_{g_\rho} \jap*{\dist_{g_{\rho_j}^\sympf}\left( \rho_j', B_j \right)}^{N_g} \jap*{\dist_{g_{\rho_j}^\sympf}\left( \rho, B_j \right)}^{N_g} .
\end{equation*}
We bound from above the last factor in the right-hand side using Lemma~\ref{lem:integr} Item~\ref{it:minorationdowns}. Plugging the resulting inequality into~\eqref{eq:derivativestwoells} and going back to the definition of $N_{l_0, k, n}(\tilde f)$ in~\eqref{eq:defuprightN}, we obtain
\begin{equation*}
N_{l_0, k', k'}(\tilde f)
	\le C N_{l_0 + \ell, k' + \ell N_g, k' + \ell N_g}(f) \jap*{\dist_{(g_{\rho_1} + g_{\rho_2})^\sympf}\left( \rho, B_1 \right) + \dist_{(g_{\rho_1} + g_{\rho_2})^\sympf}\left( \rho, B_2 \right)}^{\ell N_g (1 + N_g)} ,
\end{equation*}
with a constant $C$ depending on $\ell$ and $g$. We finally conclude that
\begin{equation*}
\abs*{\nabla_{X^\ell}^\ell \left(\e^{- \ii \frac{s}{2} \frak{P}} f(\rho, \rho)\right)}
	\le \dfrac{C N_{l_0 + \ell, k' + \ell N_g, k' + \ell N_g}(f; \rho_1, \rho_2)}{\jap*{\dist_{({\sf g}_1 + {\sf g}_2)^\sympf}(\rho, B_1) + \dist_{({\sf g}_1 + {\sf g}_2)^\sympf}(\rho, B_2)}^{k - \ell N_g (1 + N_g)}} .
\end{equation*}
which yields the sought result after substituting $k + \ell N_g (1 + N_g)$ for $k$.
\end{proof}

\subsection{Pseudo-differential calculus} \label{subsec:proofspdeudocalc}

\subsubsection{Compatible metrics}

To perform pseudo-differential calculus with two different metrics, we need some conditions on these metrics.

\begin{definition} \label{def:compatiblemetrics}
Two admissible metrics $g_1$ and $g_2$ are compatible if there exist $C > 0$ and $N \ge 0$ such that
\begin{equation} \label{eq:jointadmissibility}
\begin{split}
\forall \rho_0, \rho \in T^\star \mfd , \qquad
g_{1, \rho}
	&\le C^2 g_{1, \rho_0} \jap*{\rho - \rho_0}_{g_{2, \rho_0}^\sympf}^{2N} \\
\forall \rho_0, \rho \in T^\star \mfd , \qquad
g_{2, \rho}
	&\le C^2 g_{2, \rho_0} \jap*{\rho - \rho_0}_{g_{1, \rho_0}^\sympf}^{2N}
\end{split}
\end{equation}
and the joint gain function of $g_1$ and $g_2$ defined in~\eqref{eq:defjointgain} is such that $\gain_{g_1, g_2}(\rho) \le 1$ for all $\rho \in T^\star \mfd$.
\end{definition}

\begin{remark} \label{rmk:otherpointtemperance}
The fact that $g_1$ and $g_2$ are admissible together with {$\sympf$-duality}~\eqref{eq:sympfduality} imply that $\abs{\rho_2 - \rho_1}_{g_{j, \rho_1}^\sympf} \le C' \abs{\rho_2 - \rho_1}_{g_{j, \rho_2}^\sympf}^{N'}$ for all $\rho_1, \rho_2$, $j = 1, 2$, hence we can equivalently replace the Japanese brackets $\jap{\bullet}_{g_{j, \rho_0}^\sympf}$ by $\jap{\bullet}_{g_{j, \rho}^\sympf}$ in the right-hand side of~\eqref{eq:jointadmissibility}.
\end{remark}

We give a sufficient condition for two admissible metrics to be compatible. Recall the symplectic intermediate metric $g^\natural$ from Lemma~\ref{lem:gnatural}.

\begin{proposition} \label{prop:sufficientconditioncompatibility}
Let $g_1$ and $g_2$ be admissible metrics. If $g_1^\natural = g_2^\natural$, and $\gain_{g_1, g_2} \le 1$, then $g_1$ and $g_2$ are compatible, and the constants in~\eqref{eq:jointadmissibility} depend only on the structure constants of $g_1$ and $g_2$.
\end{proposition}

\begin{remark} \label{rmk:sufficientconditioncompatibility}
The condition $g_1^\natural = g_2^\natural$ is verified when $g_1$ and $g_2$ are conformal.
\end{remark}

\begin{proof}
This is a mere consequence of the improved admissibility property of Proposition~\ref{prop:improvedadmissibility} since $g_1^\natural = g_2^\natural \le g_2^\sympf$ and $g_2^\natural = g_1^\natural \le g_1^\sympf$.
\end{proof}

Recall that the maps $\gain_{g_1 \oplus g_2}$ and $\gain_{g_1, g_2}$ defined in~\eqref{eq:gaing1oplusg2} and~\eqref{eq:defjointgain} are related by
\begin{equation*}
{\gain_{g_1 \oplus g_2}}_{\vert \diag}
	= \gain_{g_1, g_2} .
\end{equation*}

\begin{proposition} \label{prop:admissibilitycompatiblemetrics}
Let $g_1, g_2$ be compatible admissible metrics. Then the metric $\hat g := \frac{1}{2} (g_1 + g_2)$ is admissible with slow variation radius $r_{\hat g} = 2^{-1/2} \min\{r_{g_1}, r_{g_2}\}$ and with structure constants depending only on those of $g_1, g_2$ and the constants in~\eqref{eq:jointadmissibility}. Moreover, the function $\gain_{g_1, g_2}$ is a {$\hat g$-admissible} weight and we have for all $\rho_1, \rho_2, \rho_1', \rho_2' \in T^\star \mfd$:
\begin{equation} \label{eq:temperancejointgain}
\gain_{g_1 \oplus g_2}(\rho_1', \rho_2')
	\le C \gain_{g_1 \oplus g_2}(\rho_1, \rho_2) \jap*{\dist_{(\hat g_{\rho_1} + \hat g_{\rho_1'})^\sympf}\left( B_1, B_1' \right)}^N \jap*{\dist_{(\hat g_{\rho_2} + \hat g_{\rho_2'})^\sympf}\left( B_2, B_2' \right)}^N ,
\end{equation}
where $B_j = B_{r_j}^{g_j}(\rho_j)$, $B_j' = B_{r_j}^{g_j}(\rho_j')$, $j = 1, 2$, for any $r_1 \in (0, r_{g_1}]$ and $r_2 \in (0, r_{g_2}]$ and $C, N$ depending only on the structure constants of~$g_1$ and~$g_2$.
\end{proposition}

\begin{proof}
Slow variation of~$\hat g$ follows directly from slow variation of~$g_1$ and~$g_2$: if $\zeta \in W$ is such that $\abs{\zeta}_{\hat g}^2 \le \frac{1}{2} \min\{r_{g_1}^2, r_{g_2}^2\}$, then $\abs{\zeta}_{g_j}^2 \le 2 \abs{\zeta}_{\hat g}^2 \le r_{g_j}^2$ for $j = 1, 2$, and one checks that slow variation holds with constant $C_{\hat g} = \max\{C_{g_1}, C_{g_2}\}$.
Temperance of~$\hat g$ is slightly more delicate. Hörmander shows in~\cite[(18.5.16)]{Hoermander:V3}, as a consequence of the temperance and compatibility of~$g_1$ and~$g_2$, that
\begin{equation} \label{eq:admhatg}
\forall \rho_0, \rho \in T^\star \mfd , \qquad
g_{j, \rho}
	\le C^2 g_{j, \rho_0} \jap*{\rho - \rho_0}_{\hat g_\rho^\sympf}^{2N} ,
		\qquad j = 1, 2 .
\end{equation}
Summing the two inequalities for $j = 1$ and $j = 2$ gives the temperance of $\hat g$.
As for the gain function, one can check that $\gain_{\hat g}^2 \le \frac{1}{4} (\gain_{g_1}^2 + 2 \gain_{g_1, g_2}^2 + \gain_{g_2}^2) \le 1$. See the proof of~\cite[Proposition 18.5.3]{Hoermander:V3} for details.

Finally, using the fact that $(\frac{g_1+g_2}{2})^\sympf(\zeta) = 2 \inf_{\zeta_1 + \zeta_2 = \zeta} \{g_1^\sympf(\zeta_1) + g_2^\sympf(\zeta_2)\}$ (see~\cite[(4.4.27)]{Lerner:10}) together with Remark~\ref{rmk:otherpointtemperance}, one can improve~\eqref{eq:admhatg} to
\begin{equation} \label{eq:bettertemperancehatg}
\forall \rho_0, \rho \in T^\star \mfd , \qquad
g_{j, \rho}
	\le C^2 g_{j, \rho_0} \jap*{\rho - \rho_0}_{(\hat g_{\rho_0} + \hat g_\rho)^\sympf}^{2N} ,
		\qquad j = 1, 2 ,
\end{equation}
with possibly different constants $C, N$, still depending on structure constants of $g_1, g_2$ (proceed as in~\cite[Lemma 2.2.14]{Lerner:10}). Then we can combine~\eqref{eq:bettertemperancehatg} with slow variation of~$g_1$ and~$g_2$ to deduce that for any non-zero $\zeta \in W$:
\begin{align*}
\dfrac{\abs{\zeta}_{g_1}(\rho_1')}{\abs{\zeta}_{g_2^\sympf}(\rho_2')}
	&\le \tilde C^2 \dfrac{\abs{\zeta}_{g_1}(\rho_1)}{\abs{\zeta}_{g_2^\sympf}(\rho_2)} \jap*{\dist_{(\hat g_{\rho_1} + \hat g_{\rho_1'})^\sympf}\left( B_1, B_1' \right)}^N \jap*{\dist_{(\hat g_{\rho_2} + \hat g_{\rho_2'})^\sympf}\left( B_2, B_2' \right)}^N ,
\end{align*}
and taking the supremum with respect to $\zeta$ yields the result~\eqref{eq:temperancejointgain}. Admissibility of $\gain_{g_1, g_2}$ with respect to $\hat g$ follows from this estimate with $\rho_1 = \rho_2$ and $\rho_1' = \rho_2'$.
\end{proof}

\subsubsection{Proof of pseudo-differential calculus with confined symbols} \label{subsubsec:proofpseudocalcsymbol}

From the bi-confinement estimate of Proposition~\ref{prop:bi-confinement}, we can deduce several variations of the pseudo-differential calculus. Proofs are inspired from~\cite{Lerner:10}.

\begin{proof}[Proof of Proposition~\ref{prop:pseudocalcconf}]
First of all, notice that $g \le g_0$ implies that we have the (continuous) inclusions
\begin{equation} \label{eq:2continuousinclusions}
S(m, g) \subset S(m, g_0) .
\end{equation}

Let $r > 0$ as in the statement and let $\psi \in \Conf_r^{g_0}(\rho_0)$ and $a \in \nabla^{-j} S(m, g)$. We introduce a partition of unity $(\varphi_{\rho_1})_{\rho_1 \in T^\star \mfd}$, independent of the parameter~$r$, adapted to the metric $g_0$, with radius $r_0$, whose existence is given by Proposition~\ref{prop:existencepartitionofunity}.

We claim that for any $k \ge 0$, the quantity
\begin{equation} \label{eq:Nkdefdef}
\cal{N}_k
	= \cal{N}_k(\rho_1, \rho_0)
	= \max_{0 \le \ell \le k} \sup_{\rho_1', \rho_2' \in T^\star \mfd} \jap*{\dist_{g_{0, \rho_1}^\sympf}(\rho_1', B_1)}^k \jap*{\dist_{g_{0, \rho_2}^\sympf}(\rho_2', B_2)}^k \abs*{\nabla^\ell f(\rho_1', \rho_2')}_{g_0 \oplus g_0}
\end{equation}
is bounded, where
\begin{equation} \label{eq:deff}
f = f_{\rho_1}
	= (\varphi_{\rho_1} \otimes 1) \times \frak{P}^j(a \otimes \psi) ,
\end{equation}
and $B_1 = B_{r_0}^{g_0}(\rho_1)$, $B_2 = B_r^{g_0}(\rho_0)$. We distribute derivatives in $\nabla^\ell f$ on both factors in~\eqref{eq:deff} thanks to the Leibniz formula. Hence it suffices to estimate the following terms: on the one hand,
\begin{equation} \label{eq:inequalityenvelope}
\forall k_1 \in \N , \qquad
	\abs*{\nabla^\ell (\varphi_{\rho_1} \otimes 1)(\rho_1', \rho_2')}_{g_0 \oplus g_0}
		\le \jap*{\dist_{g_{0, \rho_1}^\sympf}(\rho_1', B_1)}^{-k_1} \abs*{\varphi_{\rho_1}}_{\Conf_{r_0}^{g_0}(\rho_1)}^{(k_1)} ;
\end{equation}
on the other hand, using first that $g_0 \oplus g_0 \ge g \oplus g_0$ and then applying Lemma~\ref{lem:continuityfrakP}, we have for any~$l$:
\begin{align*}
\abs*{\nabla^l \frak{P}^j (a \otimes \psi)}_{g_0 \oplus g_0}
	&\le \abs*{\nabla^l \frak{P}^j (a \otimes \psi)}_{g \oplus g_0}
	\le (2 \dim \mfd)^j \gain_{g \oplus g_0}^j \abs*{\nabla^{j + l} (a \otimes \psi)}_{g \oplus g_0} \\
	&\le (2 \dim \mfd)^j \gain_{g \oplus g_0}^j \abs*{\nabla^{j + l} a}_g \otimes \abs*{\nabla^{j + l} \psi}_{g_0}
\end{align*}
Using that $\nabla^j a \in S(m, g)$ and the confinement of $\psi$ (recall~\eqref{eq:2continuousinclusions}), we obtain for any~$l$ and~$k_2 \in \N$:
\begin{equation} \label{eq:inequalitysemiconf}
\abs*{\nabla^l \frak{P}^j (a \otimes \psi)}_{g_0 \oplus g_0}(\rho_1', \rho_2')
	\le 2^{k_2} (2 \dim \mfd)^j \gain_{g \oplus g_0}^j m(\rho_1') \dfrac{\abs*{\nabla^j a}_{S(m, g)}^{(l)} \abs*{\psi}_{\Conf_r^{g_0}(\rho_0)}^{(j + l + k_2)}}{\jap*{ \dist_{g_{0, \rho_0}^\sympf}(\rho_2', B_r^{g_0}(\rho_0))}^{k_2}} .
\end{equation}

Since $m$ is {$g_0$-admissible}, we infer from Proposition~\ref{prop:improvedadmissibility} that
\begin{equation} \label{eq:inequalityonm}
m(\rho_1')
	\le C m(\rho_1) \jap*{\dist_{g_{0, \rho_1}^\sympf}(\rho_1', B_1)}^N
	\le C' m(\rho_0) \jap*{\dist_{g_{0, \rho_1}^\sympf}(\rho_1', B_1)}^N \jap*{\dist_{(g_{0, \rho_1} + g_{0, \rho_0})^\sympf}(B_1, B_2)}^{N'}
\end{equation}
(we use here the assumption that~$r_0$ is a slow variation radius of~$m$ for the metric~$g_0$).
Applying ~\eqref{eq:temperancejointgain} in Proposition~\ref{prop:admissibilitycompatiblemetrics} twice (recall~$g$ and~$g_0$ are compatible and have $r_0$ as a common slow variation radius), we have
\begin{align*}
\gain_{g \oplus g_0}(\rho_1', \rho_2')
	&\le C \gain_{g \oplus g_0}(\rho_1, \rho_0) \jap*{\dist_{g_{0, \rho_1}^\sympf}\left( \rho_1', B_1 \right)}^N \jap*{\dist_{g_{0, \rho_0}^\sympf}\left( \rho_2', B_2 \right)}^N \\
	&\le C' \gain_{g, g_0}(\rho_0) \jap*{\dist_{(g_{0, \rho_1} + g_{0, \rho_0})^\sympf}\left(B_1, B_2\right)}^{N'} \jap*{\dist_{g_{0, \rho_1}^\sympf}\left( \rho_1', B_1 \right)}^N \jap*{\dist_{g_{0, \rho_0}^\sympf}\left( \rho_2', B_2 \right)}^N .
\end{align*}
Here, we used the fact that $g_0 \le g + g_0 \le 2 g_0$ to obtain an estimate in terms of the metric~$g_0$ only.
The constants depend only on structure constants of the metrics involved.
Combining this with the estimate~\eqref{eq:inequalityonm} on $m$, the inequality~\eqref{eq:inequalitysemiconf} becomes
\begin{equation} \label{eq:frakPjcontinuityineq}
\begin{multlined}
\abs*{\nabla^l \frak{P}^j (a \otimes \psi)}_{g_0 \oplus g_0}(\rho_1', \rho_2')
	\le C_{j, g, g_0} 2^{k_2} \gain_{g, g_0}^j(\rho_0) m(\rho_0) \jap*{\dist_{(g_{0, \rho_1} + g_{0, \rho_0})^\sympf}\left(B_1, B_2\right)}^{jN'} \\
		\times \jap*{\dist_{g_{0, \rho_1}^\sympf}(\rho_1', B_1)}^{jN} \jap*{\dist_{g_{0, \rho_2}^\sympf}(\rho_2', B_2)}^{jN-k_2} \abs*{\nabla^j a}_{S(m, g)}^{(l)} \abs*{\psi}_{\Conf_r^{g_0}(\rho_0)}^{(j + l  + k_2)} ,
\end{multlined}
\end{equation}
and together with~\eqref{eq:inequalityenvelope}, we obtain
\begin{equation} \label{eq:nablaellfineq}
\begin{multlined}
\abs*{\nabla^\ell f}_{g_0 \oplus g_0}(\rho_1', \rho_2')
	\le C_{j, g, g_0} 2^{\ell+k_2} \gain_{g, g_0}^j(\rho_0) m(\rho_0) \jap*{\dist_{(g_{0, \rho_1} + g_{0, \rho_0})^\sympf}\left(B_1, B_2\right)}^{jN'} \\
		\times \jap*{\dist_{g_{0, \rho_1}^\sympf}(\rho_1', B_1)}^{jN-k_1} \jap*{ \dist_{g_{0, \rho_0}^\sympf}(\rho_2', B_2)}^{jN-k_2} \abs*{\nabla^j a}_{S(m, g)}^{(\ell)} \abs*{\psi}_{\Conf_r^{g_0}(\rho_0)}^{(j + \ell  + k_2)} \abs*{\varphi_{\rho_1}}_{\Conf_{r_0}^{g_0}(\rho_1)}^{(k_1)} .
\end{multlined}
\end{equation}
At this stage, we remark that taking $\rho_1 = \rho_0$ and $\rho$ in place of $\rho_1'$ and $\rho_2'$ in~\eqref{eq:frakPjcontinuityineq} yields
\begin{equation*}
\abs*{\nabla^l \frak{P}^j (a \otimes \psi)}_{g_0 \oplus g_0}(\rho, \rho)
	\le C_{j, g, g_0} 2^{k_2} \gain_{g, g_0}^j(\rho_0) m(\rho_0) \times \jap*{\dist_{g_{0, \rho_0}^\sympf}(\rho, B_2)}^{2Nj - k_2} \abs*{\nabla^j a}_{S(m, g)}^{(l)} \abs*{\psi}_{\Conf_r^{g_0}(\rho_0)}^{(j + l  + k_2)} .
\end{equation*}
Then we observe that the continuity estimate~\eqref{eq:continuitycalPj} for $\cal{P}_j$, defined in~\eqref{eq:defPj}, follows from restricting this estimate to the diagonal through the chain rule:
\begin{align*}
\abs*{\nabla^\ell \left(\frak{P}^j (a \otimes \psi)(\rho, \rho)\right)}_{g_0}
	&\le C_\ell \max_{0 \le \ell_1, \ell_2 \le \ell} \abs*{\nabla^{\ell_1} \otimes \nabla^{\ell_2} \left(\frak{P}^j (a \otimes \psi)\right)}_{g_0 \oplus g_0}(\rho, \rho) \\
	&\le C_{\ell, k_2, j, g, g_0} \gain_{g, g_0}^j(\rho_0) m(\rho_0) 	\jap*{\dist_{g_{0, \rho_0}^\sympf}(\rho, B_2)}^{2Nj - k_2} \abs*{\nabla^j a}_{S(m, g)}^{(\ell)} \abs*{\psi}_{\Conf_r^{g_0}(\rho_0)}^{(j + \ell  + k_2)} .
\end{align*}
Here, $k_2$ is arbitrary.

Now we focus on the estimate~\eqref{eq:continuityhatcalPj} for~$\widehat{\cal{P}}_j$. Since the integers~$k_1$ and~$k_2$ in~\eqref{eq:nablaellfineq} are arbitrary, we deduce that for some $k'$ sufficiently large, the quantity~$\cal{N}_k$ defined in~\eqref{eq:Nkdefdef} is estimated by:
\begin{equation} \label{eq:ineqoncalNk}
\cal{N}_k
	\le C_{k, j, g, g_0} \gain_{g, g_0}^j(\rho_0) m(\rho_0) \jap*{\dist_{(g_{0, \rho_1} + g_{0, \rho_0})^\sympf}\left(B_1, B_2\right)}^{jN'} \abs*{\nabla^j a}_{S(m, g)}^{(k')} \abs*{\psi}_{\Conf_r^{g_0}(\rho_0)}^{(k')} \abs*{\varphi_{\rho_1}}_{\Conf_{r_0}^{g_0}(\rho_1)}^{(k')} .
\end{equation}
We apply the bi-confinement estimate (Proposition~\ref{prop:bi-confinement}): for any~$k$ and~$\ell$, there exists $k'$ such that for all $\rho \in T^\star \mfd$:
\begin{equation} \label{eq:estdiagmoyal}
\begin{multlined}
\abs*{\nabla^\ell \left(\e^{- \ii \frac{s}{2} \frak{P}} f(\rho, \rho)\right)}_{g_0} \\
	\le C_{k, \ell, j, g, g_0} \gain_{g, g_0}^j(\rho_0) m(\rho_0) \dfrac{\abs*{\nabla^j a}_{S(m, g)}^{(k')} \abs*{\psi}_{\Conf_r^{g_0}(\rho_0)}^{(k')} \abs*{\varphi_{\rho_1}}_{\Conf_{r_0}^{g_0}(\rho_1)}^{(k')}}{\jap*{\dist_{(g_{0, \rho_1} + g_{0, \rho_0})^\sympf}(\rho, B_1) + \dist_{(g_{0, \rho_1} + g_{0, \rho_0})^\sympf}(\rho, B_2)}^{k - jN'}} .
\end{multlined}
\end{equation}
We absorbed the factor to the power $jN'$ in~\eqref{eq:ineqoncalNk} into the denominator of~\eqref{eq:estdiagmoyal} thanks to the triangle inequality.
Taking $k$ large enough, we observe that for fixed $\rho$, the right-hand side of~\eqref{eq:estdiagmoyal} is integrable with respect to $\rho_1$ (see Lemma~\ref{lem:integr} Item~\ref{it:integrabilityg}). Thus the following integral is absolutely convergent, and by Fubini's theorem, we deduce
\begin{equation*}
\int_{T^\star \mfd} \nabla^\ell \left(\e^{- \ii \frac{s}{2} \frak{P}} f\right) \dd \vol_{g_0}(\rho_1)
	= \nabla^\ell \e^{- \ii \frac{s}{2} \frak{P}} \int_{T^\star \mfd} (\varphi_{\rho_1} \otimes 1) \frak{P}^j(a \otimes \psi) \dd \vol_{g_0}(\rho_1)
	= \nabla^\ell \left(\e^{- \ii \frac{s}{2} \frak{P}} \frak{P}^j(a \otimes \psi)\right) ,
\end{equation*}
with the estimate
\begin{multline*}
\abs*{\nabla^\ell \left( \e^{- \ii \frac{s}{2} \frak{P}} \frak{P}^j(a \otimes \psi)(\rho, \rho) \right)}_{g_0}
	\le C_{k, \ell, j, g, g_0} \gain_{g, g_0}^j(\rho_0) m(\rho_0) \\
		\times \int_{T^\star \mfd} \dfrac{\abs*{\nabla^j a}_{S(m, g)}^{(k')} \abs*{\psi}_{\Conf_r^{g_0}(\rho_0)}^{(k')} \sup_{\rho_1 \in T^\star \mfd} \abs*{\varphi_{\rho_1}}_{\Conf_{r_0}^{g_0}(\rho_1)}^{(k')}}{\jap*{\dist_{(g_{0, \rho_1} + g_{0, \rho_0})^\sympf}(\rho, B_1) + \dist_{(g_{0, \rho_1} + g_{0, \rho_0})^\sympf}(\rho, B_2)}^{k - jN'}} \dd \vol_{g_0}(\rho_1) .
\end{multline*}
Now we check the decay as $\rho \to \infty$ as follows: the metric $g_0$ is admissible, so that choosing $k$ large enough, one can apply Item~\ref{it:minorationdowns} of Lemma~\ref{lem:integr} to obtain
\begin{multline*}
\int_{T^\star \mfd} \jap*{\dist_{(g_{0, \rho_1} + g_{0, \rho_0})^\sympf}(\rho, B_1) + \dist_{(g_{0, \rho_1} + g_{0, \rho_0})^\sympf}(\rho, B_2)}^{-(k - jN')} \dd \vol_{g_0}(\rho_1) \\
	\le \tilde C \jap*{\dist_{g_{0, \rho_0}^\sympf}(\rho, B_2)}^{-n_0} \int_{T^\star \mfd} \jap*{\dist_{g_{0, \rho_1}^\sympf}(\rho, B_1)}^{-n_0} \dd \vol_{g_0}(\rho_1) .
\end{multline*}
Notice that we can make $n_0$ as large as we wish, up to enlarging~$k$. Finally, we observe that the integral in the right-hand side is bounded independently of $\rho$ (and $\rho_0$ and $r$) by Item~\ref{it:integrabilityg} of Lemma~\ref{lem:integr}, since $g_{0, \rho_1}^\sympf \ge (g_{0, \rho} + g_{0, \rho_1})^\sympf$ (recall that {$\sympf$-duality} in non-increasing~\eqref{eq:sympfduality}).

It remains to integrate the estimate with respect to $s \in [0, 1]$ to obtain, in view of the definition of $\widehat{\cal{P}}_j$ in~\eqref{eq:defhatPj}:
\begin{equation*}
\abs*{\nabla^\ell\left(\widehat{\cal{P}}_j(a, \psi)(\rho, \rho)\right)}_{g_0}
	\le \tilde C_{k, \ell, j, g, g_0} \dfrac{\gain_{g, g_0}^j(\rho_0) m(\rho_0)}{\jap*{\dist_{g_{0, \rho_0}^\sympf}(\rho, B_2)}^{n_0}} \abs*{\nabla^j a}_{S(m, g)}^{(k')} \abs*{\psi}_{\Conf_r^{g_0}(\rho_0)}^{(k')} \sup_{\rho_1 \in T^\star \mfd} \abs*{\varphi_{\rho_1}}_{\Conf_{r_0}^{g_0}(\rho_1)}^{(k')} ,
\end{equation*}
which is the desired estimate to establish~\eqref{eq:continuityhatcalPj}. Notice that uniformity in~$r$ is due to the fact that the partition of unity~$(\varphi_{\rho_1})_{\rho_1 \in T^\star \mfd}$ does not depend on~$r$.
\end{proof}

\subsubsection{Proof of pseudo-differential calculus in symbol classes}

So far, we have discussed pseudo-differential calculus with one confined symbol and one symbol in a Weyl--Hörmander class. Here we go one step further and prove the pseudo-differential calculus for two symbols in Weyl--Hörmander classes, using partitions of unity from Proposition~\ref{prop:existencepartitionofunity}.

\begin{proof}[Proof of Proposition~\ref{prop:pseudocalcsymb}]
We indicate how to modify the above proof of Proposition~\ref{prop:pseudocalcconf} in order to obtain the desired result. Let $a_1 \in \nabla^{-j} S(m_1, g_1)$, $a_2 \in \nabla^{-j} S(m_2, g_2)$. As a first step, we use the fact that $(g_1 + g_2)^{\oplus 2} \ge g_1 \oplus g_2$ and we apply Lemma~\ref{lem:continuityfrakP} to have
\begin{align} \label{eq:ineqfrakPj0general}
\abs*{\nabla^\ell \frak{P}^j (a_1 \otimes a_2)}_{(g_1 + g_2) \oplus (g_1 + g_2)}
	&\le C_j \gain_{g_1 \oplus g_2}^j \abs*{\nabla^{j + \ell} (a_1 \otimes a_2)}_{g_1 \oplus g_2} \nonumber\\
	&\le C_j \gain_{g_1 \oplus g_2}^j m_1 \otimes m_2 \abs*{\nabla^j a_1}_{S(m_1, g_1)}^{(\ell)} \abs*{\nabla^j a_2}_{S(m_2, g_2)}^{(\ell)} .
\end{align}
We restrict this estimate to the diagonal thanks to the chain rule to prove the continuity estimate~\eqref{eq:continuitycalPjsymb} for~$\cal{P}_j$:
\begin{align*}
\abs*{\nabla^\ell \left(\frak{P}^j (a_1 \otimes a_2)(\rho, \rho)\right)}_{g_1 + g_2}
	&\le C_\ell \max_{0 \le \ell_1, \ell_2 \le \ell} \abs*{\nabla^{\ell_1} \otimes \nabla^{\ell_2} \left(\frak{P}^j (a_1 \otimes a_2)\right)}_{(g_1 + g_2) \oplus (g_1 + g_2)}(\rho, \rho) \\
	&\le C_{j, \ell} \gain_{g_1, g_2}^j(\rho) (m_1 m_2)(\rho) \abs{\nabla^j a_1}_{S(m_1, g_1)}^{(\ell)} \abs{\nabla^j a_2}_{S(m_2, g_2)}^{(\ell)} .
\end{align*}

We now turn to the proof of the continuity estimate~\eqref{eq:continuityhatcalPjsymb} for $\widehat{\cal{P}}_j$. Write for simplicity $\hat g = \frac{1}{2} (g_1 + g_2)$ and fix a partition of unity $(\varphi_{\rho_1})_{\rho_1 \in T^\star \mfd}$ adapted to the metric $\hat g$ given by Proposition~\ref{prop:existencepartitionofunity}, with radius $r$ small enough so that it is a slow variation radius of~$\hat g$, $m_1$ and $m_2$. Similarly to~\eqref{eq:deff}, we define
\begin{equation} \label{eq:deffpseudocalcsymbols}
f = f_{\rho_1, \rho_2}
	= (\varphi_{\rho_1} \otimes \varphi_{\rho_2}) \times \frak{P}^j(a_1 \otimes a_2) , 
\end{equation}
and we show that for any $k \in \N$, the quantity
\begin{equation*}
\cal{N}_k
	= \cal{N}_k(\rho_1, \rho_2)
	= \max_{0 \le \ell \le k} \sup_{\rho_1', \rho_2' \in T^\star \mfd} \jap*{\dist_{\hat g_{\rho_1}^\sympf}(\rho_1', B_1)}^k \jap*{\dist_{\hat g_{\rho_2}^\sympf}(\rho_2', B_2)}^k \abs*{\nabla^\ell f(\rho_1', \rho_2')}_{\hat g \oplus \hat g}
\end{equation*}
is bounded, where $B_1 = B_r^{\hat g}(\rho_1)$ and $B_2 = B_r^{\hat g}(\rho_2)$. Now in~\eqref{eq:ineqfrakPj0general}, we apply Proposition~\ref{prop:improvedadmissibility} with the {$\hat g$-admissible} weights~$m_1$ and~$m_2$, together with~\eqref{eq:temperancejointgain} in Proposition~\ref{prop:admissibilitycompatiblemetrics}, to obtain
\begin{align*}
\gain_{g_1 \oplus g_2}^j(\rho_1', \rho_2') m_1(\rho_1') m_2(\rho_2')
	&\le C \gain_{g_1 \oplus g_2}^j(\rho_1, \rho_2) m_1(\rho_1) m_2(\rho_2) \\
	&\qquad\qquad \times \jap*{\dist_{(\hat g_{\rho_1} + \hat g_{\rho_1'})^\sympf}\left(\rho_1', B_1\right)}^{N'} \jap*{\dist_{(\hat g_{\rho_2} + \hat g_{\rho_2'})^\sympf}\left(\rho_2', B_2\right)}^{N'} \\
	&\le C' \gain_{g_1, g_2}^j(\rho) (m_1 m_2)(\rho) \jap*{\dist_{\hat g_{\rho_1}^\sympf}\left(\rho_1', B_1\right)}^{N'} \jap*{\dist_{\hat g_{\rho_2}^\sympf}\left(\rho_2', B_2\right)}^{N'} \\
	&\qquad\qquad\qquad \times	\jap*{\dist_{(\hat g_{\rho_1} + \hat g_\rho)^\sympf}\left(\rho, B_1\right)}^{N'} \jap*{\dist_{(\hat g_{\rho_2} + \hat g_\rho)^\sympf}\left(\rho, B_2\right)}^{N'} .
\end{align*}
Plugging this into~\eqref{eq:ineqfrakPj0general}, then in view of~\eqref{eq:deffpseudocalcsymbols}, we know that for any $k$, taking advantage of the confinement of $\varphi_{\rho_1}, \varphi_{\rho_2}$, one has
\begin{equation*}
\cal{N}_k
	\le C_{j, k}' (\gain_{g_1, g_2}^j m_1 m_2)(\rho) \jap*{\dist_{\hat g_{\rho_1}^\sympf}\left(\rho, B_1\right)}^{N'} \jap*{\dist_{\hat g_{\rho_2}^\sympf}\left(\rho, B_2\right)}^{N'} \abs{\nabla^j a_1}_{S(m_1, g_1)}^{(k)} \abs{\nabla^j a_2}_{S(m_2, g_2)}^{(k)} .
\end{equation*}
Here the constant~$C_{j, k}'$ contains seminorms of~$\varphi_{\rho_1}, \varphi_{\rho_2}$, in the corresponding classes~$\Conf_r^{\hat g}$, which are bounded uniformly in $\rho_1, \rho_2 \in T^\star \mfd$.
Therefore the bi-confinement estimate (Proposition~\ref{prop:bi-confinement}) applied to~$f$ implies that for any~$k$ and~$\ell$, there exists~$k'$ such that for all $\rho \in T^\star \mfd$:
\begin{align*}
\abs*{\nabla^\ell\left(\e^{- \ii \frac{s}{2} \frak{P}} f(\rho, \rho)\right)}_{\hat g}
	&\le C_{j, k, \ell}'' (\gain_{g_1, g_2}^j m_1 m_2)(\rho) \dfrac{\abs{\nabla^j a_1}_{S(m_1, g_1)}^{(k')} \abs{\nabla^j a_2}_{S(m_2, g_2)}^{(k')}}{\jap*{\dist_{(\hat g_{\rho_1} + \hat g_{\rho_2})^\sympf}(\rho, B_1) + \dist_{(\hat g_{\rho_1} + \hat g_{\rho_2})^\sympf}(\rho, B_2)}^k} \\
	&\qquad\qquad\qquad \times \jap*{\dist_{\hat g_{\rho_1}^\sympf}\left(\rho, B_1\right)}^{N'} \jap*{\dist_{\hat g_{\rho_2}^\sympf}\left(\rho, B_2\right)}^{N'} .
\end{align*}
Since $\hat g$ is admissible (Proposition~\ref{prop:admissibilitycompatiblemetrics}), we can bound the denominator from below using Lemma~\ref{lem:integr} Item~\ref{it:minorationdowns} and we obtain
\begin{equation*}
\abs*{\nabla^\ell \left(\e^{- \ii \frac{s}{2} \frak{P}} f(\rho, \rho)\right)}_{\hat g}
		\le \tilde C_{j, n_0, \ell} (\gain_{g_1, g_2}^j m_1 m_2)(\rho) \dfrac{\abs{\nabla^j a_1}_{S(m_1, g_1)}^{(k')} \abs{\nabla^j a_2}_{S(m_2, g_2)}^{(k')}}{\jap*{\dist_{\hat g_{\rho_1}^\sympf}(\rho, B_1)}^{n_0} \jap*{\dist_{\hat g_{\rho_2}^\sympf}(\rho, B_2)}^{n_0}} .
\end{equation*}
Here, $n_0$ can be taken arbitrarily large, up to enlarging~$k'$. All constants depend only on structure constants of $g_1, g_2, m_1, m_2$. For $n_0$ large enough, we deduce from Lemma~\ref{lem:integr} Item~\ref{it:integrabilityg} that the right-hand side is integrable with respect to $\rho_1, \rho_2$, independently of $\rho$ (recall that $\hat g_{\rho_j}^\sympf \ge (\hat g_{\rho_1} + \hat g_{\rho_2})^\sympf$, $j = 1, 2$). Further integrating in $s \in [0, 1]$ provides the sought estimate.
\end{proof}

\Large
\section{Continuity properties of pseudo-differential operators} \label{app:operators}
\normalsize

In this appendix, we recall classical properties and estimates about pseudo-differential operators.

\subsection{Distributional continuity}

We first recall a reformulation of the Schwartz kernel theorem in the context of the Weyl quantization.

\begin{proposition} \label{prop:SchwartzWeylkernel}
Let $A : \sch(\mfd) \to \sch'(\mfd)$ be a continuous linear operator. Then there exists a unique tempered distribution $a \in \sch'(T^\star \mfd)$ such that $A = \Opw{a}$.
\end{proposition}

\begin{proof}
Suppose $A = \Opw{a}$ with $a \in \sch(T^\star \mfd)$. From~\cite[(2.4) {p.\ 80}]{FollandPhaseSpace}, the Schwartz kernel $K_A$ and the Weyl symbol $a$ of $A$ are related through the formula
\begin{equation} \label{eq:SchwartzWeylsymbol}
a(x, \xi)
	= \int_V K_A\left(x - \dfrac{v}{2}, x + \dfrac{v}{2}\right) \e^{\ii \xi. v} \dd v .
\end{equation}
In other words, $a$ is the (extended) Wigner transform of $K_A$. Since the Wigner transform consists in an affine change of variables and a partial Fourier transform, we deduce that the map $K_A \mapsto a$ acts continuously on $\sch(\mfd \times \mfd) \to \sch(T^\star \mfd)$, and extends to a continuous isomorphism $\sch'(\mfd \times \mfd) \to \sch'(T^\star \mfd)$ (\cite[Proposition (1.92)]{FollandPhaseSpace}). In particular, the above formula~\eqref{eq:SchwartzWeylsymbol} still makes sense for general tempered distributions $K_A$. Given a general continuous linear operator $A : \sch(\mfd) \to \sch'(\mfd)$, the existence and uniqueness of $K_A \in \sch'(\mfd \times \mfd)$ is given by the Schwartz kernel theorem (see for instance~\cite[Corollary of Theorem 51.7]{Treves:67}), which concludes the proof.
\end{proof}

Let us now recall basic continuity properties on the Schwartz class $\sch(\mfd)$ and on the space of tempered distributions $\sch'(\mfd)$.

\begin{proposition}[Continuity of pseudo-differential operators \--- {\cite[Theorem 18.6.2]{Hoermander:V3}}] \label{prop:continuityonSchwartz}
Let $g$ be an admissible metric and $m$ be a {$g$-admissible} weight. Then for any $a \in S(m, g)$, the operator $\Opw{a}$ maps $\sch(\mfd)$ to itself continuously, and can be extended as a continuous operator on $\sch'(\mfd)$.
\end{proposition}

\begin{proposition}[Continuity of operator with Schwartz symbol \--- {\cite[Theorem 4.1]{Zworski:book}}] \label{prop:continuityofSchkerneloperator}
Let $k \in \sch(T^\star \mfd)$. Then
\begin{equation*}
\Opw{k} : \sch'(T^\star \mfd) \longrightarrow \sch(T^\star \mfd)
\end{equation*}
is continuous.
\end{proposition} 

\subsection{$L^2$ boundedness}

We also have the $L^2$ boundedness of operators whose symbol lies in $S(1, g)$. We refer to~\cite[Theorem 18.6.3]{Hoermander:V3} and ~\cite[Theorem 2.5.1]{Lerner:10}.

\begin{proposition}[Calder\'{o}n--Vaillancourt] \label{prop:CV}
Let $g$ be an admissible metric. Then there exists a constant $C_0$ and an integer $\ell_0$, depending only on $\dim \mfd$ and structure constants of $g$, such that for all $a \in S(1, g)$, the operator $\Opw{a}$ extends to a bounded operator on $L^2(M)$ and
\begin{equation*}
\norm*{\Opw{a}}_{\Bop(L^2)}
	\le C_0 \abs*{a}_{S(1, g)}^{(\ell_0)} .
\end{equation*}
\end{proposition}

We will need a characterization of pseudo-differential operators, namely, we want to know when a continuous linear operator $A : \sch(\mfd) \to \sch'(\mfd)$ can be written $A = \Opw{a}$ with a symbol $a \in S(m, g)$. We refer to Lerner's book~\cite{Lerner:10} for proofs of the results below. Recall $\aff(T^\star \mfd)$ is the set of affine maps $f : T^\star \mfd \to \R$.

\begin{proposition}[Beals' theorem \--- {\cite[Theorem 2.2.6]{Lerner:10}}] \label{prop:Beals}
For all $\ell \in \N$, there exist $C_\ell > 0$ and $k_\ell$ such that for all continuous linear operator $A : \sch(\mfd) \to \sch'(\mfd)$ with $A = \Opw{a}$, we have
\begin{equation*}
\abs*{a}_{S(1, {\sf g})}^{(\ell)}
		\le C_\ell \max_{0 \le j \le k_\ell} \sup_{f_1, f_2, \ldots, f_j \in \aff(T^\star \mfd, \R) \setminus \{0\}} \dfrac{\norm*{\ad_{\Opw{f_1}} \ad_{\Opw{f_2}} \cdots \ad_{\Opw{f_j}} A}_{\Bop(L^2)}}{\abs*{H_{f_1}}_{{\sf g}} \abs*{H_{f_2}}_{{\sf g}} \cdots \abs*{H_{f_j}}_{{\sf g}}} ,
\end{equation*}
where ${\sf g}$ is any flat admissible metric on $T^\star \mfd$ (that is to say ${\sf g}$ is constant and ${\sf g} \le {\sf g}^\sympf$).
\end{proposition}

\begin{remark}
In the statement of Proposition~\ref{prop:Beals}, it is possible to have $k_\ell$ and $C_\ell$ independent of ${\sf g}$ provided ${\sf g}$ is admissible. See the proof of~\cite[Theorem 2.6.6]{Lerner:10}, relying on a suitable choice of symplectic Euclidean coordinates.
\end{remark}

\subsection{Trace class and Hilbert--Schmidt operators}

Denote by $\hsclass\left(L^2(\mfd)\right)$ the Hilbert space of Hilbert--Schmidt operators, endowed with the inner product $(A, B) \mapsto \tr(A^\ast B)$.

\begin{proposition}[Weyl quantization as an isometry \---  {\cite{Pool} and \cite[Theorem V]{Berezin:HS}}] \label{prop:quantizationisometry}
The Weyl quantization is an isometry between Hilbert spaces:
\begin{equation*}
\quantizationw : L^2(T^\star \mfd) \longrightarrow \hsclass\left(L^2(\mfd)\right) .
\end{equation*}
\end{proposition}

\Large
\section{Technical lemmata on phase space metrics} \label{app:metrics}
\normalsize

This section collects miscellaneous technical results related to Weyl--Hörmander metrics.

\subsection{Temperate growth of the Hamiltonian}

We discuss a technical result on the polynomial growth (or temperate growth) of the classical Hamiltonian.

\begin{lemma}[Temperate growth] \label{lem:temperategrowth}
Under Assumption~\ref{assum:p}, the classical Hamiltonian $p$ satisfies~\eqref{eq:temperategrowth}.
\end{lemma}

\begin{proof}
We let $\rho_0 \in T^\star \mfd$ and we do a Taylor expansion of $p$ at $\rho_0$:
\begin{multline*}
\forall \rho \in T^\star \mfd ,	\qquad
	\nabla^\ell p(\rho)
		= \nabla^\ell p (\rho_0) + \nabla^{\ell+1} p(\rho_0). (\rho - \rho_0) + \dfrac{1}{2} \nabla^{\ell+2} p(\rho_0). (\rho - \rho_0)^2 \\
			+ \int_0^1 \dfrac{(1 - s)^2}{2} \nabla^{\ell+3} p\left((1 - s)\rho_0 + s \rho\right). (\rho - \rho_0)^3 \dd s .
\end{multline*}
Now using the admissibility of the metric $g$, the fact that $\nabla^3 p \in S(\gain_g^{-1}, g)$ and the temperance of the weight $\gain_g^{-1}$, we obtain
\begin{align*}
\abs*{\nabla^{\ell+3} p\left((1 - s)\rho_0 + s \rho\right). (\rho - \rho_0)^3}_{g_{\rho_0}}
	&\le C_g^{\ell} \abs*{\nabla^{\ell+3} p\left((1 - s)\rho_0 + s \rho\right). (\rho - \rho_0)^3}_{g_{(1 - s)\rho_0 + s \rho}} \jap*{s (\rho - \rho_0)}_{g_{\rho_0}^\sympf}^{\ell N_g} \\
	&\le C_g^{\ell} \abs*{\rho - \rho_0}_{g_{(1 - s)\rho_0 + s \rho}}^3 \abs*{\nabla^{\ell+3} p}_{g_{(1 - s)\rho_0 + s \rho}} \jap*{\rho - \rho_0}_{g_{\rho_0}^\sympf}^{\ell N_g} \\
	&\le C_g^{\ell+3} \abs*{\rho - \rho_0}_{g_{\rho_0}}^3 \abs*{\nabla^{\ell+3} p}_{g_{(1 - s)\rho_0 + s \rho}} \jap*{\rho - \rho_0}_{g_{\rho_0}^\sympf}^{(\ell+3) N_g} \\
	&\le C_g^{\ell+3} \jap*{\rho - \rho_0}_{g_{\rho_0}^\sympf}^{(\ell+3) N_g + 3} \abs*{\nabla^{\ell+3} p}_{S(\gain_g^{-1}, g)}^{(0)} \gain_g^{-1}\left( (1 - s)\rho_0 + s \rho \right) \\
	&\le C_g^{\ell+3} C \jap*{\rho - \rho_0}_{g_{\rho_0}^\sympf}^{(\ell+3) N_g + 3 + N} \abs*{\nabla^{\ell+3} p}_{S(\gain_g^{-1}, g)}^{(0)} \gain_g^{-1}\left( \rho_0 \right) .
\end{align*}
The terms of order~$0$, $1$ and~$2$ of the Taylor expansion satisfy the same type of estimate. So in the end we obtain
\begin{equation*}
\abs*{\nabla^\ell p(\rho)}_{g_{\rho_0}^\sympf}
	\le \abs*{\nabla^\ell p(\rho)}_{g_{\rho_0}}
	\le C' \jap*{\rho - \rho_0}_{g_{\rho_0}^\sympf}^{N'} .
\end{equation*}
The constants $C'$ and $N'$ depend on $\ell$ and $\rho_0$. This gives the sought result with ${\sf g} = g_{\rho_0}^\sympf$.
\end{proof}

\subsection{Properties of the temperance weight} \label{subsec:justificationequalities}

Let us now provide a justification of the equalities in~\eqref{eq:deftheta}. The proof is taken from~\cite[(2.3.11)]{Lerner:10}.

\begin{lemma} \label{lem:justificationequalities}
If $g_1$ and $g_2$ are two definite positive quadratic forms on $W$, we have:
\begin{equation*}
\inf_{\zeta \in W \setminus \{0\}} \dfrac{\abs{\zeta}_{g_1^\sympf}}{\abs{\zeta}_{g_2}}
	= \inf_{\zeta \in W \setminus \{0\}} \dfrac{\abs{\zeta}_{g_2^\sympf}}{\abs{\zeta}_{g_1}}
		\qquad \rm{and} \qquad
\sup_{\zeta \in W \setminus \{0\}} \dfrac{\abs{\zeta}_{g_1^\sympf}}{\abs{\zeta}_{g_2}}
	= \sup_{\zeta \in W \setminus \{0\}} \dfrac{\abs{\zeta}_{g_2^\sympf}}{\abs{\zeta}_{g_1}} .
\end{equation*}
\end{lemma}

\begin{proof}
From the Cauchy--Schwarz inequality, we have for any positive definite quadratic form~$g$ on~$W$:
\begin{equation*}
\abs*{\sympf(\zeta_1, \zeta_2)}
	= \abs*{g(g^{-1} \sympf \zeta_1, \zeta_2)}
	\le \abs*{g^{-1} \sympf \zeta_1}_g \abs*{\zeta_2}_g
	= \abs*{\zeta_1}_{g^\sympf} \abs*{\zeta_2}_g ,
		\qquad \forall \zeta_1, \zeta_2 \in W .
\end{equation*}
Therefore, we deduce that for any non-zero $\zeta \in W$:
\begin{equation*}
\dfrac{\abs{\zeta}_{g_1^\sympf}}{\abs{\zeta}_{g_2}}
	= \dfrac{\abs{\zeta}_{g_1^\sympf}^2}{\abs{g_1^{-1} \sympf \zeta}_{g_1} \abs{\zeta}_{g_2}}
	= \dfrac{\abs{\sympf(\zeta, g_1^{-1} \sympf \zeta)}}{\abs{g_1^{-1} \sympf \zeta}_{g_1} \abs{\zeta}_{g_2}}
	\le \dfrac{\abs*{\zeta}_{g_2} \abs*{g_1^{-1} \sympf \zeta}_{g_2^\sympf}}{\abs{g_1^{-1} \sympf \zeta}_{g_1} \abs{\zeta}_{g_2}}
	= \dfrac{\abs*{g_1^{-1} \sympf \zeta}_{g_2^\sympf}}{\abs{g_1^{-1} \sympf \zeta}_{g_1}} .
\end{equation*}
Since the map $\zeta \mapsto g_1^{-1} \sympf \zeta$ is an isomorphism and $g_1, g_2$ play a symmetric role, taking the infimum and the supremum over $\zeta \neq 0$ yields the sought result.
\end{proof}

We now give proofs of technical results introduced in Section~\ref{subsec:metricsonphasespace}.
 
\begin{proof}[Proof of Proposition~\ref{prop:temperanceweight}]
Admissibility of $\theta_g$ follows directly from the admissibility of $g$ (via Proposition~\ref{prop:improvedadmissibility} for instance), with the same slow variation radius.
The fact that $\theta_g$ is comparable to $\theta_g'$~\eqref{eq:temperanceweightscomparable} simply follows from norm equivalence in finite dimension.
\end{proof}

\begin{proof}[Proof of Proposition~\ref{prop:thetag}]
It directly follows from the definition of $\theta_g$ (Definition~\ref{def:thetag}) that for any $\zeta \in W$,
\begin{equation*}
\abs{\zeta}_{g_\rho^\sympf}
	\le \theta_g(\rho) \abs*{\zeta}_{{\sf g}}
	\le \theta_g(\rho) \abs*{\zeta}_{{\sf g}^\sympf}
	\le \theta_g(\rho) \theta_g(\rho_0) \abs*{\zeta}_{g_{\rho_0}} .
\end{equation*}
The second assertion is proved as follows. For any non-zero $\zeta \in W$:
\begin{equation*}
1
	= \dfrac{\abs{\zeta}_g}{\abs{\zeta}_{g^\sympf}} \times \dfrac{\abs{\zeta}_{{\sf g}^\sympf}}{\abs{\zeta}_g} \times \dfrac{\abs{\zeta}_{g^\sympf}}{\abs{\zeta}_{{\sf g}}}
	\le \gain_g \theta_g^2 ,
\end{equation*}
where we took advantage of the fact that ${\sf g} = {\sf g}^\sympf$ in Definition~\ref{def:thetag}.
\end{proof}

\subsection{Slow variation radius} \label{app:slowvariationradius}

To any slowly varying metric $g$, we can associate a number:
\begin{equation} \label{eq:slowvarradius}
R_g
	:= \sup \set{r \in \R_+}{\exists C > 0 : \forall \rho_0 \in T^\star \mfd, \forall \rho \in B_r^g(\rho_0) , \; C^{-2} g_{\rho_0} \le g_\rho \le C^2 g_{\rho_0}} \in (0, + \infty] .
\end{equation}
Positivity of this quantity is ensured by the slow variation property (see Definition~\ref{def:admissiblemetric}). The lemma below is not essential in the present article but it clarifies the importance of the notion of slow variation radius.

\begin{lemma}
If for any $r < R_g$ we set
\begin{equation*}
C_g(r)
	:= \inf \set{C \in \R_+}{\forall \rho_0, \rho \in T^\star \mfd , \; \left(\abs{\rho - \rho_0}_{g_{\rho_0}} \le r \; \Longrightarrow \; C^{-2} g_{\rho_0} \le g_\rho \le C^2 g_{\rho_0} \right)} \in [1, + \infty) ,
\end{equation*}
we have
\begin{equation*}
R_g = + \infty
	\qquad \rm{or} \qquad
C_g(r) \strongto{r \to R_g^-} + \infty .
\end{equation*}
\end{lemma}

\begin{proof}
Assume $R_g < \infty$. Then for any $r \in(0, R_g)$ and $C > C_g(r)$, we pick $\rho_0, \rho \in T^\star \mfd$ such that $r \le \abs*{\rho - \rho_0}_{g_{\rho_0}} \le r (1 + C^{-1})$. Introducing
\begin{equation*}
\rho_1
	= \rho_0 + r \dfrac{\rho - \rho_0}{\abs{\rho - \rho_0}_{g_{\rho_0}}} ,
\end{equation*}
we have $\abs{\rho_1 - \rho_0}_{g_{\rho_0}} = r$ so that
\begin{equation} \label{eq:comparisonrho0}
C^{-2} g_{\rho_0}
	\le g_{\rho_1}
	\le C^2 g_{\rho_0} ,
\end{equation}
which yields in particular
\begin{equation*}
\abs*{\rho - \rho_1}_{g_{\rho_1}}
	\le C \abs*{\rho - \rho_1}_{g_{\rho_0}}
	= C \abs*{\abs*{\rho - \rho_0}_{g_{\rho_0}} - r}
	\le C \left( \dfrac{r}{C} \right)
	= r .
\end{equation*}
Slow variation again implies
\begin{equation*}
C^{-2} g_{\rho_1}	
	\le g_\rho
	\le C^2 g_{\rho_1} ,
\end{equation*}
and therefore, combined with~\eqref{eq:comparisonrho0}, we obtain
\begin{equation*}
C^{-4} g_{\rho_0}
	\le g_\rho
	\le C^4 g_{\rho_0} .
\end{equation*}
This means that $r (1 + C^{-1}) \le R_g$, which can be rewritten as $C \ge (\frac{R_g}{r} - 1)^{-1}$. Taking the infimum in $C$, that leads to 
\begin{equation*}
C_g(r)
	\ge \dfrac{r}{R_g - r} ,
\end{equation*}
which goes to infinity as $r \to R_g^-$.
\end{proof}

\subsection{Improved admissibility} \label{app:improvedadmissibility}

The temperance property introduced in Definition~\ref{def:admissiblemetric} is not homogeneous under multiplication by a conformal factor. This indicates that temperance as introduced in Definition~\ref{def:admissiblemetric} is not sharp. Indeed, it turns out that admissible metrics satisfy a stronger temperance property than the one we stated in Definition~\ref{def:admissiblemetric}. It involves the so-called symplectic intermediate metric.

\begin{lemma}[Symplectic intermediate metric \--- {\cite[Proposition 2.2.20]{Lerner:10}}] \label{lem:gnatural}
Let $g$ be an admissible metric. Then the symplectic intermediate metric $g^\natural$ defined as the harmonic mean of $g$ and $g^\sympf$ (see~\cite[Definition 2.2.19]{Lerner:10}) is admissible and satisfies
\begin{equation*}
g
	\le \gain_g g^\natural
	\le g^\natural
	= \left(g^\natural\right)^\sympf
	\le \gain_g^{-1} g^\natural
	\le g^\sympf .
\end{equation*}
Moreover, any {$g$-admissible} weight is also {$g^\natural$-admissible}.
\end{lemma}

Let us now give the proof of the crucial improved admissibility statement.

\begin{proof}[Proof of Proposition~\ref{prop:improvedadmissibility}]
By~\cite[Proposition 2.2.20]{Lerner:10}, there exist $C > 0$ and $N \ge 0$ such that
\begin{equation*}
\forall \rho_1, \rho_2 \in T^\star \mfd, \qquad
	g_{\rho_2}
		\le C^2 g_{\rho_1} \jap*{\rho_2 - \rho_1}_{(g_{\rho_1}^\natural + g_{\rho_2}^\natural)^\sympf}^{2N} .
\end{equation*}
Let $r = \min\{1, \frac{R_g}{2}, \frac{R_{g^\natural}}{2}\}$, with the notation~\eqref{eq:slowvarradius}. Given $\rho_0, \rho \in T^\star \mfd$, we apply the above inequality to the points $\rho_0' \in \bar B_r^g(\rho_0)$ and $\rho' \in \bar B_r^g(\rho)$ that achieve the minimal distance between the two balls, namely $\abs{\rho' - \rho_0'}_{{\sf g}} = \dist_{{\sf g}}(B_r^g(\rho_0), B_r^g(\rho))$, whatever the flat metric ${\sf g}$ is. Combining this with slow variation of $g$ in each ball yields
\begin{equation} \label{eq:tempimprovedadm}
g_\rho
	\le c_g^2 g_{\rho'}
	\le c_g^2 C^2 \jap*{\rho_0' - \rho'}_{(g_{\rho_0'}^\natural + g_{\rho'}^\natural)^\sympf}^{2N} g_{\rho_0'}
	\le c_g^4 C^2 \jap*{\rho_0' - \rho'}_{(g_{\rho_0'}^\natural + g_{\rho'}^\natural)^\sympf}^{2N} g_{\rho_0} .
\end{equation}
To finish the proof, it suffices to justify that
\begin{equation} \label{eq:slowvargnatural}
\forall \rho_1 \in T^\star \mfd , \rho_2 \in \bar B_r^g(\rho_1) , \qquad
	c_g^{-2} g_{\rho_1}^\natural \le g_{\rho_2}^\natural \le c_g^{2} g_{\rho_1}^\natural ,
\end{equation}
and apply it to $(\rho_1, \rho_2) = (\rho_0, \rho_0')$ and $(\rho, \rho')$.

To establish this, we go back to the definition of $g^\natural$ as the geometric mean of~$g$ and~$g^\sympf$~\eqref{eq:defgnatural}. From slow variation of~$g$ and {$\sympf$-duality}, we have
\begin{equation*}
\forall \rho_1 \in T^\star \mfd , \rho_2 \in \bar B_r^g(\rho_1) , \qquad
\left\{
\begin{aligned}
	c_g^{-2} g_{\rho_1} &\le g_{\rho_2} \le c_g^{2} g_{\rho_1} , \\
	c_g^{-2} g_{\rho_1}^\sympf &\le g_{\rho_2}^\sympf \le c_g^{2} g_{\rho_1}^\sympf ,
\end{aligned}
\right.
\end{equation*}
which yields
\begin{align*}
0
	\le \begin{pmatrix}
g_{\rho_1} & g_{\rho_1}^\natural \\ g_{\rho_1}^\natural & g_{\rho_1}^\sympf
\end{pmatrix}
	&\le c_g^2 \begin{pmatrix}
g_{\rho_2} & \frac{1}{c_g^2} g_{\rho_1}^\natural \\ \frac{1}{c_g^2} g_{\rho_1}^\natural & g_{\rho_2}^\sympf
\end{pmatrix} , \\
0
	\le \begin{pmatrix}
g_{\rho_2} & g_{\rho_2}^\natural \\ g_{\rho_2}^\natural & g_{\rho_2}^\sympf
\end{pmatrix}
	&\le c_g^2 \begin{pmatrix}
g_{\rho_1} & \frac{1}{c_g^2} g_{\rho_2}^\natural \\ \frac{1}{c_g^2} g_{\rho_2}^\natural & g_{\rho_1}^\sympf
\end{pmatrix} .
\end{align*}
By definition of~$g^\natural$, we deduce that
\begin{equation*}
\frac{1}{c_g^2} g_{\rho_1}^\natural
	\le g_{\rho_2}^\natural
		\qquad \textrm{and} \qquad
\frac{1}{c_g^2} g_{\rho_2}^\natural
	\le g_{\rho_1}^\natural ,
\end{equation*}
hence~\eqref{eq:slowvargnatural}. We plug this into~\eqref{eq:tempimprovedadm} to obtain
\begin{equation*}
g_\rho
	\le c_g^4 C^2 c_g^{2N} \jap*{\rho_0' - \rho'}_{(g_{\rho_0}^\natural + g_{\rho}^\natural)^\sympf}^{2N} g_{\rho_0} .
\end{equation*}
Therefore we have the desired result~\eqref{eq:unifadmnatural} for $t = 0$ with $r_g = r$, $C_g = c_g^{2+N} C$ and $N_g = N$, and the conclusion for general $t$ follows by multiplying on both sides by $\e^{2 (\Lambda + 2 \Upsilon) \abs{t}}$.

We proceed similarly for the weight function; see~\cite[Proposition 2.2.20 and Lemma 2.2.25]{Lerner:10}.
\end{proof}

\subsection{Equivalent seminorms on spaces of confined symbols} \label{app:Confseminorms}

The space $\Conf_g^r(\rho_0)$ can be equipped with two distinct families of seminorms giving rise to the same topology, by measuring the size of derivatives either with respect to $g$ or to the constant metric $g_{\rho_0}$.
For a Riemannian metric $\tilde g$, we set
\begin{equation} \label{eq:alternativeseminormsConf}
\abs*{\psi}_{\Conf_g^r(\rho_0), \tilde g}^{(\ell)}
	= \max_{0 \le j \le \ell} \; \sup_{\rho \in T^\star \mfd} \; \sup_{X_1, X_2, \ldots, X_j \in \Gamma(T^\star \mfd)} \dfrac{\abs{\nabla^j \psi (X_1, X_2, \ldots, X_j)}}{\abs{X_1}_{\tilde g} \abs{X_2}_{\tilde g} \cdots \abs{X_j}_{\tilde g}}(\rho) \jap*{\dist_{g_{\rho_0}^\sympf}\left(\rho, B_r^g(\rho_0)\right)}^\ell .
\end{equation}
The difference with the seminorms introduced in Definition~\ref{def:Conf} is that derivatives are measured with respect to the metric $\tilde g$, whereas decay away from $B_r^g(\rho_0)$ is measured with respect to $g$.

\begin{lemma}[Equivalent seminorms on $\Conf_g^r(\rho_0)$] \label{lem:equivalenceseminormsConf}
Let~$g$ be an admissible metric. Then there exists and integer $k$ and a constant $C > 0$ such that for any $r \le r_g$, we have
\begin{equation*}
\forall \ell \in \N, \forall \psi \in \Conf_g^r(\rho_0) , \qquad
\left\{
\begin{aligned}
	\abs*{\psi}_{\Conf_g^r(\rho_0), g}^{(\ell)}
		&\le C^\ell \abs*{\psi}_{\Conf_g^r(\rho_0), g_{\rho_0}}^{(\ell (1 + k))} \\
	\abs*{\psi}_{\Conf_g^r(\rho_0), g_{\rho_0}}^{(\ell)}
		&\le C^\ell \abs*{\psi}_{\Conf_g^r(\rho_0), g}^{(\ell (1 + k))}
\end{aligned}
\right. .
\end{equation*}
\end{lemma}

\begin{proof}
The proof relies on the improved admissibility property of Proposition~\ref{prop:improvedadmissibility}, which yields in particular:
\begin{equation*}
g_\rho
	\le C_g^2 \jap*{\dist_{g_{\rho_0}^\sympf}\left( \rho, B_{r_g}^g(\rho_0) \right)}^{2N_g} g_{\rho_0} ,
\end{equation*}
and the same exchanging $g_\rho$ and $g_{\rho_0}$. Plugging this in the denominator of~\eqref{eq:alternativeseminormsConf} gives the desired estimate with $C = C_g$ and $k = \lceil N \rceil$.
\end{proof}

\Large
\section{Faà di Bruno formula} \label{app:faadibruno}
\normalsize

We recall and comment the classical Faà di Bruno formula, that allows to compute derivatives of a composition of two functions. Our presentation is similar to \cite[Section 4.3.1]{Lerner:10}. We redo the proofs as we consider vector-valued functions, and not only scalar function. We use boldface letters for multi-indices, for instance $\bf{n} = (n_1, n_2, \ldots, n_j) \in \N^j$. The length of such a multi-index is
\begin{equation*}
\abs{\bf{n}}
	:= n_1 + n_2 + \cdots + n_j .
\end{equation*}
For any smooth function $f$, we also use the shorthand
\begin{equation} \label{eq:compactnotation}
\nabla^{\bf{n}} f. (X^{\bf{n}})
	:= \left(\nabla^{n_1} f.(X^{n_1}), \nabla^{n_2} f.(X^{n_2}), \ldots, \nabla^{n_j} f.(X^{n_j})\right) ,
\end{equation}
as well as
\begin{equation*}
\bf{n} !
	:= n_1! \, n_2! \, \cdots \, n_j! .
\end{equation*}

\begin{lemma}[Faà di Bruno formula] \label{lem:FaadiBruno}
Let $H : T^\star M \to T(T^\star M)$ be a vector field and $\phi : T^\star M \to T^\star M$ a smooth function. Then for any vector field $X$ on $T^\star M$ and any integer $\ell \ge 1$, we have for all $\rho \in T^\star M$:
\begin{equation} \label{eq:statementfdb}
\dfrac{1}{\ell!} \nabla^\ell (H \circ \phi)_\rho. (X^\ell)
		= \sum_{j = 1}^\ell \; \sum_{\substack{\bf{n} \in (\N^\ast)^j \\ \abs{\bf{n}} = \ell}} \dfrac{1}{j!} (\nabla^j H)_{\phi(\rho)}.\left( \dfrac{1}{\bf{n}!} \nabla^{\bf{n}} \phi.(X^{\bf{n}}) \right) .
\end{equation}
\end{lemma}

\begin{remark}[Polarization] \label{rmk:polarization}
Notice that this lemma allows to compute only diagonal elements of the symmetric tensor $\nabla^\ell(H \circ \phi)$. This is not a problem since any tuple $(X_1, X_2, \ldots, X_\ell)$ can be written as a sum of tuples having all the same components (this is nothing but polarization of multilinear forms).
\end{remark}

\begin{remark}[Combinatorics of the Faà di Bruno formula] \label{rk:numberoftermsfaadibruno}
The number of terms in the right-hand side of~\eqref{eq:statementfdb} is $2^{\ell - 1}$. To check this, one can count how many tuples $(n_1, n_2, \ldots, n_j)$ of positive integers have length $\ell$. Considering such a tuple is equivalent to the data of the $j-1$ partial sums $n_1 + n_2 + \cdots + n_k$, $k \in \{ 1, 2, \ldots, j - 1 \}$, namely an increasing sequence of~$j-1$ numbers ranging from $1$ to $\ell - 1$. There are $\binom{\ell-1}{j-1}$ such partitions. Summing over $j = 1, 2, \ldots, \ell$, one recovers that there are indeed $2^{\ell-1}$ terms in the formula.
\end{remark}

\begin{proof}
We establish the equality by induction. For $\ell = 1$, this is nothing but the chain rule. Suppose the formula is true for some $\ell \ge 1$. We differentiate once again, using the Leibniz formula:
\begin{multline} \label{eq:FdBtemp}
\dfrac{1}{\ell!} \nabla^{\ell+1} (H \circ \phi)_\rho. (X^{\ell+1})
	= \sum_{j = 1}^\ell \; \sum_{n_1 + n_2 + \cdots + n_j = \ell} \\
		\dfrac{1}{j!} (\nabla^{j+1} H)_{\phi(\rho)}. \dfrac{1}{{\bf n}!}\left(  \nabla^{n_1} \phi.(X^{n_1}), \nabla^{n_2} \phi.(X^{n_2}), \cdots, \nabla^{n_j} \phi.(X^{n_j}), \nabla \phi. X \right) \\
		+ \dfrac{1}{j!} (\nabla^j H)_{\phi(\rho)}.\left( \dfrac{1}{n_1!} \nabla^{n_1+1} \phi.(X^{n_1+1}), \dfrac{1}{n_2!} \nabla^{n_2} \phi.(X^{n_2}), \cdots, \dfrac{1}{n_j!} \nabla^{n_j} \phi.(X^{n_j}) \right) \\
		+ \dfrac{1}{j!} (\nabla^j H)_{\phi(\rho)}.\left( \dfrac{1}{n_1!} \nabla^{n_1} \phi.(X^{n_1}), \dfrac{1}{n_2!} \nabla^{n_2+1} \phi.(X^{n_2+1}), \cdots, \dfrac{1}{n_j!} \nabla^{n_j} \phi.(X^{n_j}) \right) \\
		+ \cdots
		+ \dfrac{1}{j!} (\nabla^j H)_{\phi(\rho)}.\left( \dfrac{1}{n_1!} \nabla^{n_1} \phi.(X^{n_1}), \dfrac{1}{n_2!} \nabla^{n_2} \phi.(X^{n_2}), \cdots, \dfrac{1}{n_j!} \nabla^{n_j+1} \phi.(X^{n_j+1}) \right) .
\end{multline}
The first term appears when the derivative hits $(\nabla^j H)_{\phi(\rho)}$, thus leading to an extra $\nabla \phi.X$ as input vector of the derivative of $H$ by the chain rule, while the other terms correspond to the derivative landing on each derivative of $\phi$, according to the Leibniz rule. We symmetrize the term coming from the chain rule:
\begin{multline*}
\dfrac{1}{j!} (\nabla^{j+1} H)_{\phi(\rho)}. \dfrac{1}{{\bf n}!}\left( \nabla^{n_1} \phi.(X^{n_1}), \nabla^{n_2} \phi.(X^{n_2}), \cdots, \nabla^{n_j} \phi.(X^{n_j}), \nabla \phi. X \right) \\
	= \sum_{k = 1}^{j+1} \dfrac{1}{(j+1)!} (\nabla^{j+1} H)_{\phi(\rho)}. \dfrac{1}{{\bf n}!}\left( \nabla^{n_1} \phi.(X^{n_1}), \nabla^{n_2} \phi.(X^{n_2}), \cdots, \nabla \phi.X, \cdots, \nabla^{n_j} \phi.(X^{n_j})\right) ,
\end{multline*}
where $\nabla \phi.X$ is inserted in the $k$-th slot. Here we used crucially that successive derivatives are symmetric tensors, due to the fact that the connection~$\nabla$ has vanishing torsion and curvature.
It remains to group the terms that can be represented by the same partition $n_1' + n_2' + \cdots + n_j' = \ell + 1$, $j \in \{1, 2, \ldots, \ell+1\}$.
We have two types of terms in~\eqref{eq:FdBtemp} that can be represented by the partition $n_1' + n_2' + \cdots + n_j' = \ell + 1$:
\begin{itemize}
\item those obtained from the Leibniz rule, corresponding to a partition $n_1 + n_2 + \cdots + n_j = \ell$ at step~$\ell$, so that
\begin{equation*}
n_{k_0}' = n_{k_0} + 1 ,
	\qquad
n_k' = n_k , \; \forall k \neq k_0
\end{equation*}
for some~$k_0$;
\item those obtained from the chain rule, corresponding to a partition $n_1 + n_2 + \cdots + n_{j-1} = \ell$ at step~$\ell$, so that
\begin{equation*}
n_{k_0}' = 1 ,
	\qquad
n_k' = n_k , \; \forall 1 \le k < k_0 ,
	\qquad
n_k' = n_{k-1} , \; \forall k_0 < k \le j
\end{equation*}
for some~$k_0$;
\end{itemize}
Thus we must count how many partitions of the form $n_1 + n_2 + \cdots + n_j = \ell$ and $n_1 + n_2 + \cdots + n_{j-1} = \ell$ in~\eqref{eq:FdBtemp} lead to a fixed partition $n_1' + n_2' + \cdots + n_j' = \ell + 1$ through this process. For any $k_0 \in \{1, 2, \ldots, j\}$, we look at how one can produce a new partition at step $\ell + 1$ from a partition at step $\ell$, by changing the {$k_0$-th} slot. Either $n_{k_0}' = 1$ and we can build our partition at step $\ell + 1$ from the same partition with $j-1$ slots, without $n_{k_0}'$; or $n_{k_0}' \ge 2$ and we can build our new partition from the same partition with $j$ slots, with $n_{k_0} = n_{k_0}' - 1 \ge 1$ in place of $n_{k_0}'$. In the latter case, we have to take into account an extra factor $n_{k_0}'$ in order to have the correct factorial in front of $\nabla^{n_{k_0}'} \phi$ (we rewrite $\frac{1}{n_{k_0}!} \nabla^{n_{k_0} + 1} \phi = n_{k_0}' \frac{1}{n_{k_0}'!} \nabla^{n_{k_0}'} \phi$ in~\eqref{eq:FdBtemp}). Denoting by $m$ the number of indices $k_0$ such that $n_{k_0}' = 1$, the number of terms in~\eqref{eq:FdBtemp} that can be represented by the partition $\mathbf{n}' = (n_1', n_2', \ldots, n_j')$ is
\begin{equation*}
m + \sum_{n_k' \ge 2} n_k'
	= n_1' + n_2' + \cdots + n_j'
	= \ell + 1 .
\end{equation*}
Therefore, after grouping terms in~\eqref{eq:FdBtemp} according to partitions $\mathbf{n}' \in (\N^\ast)^j$ of length $\ell+1$, the factor in front of $(\nabla^j H)_{\phi(\rho)}.\left( \dfrac{1}{\bf{n}'!} \nabla^{\bf{n}'} \phi.(X^{\bf{n}}) \right)$ is $\frac{\ell+1}{j!}$. Dividing by~$\ell + 1$ in~\eqref{eq:FdBtemp} then yields the desired result.
\end{proof}


\bibliographystyle{alpha}
{\small
\bibliography{biblio}}

\normalsize

\end{document}